\documentclass[1 [leqno,11pt]{amsart}
\usepackage{amssymb, amsmath, latexsym, amssymb, amsfonts, amsbsy, amsthm, mathtools, graphicx, CJKutf8, CJKnumb, CJKulem, color}
\usepackage{comment}
\usepackage{tikz}
\tikzstyle{startstop} = [rectangle, rounded corners, minimum width=1cm, minimum height=1cm,text centered, draw=black]
\tikzstyle{io} = [rectangle, rounded corners, minimum width=1cm, minimum height=1cm,text centered, draw=black]
\tikzstyle{process} = [rectangle, rounded corners, minimum width=1cm, minimum height=1cm,text centered, draw=black]
\tikzstyle{decision} = [rectangle, rounded corners, minimum width=1cm, minimum height=1cm,text centered, draw=black]
\tikzstyle{explain} = [minimum width=1cm, minimum height=0.5cm,text centered, draw=black]
\tikzstyle{equations} = []
\tikzstyle{arrow} = [thick,->,>=stealth]
\tikzstyle{d-arrow} = [thick,->,dashed,>=stealth]
\usetikzlibrary{shapes.geometric, arrows, plotmarks, backgrounds}
\usepackage{pgfplots}
\usepackage{tikz}
\usepackage{hyperref}
\usepackage{cancel}
\numberwithin{equation}{section}

\allowdisplaybreaks

\let\al=\alpha
\let\b=\beta

\let\d=\delta

\let\ep=\epsilon

\let\la=\lambda

\let\s=\sigma
\let\f=\frac

\let\om=\omega

\let\Om=\Omega

\let\na=\nabla
\let\th=\theta
\let\pa=\partial

\def\mfq{\mathfrak{q}}
\def\mfe{\mathfrak{e}}
\def\mfs{\mathfrak{s}}

\def\bbD{\mathbb{D}}
\def\bfD{\mathbf{D}}

\def\bfR{\mathbf{R}}

\def\tom{\tilde{\omega}}

\def\tchi{\tilde{\chi}}

\usepackage{scalerel,stackengine}
\stackMath
\newcommand\reallywidehat[1]{%
\savestack{\tmpbox}{\stretchto{%
  \scaleto{%
    \scalerel*[\widthof{\ensuremath{#1}}]{\kern-.6pt\bigwedge\kern-.6pt}%
    {\rule[-\textheight/2]{1ex}{\textheight}}
  }{\textheight}%
}{0.5ex}}%
\stackon[1pt]{#1}{\tmpbox}%
}

\def\rmA{\mathrm{A}}
\def\rmB{\mathrm{B}}

\def\rmD{\mathrm{D}}
\def\rmE{\mathrm{E}}
\def\rmH{\mathrm{H}}
\def\rmJ{\mathrm{J}}

\def\rmR{\mathrm{R}}
\def\rmK{\mathrm{K}}
\def\rmL{\mathrm{L}}
\def\rmM{\mathrm{M}}
\def\rmN{\mathrm{N}}
\def\rmT{\mathrm{T}}

\def\rmU{\mathrm{U}}

\def\R{\mathbb R}

\def\no{\noindent}

\def\dive{\mathop{\rm div}\nolimits}

\def\eqdef{\buildrel\hbox{\footnotesize def}\over =}

\def\bbT{\mathbb{T}}

\newcommand{\beq}{\begin{equation}}
\newcommand{\eeq}{\end{equation}}
\newcommand{\ben}{\begin{eqnarray}}
\newcommand{\een}{\end{eqnarray}}
\newcommand{\beno}{\begin{eqnarray*}}
\newcommand{\eeno}{\end{eqnarray*}}
\newcommand{\udl}[1]{\underline{#1}}

\newtheorem{theorem}{Theorem}[section]

\newtheorem{lemma}[theorem]{Lemma}
\newtheorem{proposition}[theorem]{Proposition}
\newtheorem{corol}[theorem]{Corollary}
\newtheorem{remark}[theorem]{Remark}

\begin{document}
\begin{CJK*}{UTF8}{gkai}
\title[Nonlinear inviscid damping]{Inviscid damping of monotone shear flows for 2D inhomogeneous Euler equation with non-constant density in a finite channel}
\author{Weiren Zhao}
\address[W. Zhao]{Department of Mathematics, New York University Abu Dhabi, Saadiyat Island, P.O. Box 129188, Abu Dhabi, United Arab Emirates.}
\email{zjzjzwr@126.com, wz19@nyu.edu}
\maketitle

\begin{abstract}
We prove the nonlinear inviscid damping for a class of monotone shear flows with non-constant background density for the two-dimensional ideal inhomogeneous fluids in $\mathbb{T}\times [0,1]$ when the initial perturbation is in Gevrey-$\f{1}{s}$ ($\f12<s<1$) class with compact support. 
\end{abstract}
\setcounter{tocdepth}{1}
{\small\tableofcontents}
\section{Introduction}
We consider the two-dimensional inhomogeneous incompressible Euler system in a finite channel $\mathbb{T}\times [0,1]$:
\beq\label{eq:inhomincompEuler-1}
\left\{
\begin{aligned}
&\pa_t\rho+{\bf v}\cdot\na \rho=0,\\
&\rho(\pa_t {\bf v}+{\bf v}\cdot\na {\bf v})+\na {\bf P}=0,\\
&\dive\, {\bf v}=0
\end{aligned}
\right.
\eeq
where $\rho>0$ is the density, ${\bf v}$ is the velocity and ${\bf P}$ is the pressure. By introducing $m=\rho^{-1}$, we have 
\beq\label{eq:inhomincompEuler}
\left\{
\begin{aligned}
&\pa_tm+{\bf v}\cdot\na m=0,\\
&\pa_t {\bf v}+{\bf v}\cdot\na {\bf v}+m\na {\bf P}=0,\\
&\dive\, {\bf v}=0.
\end{aligned}
\right.
\eeq
The system has a nontrivial steady state:
\beno
{\bf v}_s=(u(y),0)\quad m_s=\th(y)\quad {\bf P}_s=\text{cons.}.
\eeno
In this paper, we study the global stability of this steady state. It is natural to introduce the perturbation: ${\bf v}=(u(y),0)+U$, $m=\th(y)+d(t,x,y)$ and ${\bf P}=\text{cons.}+P$. We also introduce the perturbed vorticity $\om=\na \times U$ and the associated stream function $\psi$. Then we obtain the equations of perturbation $(U,d,\om,\psi,P)$:
\begin{subequations}\label{eq:vorticity}
\beq
\left\{
\begin{aligned}
&\pa_t U^x+u(y)\pa_xU^x+u'(y)\pa_x\psi+U\cdot \na U^x+\th\pa_xP+d\pa_xP=0\\
&\pa_tU^y+u(y)\pa_xU^y+U\cdot \na U^y+\th\pa_yP+d\pa_yP=0,
\end{aligned}
\right.
\eeq
and
\beq
\left\{
\begin{aligned}
&\pa_t\om+u(y)\pa_x\om-u''(y)\pa_x\psi-\th'(y)\pa_xP+\pa_yP\pa_xd-\pa_xP\pa_yd+U\cdot\nabla \om=0,\\
&\pa_t d+u(y)\pa_xd+\th'(y)\pa_x\psi+U\cdot \na d=0,
\end{aligned}
\right.
\eeq
where 
\beq
\left\{
\begin{aligned}
&U=\na^{\bot}\psi, \quad \Delta\psi=\om,\quad 
\dive\big((\th+d)\na P\big)=-\dive (u\cdot \na u)-2u'(y)\pa_{xx}\psi,\\
&\psi(t,x,0)=\psi(t,x,1)=0,\quad \pa_yP(t,x,0)=\pa_yP(t,x,1)=0,\\
&\om|_{t=0}=\om_{in}(x,y),\quad d|_{t=0}=d_{in}(x,y).
\end{aligned}
\right.
\eeq
\end{subequations}
In the homogeneous density case $\th=1, d=0$, \eqref{eq:vorticity} is the Euler equation around a shear flow $(u(y),0)$. In \cite{Orr1907}, Orr observed that for the Couette flow case $u(y)=y$, the velocity will tend to a shear flow as $t\to \infty$. This phenomenon is called inviscid damping. 
Bedrossian and Masmoudi \cite{BM2015} proved nonlinear inviscid damping around the Couette flow in Gevrey-$m$  class ($1\leq m<2$). 
Deng and Masmoudi \cite{DM2018}  proved some instability for initial perturbations in the Gevrey-$m$ class ($m>2$). 
We refer to \cite{IonescuJia2020cmp, IonescuJia2021} and references therein for other related interesting results.
For the monotone shear flow setting, Case \cite{Case1960} predicted the $t^{-1}$ decay for the velocity. It is later proved by Wei, Zhang, and Zhao \cite{WeiZhangZhao2018}. We also refer to \cite{Zillinger2017, Jia2020siam, Jia2020arma}. The nonlinear inviscid damping for stable monotone shear flow was proved by Ionescu-Jia \cite{IJ2020}, and Masmoudi-Zhao \cite{MasmoudiZhao2020}. 
For non-monotone flows such as the Poiseuille flow and the Kolmogorov flow, another dynamic phenomenon should be taken into consideration, which is the so-called vorticity depletion, predicted by Bouchet and Morita \cite{BM2010} and later proved in \cite{WeiZhangZhao2019, WeiZhangZhao2020, ionescuSameerJia2022linear}. See also \cite{BCV2017, LMZZ2021, IonescuJia2021} for similar depletion in various systems. For the inhomogeneous Euler equation, very recently, Chen, Wei, Zhang, and Zhang \cite{ChenWeiZhangZhang2023} proved the nonlinear inviscid damping for the Couette flow with the density near a constant, namely, $\th=1$. 

Our main result is
\begin{theorem}\label{Thm: main}
Suppose $u(y), \th(y)$ are smooth functions defined on $[0,1]$ which satisfy: 
\begin{enumerate}
\item (Monotone) There exists $c_0>0$ such that $u'\geq c_0>0$.
\item (Positive density) There exists $c_0>0$ such that $\th\geq c_0>0$. 
\item (Compact support) There exists $\kappa_0\in (0,\f{1}{10}]$ such that $\mathrm{supp}\, u''\subset [4\kappa_0,1-4\kappa_0]$ and $\mathrm{supp}\, \th'\subset [4\kappa_0,1-4\kappa_0]$. 
\item (Linear stability) The distorted Rayleigh operator $u\mathrm{Id}-\left(\f{u'}{\th}\right)'\tilde{\Delta}^{-1}$ has no eigenvalue and no embedded eigenvalue. 
\item (Regularity) There exist $K>1$ and $s_0\in (0,1)$ such that for all integers $m\geq 0$,
\beno
\sup_{y\in\R}\left|\f{d^m(\th'(y))}{dy^m}\right|+\sup_{y\in\R}\left|\f{d^m(u''(y))}{dy^m}\right|\leq K^m(m!)^{\f{1}{s_0}}(m+1)^{-2}.
\eeno
\end{enumerate}
For all $1>s_0\geq s> 1/2$ and $\la_{in}>0$, there exist $\la_{in}>\la_{\infty}=\la_{\infty}(\la_{in}, K,\kappa_0,s)>0$ and $0<\ep_0=\ep_0(\la_{in},\la_{\infty},\kappa_0,s)\leq \f12$ such that for all $\ep\leq \ep_0$ if $\om_{in},\ d_{in}$ have compact support in $\mathbb{T}\times [3\kappa_0,1-3\kappa_0]$ and satisfy
\beno
\|\om_{in}\|_{\mathcal{G}^{s, \la_{in}}}^2=\sum_{k}\int\left|\widehat{\om}_{in}(k,\eta)\right|^2e^{2\la_{in}\langle k,\eta\rangle^s}d\eta\leq \ep^2,
\quad \int_{\mathbb{T}\times [0,1]}\om_{in}(x,y)dxdy=0
\eeno
and
\beno
\|d_{in}\|_{\mathcal{G}^{s,\la_{in}}}^2=\sum_{k}\int\left|\widehat{d}_{in}(k,\eta)\right|^2e^{2\la_{in}\langle k,\eta\rangle^s}d\eta\leq \ep^2,
\quad \int_{\mathbb{T}\times [0,1]}d_{in}(x,y)dxdy=0
\eeno
then the smooth solution $\om(t)$ and $d(t)$ to \eqref{eq:vorticity} satisfy: 
\begin{itemize}
\item[1.] (Compact support) For all $t\geq 0$, $\mathrm{supp}\, \om(t)\subset \mathbb{T}\times [2\kappa_0,1-2\kappa_0]$ and $\mathrm{supp}\, d(t)\subset \mathbb{T}\times [2\kappa_0,1-2\kappa_0]$. 
\item[2.] (Scattering) There exist $f_{\infty}, d_{\infty}\in \mathcal{G}^{s,\la_{\infty}}$ with $\mathrm{supp}\, f_{\infty}\subset \mathbb{T}\times [2\kappa_0,1-2\kappa_0]$ and $\mathrm{supp}\, d_{\infty}\subset \mathbb{T}\times [2\kappa_0,1-2\kappa_0]$ such that for all $t\geq 0$, 
\beq\label{eq: Scattering om} 
\left\|\om(t,x+tu(y)+\Phi(t,y),y)-f_{\infty}(x,y)\right\|_{\mathcal{G}^{s, \la_{\infty}}}\lesssim \f{\ep}{\langle t\rangle},
\eeq
and
\beq\label{eq: Scattering den} 
\left\|d(t,x+tu(y)+\Phi(t,y),y)-d_{\infty}(x,y)\right\|_{\mathcal{G}^{s,\la_{\infty}}}\lesssim \f{\ep}{\langle t\rangle},
\eeq
where $\Phi(t,y)$ is given explicitly by 
\beq\label{eq:Phi(t,y)}
\Phi(t,y)=\f{1}{2\pi}\int_0^t\int_{\mathbb{T}}U^x(\tau,x,y)dxd\tau=u_{\infty}(y)t+O(\ep^2),
\eeq
with $u_{\infty}(y)=-\f{1}{2\pi}\pa_y\int_{\mathbb{T}}\Delta^{-1}f_{\infty}(x, y)dx$. 
\item[3.] (Inviscid damping) 
The velocity field $U$ satisfies
\begin{align}\label{eq: inviscid damping}
\left\|\f{1}{2\pi}\int U^x(t,x,\cdot)dx-u_{\infty}\right\|_{\mathcal{G}^{s,\la_{\infty}}}&\lesssim \f{\ep^2}{\langle t\rangle^2},\\
\left\|U^x(t)-\f{1}{2\pi}\int U^x(t,x,\cdot)dx\right\|_{L^2}&\lesssim \f{\ep}{\langle t\rangle},\\
\big\|U^y(t)\big\|_{L^2}&\lesssim \f{\ep}{\langle t\rangle^2}.
\end{align}
\end{itemize}
\end{theorem}
We list some remarks here. 

\begin{enumerate}
\item By modifying the Fourier multiplier, one can prove the asymptotic stability in the Gevrey-$2$ class. 
\item The compact support assumptions on $u''$, $\th'$ and $\om_{in}, d_{in}$ seem necessary, which ensures the Fourier analysis works in the estimate. 
\item In the proof, we construct a wave operator to absorb the nonlocal term. Such kind of operators appears in the study of dispersive equations with potential, (also called distorted Fourier transform, see \cite{DelortMasmoudi, GP20,  LLS21, GPR18, Schlag, Yajima} for more details. Let us also mention that recently, the wave operator was successfully used to solve important problems in fluid mechanics.
In \cite{LWZ}, the authors use the wave operator method to solve Gallay's conjecture on pseudospectral and spectral bounds of the Oseen vortices operator. In \cite{WeiZhangZhao2020}, the wave operator method was used to solve Beck and Wayne's conjecture. In \cite{MasmoudiZhao2020}, the authors use the wave operator method to prove the nonlinear inviscid damping for stable monotone shear flows. 
\item In \cite{ChenWeiZhangZhang2023}, the authors introduce a well-constructed weight $\rmA^*$ which captures the important difference of regularities between the vorticity and density. There are other choices. 
The weight introduced in \cite{ChenWeiZhangZhang2023} is constructed in a clever way. 
We use the same weight in this paper.
It helps us to simplify the proof and reduce more than 20 pages of this paper.
\item It is quite interesting to study the stability threshold problem for the 2D inhomogeneous incompressible Navier-Stokes equations at a high Reynolds number. For the Navier-Stokes equations, we refer to \cite{BGM2015, BGM2017, BGM2020, BMV2016, BVW2018, CLWZ2020, ChenWeiZhang2020,  LMZ-G2022, LMZ2022, MasmoudiZhao2020cpde, MasmoudiZhao2019} for the recent progress. 
\end{enumerate}

\subsection{Main idea}
The paper contains mainly two parts. The first part, from section \ref{sec: step1} to section \ref{sec: good system}, is the more important part, where we do the reduction. In this part, we introduce a good unknown, pressure decomposition, nonlinear change of coordinates, and the wave operator to deduce a good working PDE system. The following map (\ref{map}) shows the main idea of reaching a good working system. In each step, we ensure that the system is better for nonlinear energy estimates. In the second part, we close the energy estimate by using time-dependent Fourier multipliers which capture the growth of the solution at each frequency and each time. 

\begin{figure*}[tbp]
\caption{The map of main ideas}
\label{map}

\medskip

{\small
\begin{tikzpicture}[node distance=4cm]
\node (start) [startstop]{$\pa_t\om+u(y)\pa_x\om-u''(y)\pa_x\psi -\boxed{\th'(y)\pa_xP}+...=0$};
\node (exp1) [explain, below of=start, yshift=2.5cm] {Step 1: good unknown $\tilde{\om}=\pa_y\left(\f{1}{\th}\pa_y\psi\right)+\f{1}{\th}\pa_{xx}\psi\eqdef\tilde{\Delta}\psi$};
\node (io1) [io, below of=exp1, yshift=2.5cm] {$\dive\big((\th+d)\na P\big)=\boxed{-\dive (U\cdot \na U)-2u'(y)\pa_{xx}\psi}$};
\node (exp2) [explain, below of=io1, yshift=2.5cm] {Step 2: pressure decomposition};
\node (io7) [io, below of=exp2, yshift=2.5cm] {$\pa_tU^y+u(y)\pa_xU^y+U\cdot \na U^y+\th\pa_yP+d\pa_yP=0$};
\node (exp9) [explain, below of=io7, yshift=2.5cm] {zero mode pressure};
\node (io2) [io, below of=exp9, yshift=2.5cm] {$\pa_t\tilde{\om}+\boxed{u(y)\pa_x\tilde{\om}+P_0(U^x)\pa_x\tilde{\om}}-\left(\f{u'}{\th}\right)' U^y
+...=0$};
\node (exp3) [explain, below of=io2, yshift=2.5cm] {Step 3: change of coordinates $(t, x, y)\to (t, z, v)$};
\node (io3) [io, below of=exp3, yshift=2.5cm] {$\pa_t\Om-\boxed{\udl{\varphi_1}\pa_zP_{\neq}\Psi}+\mathrm{U}\cdot\na_{z,v}\Om+\mathcal{N}_{\Om}[\Psi]+\mathcal{N}_a[\Pi]
=0$};
\node (exp4) [explain, below of=io3, yshift=2.5cm] {Step 4: wave operator};
\node (exp6) [explain, right of=exp4, yshift=1cm, xshift=4cm] {inverse change of coordinates};
\node (io6) [io, below of=exp6, yshift=2.5cm] {$\pa_t\tilde{\om}+\pa_x\mathcal{R}\tilde{\om}=0$};
\node (io4) [io, below of=exp4, yshift=2.5cm] {The working system \eqref{eq: main system}};
\node (exp5) [explain, right of=exp1, yshift=0cm, xshift=4cm] {distorted elliptic equation}; 
\node (exp7) [explain, below of=exp5, yshift=2.5cm] {Sturm-Liouville Equation};
\node (io5) [io, right of=exp2, yshift=0cm, xshift=4cm] {
$\dive\big((\th+d)\na P_{l,2}\big)=(\chi_2''u)(y)\boxed{\pa_{y}\psi_{\neq}}$}; 
\node (exp8) [explain, below of=io5, yshift=2.5cm, xshift=0cm] {elliptic commutator}; 
\draw [arrow](start) -- (exp1);
\draw [arrow](exp1) -- (exp5);
\draw [arrow](exp5) -- (exp7);
\draw [arrow](exp2) -- (io5);
\draw [arrow](io5) -- (exp8);
\draw [arrow](exp8) -- (io2);
\draw [arrow](exp7) -- (io1);
\draw [arrow](io1) -- (exp5);
\draw [arrow](exp1) -- (io1);
\draw [arrow](io1) -- (exp2);
\draw [arrow](exp2) -- (io7);
\draw [arrow](io7) -- (exp9);
\draw [arrow](exp9) -- (io2);
\draw [arrow](io2) -- (exp3);
\draw [arrow](exp3) -- (io3);
\draw [arrow](io3) -- (exp4);
\draw [arrow](exp4) -- (io4);
\draw [arrow](io6) -- (exp4);
\draw [arrow](exp6) -- (io6);
\draw [arrow](exp4) -- (exp6);
\end{tikzpicture}

}
\end{figure*}
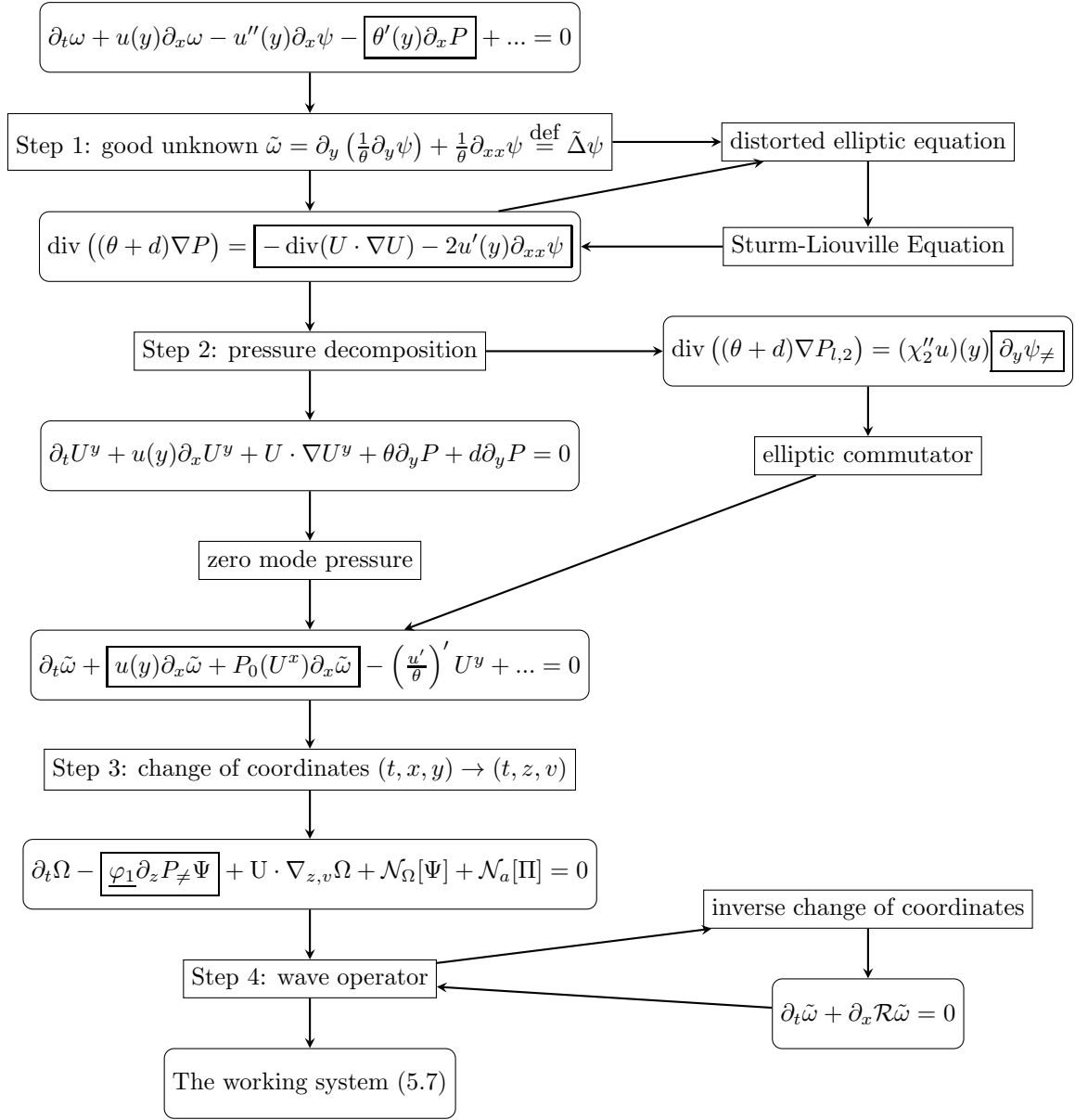

We first introduce a good unknown to get rid of the linear pressure term $\th'(y)\pa_xP$, see section \ref{sec: step1} for more details. This offers us a better equation (see equation \eqref{eq: linearsystem}), at least at the linear level. One can easily compare the linearized equation \eqref{eq: linearsystem} of the good unknown with the classical linearized Euler equation and obtain linear (in)stability results parallelly. Due to the new good unknown, we have to study a distorted elliptic equation to recover the stream function, even at $(t, x, y)$ coordinate. Moreover, the equation for the pressure is also distorted by the background density. In Appendix \ref{sec: ST equ}, we study the Sturm-Liouville equation and obtain the decay estimate of the Fourier transform of the corresponding Green function, see Proposition \ref{eq: T_1} and Proposition \ref{eq: T_2}, which will be used in the elliptic estimates for the stream function and pressure. This part is new. Roughly speaking, after being distorted by the density and nonlinear change of coordinates, the main part of the elliptic operator still behaves as $\pa_z^2+(\pa_v-t\pa_z)^2$, namely, the Fourier transform of the Green function $\mathcal{G}$ satisfies
\beno
|\mathcal{G}(t, k, \xi, \eta)|\leq C\min\left\{\f{e^{-\la_{\Delta}\langle \xi-\eta\rangle^s}}{1+k^2+(\xi-kt)^2}, \f{e^{-\la_{\Delta}\langle \xi-\eta\rangle^s}}{1+k^2+(\eta-kt)^2}\right\}.
\eeno
Note that the denominator $k^2+(\xi-kt)^2$ is the Fourier transform of $-\pa_z^2-(\pa_v-t\pa_z)^2$. 

The pressure term is one of the main difficulties: 1, the right-hand side of the pressure equation \eqref{eq: new system} is not compactly supported; 2, the pressure equation is an elliptic equation with Neumann boundary conditions, which is not friendly to the zero mode, namely, the compatibility condition should be kept in the decomposition. The second step is pressure decomposition, which is mainly because of the finite channel setting. Instead of solving the elliptic equation, we obtain the zero mode of the pressure from the vertical velocity equation. See \eqref{eq: pressure-zero-mode} for the equation and section \ref{sec: pressure-zero} for detailed estimates. For the nonzero modes, we first divide the pressure into the linear part $P_l$ and the nonlinear part $P_n$. We divide the pressure into 6 terms,
\begin{align*}
P=P_{l,1}+P_{l,2}+P_{l, b}+P_{n,1}+P_{n,2}+P_{n, b}.
\end{align*}
Now we introduce the main idea of decomposition. For the nonlinear part $P_n$, we consider the steady case, namely 
\beno
-(\th+d)\pa_xP_n=U^x\pa_xU^x+U^y\pa_yU^x=\f12\pa_x(U^x)^2+\pa_x\psi(\Om+\pa_x\pa_x\psi).
\eeno
Noticing that near the boundary $d+\th$ is constant, $\Om=0$, we then have 
\beno
P_n=-\f{((\pa_x\psi)^2+(\pa_y\psi)^2)+C}{2(\th+d)}
\eeno 
near the boundary. We use this idea to construct two boundary correction terms $P_{l,b}$ and $P_{n,b}$ to get rid of the boundary value on the right-hand side of the pressure equation and make them compactly supported. The same idea also works for the decomposition of the linear part, see section \ref{sec: step1} for more details. There is a derivative loss after the decomposition, which does not appear in the whole space $\mathbb{T}\times \mathbb{R}$ setting, discussed in \cite{ChenWeiZhangZhang2023}. For example, for the linear part, before the decomposition, one can formally regard $\Delta P_l\approx -\pa_{xx}\psi$. After the decomposition, there is $\chi_2''\pa_y\psi$ on the right-hand side of the equation of $P_{l,2}$, see \eqref{eq: P_l,2}. Note that the derivative loss in $y$ is bad because it will lead to growth. Luckily, $\chi_2''$ vanishes on the support of $\th', d, \chi_1$. So we first study $\Delta P_{l,2}\approx \chi_2''\pa_y\psi$ and control $\langle \pa_x,\pa_y\rangle P_{l,2}$ by $\psi$. Then we study $\Delta (P_{l,2}\chi_1)\approx 2\chi_1'\pa_yP_{l,2}$ and control $\langle \pa_x,\pa_y\rangle^2 P_{l,2}\chi_1$ by $\langle \pa_x,\pa_y\rangle P_{l,2}$, which can be formally regarded as the estimate of the commutator 
\beno
[\Delta^{-1},\chi_1](\chi_2''\pa_y\psi).
\eeno 
The same problem appears in the pressure decomposition of the nonlinear part. Roughly speaking, the additional growth of $\Pi_{l,2}$ and $\Pi_{n,2}$ happens only near the boundary, which does not affect the vorticity equation due to the compact support of $a$ and $\th'$. We use the same strategy to treat $\Pi_{l,2}$ and $\Pi_{n,2}$, see section \ref{sec: ell-pressure} for more details. 

The third step is the nonlinear change of coordinates, which helps us to get rid of the non-decaying zero mode in the transport term. The step will also distort the elliptic equation for the stream function and pressure. But after separating the linear equation with the time-independent coefficients, the time-dependent (solution-dependent) part is small, see \eqref{eq: decomposition stream} and \eqref{eq: elliptic pressure-1} for the decompositions, also see Proposition \ref{prop: estimate diff-1}, Lemma \ref{lem: S_2}, Lemma \ref{lem: S_a}, Lemma \ref{lem: S_2-1}, and Lemma \ref{lem: S_a-1} for the estimates of the solution-dependent part. 

The original nonlocal term $-u''\pa_x\psi-\th'\pa_xP$, which does not appear in the Couette flow with a constant background density case, is hard to deal with. After introducing the good unknown $\tom$ (see \ref{eq: good unknown}) and doing the nonlinear change of coordinates, we have a slightly better new nonlocal term $\udl{\varphi_1}\pa_zP_{\neq}\Psi$. The new nonlocal term may lead to instability. However, under the linear stability assumption (4) in Theorem \ref{Thm: main}, we can construct a wave operator to eliminate it. Similar to \cite{MasmoudiZhao2020}, we have the following observations: 
1, the operator can be constructed for each frequency in $z$ separately; 
2, the nonlocal term is small for high frequencies in the $z$ variable, which also corresponds to the linear stability of the distorted Rayleigh operator in large wave numbers. So the wave operator is not needed for higher frequencies;
3, the nonlocal term $(\varphi_1\circ u^{-1})(v)\pa_zP_{\neq}\mathring{\mathcal{T}}_{1,D}^{-1}[\Om]$ with time-independent coefficient $(\varphi_1\circ u^{-1})(v)$ is responsible for the linear instability. Thus, we only construct a wave operator for lower frequencies in $z$ to eliminate the time-independent part. Due to the change of coordinates, even the linearized nonlocal term $\mathring{\mathcal{T}}_{1, D}^{-1}[\Om]$ is time-dependent. We then introduce the inverse linear change of coordinates to make the nonlocal term time-independent. We arrive at \eqref{eq: linearsystem}, which is the same as the linearized equation in the original variables but has a different meaning. The construction of the wave operator is highly non-trivial. In section \ref{sec: LID}, we study \eqref{eq: linearsystem}. This part is new. We obtain the representation formula of the solution, which gives the linear inviscid damping. The representation formula leads us to the formula of the wave operator. In section \ref{sec: wave operator}, we study the basic properties of the wave operator. The nice estimates of its Fourier kernel, make the wave operator well-adapted to nonlinear interactions. More precisely, the Fourier kernel $\mathcal{D}(t,k,\xi_1,\xi_2)$ of the wave operator satisfies 
\beno
\left|\mathcal{D}(t,k,\xi_1,\xi_2)\right|
\lesssim e^{-\la_{\mathcal{D}}|\xi_1-\xi_2|^{s_0}}. 
\eeno
In terms of regularity, one can simply regard the effect of the wave operator acting on the function $f$ as multiplying $f$ by a smooth function, whose Fourier transform is $e^{-\la_{\mathcal{D}}|\xi|^{s_0}}$. This is one reason why the wave operator is well-adapted to the nonlinear system. The second reason is related to the famous Orr mechanism \cite{Orr1907}. In physical terms, because the mode of the vorticity in question is initially well-mixed, and then proceeds to un-mix under the shear flow evolution, a transient growth happens near the Orr critical time. In the breakthrough paper \cite{BM2015}, the authors estimate this transient growth mathematically rigorously by introducing a multiplier $\rmA(t,\na)$. In terms of frequency, the wave operator has two effects: 1, it changes the amplitude of the function at each frequency; 2, it essentially mixes the information near each frequency, due to the fast decay of its Fourier kernel. It ensures that the critical time does not move too much. These good properties of the wave operator allow us to use the same multiplier $\rmA$ of \cite{BM2015} and use the same strategy as in \cite{MasmoudiZhao2020} to study the effect of the wave operator on the nonlinear terms. 

With all the preparations done, we arrive at a good working system \eqref{eq: main system}. In the second part, we study the nonlinear problem and close the energy estimates. The main idea is to introduce time-dependent Fourier multipliers which capture the different growths of different frequencies at different times. We use the classical multiplier $\rmA(t,\na)$ (see Appendix \ref{sec: operator A} for details) of \cite{BM2015} to capture the nonlinear growth from the nonlinear interactions of the transport term $\rmU\cdot \na\Om$. We use an additional multiplier $\rmB(t,\na)$ introduced in \cite{ChenWeiZhangZhang2023} to gain the derivative for the density $a$ for short times. It is useful in the estimate of the nonlinear interactions between $a$ and the pressure $\Pi$ when $a$ is in higher frequencies. We also use the multipliers $\mathcal{M}_3, \mathcal{M}_4, $ and $\mathcal{M}_5$ by using the ideas from  \cite{ChenWeiZhangZhang2023} to control the pressure. We refer to \cite{ChenWeiZhangZhang2023} for some explanations about the new multipliers. Some ideas related to the energy estimates and applying the weight are given in section \ref{sec energy functional} after we introduce the working system and the Fourier multipliers, which are more technical. 

The proof of the main theorem is based on a standard bootstrap argument. In section \ref{sec energy functional}, we introduce the bootstrap Proposition \ref{prop: btsp}. In section \ref{sec: energy estimate}, we give the key propositions which will be used to prove Proposition \ref{prop: btsp} and finish the proof of Theorem \ref{Thm: main} by admitting those propositions. The rest of the paper is to prove propositions in section \ref{sec: energy estimate}. Some estimates are done in the previous papers \cite{BM2015, MasmoudiZhao2020}. The key estimate is to control the pressure term, which is in section \ref{sec: pressure-zero} and section \ref{sec: ell-pressure}. 

\subsection{Notations}
For $f(z,v)$ smooth with compact support in $\mathbb{T}\times (u(0),u(1))$, we define the Fourier transform in the first variable $\mathcal{F}_{1}f(k,v)$, the Fourier transform in the second variable $\mathcal{F}_{2}f(z,\eta)$ and the Fourier transform in both variables $\hat{f}_k(\eta)$ where $(k,\eta)\in \mathbb{Z}\times \R$, 
\begin{align*}
&\mathcal{F}_{1}f(k,v)=\f{1}{2\pi}\int_{\mathbb{T}}f(z,v)e^{-izk}dz,\\
&\mathcal{F}_{2}f(z,\eta)=\int_{\R}f(z,v)e^{-iv\eta}dv,\\
&\hat{f}_k(\eta)=\hat{f}(k,\eta)=\f{1}{2\pi}\int_{\mathbb{T}\times \R}f(z,v)e^{-izk-iv\eta}dzdv. 
\end{align*}

We use a convention is to denote $\rmM, \rmN$ as dyadic integers $\rmM, \rmN\in \mathbb D$  where
\begin{align*}
  \mathbb D=\left\{\frac{1}{2},1,2,\dots,2^j,\dots\right\}=2^{\mathbb N}\cup \frac{1}{2}.
\end{align*}
This will be useful when defining Littlewood-Paley projections and paraproduct decompositions. We will use the same $\rmA$ for $\rmA(t,\na) f=(\rmA(\eta) \hat{f}(\eta))^{\vee}$ or $\rmA \hat{f}=\rmA(\eta) \hat{f}(\eta)$, where $\rmA(t,\na)$ is a Fourier multiplier. For function $\rmA(t, k, \xi)$, we denote the time derivative of $\rmA$ by $\dot{\rmA}(t,k,\xi)=\pa_t\rmA(t,k,\xi)$. We also denote the time derivative on the Fourier multiplier ${\rmA}(t,\na)$ by 
\beno
\dot{\rmA}(t,\na) f=(\dot{\rmA}(t, k,\eta) \hat{f}(t, k, \eta))^{\vee}.
\eeno

We use the notation $f \lesssim g$ when there exists a constant $C>0$ independent of the parameters of interest such that $f \leq C g$ (we define $g \gtrsim f$ analogously). Similarly, we use the notation $f \approx g$ when there exists $C>0$ such that $C^{-1} g \leq f \leq C g$. 
We use the notation $c$ to denote a constant such that $0\le c<1$ which may be different from line to line.

We will denote the $l^{1}$ vector norm $|k, \eta|=|k|+|\eta|$, which by convention is the norm taken in our work. Similarly, given a scalar or vector in $\mathbb{R}^{n}$ we denote
$$
\langle v\rangle=\left(1+|v|^{2}\right)^{\frac{1}{2}} .
$$

We denote the projection to the $k$th mode of $f(x,y)$ by
\begin{align*}
  P_kf(x,y)=f_k(x,y)=\frac{1}{2\pi}\left(\int_{\mathbb{T}}f(x',y)e^{-ikx'}dx'\right)e^{ikx},
\end{align*}
and denote the projection to the non-zero mode by
\begin{align*}
  P_{\neq}f(x,y)=f_{\neq}=\sum_{k\in\mathbb Z\setminus\{0\}}f_k(x,y).
\end{align*}
and the projection to zero mode by
\begin{align*}
  \left(P_{0}f\right)(y)= \frac{1}{2\pi}\left(\int_{\mathbb{T}}f(x',y)dx'\right).
\end{align*}
We also use $\hat f_k(\xi)$ to denote $\hat f(k,\xi)$ to emphasize it is the Fourier transform of the $k$ mode. 

In some proofs, we use the Littlewood-Paley dyadic decomposition. Let $\varrho_0$ be a smooth function such that $\varrho_0=1$ for $|\xi|\leq \f12$ and $\varrho(\xi)=0$ for $|\xi|\geq \f34$. We define $\varrho=\varrho_0(\xi/2)-\varrho_0(\xi)$ and $\varrho_{\rmM}(\xi)=\varrho(\rmM^{-1}\xi)$ with dyadic numbers $\rmM=1,2,4,...$.
\begin{align*}
&    f_{\rmM}=\varrho_{\rmM}(|\na|)f,\quad \widehat{f}_k(\xi)_{\rmM}=\widehat{f}_k(\xi)\rho_{\rmM}(k,\xi)\\
 & f_{<\rmM}= \varrho_0(|\na|)f+\sum_{\rmK\in 2^{\mathbb{N}}: \rmK<\rmM}f_{\rmK}.
\end{align*}

We denote the standard $L^{2}$ norms by $\|\cdot\|_{2}$. We make common use of the Gevrey-$\frac{1}{s}$ norm with Sobolev correction defined by
\begin{align*}
  \|f\|_{\mathcal G^{s,\lambda,\sigma}}^2=\sum_k \int \left|\hat f_k(\eta)\right|^2 e^{2\lambda|k,\eta|^s}\langle k,\eta\rangle^{2\sigma}  d\eta.
\end{align*}
In the proof, we sometimes use $\mathcal G^{s,\sigma}$ instead where we omit $\la=\la(t)$ for simplification and keep the important index $\s$ which represents the Sobolev corrections. We refer to this norm as the $\mathcal G^{s,\lambda,\sigma}$ norm and occasionally refer to the space of functions
\begin{align*}
  \mathcal G^{s,\lambda,\sigma}= \left\{f\in L^2:\|f\|_{\mathcal G^{s,\lambda,\sigma}}<\infty\right\}.
\end{align*}
We refer to {\it section A.2} in \cite{BM2015} for a discussion of the basic properties of this norm and some related useful inequalities. 
For $\eta\geq 0$, we define $\rmE(\eta)\in \mathbb{Z}$ to be the integer part. We define for $\eta\in \R$ and $1\leq |k|\leq \rmE(\sqrt{|\eta|})$ with $\eta k\geq 0$, $t_{k,\eta}=\big|\f{\eta}{k}\big|-\f{|\eta|}{2|k|(|k|+1)}$ and $t_{0,\eta}=2|\eta|$ and the critical intervals
\beno
\mathrm{I}_{k,\eta}=\left\{
\begin{aligned}
&[t_{|k|,\eta},t_{|k|-1,\eta}]\quad &\text{if}\ \eta k\geq 0 \ \text{and } 1\leq |k|\leq \rmE(\sqrt{|\eta|}),\\
&\emptyset \quad &\text{otherwise}.
\end{aligned}\right.
\eeno
For minor technical reasons, we define a slightly restricted subset as the resonant intervals
\beno
\mathbf{I}_{k,\eta}=\left\{
\begin{aligned}
&\mathrm{I}_{k,\eta}\quad &\sqrt{|\eta|}\leq t_{k,\eta},\\
&\emptyset \quad &\text{otherwise}.
\end{aligned}\right.
\eeno

We use $\mathbf{1}_{A}$ be the characteristic function which means $\mathbf{1}_{A}(x)=\left\{\begin{aligned}&1\quad x\in A\\&0\quad x\notin A\end{aligned}\right.$. 
We also use the smooth cut-off functions $\chi_1$ and $\chi_2$ in Gevrey-$\f{2}{s_0+1}$ class which satisfies 
\begin{align*}
&\mathrm{supp}\, \chi_1\subset [\kappa_0,1-\kappa_0],\quad 
\chi_1(y)\equiv 1\  \text{for}\  y\in [1-2\kappa_0,2\kappa_0],\\
&\mathrm{supp}\, \chi_2\subset \big[\f{\kappa_0}{2},1-\f{\kappa_0}{2}\big],\quad  
\chi_2(y)\equiv 1\  \text{for}\  y\in [\kappa_0,1-\kappa_0],
\end{align*} 
and there exist $K_1>1$ such that for all integers $m\geq 0$
\ben\label{eq: chi}
\sup_{y\in\R}\left|\f{d^m\chi_i(y)}{dy^m}\right|\leq K_1^m(m!)^{\f{2}{s_0+1}}(m+1)^{-2},\quad i=1,2. 
\een

\section{Good unknown and coordinate system}\label{sec: step1}
There are two linear nonlocal terms $-u''\pa_x\psi$ and $-\th'\pa_xP$ in the vorticity equation of \eqref{eq:vorticity}. The first nonlocal term also appears in the Euler case \cite{MasmoudiZhao2020}, where the authors introduce the wave operator to eliminate it. The second nonlocal term is new, which does not exist in the constant background density case $\th=1$. The idea of eliminating this term is to use the horizontal velocity equation, which also contains $\pa_xP$ in the equation. We introduce the good unknown 
\ben\label{eq: good unknown}
\tilde{\om}=\f{\om}{\th}-\f{\th'}{\th^2}\pa_y\psi=\pa_y\left(\f{1}{\th}\pa_y\psi\right)+\f{1}{\th}\pa_{xx}\psi\eqdef\tilde{\Delta}\psi.
\een

\subsection{New system in good unknown}
A direct calculation gives
\beq\label{eq: new system}
\left\{\begin{aligned}
&\pa_t\tilde{\om}+u(y)\pa_x\tilde{\om}-\left(\f{u'}{\th}\right)' U^y
+U\cdot\na\tilde{\om}+\f{\th'}{\th}\pa_x\psi\tilde{\om}+\left(\f{\th''}{\th^2}-\f{(\th')^2}{\th^3}\right)\pa_x\psi\pa_y\psi\\
&\quad\quad\quad\quad\quad\quad\quad\quad\quad\quad\quad\quad\quad\quad\quad\quad
+\f{\pa_xd\pa_yP-\pa_yd\pa_xP}{\th}+\f{\th'd\pa_xP}{\th^2}=0,\\
&\tilde{\Delta}\psi\eqdef\pa_y\left(\f{1}{\th}\pa_y\psi\right)+\f{1}{\th}\pa_{xx}\psi=\tilde{\om},
\quad U=(U^x, U^y)=\na^{\bot}\psi,\\
&\psi(t, x,0)=\psi(t, x, 1)=0\\
&\pa_t d+u(y)\pa_xd+\th'(y)\pa_x\psi+U\cdot \na d=0,\\
&\dive\big((\th+d)\na P\big)=-\dive (U\cdot \na U)-2u'(y)\pa_{xx}\psi,
\end{aligned}
\right.
\eeq
Let us mention some observations of the new system here. 
\begin{itemize}
\item The new linear nonlocal term $\left(\f{u'}{\th}\right)'U^y$ has compact support. 
\item If $d=0$, $\th=1$, we are back to the Euler equation around shear flows. 
\end{itemize}
\subsection{Pressure decomposition}
The pressure equation is not good for applying the Fourier transform, because the right-hand side $-\dive (U\cdot \na U)-2u'(y)\pa_{xx}\psi$ is not compactly supported. We divide $P$ into six terms: two linear terms $P_{l,1}$ and $P_{l, 2}$, two nonlinear terms $P_{n,1}$ and $P_{n,2}$, and two boundary corrections $P_{l, b}$ $P_{n, b}$, namely, 
\begin{align*}
P=P_{l,1}+P_{l,2}+P_{l, b}+P_{n,1}+P_{n,2}+P_{n, b}.
\end{align*}
We define $P_{l, b}=\f{u\pa_y\psi_{\neq}-u'\psi_{\neq}}{\th+d}(1-\chi_2)$, where $\psi_{\neq}=\psi-\f{1}{2\pi}\int_{\mathbb{T}}\psi(t, x, y)dx$ is the projection of $\psi$ onto the non-zero mode. Then by using the facts 
\beno
\begin{aligned}
(1-\chi_2)\om=0,\ u''(1-\chi_2)=0,\ (1-\chi_2)\na(\f{1}{\th+d})=0,
\end{aligned}
\eeno 
we have
\begin{align*}
\pa_yP_{l, b}&=\f{u(P_{\neq}[\om]-\pa_{xx}\psi_{\neq})}{\th+d}(1-\chi_2)-\f{u\pa_y\psi_{\neq}-u'\psi_{\neq}}{\th+d}\chi_2'\\
\dive\big((\th+d)\na P_{l, b}\big)&=\Delta\big((u\pa_y\psi-u'\psi)(1-\chi_2)\big)\\
&=(1-\chi_2)\big(u\pa_{yxx}\psi_{\neq}-u'\pa_{xx}\psi_{\neq}\big)
-\chi_2''(u\pa_y\psi_{\neq}-u'\psi_{\neq})\\
&\quad-2\chi_2'(u\pa_{yy}\psi_{\neq}-u'\pa_y\psi_{\neq})-2\chi_2'(u'\pa_y\psi_{\neq}-u''\psi_{\neq})\\
&\quad+(1-\chi_2)(u\pa_{yyy}\psi_{\neq}+2u'\pa_{yy}\psi_{\neq}-u'\pa_{yy}\psi_{\neq}-u''\pa_y\psi_{\neq}-u'''\psi_{\neq})\\
&=-2(1-\chi_2)u'\pa_{xx}\psi_{\neq}+2\chi_2'(u\pa_{xx}\psi_{\neq})
-\chi_2''(u\pa_y\psi_{\neq}-u'\psi_{\neq})
\end{align*}
which gives $\pa_yP_{l, b}(t, x, 0)=\pa_yP_{l, b}(t, x, 1)=0$ and 
\ben\label{eq: P_lb, property}
\f{\pa_xd\pa_yP_{l, b}-\pa_yd\pa_xP_{l, b}}{\th}+\f{\th'd\pa_xP_{l, b}}{\th^2}\equiv 0,\quad \text{and}\quad d\pa_xP_{l,b}\equiv0.
\een
This means that the linear boundary correction does not affect the equation of $\tom$. 
Then we define $P_l$ by solving the following equation: 
\ben\label{eq: P_l,1}
\left\{\begin{aligned}
&\dive\big((\th+d)\na P_{l,1}\big)=-2(\chi_2u')(y)\pa_{xx}\psi-2(u\chi_2')(y)\pa_{xx}\psi-(u'\chi_2'')(y)\psi_{\neq},\\
&\pa_yP_{l,1}(t, x, 0)=\pa_yP_{l,1}(t, x, 1)=0,
\end{aligned}\right.
\een
and
\ben\label{eq: P_l,2}
\left\{\begin{aligned}
&\dive\big((\th+d)\na P_{l,2}\big)=(\chi_2''u)(y)\pa_{y}\psi_{\neq},\\
&\pa_yP_{l,2}(t, x, 0)=\pa_yP_{l,2}(t, x, 1)=0.
\end{aligned}\right.
\een
It is easy to check that $P_{l,1}+P_{l,2}+P_{l,b}$ solves
\ben
\left\{\begin{aligned}
&\dive\big((\th+d)\na (P_{l,1}+P_{l,2}+P_{l,b})\big)=-2u'(y)\pa_{xx}\psi,\\
&\pa_y(P_{l,1}+P_{l,2}+P_{l,b})(t, x, 0)=\pa_y(P_{l,1}+P_{l,2}+P_{l,b})(t, x, 1)=0.
\end{aligned}\right.
\een
We then define the nonlinear boundary correction: 
\ben
P_{n, b}(t, x, y)=-\f12\f{1}{\th+d}\Big[(\pa_x\psi)^2+(\pa_y\psi)^2\Big](1-\chi_2).
\een
Notice that $\pa_yP_{n, b}(t, x, 0)=\pa_yP_{n, b}(t, x, 1)=0$ and 
\beno
\f{\pa_xd\pa_yP_{n, b}-\pa_yd\pa_xP_{n, b}}{\th}+\f{\th'd\pa_xP_{n, b}}{\th^2}\equiv 0,\quad \text{and}\quad d\pa_xP_{n,b}\equiv0
\eeno
which means that the nonlinear boundary correction does not affect the equation of $\tom$. 

Now let us deduce the equation for $P_{n}$. A direct calculation gives
\begin{align*}
\dive\big((\th+d)\na P_{n}\big)
&=-\dive(U\cdot \na U)+\f12\dive\Big((\th+d)\na\Big(\f{1}{\th+d}\Big[(\pa_x\psi)^2+(\pa_y\psi)^2\Big](1-\chi_2)\Big)\Big)\\
&=-2(\pa_{xy}\psi)^2-2(\pa_{xx}\psi)^2+2\om\pa_{xx}\psi\\
&\quad +\f12\dive\Big(\Big[(\pa_x\psi)^2+(\pa_y\psi)^2\Big](1-\chi_2)(\th+d)\na\Big(\f{1}{\th+d}\Big)\Big)\\
&\quad +\f12\dive\Big((1-\chi_2)\na\Big(\Big[(\pa_x\psi)^2+(\pa_y\psi)^2\Big]\Big)\Big)\\
&\quad -\f12\pa_y\Big(\Big[(\pa_x\psi)^2+(\pa_y\psi)^2\Big]\chi_2'\Big).
\end{align*}
Noticing that $(1-\chi_2)=0$ on the compact support of $\na\Big(\f{1}{\th+d}\Big)$, we obtain that
\begin{align*}
\dive\big((\th+d)\na P_{n}\big)
&=-\dive(U\cdot \na U)+\dive\Big((\th+d)\na\Big(\f{1}{\th+d}\Big[(\pa_x\psi)^2+(\pa_y\psi)^2\Big](1-\chi_2)\Big)\Big)\\
&=-2(\pa_{xy}\psi)^2-2(\pa_{xx}\psi)^2+2\om\pa_{xx}\psi\\
&\quad +\f12(1-\chi_2)\dive\Big(\na\Big(\Big[(\pa_x\psi)^2+(\pa_y\psi)^2\Big]\Big)\Big)\\
&\quad -\chi_2'\pa_y\Big(\Big[(\pa_x\psi)^2+(\pa_y\psi)^2\Big]\Big)
 -\f12\chi_2''\Big[(\pa_x\psi)^2+(\pa_y\psi)^2\Big]\\
 &=-2\chi_2(\pa_{xy}\psi)^2-2\chi_2(\pa_{xx}\psi)^2+2\om\pa_{xx}\psi\\
&\quad +(1-\chi_2) \pa_x\om\pa_x\psi+(1-\chi_2)\pa_y\om\pa_y\psi\\
&\quad -\chi_2'\Big(\Big[2(\pa_x\psi)\pa_{xy}\psi+2(\pa_y\psi)\pa_{xx}\psi\Big]\Big)
 -\f12\chi_2''\Big[(\pa_x\psi)^2+(\pa_y\psi)^2\Big]. 
\end{align*}
Noticing that $(1-\chi_2)=0$ on the compact support of $\na\om$, we obtain the equation of $P_{n}$
\beq\label{eq: P_n}
\left\{
\begin{aligned}
\dive\big((\th+d)\na P_{n}\big)&=-2\chi_2(\pa_{xy}\psi)^2-2\chi_2(\pa_{xx}\psi)^2+2\om\pa_{xx}\psi\\
&\quad -\chi_2'\Big(\Big[2(\pa_x\psi)\pa_{xy}\psi+2(\pa_y\psi)\pa_{xx}\psi\Big]\Big)
 -\f12\chi_2''\Big[(\pa_x\psi)^2+(\pa_y\psi)^2\Big]. \\
 \pa_y P_{n}(t, x, 0)&=\pa_yP_{n}(t, x, 1)=0.
\end{aligned}
\right.
\eeq
Note that the last term $(\pa_y\psi)^2$ in \eqref{eq: P_n} requires an additional decomposition for $P_n=P_{n,1}+P_{n,2}$ with 
\beq\label{eq: P_n1}
\left\{
\begin{aligned}
\dive\big((\th+d)\na P_{n,1}\big)&=-2\chi_2(\pa_{xy}\psi)^2-2\chi_2(\pa_{xx}\psi)^2+2\om\pa_{xx}\psi\\
&\quad -\chi_2'\Big(\Big[2(\pa_x\psi)\pa_{xy}\psi+2(\pa_y\psi)\pa_{xx}\psi\Big]\Big)
 -\f12\chi_2''\Big[(\pa_x\psi)^2+(\pa_yP_{\neq}\psi)^2\Big]. \\
 \pa_y P_{n,1}(t, x, 0)&=\pa_yP_{n,1}(t, x, 1)=0.
\end{aligned}
\right.
\eeq
and
\beq\label{eq: P_n2}
\left\{
\begin{aligned}
\dive\big((\th+d)\na P_{n,2}\big)&=
 -\chi_2''(\pa_yP_{\neq}\psi)(\pa_yP_{0}\psi)-\f12\chi_2''(\pa_yP_{0}\psi)^2. \\
 \pa_y P_{n,2}(t, x, 0)&=\pa_yP_{n,2}(t, x, 1)=0.
\end{aligned}
\right.
\eeq

By this decomposition, the functions of the right-hand side of equations \eqref{eq: P_l,1}, \eqref{eq: P_l,2}, \eqref{eq: P_n}, and \eqref{eq: P_n2} have compact support. Let us give a formal discussion about the behavior of each pressure. By dropping the lower order terms, coefficients, and the distorting effect in the elliptic equations, in terms of regularity, one can formally regard that these equations are
\begin{align*}
&\Delta P_{l,1}\approx \chi_2\pa_{xx}\psi,\quad
\Delta P_{l,2}\approx \chi_2''\pa_y\psi_{\neq},\quad
\Delta P_{n,1}\approx \om\pa_{xx}\psi, \quad \Delta P_{n,2}\approx \chi_2''(\pa_y\psi_{\neq}).
\end{align*}

\subsection{Zero mode of the pressure}
For the zero mode of the pressure, we use the equation
\beno
\pa_tU^y+u(y)\pa_xU^y+U\cdot \na U^y+\th\pa_yP+d\pa_yP=0
\eeno
to estimate the zero mode of the pressure. We have 
\beno
(\th+P_0(d))\pa_yP_0(P)+P_0[P_{\neq}(d)\pa_yP_{\neq}(P)]=P_0(U^x\pa_xU^y-U^y\pa_xU^x)=P_0(-\pa_y\psi\pa_{xx}\psi+\pa_x\psi\pa_{xy}\psi).
\eeno

\subsection{Nonlinear coordinate transform}
Let us now introduce the change of variables $(t,x,y)\to (t,z,v)$ to eliminate the non-decaying zero mode: 
\beq\label{eq: coor-change}
\begin{aligned}
v(t,y)&=u(y)-\f{\chi_1(y)}{t}\int_0^t\f{1}{2\pi}\int_{\bbT}\pa_y\psi(t',x,y)dx dt',\\
z(t,x,y)&=x-tv(t,y).
\end{aligned}
\eeq
Thus $\mathrm{Ran}\, u=\mathrm{Ran}\, v=[u(0),u(1)]$.  
Let us also define $\udl{\pa_tv}(t,v)$ and $\udl{\pa_yv}(t,v)$ so that
\begin{align*}
&\udl{\pa_tv}(t,v(t,y))=\pa_tv(t,y),\quad
\udl{\pa_yv}(t,v(t,y))=\pa_yv(t,y). 
\end{align*}
In order to simplify the notations, we introduce 
\beq\label{eq:definition vaprhi}
\begin{aligned}
&\varphi_1(y)=\left(\f{u'}{\th}\right)'(y),\quad
\varphi_2(y)=\f{\th'(y)}{\th(y)},\quad
\varphi_3(y)=\left(\f{\th''}{\th^2}-\f{(\th')^2}{\th^3}\right),\\
&\varphi_4(y)=\f{1}{\th(y)},\quad 
\varphi_5(y)=\varphi_2(y)\varphi_4(y),\quad \varphi_6(y)=-2(\chi_2u')(y)-2(u\chi_2')(y),\\
&\varphi_7(y)=-(u'\chi_2'')(y),\quad \varphi_8(y)=(\chi_2''u)(y),
\quad \varphi_9(y)=u'(y),\quad \varphi_{10}(y)=\th'(y),\\
&\varphi_{11}(y)=\th(y).
\end{aligned}
\eeq
Here and below, for any function $\varphi(y)$, we denote 
\beno
\udl{\varphi}(t,v(t,y))=\varphi(y)\quad \text{and}\quad 
\widetilde{\varphi}(u(y))=\varphi(y).
\eeno
We write two functions $\udl{\varphi}(t,v)$ and $\widetilde{\varphi}(v)$ both in $(t,v)$ variable and regard $\widetilde{\varphi}(v)$ as the linear part of $\udl{\varphi}(t,v)$. So we can expect 
\ben\label{eq: varphidelta}
\varphi^{\d}(t,v)\eqdef\udl{\varphi}(t,v)-\widetilde{\varphi}(v)
\een
is small. 
\begin{remark}\label{Rmk: equation of coefficients}
Recall the notations $\udl{\varphi}(t,v(t,y))=\varphi(y)$ and $\widetilde{\varphi}(u(y))=\varphi(y)$, then $\udl{\varphi}(t,v)$ satisfies the following transport equation
\beno
\pa_t\udl{\varphi}(t,v)+\udl{\pa_tv}(t,v)\pa_v\udl{\varphi}(t,v)=0,
\eeno
and $\varphi^{\d}(t,v)=\udl{\varphi}(t,v)-\widetilde{\varphi}(v)$ satisfies
\beno
\pa_t\varphi^{\d}+\udl{\pa_tv}(t,v)\pa_v\varphi^{\d}=-\udl{\pa_tv}(t,v)\pa_v\widetilde{\varphi}(v). 
\eeno
Also note that $\varphi_{4}(t,v)$, $\varphi_9(t,v)$, and $\varphi_{11}(t,v)$ do not have compact support, but $\varphi^{\d}_{j}(t,v)$ has compact support for $j=1,2,...,11$
\end{remark}

Define $\Om(t,z,v)$, $a(t,z,v)$, $\Pi(t,z,v)$ and $\Psi(t,z,v)$ so that 
\beq\label{eq: Om, a, Pi, Psi}
\begin{aligned}
&\Om(t,z(t,x,y),v(t,y))=\tilde{\om}(t,x,y),\quad a(t,z(t,x,y),v(t,y))=d(t,x,y),\\
&\Pi_{l,1}(t,z(t,x,y),v(t,y))=P_{l,1}(t,x,y),\quad \Pi_{l,2}(t,z(t,x,y),v(t,y))=P_{l,2}(t,x,y),\\
&\Pi_{n,1}(t, z(t, x, y), v(t, y))=P_{n,1}(t, x, y),\quad
\Pi_{n,2}(t, z(t, x, y), v(t, y))=P_{n,2}(t, x, y),\\
&\Psi(t,z(t,x,y),v(t,y))=\psi(t,x,y),\quad
\udl{\om}(t, z(t, x, y), v(t, y))=\om(t, x, y)
\end{aligned}
\eeq
hence the original 2D inhomogeneous Euler system \eqref{eq:vorticity} is expressed in the $(t,z,v)$ variables as
\beq\label{eq: Om}
\left\{
\begin{aligned}
&\pa_tP_0(\udl{\om})+\udl{\pa_yv}P_0(\pa_z\Psi\pa_v\udl{\om}-\pa_v\Psi\pa_z\udl{\om})-\udl{\pa_yv}P_0(\pa_z\Pi_{l}\pa_va-\pa_v\Pi_{l}\pa_za)\\
&\quad\quad\quad\ 
-\udl{\pa_yv}P_0(\pa_z\Pi_{n}\pa_va-\pa_v\Pi_{n}\pa_za)=0,\\
&\pa_t\Om-\udl{\varphi_1}\pa_zP_{\neq}\Psi+\mathrm{U}\cdot\na_{z,v}\Om+\mathcal{N}_{\Om}[\Psi]+\mathcal{N}_a[\Pi]
=0,\\
&\pa_ta+\udl{\th'}\pa_z\Psi+\mathrm{U}\cdot\na_{z,v}a=0,\\
&\udl{\varphi_4}\pa_{zz}\Psi+\udl{\pa_yv}(\pa_v-t\pa_z)\Big(\udl{\varphi_4}\udl{\pa_yv}(\pa_v-t\pa_z)\Psi\Big)=\Om,\\
&\Psi(t, z, u(0))=\Psi(t, z, u(1))=0,\\
&\mathrm{U}=(0,\udl{\pa_tv})+\udl{\pa_yv}\na^{\bot}_{v,z}P_{\neq}(\Psi),
\end{aligned}\right.
\eeq 
where $\mathcal{N}$ represents the nonlinear terms: 
\ben
\mathcal{N}_{\Om}[\Psi]=\udl{\varphi_2}\pa_z\Psi\Om+\udl{\varphi_3}\pa_z\Psi\udl{\pa_yv}(\pa_v-t\pa_z)\Psi
\een 
and 
$\mathcal{N}_a[\Pi]=\mathcal{N}_a[\Pi_{l}]+\mathcal{N}_a[\Pi_{n}]$ with $\Pi_l=\Upsilon_2\Pi_{l,1}+\Upsilon_1\Pi_{l,2}$, $\Pi_n=\Upsilon_2\Pi_{n,1}+\Upsilon_1\Pi_{n,2}$, and 
\beq\label{eq: N_aPi}
\begin{aligned}
&\mathcal{N}_a[\Pi_{l}]=\udl{\varphi_4}\udl{\pa_yv}(\pa_v\Pi_{l}\pa_za-\pa_z\Pi_{l}\pa_va)+\udl{\varphi_5}a\pa_z\Pi_{l},\\
&\mathcal{N}_a[\Pi_n]=\udl{\varphi_4}\udl{\pa_yv}(\pa_v\Pi_n\pa_za-\pa_z\Pi_n\pa_va)+\udl{\varphi_5}a\pa_z\Pi_n. 
\end{aligned}
\eeq
Here $(P_{l,1}, P_{l,2}, P_{n,1}, P_{n,2})$ solve the following equations.
\begin{subequations}\label{eq: Pressure}
\beq
\left\{
\begin{aligned}
&\pa_z\Big((a+\udl{\th})\pa_z\Pi_{l,1}\Big)+\udl{\pa_yv}(\pa_v-t\pa_z)\Big((a+\udl{\th})\udl{\pa_yv}(\pa_v-t\pa_z)\Pi_{l,1}\Big)\\
&\quad\quad \quad \quad\quad \quad 
=\udl{\varphi_6}\pa_{zz}\Psi+\udl{\varphi_7}P_{\neq}(\Psi),\\
&(\pa_v-t\pa_z)\Pi_{l,1}(t, z, u(0))=(\pa_v-t\pa_z)\Pi_{l,1}(t, z, u(1))=0,
\end{aligned}\right.
\eeq 
and
\beq
\left\{
\begin{aligned}
&\pa_z\Big((a+\udl{\th})\pa_z\Pi_{l,2}\Big)+\udl{\pa_yv}(\pa_v-t\pa_z)\Big((a+\udl{\th})\udl{\pa_yv}(\pa_v-t\pa_z)\Pi_{l,2}\Big)\\
&\quad\quad \quad \quad\quad \quad
=\udl{\varphi_8}\udl{\pa_yv}(\pa_v-t\pa_z)P_{\neq}(\Psi),\\
&(\pa_v-t\pa_z)\Pi_{l,2}(t, z, u(0))=(\pa_v-t\pa_z)\Pi_{l,2}(t, z, u(1))=0,
\end{aligned}\right.
\eeq 
and
\beq
\left\{
\begin{aligned}
&\pa_z\Big((a+\udl{\th})\pa_z\Pi_{n,1}\Big)+\udl{\pa_yv}(\pa_v-t\pa_z)\Big((a+\udl{\th})\udl{\pa_yv}(\pa_v-t\pa_z)\Pi_{n,1}\Big)\\
&\quad\quad \quad \quad\quad \quad 
=-2\udl{\chi_2}(\udl{\pa_yv}(\pa_v-t\pa_z)\pa_z\Psi)^2-2\pa_{zz}\Psi\Big(\Om-\udl{\chi_2}\pa_{zz}\Psi\Big),\\
&\quad\quad \quad \quad\quad \quad\quad
-2\udl{\chi_2'}\udl{\pa_yv}\Big[(\pa_z\Psi)(\pa_v-t\pa_z)\pa_z\Psi+(\pa_v-t\pa_z)\Psi\pa_{zz}\Psi\Big]\\
&\quad\quad \quad \quad\quad \quad\quad 
-\f{1}{2}\udl{\chi_2''}\Big[(\pa_z\Psi)^2+(\udl{\pa_yv})^2((\pa_v-t\pa_z)P_{\neq}\Psi)^2\Big],\\
&(\pa_v-t\pa_z)\Pi_{n,1}(t, z, u(0))=(\pa_v-t\pa_z)\Pi_{n,1}(t, z, u(1))=0,
\end{aligned}\right.
\eeq 
and
\beq
\left\{
\begin{aligned}
&\pa_z\Big((a+\udl{\th})\pa_z\Pi_{n,2}\Big)+\udl{\pa_yv}(\pa_v-t\pa_z)\Big((a+\udl{\th})\udl{\pa_yv}(\pa_v-t\pa_z)\Pi_{n,2}\Big)\\
&\quad\quad \quad \quad\quad \quad 
= -\f{1}{2}\udl{\chi_2''}\Big[2(\udl{\pa_yv})^2((\pa_v-t\pa_z)P_{\neq}\Psi)\pa_vP_{0}\Psi+(\udl{\pa_yv})^2(\pa_vP_{0}\Psi)\pa_vP_{0}\Psi\Big],\\
&(\pa_v-t\pa_z)\Pi_{n,2}(t, z, u(0))=(\pa_v-t\pa_z)\Pi_{n,2}(t, z, u(1))=0.
\end{aligned}\right.
\eeq 
\end{subequations}
We also have that the zero mode of the pressure solves
\beno
(\udl{\th}+P_0(a))\pa_vP_0(\Pi)+P_0\big[P_{\neq}(a)P_{\neq}(\pa_v-t\pa_z)P_{\neq}(\Pi)\big]=P_0\Big(\pa_z\Psi\pa_z(\pa_v-t\pa_z)\Psi-\pa_{zz}\Psi\pa_v\Psi\Big).
\eeno
Note that \eqref{eq: Om} is not our final working system. The main part of the nonlocal quasi-linear term $-\underline{\varphi_1}\pa_zP_{\neq}(\Psi)$ will be eliminated by applying the wave operator.

\subsection{Linearization and the inverse linear change of coordinate}

Let us also introduce the linearization of \eqref{eq: Om} in $(t,z,v)$ coordinates
\beq
\label{eq:linvorticity}
\left\{
\begin{aligned}
&\pa_t\Om-\widetilde{\varphi_1}\pa_zP_{\neq}\Psi
=0,\\
&\tilde{\Delta}_{t}^l\Psi\eqdef\widetilde{\varphi_4}\pa_{zz}\Psi+\widetilde{u'}(\pa_v-t\pa_z)\Big(\widetilde{\varphi_4}\widetilde{u'}(\pa_v-t\pa_z)\Psi\Big)=\Om,
\end{aligned}
\right.
\eeq
Here we recall that $\widetilde{\varphi_1}(v)=\left(\left(\f{u'}{\th}\right)'\circ u^{-1}\right)(v)$, $\widetilde{\varphi_4}(v)=\f{1}{(\th\circ u^{-1})(v)}$ and $\widetilde{u'}(v)=(u'\circ u^{-1})(v)$. Let us now introduce the inverse linear change of coordinate, namely,  $(t,z,v)\to (t, x, y)$: 
\ben\label{eq: inversecoordinate}
y=u^{-1}(v), \quad x=z+t u(y)
\een
Define $\tilde{\om}(t,x,y)=\Om(t, x-tu(y), u(y))$ and $\psi(t, x, y)=\Psi(t, x-tu(y), u(y))$. Then we arrive at
\ben\label{eq: linearsystem}
\pa_t\tilde{\om}+\pa_x\mathcal{R}\tilde{\om}=0, \quad \psi=\tilde{\Delta}^{-1}\tilde{\om},
\een
where $\mathcal{R}=u\mathrm{Id}-\left(\f{u'}{\th}\right)'\tilde{\Delta}^{-1}$ is the distorted Rayleigh operator. Note that here we use the same notation $(t, x, y)$, $\tilde{\om}$, and $\psi$. They are different from the functions in \eqref{eq: new system} written in original $(t, x, y)$ coordinates.

\section{Linear inviscid damping}\label{sec: LID}
In this section, we study the linearized equation \eqref{eq: linearsystem}. The main purpose is to introduce the wave operator associated with the distorted Rayleigh operator $\mathcal{R}$. 

The wave operator related to a self-adjoint operator is well-known \cite{DelortMasmoudi,GPR18,GP20,LLS21,Schlag,Yajima}. Let $A, B$ be two self-adjoint operators in the Hilbert space $H$, then the wave operator $\bbD$ related to $A$ and $B$ satisfies
\beno
A\bbD=\bbD B.
\eeno
It can be defined by
\beno
\bbD=\lim_{t\to \infty}e^{-itA}e^{itB}. 
\eeno
However, the wave operator related to non-self-adjoint operators is usually not easy to construct and estimate. 
Let us explain the intuition of the construction on the operator $i\pa_x$ associated with the Fourier transform. More precisely, we know that $\la$ and $e^{-i\la x}$ are the generalized eigenvalue and eigenfunction. For any function $f\in L^1\cap L^2$ we have the representation formula
\ben\label{eq:Fourier and wave operator}
f(x)=\int_{\R}\langle f,e^{-i\la x}\rangle e^{-i\la x}d\la.
\een
Let us define $T(f)(\la)=\langle f,e^{-i\la x}\rangle=\int_{\R}f(x)e^{ix\la}dx$, then $T\circ (i\pa_x)=\la T$, and we can regard $T$ as a wave operator conjugating $i\pa_x$ to the multiplication by $\la$. 
\subsection{The distorted Rayleigh equation}
By taking the Fourier transform in $x$, we have 
\beno
\pa_t\tilde{\Delta}_k\hat{\psi}+ik u(y)\tilde{\Delta}_k\hat{\psi}-ik \big(\f{u'}{\th}\big)'\hat{\psi}=0.
\eeno
Let $\tilde{\mathcal{L}}_k\hat{\psi}=\tilde{\Delta}_k^{-1}\Big(u(y)\tilde{\Delta}_k\hat{\psi}-\big(\f{u'}{\th}\big)'\hat{\psi}\Big)$. 
We then have 
\beq\label{eq: represent}
\hat{\psi}(t, k, y)=\f{1}{2\pi i}\oint_{\pa \mathcal{S}} e^{-i c k t}(c-\tilde{\mathcal{L}}_k)^{-1}\hat{\psi}_{in} dc
\eeq
where $\mathcal{S}$ contains the spectrum of $\tilde{\mathcal{L}}_k$. 
Next, we study the resolvent $\Psi^{I}=(c-\tilde{\mathcal{L}}_k)^{-1}\hat{\psi}_{in}$. Then $\Psi^{I}$ solves the distorted inhomogeneous Rayleigh equation
\ben\label{eq: inhomRay1}
(u-c)\pa_y\left(\f{1}{\th}\pa_y\Psi^{I}\right)-(u-c)\f{k^2}{\th}\Psi^{I}-\big(\f{u'}{\th}\big)'\Psi^{I}=-\tilde{\Delta}_k\hat{\psi}_{in}=-\hat{\tilde{\om}}_{in}
\een
with boundary condition $\Psi^{I}(0)=\Psi^{I}(1)=0$. We also introduce the homogeneous distorted Rayleigh equation
\ben\label{eq: homRay1}
(u-c)\pa_y\left(\f{1}{\th}\pa_y\Psi^{H}\right)-(u-c)\f{k^2}{\th}\Psi^{H}-\big(\f{u'}{\th}\big)'\Psi^{H}=0.
\een
\begin{remark}
Let $u\in C^4$, $u'\neq 0$ and $\th>0$. 
Then we have the following facts:
\begin{itemize}
\item $\mathrm{Ran}\, u\subset \s(\tilde{\mathcal{L}}_k)$ is the continuous spectrum. 
\item If $c\notin \mathrm{Ran}\, u$ is an eigenvalue of $\tilde{\mathcal{L}}_k$, then \eqref{eq: homRay1} has a non-trivial solution in $H^2(0,1)\cap H_0^1(0,1)$. 
\item If $c=u(y_c)\in \mathrm{Ran}\, u$ is an embedded eigenvalue, then $\big(\f{u'}{\th}\big)'(y_c)=0$. The boundary value $u(0), u(1)$ are not the embedded eigenvalue.
\item If $\big(\f{u'}{\th}\big)'$ does not change sign in $[0,1]$, then $\tilde{\mathcal{L}}_k$ has no eigenvalue or embedded eigenvalue. 
\end{itemize}
\end{remark}
\subsection{Homogeneous equation}
In this section, we show the existence of a smooth solution to the homogeneous equation \eqref{eq: homRay1}. 
\begin{proposition}\label{prop: homsol}
Let $y,y'\in [0,1]$, and let $\phi_1$ solve
\beq\label{eq:phi_1}
\phi_1(y, y', k)=1+\int_{y'}^y\f{k^2\th(z')}{(u(z')-u(y'))^2}\int_{y'}^{z'}\f{(u(z'')-u(y'))^2}{\th(z'')}\phi_1(z'',y',k)dz''dz'
\eeq
or in the differential equation form
\beq\label{eq:phi_1DE}
\left\{\begin{aligned}
&\pa_y\left(\f{(u(y)-u(y'))^2}{\th(y)}\pa_y\phi_1(y, y', k)\right)=k^2\f{(u(y)-u(y'))^2}{\th(y)}\phi_1(y, y', k),\\
&\phi_1(y', y', k)=1\quad \pa_y\phi_1(y', y', k)=0.
\end{aligned}\right.
\eeq 
Then $\phi(y, y',k)=(u(y)-u(y'))\phi_1(y, y',k)$ solves \eqref{eq: homRay1} with $c=u(y')\in \mathrm{Ran}\, u$. 

Let $\ep_0>0$ and for $|\ep|<\ep_0$, let $\phi_1(y, y', \ep, k)$ solve 
\beq\label{eq:phi_1,ep}
\phi_1(y, y', \ep, k)=1+\int_{y'}^y\int_{y'}^{z'}\f{k^2\th(z')(u(z'')-u(y')-i\ep)^2}{\th(z'')(u(z')-u(y')-i\ep)^2}\phi_1(z'',y', \ep,k)dz''dz'
\eeq
or in the differential equation form
\beq\label{eq:phi_1epDE}
\left\{\begin{aligned}
&\pa_y\left(\f{(u(y)-u(y')-i\ep)^2}{\th(y)}\pa_y\phi_1(y, y', \ep, k)\right)=k^2\f{(u(y)-u(y')-i\ep)^2}{\th(y)}\phi_1(y, y', \ep, k),\\
&\phi_1(y', y', \ep, k)=1\quad \pa_y\phi_1(y', y', \ep, k)=0.
\end{aligned}\right.
\eeq 
Then $\phi(y, y', \ep, k)=(u(y)-u(y')-i\ep)\phi_1(y, y', \ep, k)$ solves \eqref{eq: homRay1} with $c=u(y')+i\ep$. 

Moreover, the following estimates hold for $y, y'\in [0,1]$, $k\in \mathbb{Z}\setminus\{0\}$, and $|\ep|<\ep_0$. 
\begin{align}
\label{eq: est1}   &\phi_{1}(y,y', k)\ge1,\quad (y-y')\pa_y\phi_{1}(y,y', k)\geq 0\\
\label{eq: est2}   &C^{-1}e^{C^{-1}|k|(|y-y'|)}\le\phi_{1}(y,y', k)\le Ce^{C|k||y-y'|},\\
\label{eq: est3}    &C^{-1}|k|\min\{|k|(y-y'),1\}\leq \f{|\pa_y\phi_{1}(y,y', k)|}{\phi_{1}(y,y', k)}\leq C|k|\min\{|k|(y-y'),1\}\\
\label{eq: est4}    &|\pa_{yy}\phi_{1}(y,y', k)|\leq Ck^2\phi_{1}(y,y', k),\\
\label{eq: est5}    & 0\leq \phi_{1}(y,y', k)-1\le C\min \{1, |k|^2|y-y'|^2\}\phi_{1}(y,y', k),
\end{align}
and $\phi_1(y, y', \ep, k)\to \phi_{1}(y,y', k)$ as $\ep\to 0$, and for $\ep_0$ small enough (may depend on $k$)
\begin{align}
\label{eq: est6}
\f12\leq |\phi_1(y, y', \ep, k)|\le Ce^{C|k||y-y'|}. 
\end{align}
Here $C$ is a constant independent of $k$. 
\end{proposition}
\begin{proof}
{\bf Existence part.} We first show the existence of $\phi_1(y, y', \ep, k)$ for any fixed $k\in \mathbb{Z}\setminus\{0\}$ and $(y, y', \ep )\in [0,1]^2\times [-\ep_0,\ep_0]$. 
We introduce the weighted norm
\begin{align*}
&\|f\|_{X_0}\eqdef \sup_{(y,y')\in [0,1]^2}\left|\f{f(y, y')}{\cosh A(y-y')}\right|,\\
&\|f\|_{X}\eqdef \sup_{(y,y',\ep)\in [0,1]^2\times [-\ep_0,\ep_0]}\left|\f{f(y, y',\ep)}{\cosh A(y-y')}\right|,
\end{align*}
and
\beno
&&\|f\|_{Y_0}\eqdef\sum_{k=0}^2\sum_{|\b|=k}A^{-k}\|\na_{y,y'}^{\b}f\|_{X_0},\\
&&\|f\|_{Y}\eqdef\|f\|_{X}+\frac{1}{A}\big(\|\partial_yf\|_{X}+\|\partial_{y'}f\|_{X}+\|\pa_{\ep}f\|_{X}\big).
\eeno
Consider the linear operator for $0\leq \ep\leq |\ep_0|$
\beno
T_{\ep}[\phi_1]=\int_{y'}^y\int_{y'}^{z'}\f{\th(z')(u(z'')-u(y')-i\ep)^2}{\th(z'')(u(z')-u(y')-i\ep)^2}\phi_1(z'',y', \ep,k)dz''dz',
\eeno
it is easy to check that 
\beno
\|T_{0}[\phi_1]\|_{X_0}\leq \f{C}{A^2}\|\phi_1\|_{X_0},
\quad \|T_{\ep}[\phi_1]\|_{X}\leq \f{C}{A^2}\|\phi_1\|_{X}.
\eeno
A direct calculation gives
\begin{align*}
\pa_y T_{\ep}[\phi_1]=\int_{y'}^{y}\f{\th(y)(u(z'')-u(y')-i\ep)^2}{\th(z'')(u(y)-u(y')-i\ep)^2}\phi_1(z'',y', \ep,k)dz'',
\end{align*}
which gives
\beno
\|\pa_yT_{0}[\phi_1]\|_{X_0}\leq \f{C}{A}\|\phi_1\|_{X_0},
\quad \|\pa_yT_{\ep}[\phi_1]\|_{X}\leq \f{C}{A}\|\phi_1\|_{X}.
\eeno
We also have 
\begin{align*}
\pa_{yy} T_{\ep}[\phi_1]
&=\phi_1+\int_{y'}^{y}\f{\th'(y)(u(z'')-u(y')-i\ep)^2}{\th(z'')(u(y)-u(y')-i\ep)^2}\phi_1(z'',y', \ep,k)dz''\\
&\quad -2\int_{y'}^{y}\f{\th(y)u'(y)(u(z'')-u(y')-i\ep)^2}{\th(z'')(u(y)-u(y')-i\ep)^3}\phi_1(z'',y', \ep,k)dz'',
\end{align*}
which gives 
\beno
\|\pa_{yy}T_{0}[\phi_1]\|_{X_0}\leq C\|\phi_1\|_{X_0},
\quad \|\pa_{yy}T_{\ep}[\phi_1]\|_{X}\leq C\|\phi_1\|_{X}.
\eeno
A direct calculation gives for $-\ep_0\leq \ep\leq \ep_0$
\begin{align*}
(\pa_y+\pa_{y'})T_{\ep}[\phi_1]
&=\int_{y'}^y\int_{y'}^{z'}(\pa_{y'}+\pa_{z'}+\pa_{z''})\left(\f{\th(z')(u(z'')-u(y')-i\ep)^2}{\th(z'')(u(z')-u(y')-i\ep)^2}\phi_1(z'',y', \ep,k)\right)dz''dz'\\
&=\int_{y'}^y\int_{y'}^{z'}(\pa_{y'}+\pa_{z'}+\pa_{z''})\left(\f{\th(z')(u(z'')-u(y')-i\ep)^2}{\th(z'')(u(z')-u(y')-i\ep)^2}\right)\phi_1(z'',y', \ep,k)dz''dz'\\
&\quad+\int_{y'}^y\int_{y'}^{z'}\f{\th(z')(u(z'')-u(y')-i\ep)^2}{\th(z'')(u(z')-u(y')-i\ep)^2}(\pa_{y'}+\pa_{z''})\phi_1(z'',y', \ep,k)dz''dz'
\end{align*}
and
\begin{align*}
&(\pa_y+\pa_{y'})^2T_{\ep}[\phi_1]\\
&=\int_{y'}^y\int_{y'}^{z'}(\pa_{y'}+\pa_{z'}+\pa_{z''})^2\left(\f{\th(z')(u(z'')-u(y')-i\ep)^2}{\th(z'')(u(z')-u(y')-i\ep)^2}\right)\phi_1(z'',y', \ep,k)dz''dz'\\
&\quad+\int_{y'}^y\int_{y'}^{z'}\f{\th(z')(u(z'')-u(y')-i\ep)^2}{\th(z'')(u(z')-u(y')-i\ep)^2}(\pa_{y'}+\pa_{z''})^2\phi_1(z'',y', \ep,k)dz''dz'\\
&\quad+\int_{y'}^y\int_{y'}^{z'}(\pa_{y'}+\pa_{z'}+\pa_{z''})\left(\f{\th(z')(u(z'')-u(y')-i\ep)^2}{\th(z'')(u(z')-u(y')-i\ep)^2}\right)(\pa_{y'}+\pa_{z''})\phi_1(z'',y', \ep,k)dz''dz',
\end{align*}
which together with the fact that for $0\leq y'\leq z''\leq z'\leq 1$ or $0\leq z'\leq z''\leq y'\leq 1$
\beno
\left|(\pa_{y'}+\pa_{z'}+\pa_{z''})^{j}\left(\f{\th(z')(u(z'')-u(y')-i\ep)^2}{\th(z'')(u(z')-u(y')-i\ep)^2}\right)\right|\leq C,\quad \text{for}\quad j=0,1,2,
\eeno
gives that
\begin{align*}
&\|(\pa_y+\pa_{y'})T_{0}[\phi_1]\|_{X_0}\leq \f{C}{A^2}\big(\|(\pa_y+\pa_{y'})\phi_1\|_{X_0}
+\|\phi_1\|_{X_0}\big),\\
&\|(\pa_y+\pa_{y'})T_{\ep}[\phi_1]\|_{X}\leq \f{C}{A^2}\big(\|(\pa_y+\pa_{y'})\phi_1\|_{X}
+\|\phi_1\|_{X}\big),\\
&\|(\pa_y+\pa_{y'})^2T_{0}[\phi_1]\|_{X_0}\leq \f{C}{A^2}\big(\|(\pa_y+\pa_{y'})^2\phi_1\|_{X_0}+\|(\pa_y+\pa_{y'})\phi_1\|_{X_0}
+\|\phi_1\|_{X_0}\big),\\
&\|(\pa_y+\pa_{y'})^2T_{\ep}[\phi_1]\|_{X}\leq \f{C}{A^2}\big(\|(\pa_y+\pa_{y'})^2\phi_1\|_{X}+\|(\pa_y+\pa_{y'})\phi_1\|_{X}
+\|\phi_1\|_{X}\big).
\end{align*}
We also have
\begin{align*}
\pa_{\ep}T_{\ep}[\phi_1]
&=-2i\int_{y'}^y\int_{y'}^{z'}\f{\th(z')(u(z'')-u(y')-i\ep)}{\th(z'')(u(z')-u(y')-i\ep)^2}\phi_1(z'',y', \ep,k)dz''dz'\\
&\quad+2i\int_{y'}^y\int_{y'}^{z'}\f{\th(z')(u(z'')-u(y')-i\ep)^2}{\th(z'')(u(z')-u(y')-i\ep)^3}\phi_1(z'',y', \ep,k)dz''dz'\\
&\quad+\int_{y'}^y\int_{y'}^{z'}\f{\th(z')(u(z'')-u(y')-i\ep)^2}{\th(z'')(u(z')-u(y')-i\ep)^2}\pa_{\ep}\phi_1(z'',y', \ep,k)dz''dz',
\end{align*}
which gives 
\begin{align*}
\|\pa_{\ep}T_{\ep}[\phi_1]\|_{X}\leq \f{C}{A}\|\phi_1\|_{X}+\f{C}{A^2}\|\pa_{\ep}\phi_1\|_{X}. 
\end{align*}
By combining all the estimates, we arrive at 
\beno
\|T_{0}[\phi_1]\|_{Y_0}\leq \f{C}{A^2}\|\phi_1\|_{Y_0}, \quad\|T_{\ep}[\phi_1]\|_{Y}\leq \f{C}{A^2}\|\phi_1\|_{Y}.
\eeno
By using the equations
\beno
\phi_1(y, y', \ep, k)=1+k^2T_{\ep}[\phi_1](y, y', \ep, k) \quad \text{and}\quad \phi_1(y, y', k)=1+k^2T_{\ep}[\phi_1](y, y', k)
\eeno
and taking $A$ large enough, we show that $T_{\ep}$ is the contraction map yields the existence of $\phi_1(y, y', \ep, k)$ and $\phi_1(y, y', k)$. 

\medskip

\no{\bf Estimate part. } The estimate \eqref{eq: est1} follows directly from the fact that $T_0$ is a positive operator. We also get that for $0\leq y'\leq z''\leq y\leq 1$ or $0\leq y\leq z''\leq y'\leq 1$, 
\beno
1\leq \phi_1(z'', y', k)\leq \phi_1(y, y', k).
\eeno
Thus
\begin{align*}
\phi_1(y, y', k)-1
&=k^2\int_{y'}^y\int_{y'}^{z'}\f{\th(z')(u(z'')-u(y'))^2}{\th(z'')(u(z')-u(y'))^2}\phi_1(z'',y', k)dz''dz'\\
&\leq Ck^2|y-y'|^2\phi_1(y, y', k),\\
|\pa_y\phi_1(y, y', k)|
&=k^2\left|\int_{y'}^{y}\f{\th(y)(u(z'')-u(y'))^2}{\th(z'')(u(y)-u(y'))^2}\phi_1(z'',y',k)dz''\right|\\
&\leq k^2|y-y'|\phi_1(y, y', k),
\end{align*}
which gives \eqref{eq: est5} and part of \eqref{eq: est3}.  

We introduce $f(y, y', k)=\f{\pa_y\phi_1(y, y', k)}{\th(y)\phi_1(y, y', k)}$ then $f(y, y', k)\geq 0$ for $y\geq y'$ and $f(y, y', k)\leq 0$ for $y\leq y'$. We also have 
\beno
f'+\f{2u'}{u(y)-u(y')}f+f^2\th=\f{k^2}{\th},
\eeno
and $f'(y', y', k)=\f{k^2}{3\th(y')}>0$. By the continuity of $f$ and  the fact that $\f{2u'}{u(y)-u(y')}f\geq 0$, using a classical contradiction argument, we have $|f(y, y', k)|\leq \f{|k|}{\th(y)}$, which gives \eqref{eq: est3}. We also get that 
\beq\label{eq: f'}
|f'(y, y', k)|\leq Ck^2
\eeq
and
\begin{align*}
\f{\phi_1(z'',y',k)}{\phi_1(y, y', k)}=\exp\Big(\int_y^{z''}\th(z')f(z', y', k)dz'\Big)\geq e^{-|k||y-z''|}. 
\end{align*}
Thus we have for $|y-y'|\leq \f{1}{|k|}$, $1\leq \phi_1(z'', y, k)\leq \phi_1(y, y', k)\leq C$
\begin{align*}
\left|\f{\pa_y\phi_1(y, y', k)}{\phi_1(y, y', k)}\right|
&=k^2\left|\int_{y'}^{y}\f{\th(y)(u(z'')-u(y'))^2}{\th(z'')(u(y)-u(y'))^2}\f{\phi_1(z'',y',k)}{\phi_1(y, y', k)}dz''\right|\\
&\geq C^{-1}k^2|y-y'|
\end{align*}
and for $|y-y'|\geq \f{1}{|k|}$ and if $y>y'$
\begin{align*}
\left|\f{\pa_y\phi_1(y, y', k)}{\phi_1(y, y', k)}\right|
&=k^2\int_{y'}^{y}\f{\th(y)(u(z'')-u(y'))^2}{\th(z'')(u(y)-u(y'))^2}\f{\phi_1(z'',y',k)}{\phi_1(y, y', k)}dz''\\
&\geq C^{-1}k^2\int_{y'}^{y}\f{(u(z'')-u(y'))^2}{(u(y)-u(y'))^2}e^{-|k||y-z''|}dz''\\
&\geq C^{-1}k^2\int_{y-\f{1}{2|k|}}^{y}\f{(u(z'')-u(y'))^2}{(u(y)-u(y'))^2}e^{-|k||y-z''|}dz''\geq C^{-1}|k|,
\end{align*}
and if $y<y'$, 
\begin{align*}
\left|\f{\pa_y\phi_1(y, y', k)}{\phi_1(y, y', k)}\right|
&\geq C^{-1}k^2\int^{y+\f{1}{2|k|}}_{y}\f{(u(z'')-u(y'))^2}{(u(y)-u(y'))^2}e^{-|k||y-z''|}dz''\geq C^{-1}|k|.
\end{align*}
We now use 
\begin{align*}
\phi_1(y,y',k)=\exp\Big(\int_{y'}^{y}\f{\pa_y\phi_1(z', y', k)}{\phi_1(z', y', k)}dz'\Big)\geq e^{C^{-1}|k||y-y'|},
\end{align*}
which give \eqref{eq: est2}. The estimate \eqref{eq: est4} follows directly from \eqref{eq: f'} and \eqref{eq: est4}. The estimate \eqref{eq: est6} follows from the fact that $\phi_1(y,y', \ep,k)$ is continuous in $\ep$ and \eqref{eq: est2}. Thus we proved the proposition. 
\end{proof}
\subsection{Inhomogeneous solution}
In this section, we give the representation formula of the inhomogeneous solution to \eqref{eq: inhomRay1} and study the limit as $\ep \to 0$. Let us first introduce 
\beno
\rho(y')&=(u(y')-u(0))(u(1)-u(y')),
\eeno
and
\begin{subequations}\label{eq:j_1}
\beq
\begin{aligned}
J_1(y',k)&\eqdef-(u(1)-u(y'))\f{\th(0)}{u'(0)}-(u(y')-u(0))\f{\th(1)}{u'(1)}\\
&\quad+\rho(y')\int_{0}^1\f{\th(z')}{(u(z')-u(y'))^2}\Big(\f{1}{\phi_1(z',y',k)^2}-1\Big)dz'\\
&\quad-\rho(y')\mathcal{H}\big[\big((\th\circ u^{-1})(u^{-1})'\big)'\chi_{[u(0),u(1)]}\big](u(y'))\\
&=-(u(1)-u(y'))\f{\th(0)}{u'(0)}-(u(y')-u(0))\f{\th(1)}{u'(1)}\\
&\quad+\rho(y')\int_{0}^1\f{\th(z')}{(u(z')-u(y'))^2}\Big(\f{1}{\phi_1(z',y',k)^2}-1\Big)dz'\\
&\quad+\rho(y')P.V.\int_0^1\f{1}{u(z)-u(y')}\Big(\f{\th(z)}{u'(z)}\Big)'dz,
\end{aligned}
\eeq
and
\beq
\begin{aligned}
J_2(y')= \rho(y')\big((\th\circ u^{-1})(u^{-1})\big)'(u(y'))
=- \rho(y')\f{\th(y')^2}{u'(y')^3}\Big(\f{u'}{\th}\Big)'(y').
\end{aligned}
\eeq
\end{subequations}
We then have the following lemma. 
\begin{lemma}\label{lem: embedded}
Let $y'\in [0,1]$ and $k\in \mathbb{Z}\setminus\{0\}$. Then $u(y')\in\mathrm{Ran}\, u$ is an embedded eigenvalue of $\tilde{\mathcal{L}}_k$, if and only if 
\beno
J_1(y',k)^2+J_2(y')^2=0.
\eeno
\end{lemma}
\begin{proof}
Suppose $J_1(y',k)^2+J_2(y')^2=0$, then $\Big(\f{u'}{\th}\Big)'(y')=0$ and $y'\neq 0,1$. We have for $0\leq y< y'$,
\begin{align*}
\varphi_1(y,y')=\phi(0,y',k)\phi(y,y',k) \int_0^{y}\f{\th(z)}{\phi(z,y',k)^2}dz
\end{align*}
solves the homogeneous equation \eqref{eq: homRay1} with boundary condition $\varphi_1(0,y')=0$ and $\pa_y\varphi_1(0,y')=1$. Next, we show that $\varphi_1$ can naturally $C^2$ extend to $y'\leq y\leq 1$ under the assumption $\Big(\f{u'}{\th}\Big)'(y')=0$. A direct calculation gives for $0\leq y< y'$, 
\begin{align*}
\varphi_1(y,y')=&\phi(0,y',k)\phi(y,y',k) \int_0^{y}\f{\th(z)}{(u(z)-u(y'))^2}dz\\
&+\phi(0,y',k)\phi(y,y',k) \int_0^{y}\f{\th(z)}{(u(z)-u(y'))^2}\Big(\f{1}{\phi_1(z,y',k)^2}-1\Big)dz\\
=&\phi(0,y',k)\phi(y,y',k) \int_{u(0)}^{u(y)}\f{(\th\circ u^{-1})(\tilde{u})(u^{-1})'(\tilde{u})}{(\tilde{u}-u(y'))^2}d\tilde{u}\\
&+\phi(0,y',k)\phi(y,y',k) \int_0^{y}\f{\th(z)}{(u(z)-u(y'))^2}\Big(\f{1}{\phi_1(z,y',k)^2}-1\Big)dz\\
=&\phi(0,y',k)\phi(y,y',k) \int_{u(0)}^{u(y)}\f{(\th\circ u^{-1})(\tilde{u})(u^{-1})'(\tilde{u})-(\th\circ u^{-1})(u(y'))(u^{-1})'(u(y'))}{(\tilde{u}-u(y'))^2}d\tilde{u}\\
&-\f{\phi(0,y',k)\phi(y,y',k)(\th\circ u^{-1})(u(y'))(u^{-1})'(u(y'))}{\tilde{u}-u(y')}\bigg|_{u(0)}^{u(y)}\\
&+\phi(0,y',k)\phi(y,y',k) \int_0^{y}\f{\th(z)}{(u(z)-u(y'))^2}\Big(\f{1}{\phi_1(z,y',k)^2}-1\Big)dz=I_1(y)+I_2(y)+I_3(y).
\end{align*}
Note that the assumption $\Big(\f{u'}{\th}\Big)'(y')=0$ gives $\big((\th\circ u^{-1})(u^{-1})'\big)'(u(y'))=0$, which gives that the integral in $I_1$
\begin{align*}
&\int_{u(0)}^{u(y)}\f{[(\th\circ u^{-1})(u^{-1})'](\tilde{u})-[(\th\circ u^{-1})(u^{-1})'](u(y'))}{(\tilde{u}-u(y'))^2}d\tilde{u}\\
&=\int_{u(0)}^{u(y)}\f{[(\th\circ u^{-1})(u^{-1})'](\tilde{u})-[(\th\circ u^{-1})(u^{-1})'](u(y'))-\big((\th\circ u^{-1})(u^{-1})'\big)'(u(y'))(\tilde{u}-u(y'))}{(\tilde{u}-u(y'))^2}d\tilde{u}
\end{align*}
is well defined for all $y\in [0,1]$ and $I_j(y)\in C^2$, for $j=1,2,3$. It is easy to check that $I_1+I_2+I_3$ is also well defined for all $y\in [0,1]$ which also solves homogeneous equation \eqref{eq: homRay1} with $c=u(y')$. We also use $\varphi_1(y,y')$ to represent its extension. A direct calculation gives
\beno
\varphi_1(1,y')=I_1(1)+I_2(1)+I_3(1)=-\phi_1(0,y',k)\phi_1(1,y',k)J_1(y',k)=0.
\eeno
Thus $\varphi_1(y,y')$ is an eigenfunction associated with the embedded eigenvalue $u(y')$. 

\medskip

Suppose $u(y')$ is an embedded eigenvalue and $\Phi^{H}\in H^2$ is the corresponding eigenfunction. Then $\pa_y\big(\f{1}{\th}\pa_y\Phi^{H}\big)=\f{k^2}{\th}\Phi^{H}+\f{\big(\f{u'}{\th}\big)'\Phi^{H}(y)}{u(y)-u(y')}\in L^2$. Thus either $\Phi^{H}(y')=0$ or $\big(\f{u'}{\th}\big)'(y')=0$. If $\Phi^{H}(y')=0$, we have
\begin{align*}
0&=\int_0^1\f{1}{\th}(|\pa_y\Phi^{H}|^2+k^2|\Phi^{H}|^2)dy
+\int_0^1\f{\big(\f{u'}{\th}\big)'|\Phi^{H}(y)|^2}{u(y)-u(y')}dy\\
&=\int_0^1\f{1}{\th}(|\pa_y\Phi^{H}|^2+k^2|\Phi^{H}|^2)dy
-2\mathrm{Re}\,\int_0^1\f{\big(\f{u'}{\th}\big)\pa_y\Phi^{H}(y)\overline{\Phi^{H}(y)}}{u(y)-u(y')}dy
+\int_0^1\f{\big(\f{u'}{\th}\big)u'|\Phi^{H}(y)|^2}{(u(y)-u(y'))^2}dy\\
&=\int_0^1\f{1}{\th}\left|\pa_y\Phi^{H}-\f{u'\Phi^{H}(y)}{u(y)-u(y')}\right|^2+\f{k^2}{\th}|\Phi^{H}|^2dy,
\end{align*}
which gives $\Phi^{H}\equiv 0$. Thus $\big(\f{u'}{\th}\big)'(y')=0$ which means $J_2(y')=0$. Then by the argument in the first part, $\varphi_1(y, y')$ is a $C^2$ solution that solves the homogeneous equation \eqref{eq: homRay1} with $c=u(y')$. Then $\varphi_1(y, y')$ is an eigenfunction if and only if $\varphi_1(1,y')=0$, which gives $J_1(y',k)=0$. 
Thus we proved the lemma. 
\end{proof}
By following a similar idea, we have the following lemma. 
\begin{lemma}\label{eq:Wronskian}
Let $k\in \mathbb{Z}\setminus \{0\}$ and $c=u(y')+i\ep$ with $0<|\ep|\leq \ep_0$. Then $c$ is an eigenvalue of $\mathcal{\tilde{L}}_k$ if and only if 
\beno
\int_{0}^1\f{\th(z')}{\phi(z',y',\ep,k)^2}dz'=0
\eeno
\end{lemma}
\begin{proof}
The function 
\beno
\varphi(y, y', \ep, k)=\phi(0, y', \ep, k)\phi(y, y', \ep, k)\int_{0}^y\f{\th(z')}{\phi(z',y',\ep,k)^2}dz'\in H^2
\eeno
is a solution to the homogeneous equation \eqref{eq: homRay1} with $c=u(y')+i\ep$ with boundary value $\varphi(0, y', \ep, k)=0$ and $\pa_y\varphi(0, y', \ep, k)=1$. The integral is well-defined for all $y\in [0,1]$ by taking $\ep_0$ small enough. Thus $\varphi(y, y', \ep, k)$ is an eigenfunction associated with $c=u(y')+i\ep$ if and only if $\varphi(1, y', \ep, k)=0$. The lemma follows directly from the fact that $\phi(0, y', \ep, k)\phi(1, y', \ep, k)\neq 0$ for $\ep\neq 0$. 
\end{proof}
\begin{lemma}\label{eq: J_1^2+J_2^2}
Under the assumption that $\mathcal{\tilde{L}}_k$ has no embedded eigenvalue, we have 
\beno
C^{-1}(1+k^2\rho(y')^2)\leq J_1(y',k)^2+\pi^2J_2(y')^2\leq  C(1+k^2\rho(y')^2).
\eeno
\end{lemma}
\begin{proof}
For $|k|\leq M_0$, Lemma \ref{lem: embedded} implies 
\beno
J_1(y',k)^2+\pi^2J_2(y')^2\approx_{M_0} 1. 
\eeno
Let us now consider the case $|k|>M_0$ with $M_0$ large enough. By Proposition \ref{prop: homsol}, we have 
\begin{align*}
&\int_{0}^1\f{\th(z')}{(u(z')-u(y'))^2}\Big(\f{1}{\phi_1(z',y',k)^2}-1\Big)dz'\\
&\leq -C^{-1}\int_{|z'-y'|\leq \f{1}{|k|}}\f{1}{(z'-y')^2}\left|\int_{z'}^{y'}\pa_y\phi_1(z'', y', k)dz''\right|dz'\\
&\leq -C^{-1}\int_{|z'-y'|\leq \f{1}{|k|}}k^2dz'\leq -C^{-1}|k|.
\end{align*}
It is easy to check that 
\begin{align*}
&\left|\int_{0}^1\f{\th(z')}{(u(z')-u(y'))^2}\Big(\f{1}{\phi_1(z',y',k)^2}-1\Big)dz'\right|\\
&\leq C\int_0^1\f{1}{(z'-y')^2}\min\{k^2|z'-y'|^2,1\}dz'
\leq C|k|,
\end{align*}
and
\beno
\left|\mathcal{H}\big[\big((\th\circ u^{-1})(u^{-1})'\big)'\chi_{[u(0),u(1)]}\big](u(y'))\right|\leq C\|\big((\th\circ u^{-1})(u^{-1})'\big)'\|_{H^1}.
\eeno
Thus we have 
\begin{align*}
J_1(y', k)^2
&\geq \Big(-(u(1)-u(y'))\f{\th(0)}{u'(0)}-(u(y')-u(0))\f{\th(1)}{u'(1)}\Big)^2\\
&\quad+\rho(y')^2\left|\int_{0}^1\f{\th(z')}{(u(z')-u(y'))^2}\Big(\f{1}{\phi_1(z',y',k)^2}-1\Big)dz'\right|^2\\
&\quad -\Big(\rho(y')\mathcal{H}\big[\big((\th\circ u^{-1})(u^{-1})'\big)'\chi_{[u(0),u(1)]}\big](u(y'))\Big)^2\\
&\geq C^{-1}(1+k^2\rho(y')^2)
\end{align*}
and
\begin{align*}
J_1(y', k)^2
&\leq \Big(-(u(1)-u(y'))\f{\th(0)}{u'(0)}-(u(y')-u(0))\f{\th(1)}{u'(1)}\Big)^2\\
&\quad+\rho(y')^2\left|\int_{0}^1\f{\th(z')}{(u(z')-u(y'))^2}\Big(\f{1}{\phi_1(z',y',k)^2}-1\Big)dz'\right|^2\\
&\quad +\Big(\rho(y')\mathcal{H}\big[\big((\th\circ u^{-1})(u^{-1})'\big)'\chi_{[u(0),u(1)]}\big](u(y'))\Big)^2\\
&\leq C(1+k^2\rho(y')^2). 
\end{align*}
Thus we prove the lemma. 
\end{proof}
Now let us study the inhomogeneous equation \eqref{eq: inhomRay1}. By using the homogeneous solution, we rewrite it as 
\beno
\pa_y\left(\f{\phi^2}{\th}\pa_y\left(\f{\Psi^{I}}{\phi}\right)\right)=-\hat{\tilde{\om}}_{in}\phi_1.
\eeno
Thus,  for $c=u(y')+i\ep$ with $y'\in [0,1]$ and $0<|\ep|\leq \ep_0$, it holds that
\beq\label{eq: inhomsol}
\begin{aligned}
\Psi^{I}(y, y', \ep, k)
&=-\phi(y, y', \ep, k)\int_{0}^y\f{\th(z')\int_{y'}^{z'}\hat{\tilde{\om}}_{in}(z'',k)\phi_1(z'',y',\ep,k)dz''}{(u(z')-u(y')-i\ep)^2\phi_1(z',y',\ep,k)^2}dz'\\
&\quad+\f{\int_{0}^1\f{\th(z')\int_{y'}^{z'}\hat{\tilde{\om}}_{in}(z'',k)\phi_1(z'',y',\ep,k)dz''}{(u(z')-u(y')-i\ep)^2\phi_1(z',y',\ep,k)^2}dz'}{\int_{0}^1\f{\th(z')}{(u(z')-u(y')-i\ep)^2\phi_1(z',y',\ep,k)^2}dz'}\phi(y, y', \ep, k)\int_{0}^y\f{\th(z')}{\phi(z',y',\ep,k)^2}dz'\\
&=-\phi(y, y', \ep, k)\int_{1}^y\f{\th(z')\int_{y'}^{z'}\hat{\tilde{\om}}_{in}(z'',k)\phi_1(z'',y',\ep,k)dz''}{(u(z')-u(y')-i\ep)^2\phi_1(z',y',\ep,k)^2}dz'\\
&\quad+\f{\int_{0}^1\f{\th(z')\int_{y'}^{z'}\hat{\tilde{\om}}_{in}(z'',k)\phi_1(z'',y',\ep,k)dz''}{(u(z')-u(y')-i\ep)^2\phi_1(z',y',\ep,k)^2}dz'}{\int_{0}^1\f{\th(z')}{(u(z')-u(y')-i\ep)^2\phi_1(z',y',\ep,k)^2}dz'}\phi(y, y', \ep, k)\int_{1}^y\f{\th(z')}{\phi(z',y',\ep,k)^2}dz'.
\end{aligned}
\eeq
Note that Lemma \ref{eq:Wronskian} ensures that the above representation formula is well-defined by taking $\ep_0$ small enough. 
By the continuity of $\phi(y,y',\ep,k)$ in $\ep$, it holds that 
\begin{align*}
&\lim_{\ep\to 0}\phi(y, y', \ep, k)\int_{1}^y\f{\th(z')}{\phi(z',y',\ep,k)^2}dz'=\phi(y, y', k)\int_{1}^y\f{\th(z')}{\phi(z',y',k)^2}dz'\quad \text{for}\quad y'<y\leq 1\\
&\lim_{\ep\to 0}\phi(y, y', \ep, k)\int_{0}^y\f{\th(z')}{\phi(z',y',\ep,k)^2}dz'=\phi(y, y', k)\int_{0}^y\f{\th(z')}{\phi(z',y',k)^2}dz'\quad \text{for}\quad 0\leq y<y'\\
&\lim_{\ep\to 0}\phi(y, y', \ep, k)\int_{j}^y\f{\th(z')\int_{y'}^{z'}\hat{\tilde{\om}}_{in}(z'',k)\phi_1(z'',y',\ep,k)dz''}{(u(z')-u(y')-i\ep)^2\phi_1(z',y',\ep,k)^2}dz'\\
&\quad =\phi(y, y', k)\int_{j}^y\f{\th(z')\int_{y'}^{z'}\hat{\tilde{\om}}_{in}(z'',k)\phi_1(z'',y',k)dz''}{(u(z')-u(y'))^2\phi_1(z',y',k)^2}dz'\quad \text{for}\quad j=0,1,\ y\in [0,1].
\end{align*}
We have for $y'\in (0,1)$
\begin{align*}
&\int_{0}^1\f{\th(z')}{(u(z')-u(y')-i\ep)^2\phi_1(z',y',\ep,k)^2}dz'\\
&=\int_{0}^1\f{\th(z')}{(u(z')-u(y')-i\ep)^2}dz'+\int_{0}^1\f{\th(z')}{(u(z')-u(y')-i\ep)^2}\Big(\f{1}{\phi_1(z',y',\ep,k)^2}-1\Big)dz'\\
&=\int_{u(0)}^{u(1)}\f{(\th\circ u^{-1})(u)(u^{-1})'(u)}{(u-u(y')-i\ep)^2}du
+\int_{0}^1\f{\th(z')}{(u(z')-u(y')-i\ep)^2}\Big(\f{1}{\phi_1(z',y',\ep,k)^2}-1\Big)dz'\\
&=\f{-(\th\circ u^{-1})(u)(u^{-1})'(u)}{(u-u(y')-i\ep)}\bigg|_{u(0)}^{u(1)}+\int_{u(0)}^{u(1)}\f{\big((\th\circ u^{-1})(u^{-1})\big)'(u)}{u-u(y')-i\ep}du\\
&\quad +\int_{0}^1\f{\th(z')}{(u(z')-u(y')-i\ep)^2}\Big(\f{1}{\phi_1(z',y',\ep,k)^2}-1\Big)dz'\\
&\to \f{\th(0)(u^{-1})'(u(0))}{(u(0)-u(y'))}-\f{\th(1)(u^{-1})'(u(1))}{(u(1)-u(y'))}\\
&\quad -\mathcal{H}\big[\big((\th\circ u^{-1})(u^{-1})'\big)'\chi_{[u(0),u(1)]}\big](u(y'))\pm i\pi \big((\th\circ u^{-1})(u^{-1})\big)'(u(y'))\\
&\quad +\int_{0}^1\f{\th(z')}{(u(z')-u(y'))^2}\Big(\f{1}{\phi_1(z',y',k)^2}-1\Big)dz'
\quad \text{as}\quad \ep\to 0^{\pm},
\end{align*}
where $\mathcal{H}[f](u)=P.V.\int\f{f(u')}{u-u'}du'$ is the Hilbert transform and 
$\chi_{[u(0),u(1)]}(u)=\left\{
\begin{aligned}&1,\ u\in [u(0),u(1)]\\
&0,\ \text{else},
\end{aligned}\right.$ is the characteristic function. 
We also have for $y'\in (0,1)$,
\begin{align*}
&\int_{0}^1\f{\th(z')\int_{y'}^{z'}\hat{\tilde{\om}}_{in}(z'',k)\phi_1(z'',y',\ep,k)dz''}{(u(z')-u(y')-i\ep)^2\phi_1(z',y',\ep,k)^2}dz'\\
&=\int_{0}^1\f{\th(z')\int_{y'}^{z'}\hat{\tilde{\om}}_{in}(z'',k)dz''}{(u(z')-u(y')-i\ep)^2}dz'
+\int_{0}^1\int_{y'}^{z'}\f{\th(z')\hat{\tilde{\om}}_{in}(z'',k)}{(u(z')-u(y')-i\ep)^2}\Big(\f{\phi_1(z'',y',\ep,k)}{\phi_1(z',y',\ep,k)^2}-1\Big)dz''dz'\\
&=\int_{u(0)}^{u(1)}\f{\int_{u(y')}^{\tilde{u}}\hat{\tilde{\om}}_{in}(u^{-1}(\tilde{\tilde{u}}),k)(u^{-1})'(\tilde{\tilde{u}})d\tilde{\tilde{u}}}{(\tilde{u}-u(y')-i\ep)^2}(\th\circ u^{-1})(\tilde{u})(u^{-1})'(\tilde{u})d\tilde{u}\\
&\quad +\int_{0}^1\int_{y'}^{z'}\f{\th(z')\hat{\tilde{\om}}_{in}(z'',k)}{(u(z')-u(y')-i\ep)^2}\Big(\f{\phi_1(z'',y',\ep,k)}{\phi_1(z',y',\ep,k)^2}-1\Big)dz''dz'\\
&=-\f{(\th\circ u^{-1})(\tilde{u})\int_{u(y')}^{\tilde{u}}\hat{\tilde{\om}}_{in}(u^{-1}(\tilde{\tilde{u}}),k)(u^{-1})'(\tilde{\tilde{u}})d\tilde{\tilde{u}}}{(\tilde{u}-u(y')-i\ep)}(u^{-1})'(\tilde{u})\bigg|_{u(0)}^{u(1)}\\
&\quad+\int_{u(0)}^{u(1)}\f{(\th\circ u^{-1})(\tilde{u})\hat{\tilde{\om}}_{in}(u^{-1}({\tilde{u}}),k)\big((u^{-1})'({\tilde{u}})\big)^2}{(\tilde{u}-u(y')-i\ep)}d\tilde{u}\\
&\quad+\int_{u(0)}^{u(1)}\f{\int_{u(y')}^{\tilde{u}}\hat{\tilde{\om}}_{in}(u^{-1}(\tilde{\tilde{u}}),k)(u^{-1})'(\tilde{\tilde{u}})d\tilde{\tilde{u}}}{(\tilde{u}-u(y')-i\ep)}\Big((\th\circ u^{-1})(\tilde{u})(u^{-1})'\Big)'(\tilde{u})d\tilde{u}\\
&\quad +\int_{0}^1\int_{y'}^{z'}\f{\th(z')\hat{\tilde{\om}}_{in}(z'',k)}{(u(z')-u(y')-i\ep)^2}\Big(\f{\phi_1(z'',y',\ep,k)}{\phi_1(z',y',\ep,k)^2}-1\Big)dz''dz'\\
&\xrightarrow{\ep\to 0\pm} -\f{(\th\circ u^{-1})(\tilde{u})\int_{u(y')}^{\tilde{u}}\hat{\tilde{\om}}_{in}(u^{-1}(\tilde{\tilde{u}}),k)(u^{-1})'(\tilde{\tilde{u}})d\tilde{\tilde{u}}}{(\tilde{u}-u(y'))}(u^{-1})'(\tilde{u})\bigg|_{u(0)}^{u(1)}\\
&\quad\quad\quad -\mathcal{H}\big[(\th\circ u^{-1})\hat{\tilde{\om}}_{in}(u^{-1}(\cdot),k)\big((u^{-1})'\big)^2\chi_{[u(0),u(1)]}\big](u(y'))
\pm i\pi\f{\hat{\tilde{\om}}_{in}(y',k)\th(y')}{(u'(y'))^2}\\
&\quad\quad\quad+\int_{u(0)}^{u(1)}\f{\int_{u(y')}^{\tilde{u}}\hat{\tilde{\om}}_{in}(u^{-1}(\tilde{\tilde{u}}),k)(u^{-1})'(\tilde{\tilde{u}})d\tilde{\tilde{u}}}{(\tilde{u}-u(y'))}\Big((\th\circ u^{-1})(\tilde{u})(u^{-1})'\Big)'(\tilde{u})d\tilde{u}\\
&\quad\quad\quad +\int_{0}^1\int_{y'}^{z'}\f{\th(z')\hat{\tilde{\om}}_{in}(z'',k)}{(u(z')-u(y'))^2}\Big(\f{\phi_1(z'',y',k)}{\phi_1(z',y',k)^2}-1\Big)dz''dz'\\
&\quad\quad=P.V.\int_{0}^1\f{\th(z')\int_{y'}^{z'}\hat{\tilde{\om}}_{in}(z'',k)dz''}{(u(z')-u(y'))^2}dz'\\
&\quad\quad\quad +\int_{0}^1\int_{y'}^{z'}\f{\th(z')\hat{\tilde{\om}}_{in}(z'',k)}{(u(z')-u(y'))^2}\Big(\f{\phi_1(z'',y',k)}{\phi_1(z',y',k)^2}-1\Big)dz''dz'
\pm i\pi\f{\th(y')\hat{\tilde{\om}}_{in}(y',k)}{(u'(y'))^2}\\
&\quad\quad\eqdef \Pi_1[\hat{\tilde{\om}}_{in}]\pm i\pi \Pi_2[\hat{\tilde{\om}}_{in}].
\end{align*}
\begin{lemma}\label{lem: L^2 est}
It holds that 
\beno
\|\Pi_1[\hat{\tilde{\om}}_{in}]\|_{L^2}\leq C\|\hat{\tilde{\om}}_{in}\|_{L^2},
\eeno
where $C>0$ is a constant independent of $k$. 
\end{lemma}
\begin{proof}
By the $L^2$ estimate of the Hilbert transform and the maximum function, we have 
\beno
\left\|-\mathcal{H}\big[(\th\circ u^{-1})\hat{\tilde{\om}}_{in}(u^{-1}(\cdot),k)\big((u^{-1})'\big)^2\chi_{[u(0),u(1)]}\big](u(y'))\right\|_{L^2}\leq C\|\hat{\tilde{\om}}_{in}\|_{L^2},\\
\left\|-\f{(\th\circ u^{-1})(\tilde{u})\int_{u(y')}^{\tilde{u}}\hat{\tilde{\om}}_{in}(u^{-1}(\tilde{\tilde{u}}),k)(u^{-1})'(\tilde{\tilde{u}})d\tilde{\tilde{u}}}{(\tilde{u}-u(y'))}(u^{-1})'(\tilde{u})\bigg|_{u(0)}^{u(1)}\right\|_{L^2}\leq C\|\hat{\tilde{\om}}_{in}\|_{L^2},
\eeno
and
\beno
\left\|\int_{u(0)}^{u(1)}\f{\int_{u(y')}^{\tilde{u}}\hat{\tilde{\om}}_{in}(u^{-1}(\tilde{\tilde{u}}),k)(u^{-1})'(\tilde{\tilde{u}})d\tilde{\tilde{u}}}{(\tilde{u}-u(y'))}\Big((\th\circ u^{-1})(\tilde{u})(u^{-1})'\Big)'(\tilde{u})d\tilde{u}\right\|_{L^2}\leq  C\|\hat{\tilde{\om}}_{in}\|_{L^2}.
\eeno
We write
\begin{align*}
&\int_{0}^1\int_{y'}^{z'}\f{\th(z')\hat{\tilde{\om}}_{in}(z'',k)}{(u(z')-u(y'))^2}\Big(\f{\phi_1(z'',y',k)}{\phi_1(z',y',k)^2}-1\Big)dz''dz'\\
&=\int_{|u(z')-u(y')|\leq \f{2}{|k|}}\int_{y'}^{z'}\f{\th(z')\hat{\tilde{\om}}_{in}(z'',k)}{(u(z')-u(y'))^2}\Big(\f{\phi_1(z'',y',k)}{\phi_1(z',y',k)^2}-1)\Big)dz''dz'\\
&\quad-\int_{|u(z')-u(y')|\geq \f{2}{|k|}}\f{\th(z')\int_{y'}^{z'}\hat{\tilde{\om}}_{in}(z'',k)dz''}{(u(z')-u(y'))^2}dz'\\
&\quad+\int_{|u(z')-u(y')|\geq \f{2}{|k|}}\int_{y'}^{z'}\f{\th(z')\hat{\tilde{\om}}_{in}(z'',k)}{(u(z')-u(y'))^2}\f{\phi_1(z'',y',k)}{\phi_1(z',y',k)^2}dz''dz'\\
&=I_1+I_2+I_3.
\end{align*}
By Proposition \ref{prop: homsol}, we have
\begin{align*}
&\|I_1\|_{L^2}\leq C\left\|\int_{|u(z')-u(y')|\leq \f{2}{|k|}}\f{|k|}{|y'-z'|}\int_{y'}^{z'}|\hat{\tilde{\om}}_{in}(z'',k)|dz''dz'\right\|_{L^2}\leq C\|\hat{\tilde{\om}}_{in}\|_{L^2},\\
&\|I_3\|_{L^2}\leq C\left\|\int_{|u(z')-u(y')|\geq \f{2}{|k|}}\f{|k|}{|y'-z'|}\int_{y'}^{z'}|\hat{\tilde{\om}}_{in}(z'',k)|dz''e^{-C^{-1}|k||y'-z'|}dz'\right\|_{L^2}\leq C\|\hat{\tilde{\om}}_{in}\|_{L^2}.
\end{align*}
We have 
\begin{align*}
I_2&=-\chi_{\{|u(y')-u(0)|\geq \f{2}{|k|}\}}(u(y'))\int_{u(0)}^{u(y')-\f{2}{|k|}}\f{(\th\circ u^{-1})(u)(u^{-1})'(u)\int_{u(y')}^{u}\hat{\tilde{\om}}_{in}(u^{-1}(\tilde{u}),k)(u^{-1})'d\tilde{u}}{(u-u(y'))^2}du\\
&\quad-\chi_{\{|u(y')-u(1)|\geq \f{2}{|k|}\}}(u(y'))\int_{u(y')+\f{2}{|k|}}^{u(1)}\f{(\th\circ u^{-1})(u)(u^{-1})'(u)\int_{u(y')}^{u}\hat{\tilde{\om}}_{in}(u^{-1}(\tilde{u}),k)(u^{-1})'d\tilde{u}}{(u-u(y'))^2}du\\
&=\chi_{\{|u(y')-u(0)|\geq \f{2}{|k|}\}}(u(y'))\f{(\th\circ u^{-1})(u)(u^{-1})'(u)\int_{u(y')}^{u}\hat{\tilde{\om}}_{in}(u^{-1}(\tilde{u}),k)(u^{-1})'d\tilde{u}}{(u-u(y'))}\bigg|_{u(0)}^{u(y')-\f{2}{|k|}}\\
&\quad+\chi_{\{|u(y')-u(1)|\geq \f{2}{|k|}\}}(u(y'))\f{(\th\circ u^{-1})(u)(u^{-1})'(u)\int_{u(y')}^{u}\hat{\tilde{\om}}_{in}(u^{-1}(\tilde{u}),k)(u^{-1})'d\tilde{u}}{(u-u(y'))}\bigg|_{u(y')+\f{2}{|k|}}^{u(1)}\\
&\quad-\int_{|u-u(z')|\geq \f{2}{|k|}}\f{\Big((\th\circ u^{-1})(u)(u^{-1})'(u)\Big)'\int_{u(y')}^{u}\hat{\tilde{\om}}_{in}(u^{-1}(\tilde{u}),k)(u^{-1})'d\tilde{u}}{u-u(y')}du\\
&\quad-\int_{|u-u(z')|\geq \f{2}{|k|}}\f{(\th\circ u^{-1})(u)((u^{-1})'(u))^2\hat{\tilde{\om}}_{in}(u^{-1}(u),k)}{u-u(y')}du.
\end{align*}
By the $L^2$ estimate of the maximal Hilbert transform, and the maximum function, we have 
\beno
\|I_2\|_{L^2(0,1)}\leq C\|\hat{\tilde{\om}}_{in}\|_{L^2}. 
\eeno
Combining the above estimates, we prove the lemma. 
\end{proof}
We conclude the above calculations by the following proposition. 
\begin{proposition}\label{prop: limit}
It holds that $\lim\limits_{\ep\to 0^{\pm}}\Psi^{I}(y, y', \ep, k)=\Psi_{\pm}^{I}(y, y', k)$.
where 
\beno
\Psi_{\pm}^{I}(y, y', k)=\left\{
\begin{aligned}
&-\phi(y, y', k)\int_{0}^y\f{\th(z')\int_{y'}^{z'}\hat{\tilde{\om}}_{in}(z'',k)\phi_1(z'',y',k)dz''}{(u(z')-u(y'))^2\phi_1(z',y',k)^2}dz'\\
&\quad +\f{\rho(y')(\Pi_1[\hat{\tilde{\om}}_{in}]\pm i\pi \Pi_2[\hat{\tilde{\om}}_{in}])}{J_1(y',k)\pm i\pi J_2(y') }\phi(y, y', k)\int_{0}^y\f{\th(z')}{\phi(z',y',k)^2}dz',\quad \text{for}\quad y<y'\\
&-\phi(y, y', k)\int_{1}^y\f{\th(z')\int_{y'}^{z'}\hat{\tilde{\om}}_{in}(z'',k)\phi_1(z'',y',k)dz''}{(u(z')-u(y'))^2\phi_1(z',y',k)^2}dz'\\
&\quad +\f{\rho(y')(\Pi_1[\hat{\tilde{\om}}_{in}]\pm i\pi \Pi_2[\hat{\tilde{\om}}_{in}])}{J_1(y',k)\pm i\pi J_2(y') }\phi(y, y', k)\int_{1}^y\f{\th(z')}{\phi(z',y',k)^2}dz'\quad \text{for}\quad y'<y.
\end{aligned}\right.
\eeno
\end{proposition}

We have the following uniform estimate. 
\begin{proposition}\label{prop: uniform bound}
For any fixed $k\in \mathbb{Z}\setminus\{0\}$, suppose that $\mathcal{\tilde{L}}_k$ has no eigenvalue or embedded eigenvalue. Then for all $c\notin \mathrm{Ran}\, u$, let $\Psi^{I}$ solve \eqref{eq: inhomRay1} with boundary condition $\Psi^{I}(0)=\Psi^{I}(1)=0$. Then it holds that 
\ben\label{eq: uniform est}
\|\Psi^{I}\|_{H^1}\leq C\|\hat{\tilde{\om}}_{in}\|_{H^1},
\een
where $C$ is a constant independent of $c$. 
\end{proposition}
\begin{proof}
It is easy to check that for $c\in \{c: \mathrm{dist}(c,\mathrm{Ran}\, u)> M\}$ with $M$ large enough, it holds that
\beno
\|\Psi^{I}\|_{H^1}\leq C\|\hat{\tilde{\om}}_{in}\|_{L^2}.
\eeno
So we only focus on the case $c\in \{c: \mathrm{dist}(c,\mathrm{Ran}\, u)\leq M\}$. 
We prove the estimate \eqref{eq: uniform est} by a contradiction argument. Suppose \eqref{eq: uniform est} does not hold, namely, there is $(c_n, \Psi_n^{I}, \hat{\tilde{\om}}_{in, n})$ such that $\{c_n\}_{n\geq 1}\subset \{c: 0<\mathrm{dist}(c,\mathrm{Ran}\, u)\leq M\}$ and as $n\to \infty$, $\|\Psi_n^{I}\|_{H^1}=1$, $\|\hat{\tilde{\om}}_n\|_{H^1}\to 0$ and 
\beno
(u-c_n)\pa_y\left(\f{1}{\th}\pa_y\Psi_n^{I}\right)-(u-c_n)\f{k^2}{\th}\Psi_n^{I}-\big(\f{u'}{\th}\big)'\Psi_n^{I}=-\hat{\tilde{\om}}_{in, n}. 
\eeno 
Then there exists a subsequence also denoted by $(c_n, \Psi_n^{I}, \hat{\tilde{\om}}_{in, n})$, so that 
\begin{align*}
c_{n}\to c_{\infty},\quad \Psi_n^{I}\rightharpoonup \Psi_{\infty}^I\quad \hat{\tilde{\om}}_{in, n}\to 0\quad\text{in}\quad H^1.
\end{align*}
If $c_{\infty}\notin \mathrm{Ran}\, u$, then $\|\Psi_{n}^l\|_{H^3}\leq C(\|\hat{\tilde{\om}}_{in, n}\|_{H^1}+\|\Psi_{n}^l\|_{H^1})$, thus $\Psi_n^{I}\to \Psi_{\infty}^I$ and 
\beno
(u-c_{\infty})\pa_y\left(\f{1}{\th}\pa_y\Psi_{\infty}^{I}\right)-(u-c_{\infty})\f{k^2}{\th}\Psi_{\infty}^{I}-\big(\f{u'}{\th}\big)'\Psi_{\infty}^{I}=0. 
\eeno 
It implies that $c_{\infty}$ is an eigenvalue that leads to a contradiction. 

If $c_{\infty}=u(y_{\infty})\in \mathrm{Ran}\, u$ (without loss of generality, we assume that $\mathrm{Im}\, c_n>0$), then we have 
\begin{align*}
\pa_y\left(\f{1}{\th}\pa_y\Psi_n^{I}\right)
&=\f{k^2}{\th}\Psi_n^{I}+\f{\big(\f{u'}{\th}\big)'\Psi_n^{I}-\hat{\tilde{\om}}_{in, n}}{u-c_n}\\
&=\f{k^2}{\th}\Psi_n^{I}+\pa_y\left(\f{\big(\f{u'}{\th}\big)'\Psi_n^{I}-\hat{\tilde{\om}}_{in, n}}{u'}\ln (u-c_n)\right)-\pa_y\left(\f{\big(\f{u'}{\th}\big)'\Psi_n^{I}-\hat{\tilde{\om}}_{in, n}}{u'}\right)\ln (u-c_n)
\end{align*}
which gives that for $1<p<2$
\ben
\left\|\f{1}{\th}\pa_y\Psi_n^{I}-\f{\big(\f{u'}{\th}\big)'\Psi_n^{I}-\hat{\tilde{\om}}_{in, n}}{u'}\ln (u-c_n)\right\|_{L^2\cap \dot{W}^{1,p}}\leq C.
\een
Therefore, we have for any $\d_0>0$
\beno
\|\pa_y\Psi_n^{I}\|_{L^2((y_{\infty}-\d_0, y_{\infty}+\d_0)\cap(0,1))}\leq C\d_0^{\f13}
\eeno
Here the constant $C$ is independent of $n$. We also have 
\beno
\Psi_n^{I}\to \Psi_{\infty}^{I} \quad \text{in}\quad H^1(0,1)\setminus (y_{\infty}-\d_0, y_{\infty}+\d_0). 
\eeno
Thus we have 
\beno
\Psi_n^{I}\to \Psi_{\infty}^{I} \quad \text{in}\quad H^1(0,1). 
\eeno
For any $\varphi\in H^1_0$, we also have 
\begin{align*}
0&=-\int_0^1\f{\pa_y\Phi_n^I(y)\pa_y\varphi(y)}{\th(y)}dy-\int_0^1\f{\Phi_n^I(y)\varphi(y)}{\th(y)}dy
-\int_0^1\big(\f{u'}{\th}\big)'(y)\f{\Phi_n^{I}\varphi}{u-c_n}dy
+\int_0^1\f{\hat{\tilde{\om}}_{in,n}\varphi}{u-c_n}dy\\
&=-\int_0^1\f{\pa_y\Phi_n^I(y)\pa_y\varphi(y)}{\th(y)}dy-\int_0^1\f{\Phi_n^I(y)\varphi(y)}{\th(y)}dy
-\int_0^1\big(\f{u'}{\th}\big)'(y)\f{\Phi_n^{I}\varphi}{u'}\pa_y(\ln (u-c_n)) dy\\
&\quad+\int_0^1\f{\hat{\tilde{\om}}_{in,n}\varphi}{u'}\pa_y(\ln(u-c_n))dy\\
&=-\int_0^1\f{\pa_y\Phi_n^I(y)\pa_y\varphi(y)}{\th(y)}dy-\int_0^1\f{\Phi_n^I(y)\varphi(y)}{\th(y)}dy
+\int_0^1\pa_y\left(\big(\f{u'}{\th}\big)'(y)\f{\Phi_n^{I}\varphi}{u'}\right)\ln (u-c_n) dy\\
&\quad -\big(\f{u'}{\th}\big)'(y)\f{\Phi_n^{I}\varphi}{u'}\ln (u(y)-c_n)\bigg|_{0}^1+\f{\hat{\tilde{\om}}_{in,n}\varphi}{u'}\ln(u(y)-c_n)\bigg|_{0}^1\\
&\quad -\int_0^1\pa_y\left(\f{\hat{\tilde{\om}}_{in,n}\varphi}{u'}\right)\ln(u-c_n)dy\\
&\to -\int_0^1\f{\pa_y\Phi_{\infty}^I(y)\pa_y\varphi(y)}{\th(y)}dy-\int_0^1\f{\Phi_{\infty}^I(y)\varphi(y)}{\th(y)}dy
+\int_0^1\pa_y\left(\big(\f{u'}{\th}\big)'(y)\f{\Phi_{\infty}^{I}\varphi}{u'}\right)\ln (u-c_n) dy\\
&\quad -i\pi \big(\f{u'}{\th}\big)'(y_{\infty})\f{\Phi_{\infty}^{I}(y_{\infty})\varphi(y_{\infty})}{u'(y_{\infty})}.
\end{align*}
By taking $\varphi(y)=\overline{\Phi_{\infty}^{I}(y)}$, we have $\big(\f{u'}{\th}\big)'(y_{\infty})|\Phi_{\infty}^{I}(y_{\infty})|^2=0$. Thus we have $\Phi_{\infty}^{I}\in H^2\cap H^1_0$ and $c_{\infty}$ is an embedded eigenvalue, which leads to a contradiction. 
\end{proof}
\subsection{Linear damping, representation formula, and the wave operator}
In this section, we prove the weak linear inviscid damping and introduce the representation formula and the wave operator which will be used in the nonlinear problem. 

\begin{lemma}
Let $\psi$ be the velocity solving the linearized equation \eqref{eq: linearsystem}. Then it holds that 
\beno
\|(\pa_y\hat{\psi}, \hat{\psi})\|_{H^1_tL^2_{y}}\leq C\|\hat{\tilde{\om}}_{in}\|_{H^1_{y}}, 
\eeno
which implies $\lim\limits_{t\to \infty}\|(\pa_y\hat{\psi}, \hat{\psi})\|_{L^2_{y}}=0$.
\end{lemma}
\begin{proof}
We have by \eqref{eq: represent}, Proposition \ref{prop: uniform bound} , and Proposition \ref{prop: limit} that
\begin{align*}
\hat{\psi}(t,k,y)&=\f{1}{2\pi i}\int_{0}^1e^{-i u(y') k t}(\Psi_{-}^{I}(y, y', k)-\Psi_{+}^{I}(y, y', k))u'(y')dy'\\
\pa_t\hat{\psi}(t,k,y)&=\f{1}{2\pi i}\int_{0}^1(-i u(y') k) e^{-i u(y') k t}(\Psi_{-}^{I}(y, y', k)-\Psi_{+}^{I}(y, y', k))u'(y')dy'
\end{align*}
and
\beno
\|\Psi_{\pm}^{I}(y, y', k)\|_{L^{\infty}_{y'}H^1_y}\leq C\|\hat{\tilde{\om}}_{in}\|_{H^1_{x,y}}.
\eeno
Therefore by Plancherel's formula, we infer that
\begin{align*}
&\|(\pa_y\hat{\psi}, \hat{\psi})\|_{H_t^1L_y^2}^2
=\int_{\R}\big(\|(\pa_y\hat{\psi}, \hat{\psi})\|_{L_y^2}^2
+\|\partial_t(\pa_y\hat{\psi}, \hat{\psi})\|_{L_y^2}^2\big)dt\\
&\lesssim \int_{0}^1\int_{\R}\big(|\hat{\psi}(t,k,y)|^2+|\partial_y\hat{\psi}(t,k,\cdot)|^2
+k^2|\partial_t\hat{\psi}(t,k,y)|^2
+|\partial_t\partial_y\hat{\psi}(t,k,y)|^2\big)dtdy\\
&\lesssim \int_{0}^1\int_{\text{Ran}\,u}\big(k^2|\Psi_{\pm}^{I}(y, y', k)|^2+|\partial_y\Psi_{\pm}^{I}(y, y', k)|^2\big)dcdy\\
&\leq C\int_{\text{Ran}\,u}\|\hat{\tilde{\om}}_{in}(k,\cdot)\|_{H^1_y}^2dc\lesssim C\|\hat{\tilde{\om}}_{in}(k,\cdot)\|_{H^1_y}^2.
\end{align*}
The lemma follows directly from the Sobolev embedding. 
\end{proof}
In order to obtain the point-wise linear inviscid damping, one may need to study the regularity of $\Psi_{\pm}^{I}(y, y', k)$ in $y'$. We may not discuss it here. 

We now introduce the representation formula. 
\begin{proposition}\label{prop: represent psi}
Let $\psi$ be the velocity solving the linearized equation \eqref{eq: linearsystem}. Then it holds that 
\begin{align*}
\hat{\psi}(t,k,y)&=-\int_{0}^1e^{-i u(y') k t}\f{J_1(y', k)\Pi_2[\hat{\tilde{\om}}_{in}](y')-J_2(y')\Pi_1[\hat{\tilde{\om}}_{in}](y',k)}{J_1(y', k)^2+\pi^2J_2(y')^2}\rho(y')u'(y')e(y,y', k)dy'
\end{align*}
where 
\ben\label{eq:e}
e(y,y', k)=
\left\{
\begin{aligned}
&\phi(y,y', k)\int_{0}^y\frac{\th(z)}{\phi(z,y',k )^2}dz\quad 0\leq y<y',\\
&\phi(y,y', k)\int_1^y\frac{\th(z)}{\phi(z,y', k)^2}dz\quad y'<y\leq 1. 
\end{aligned}
\right.
\een
\end{proposition}
\begin{proof}
The proposition follows directly from \eqref{eq: represent}, Proposition \ref{prop: uniform bound}, and Proposition \ref{prop: limit} with a simple calculation. 
\end{proof}
The representation formula in Proposition \ref{prop: represent psi} leads to the wave operator. By taking $t=0$, we have the following representation formula:
\beq\label{eq:rep-form}
\hat{\psi}_{in}(k,y)=-\int_0^1\bbD_{u,k}\left[\tilde{\Delta}_k\hat{\psi}_{in}\right](y',k)\f{\f{\th(y')}{u'(y')}\rho(y')e(y,y', k)}{\sqrt{J_1(y', k)^2+\pi^2J_2(y')^2}}dy',
\eeq
where 
\beq\label{eq: wave operator}
\begin{aligned}
\bbD_{u,k}\left[\tilde{\om}\right](y',k)
&\eqdef 
\f{u'(y')}{\th(y')}\f{u'(y')J_1(y', k)\Pi_2[\tilde{\om}](y')-u'(y')J_2(y')\Pi_1[\tilde{\om}](y',k)}{\sqrt{J_1(y', k)^2+\pi^2J_2(y')^2}}\\
&=\f{J_1(y', k)\tilde{\om}(y')+\rho(y')\f{\th(y')}{u(y')}\left(\f{u'}{\th}\right)'(y')\Pi_1[\tilde{\om}](y',k)}{\sqrt{J_1(y', k)^2+\pi^2J_2(y')^2}}.
\end{aligned}
\eeq

\section{Wave operator}\label{sec: wave operator}
In this section, we give the basic properties of the wave operator and study its Fourier kernel. 
\begin{proposition}\label{prop: general-wave}
For $k\neq 0$, let $\mathcal{R}_{u,k}=u\mathrm{Id}-\left(\f{u'}{\th}\right)'\tilde{\Delta}_k^{-1}$. Suppose that for all $k\neq 0$, $\mathcal{R}_{u,k}$ has no eigenvalue and no embedded eigenvalue. Let  $\bbD_{u,k}$ be defined in \eqref{eq: wave operator}. Then there exists $\bbD^1_{u,k}$ such that
\ben\label{eq:DR=uD}
\bbD_{u,k}\mathcal{R}_{u,k}=u\bbD_{u,k},
\een
and for $\om,g\in L^2(0,1)$
\ben\label{eq: id1}
\int_0^1\bbD_{u,k}(\om)(y)\bbD^1_{u,k}(g)(y)dy=\int_0^1\om(y)g(y)dy. 
\een
Moreover, there exists $C>1$ independent of $k$ such that 
\ben\label{eq: est11}
C^{-1}\leq \|\bbD_{u,k}\|_{L^2\to L^2}\leq C,\quad C^{-1}\leq\|\bbD^1_{u,k}\|_{L^2\to L^2}\leq C.
\een
\end{proposition}
Note that \eqref{eq: est11} implies that the wave operator $\bbD_{u,k}$ is invertible and 
\ben\label{eq:D^1dual}
\bbD^1_{u,k}=(\bbD_{u,k}^{-1})^*.
\een
We give the representation formula of $\bbD_{u, k}^{-1}$ in \eqref{eq: D_uk^{-1}}. 
\begin{proof}
{\bf Proof of \eqref{eq:DR=uD}. }
Let $\psi=\tilde{\Delta}_k^{-1}\om$, namely $\pa_y\left(\f{1}{\th}\pa_y\psi\right)-\f{k^2}{\th}\psi=\om$ and $\psi(0)=\psi(1)=0$. Then $\mathcal{R}_{u,k}[\om](y')=u(y')\om(y')-\left(\f{u'}{\th}\right)'\psi(y')$. A direct calculation gives
\begin{align*}
&\bbD_{u,k}\left[\mathcal{R}_{u,k}[\om]\right](y',k)\\
&=\f{u'(y')}{\th(y')}\f{u'(y')J_1(y', k)\Pi_2[\mathcal{R}_{u,k}[\om]](y')-u'(y')J_2(y')\Pi_1[\mathcal{R}_{u,k}[\om]](y',k)}{\sqrt{J_1(y', k)^2+\pi^2J_2(y')^2}}\\
&=u(y')\f{u'(y')J_1(y', k)\Pi_2[\om]}{\sqrt{J_1(y', k)^2+\pi^2J_2(y')^2}}\f{u'(y')}{\th(y')}
-\f{J_1(y', k)\f{\th(y')}{u'(y')}\left(\f{u'}{\th}\right)'\psi(y')}{\sqrt{J_1(y', k)^2+\pi^2J_2(y')^2}}\f{u'(y')}{\th(y')}\\
&\quad-\f{u'(y')}{\th(y')}\f{u'(y')J_2(y')P.V.\int_{0}^1\f{\th(z')\int_{y'}^{z'}\left(u(z'')\om(z'')-\left(\f{u'}{\th}\right)'\psi(z'')\right)\phi_1(z'', y', k)dz''}{(u(z')-u(y'))^2\phi_1(z', y', k)^2}dz'}{\sqrt{J_1(y', k)^2+\pi^2J_2(y')^2}}.
\end{align*}
Let us now simplify the second term. Recall that $\phi(y, y', k)=(u(y)-u(y'))\phi_1(y, y', k)$ solves 
\beno
\pa_y\left(\f{1}{\th}\pa_y\phi\right)-\f{k^2}{\th}\phi=\left(\f{u'}{\th}\right)'\phi_1.
\eeno
Then by using the integration by part, we have 
\begin{align*}
&P.V.\int_{0}^1\f{\th(z')\int_{y'}^{z'}\left(u(z'')\om(z'')-\left(\f{u'}{\th}\right)'\psi(z'')\right)\phi_1(z'', y', k)dz''}{(u(z')-u(y'))^2\phi_1(z', y', k)^2}dz'\\
&=P.V.\int_{0}^1\f{\th(z')\int_{y'}^{z'}u(z'')\om(z'')\phi_1(z'', y', k)dz''}{(u(z')-u(y'))^2\phi_1(z', y', k)^2}dz'\\
&\quad-P.V.\int_{0}^1\f{\th(z')\int_{y'}^{z'}\psi(z'')\left(\pa_y\left(\f{1}{\th}\pa_y\phi(z'', y', k)\right)-\f{k^2}{\th}\phi(z'', y', k)\right) dz''}{(u(z')-u(y'))^2\phi_1(z', y', k)^2}dz'\\
&=P.V.\int_{0}^1\f{\th(z')\int_{y'}^{z'}u(z'')\om(z'')\phi_1(z'', y', k)dz''}{(u(z')-u(y'))^2\phi_1(z', y', k)^2}dz'
+P.V.\int_{0}^1\f{\th(z')\int_{y'}^{z'}\psi(z'')\f{k^2}{\th}\phi(z'', y', k) dz''}{(u(z')-u(y'))^2\phi_1(z', y', k)^2}dz'\\
&\quad+P.V.\int_{0}^1\f{\th(z')\int_{y'}^{z'}\pa_y\psi(z'')\left(\f{1}{\th}\pa_y\phi(z'', y', k)\right) dz''}{(u(z')-u(y'))^2\phi_1(z', y', k)^2}dz'\\
&\quad-P.V.\int_{0}^1\f{(u(z')-u(y'))\psi(z')\pa_y\phi_1(z', y', k)+u'(z')\psi(z')\phi_1(z', y', k)-\th(z')\psi(y')\left(\f{u'(y')}{\th(y')}\right)}{(u(z')-u(y'))^2\phi_1(z', y', k)^2}dz'\\
&=P.V.\int_{0}^1\f{\th(z')\int_{y'}^{z'}u(z'')\om(z'')\phi_1(z'', y', k)dz''}{(u(z')-u(y'))^2\phi_1(z', y', k)^2}dz'
-P.V.\int_{0}^1\f{\th(z')\int_{y'}^{z'}\om(z'')\phi(z'', y', k) dz''}{(u(z')-u(y'))^2\phi_1(z', y', k)^2}dz'\\
&\quad+P.V.\int_{0}^1\f{\pa_y\psi(z') }{(u(z')-u(y'))\phi_1(z', y', k)}dz'
-\int_{0}^1\f{\psi(z')\pa_y\phi_1(z', y', k)}{(u(z')-u(y'))\phi_1(z', y', k)^2}dz'\\
&\quad-P.V.\int_{0}^1\f{u'(z')\psi(z')\phi_1(z', y', k)-\th(z')\psi(y')\left(\f{u'(y')}{\th(y')}\right)}{(u(z')-u(y'))^2\phi_1(z', y', k)^2}dz'\\
&=u(y')\Pi_1[\om]+I_1[\psi]+I_2[\psi]+I_3[\psi].
\end{align*}
We write
\begin{align*}
I_1&=\int_{0}^1\f{\pa_y\psi(z') }{(u(z')-u(y'))}\Big(\f{1}{\phi_1(z', y', k)}-1\Big)dz'+P.V.\int_{0}^1\f{\pa_y\psi(z') }{(u(z')-u(y'))}dz'\\
I_2&=\int_{0}^1\f{\psi(z') }{(u(z')-u(y'))}\pa_y\Big(\f{1}{\phi_1(z', y', k)}-1\Big)dz'\\
I_3&=\int_{0}^1\psi(z')\pa_y\Big(\f{1}{(u(z')-u(y'))}\Big)\Big(\f{1}{\phi_1(z', y', k)}-1\Big)dz'\\
&\quad-P.V.\int_0^1\f{u'(z')\psi(z')-\th(z')\psi(y')\left(\f{u'(y')}{\th(y')}\right)}{(u(z')-u(y'))^2}dz'\\
&\quad+\psi(y')\f{u'(y')}{\th(y')}\int_{0}^1\f{\th(z')}{(u(z')-u(y'))^2}\Big(\f{1}{\phi_1(z', y', k)^2}-1\Big)dz'.
\end{align*}
Therefore, we obtain that
\begin{align}
\label{eq:I_1+I_2+I_3}
I_1+I_2+I_3
&=P.V.\int_{0}^1\f{\pa_y\psi(z') }{(u(z')-u(y'))}dz'\\
\nonumber&\quad+P.V.\int_0^1\pa_y\Big(\f{1}{u(z')-u(y')}\Big)\Big(\psi(z')-\f{\th(z')}{u'(z')}\psi(y')\left(\f{u'(y')}{\th(y')}\right)\Big)dz'\\
\nonumber&\quad+\psi(y')\f{u'(y')}{\th(y')}\int_{0}^1\f{\th(z')}{(u(z')-u(y'))^2}\Big(\f{1}{\phi_1(z', y', k)^2}-1\Big)dz'\\
\nonumber&=\psi(y')\f{u'(y')}{\th(y')}P.V.\int_0^1\f{1}{u(z')-u(y')}\Big(\f{\th(z')}{u'(z')}\Big)'dz'\\
\nonumber&\quad-\f{\f{\th(1)}{u'(1)}\psi(y')\left(\f{u'(y')}{\th(y')}\right)}{u(1)-u(y')}+\f{\f{\th(0)}{u'(0)}\psi(y')\left(\f{u'(y')}{\th(y')}\right)}{u(0)-u(y')}\\
\nonumber&\quad+\psi(y')\f{u'(y')}{\th(y')}\int_{0}^1\f{\th(z')}{(u(z')-u(y'))^2}\Big(\f{1}{\phi_1(z', y', k)^2}-1\Big)dz'\\
\nonumber&=\psi(y')\f{u'(y')}{\th(y')\rho(y')}J_1(y',k),
\end{align}
which gives that
\begin{align*}
&P.V.\int_{0}^1\f{\th(z')\int_{y'}^{z'}\left(u(z'')\om(z'')-\left(\f{u'}{\th}\right)'\psi(z'')\right)\phi_1(z'', y', k)dz''}{(u(z')-u(y'))^2\phi_1(z', y', k)^2}dz'\\
&=u(y')\Pi_1[\om]+\psi(y')\f{u'(y')}{\th(y')\rho(y')}J_1(y',k).
\end{align*}
Then we can conclude that
\begin{align*}
&\bbD_{u,k}\left[\mathcal{R}_{u,k}[\om]\right](y',k)\\
&=u(y')\f{u'(y')}{\th(y')}\f{u'(y')J_1(y', k)\Pi_2[\om]}{\sqrt{J_1(y', k)^2+\pi^2J_2(y')^2}}-u(y')\f{u'(y')}{\th(y')}\f{u'(y')J_2(y')\Pi_1[\om]}{\sqrt{J_1(y', k)^2+\pi^2J_2(y')^2}}=u(y')\bbD_{u,k}[\om]. 
\end{align*}
{\bf Proof of \eqref{eq: id1}. } Let $g\in H^{2}(0,1)\cap H_0^1(0,1)$ and $\psi=\tilde{\Delta}_k^{-1}\om$, namely $\pa_y\left(\f{1}{\th}\pa_y\psi\right)-\f{k^2}{\th}\psi=\om$ and $\psi(0)=\psi(1)=0$. 
Then 
\begin{align*}
\int_0^1\om(y)g(y)dy=\int_0^1\psi(y)\tilde{\Delta}_kg(y)dy. 
\end{align*}
By the representation formula \eqref{eq:rep-form}, we have
\begin{align*}
\int_0^1\psi(y)\tilde{\Delta}_k[g]dy
&=-\int_0^1\int_0^1\bbD_{u,k}\left[\om\right](y',k)\f{\f{\th(y')}{u'(y')}\rho(y')e(y,y', k)}{\sqrt{J_1(y', k)^2+\pi^2J_2(y')^2}}dy'\tilde{\Delta}_kg(y)dy\\
&=-\int_0^1\int_0^y\bbD_{u,k}\left[\om\right](y',k)\f{\th(y')}{u'(y')}\f{\rho(y')\phi(y,y', k)\int_1^y\frac{\th(z)}{\phi(z,y', k)^2}dz}{\sqrt{J_1(y', k)^2+\pi^2J_2(y')^2}}dy'\tilde{\Delta}_kg(y)dy\\
&\quad-\int_0^1\int_y^1\bbD_{u,k}\left[\om\right](y',k)\f{\th(y')}{u'(y')}\f{\rho(y')\phi(y,y', k)\int_{0}^y\frac{\th(z)}{\phi(z,y',k )^2}dz}{\sqrt{J_1(y', k)^2+\pi^2J_2(y')^2}}dy'\tilde{\Delta}_kg(y)dy\\
&=\int_0^1\f{\bbD_{u,k}\left[\om\right](y',k)\rho(y')}{\sqrt{J_1(y', k)^2+\pi^2J_2(y')^2}}\f{\th(y')}{u'(y')}\int_{y'}^1\f{\th(z)\int_{y'}^z\phi(y,y', k)\tilde{\Delta}_kg(y)dy}{\phi(z,y',k )^2}dzdy'\\
&\quad-\int_0^1\f{\bbD_{u,k}\left[\om\right](y',k)\rho(y')}{\sqrt{J_1(y', k)^2+\pi^2J_2(y')^2}}\f{\th(y')}{u'(y')}\int_0^{y'}\f{\th(z)\int_{z}^{y'}\phi(y,y', k)\tilde{\Delta}_kg(y)dy}{\phi(z,y',k )^2}dzdy'\\
&=\int_0^1\f{\bbD_{u,k}\left[\om\right](y',k)\rho(y')}{\sqrt{J_1(y', k)^2+\pi^2J_2(y')^2}}\f{\th(y')}{u'(y')}\int_{0}^1\f{\th(z)\int_{y'}^z\phi(y,y', k)\tilde{\Delta}_kg(y)dy}{\phi(z,y',k )^2}dzdy'.
\end{align*}
We now simplify the integral by using integration by parts
\begin{align*}
&\int_{0}^1\f{\th(z)\int_{y'}^z\phi(z'',y', k)\tilde{\Delta}_kg(z'')dy}{\phi(z,y',k )^2}dz\\
&=\int_{0}^1\f{\th(z)\int_{y'}^z\tilde{\Delta}_k\phi(z'',y', k)g(z'')dy}{\phi(z,y',k )^2}dz
+P.V.\int_{0}^1\f{\pa_yg(z)}{(u(z)-u(y'))\phi_1(z,y',k )}dz\\
&\quad-\int_{0}^1\f{\pa_y\phi_1(z,y', k)g(z)}{(u(z)-u(y'))\phi_1(z,y',k )^2}dz
-\int_{0}^1\f{u'(z)\phi_1(z,y', k)g(z)-\th(z)\f{1}{\th(y')}u'(y')g(y')}{\phi(z,y',k )^2}dz\\
&=P.V.\int_{0}^1\f{\th(z)\int_{y'}^z\left(\f{u'}{\th}\right)'\phi_1(z'', y', k)g(z'')dy}{\phi(z,y',k )^2}dz
+I_1[g]+I_2[g]+I_3[g],
\end{align*}
where $I_1, I_2, I_3$ are defined in the previous step. 
Thus by \eqref{eq:I_1+I_2+I_3}, we get 
\begin{align*}
&\int_{0}^1\f{\th(z)\int_{y'}^z\phi(z'',y', k)\tilde{\Delta}_kg(z'')dy}{\phi(z,y',k )^2}dz=\Pi_1\left[\left(\f{u'}{\th}\right)'g\right]+\f{u'(y')}{\th(y')\rho(y')}J_1(y',k)g(y').
\end{align*}
We define
\ben\label{eq: def D_1}
\bbD^{1}_{u,k}[g]=\f{\f{\th(y')}{u'(y')}\rho(y')\Pi_1\left[\left(\f{u'}{\th}\right)'g\right]+J_1(y',k)g(y')}{\sqrt{J_1(y', k)^2+\pi^2J_2(y')^2}},
\een
and then we obtain \eqref{eq: id1} for $g\in H^{2}(0,1)\cap H_0^1(0,1)$. 

By Lemma \ref{eq: J_1^2+J_2^2} and Lemma \ref{lem: L^2 est}, it is easy to obtain for any $\om\in L^2(0,1)$ and $k\neq 0$,
\ben\label{eq: D, D^1 upper}
\|\bbD_{u,k}[\om]\|_{L^2}+\|\bbD^{1}_{u,k}[\om]\|_{L^2}\leq C\|\om\|_{L^2}. 
\een
Here $C>0$ is a constant independent of $k$.  

Due to the fact that $H^2\cap H_0^1$ is dense in $L^2$, for any $g\in L^2$, there is a sequence $\{g_n\}_{n\geq 1}$ such that $H^2\cap H_0^1 \ni g_n\to g$ in $L^2$ and 
\begin{align*}
\int_{0}^1\bbD_{u,k}(\om)(y'){\bbD}_{u,k}^1(g_n)(y')dy'=\int_0^1\om(y)g_n(y)dy. 
\end{align*}
Then 
\begin{align*}
&\left|\int_{0}^1\bbD_{u,k}(\om)(y_c){\bbD}_{u,k}^1(g)(y')dy'-\int_0^1\om(k,y)g(k,y)dy\right|\\
&=\left|\int_{0}^1\bbD_{u,k}(\om)(y'){\bbD}_{u,k}^1(g_n)(y')dy'-\int_0^1\om(k,y)(g(y)-g_n(y))dy\right|\\
&\leq C\|\bbD_{u,k}(\om)\|_{L^2}\|{\bbD}_{u,k}^1(g-g_n)\|_{L^2}+\|\om\|_{L^2}\|g-g_n\|_{L^2}\to 0.
\end{align*}
Thus we proved \eqref{eq: id1}. The lower bound of ${\bbD}_{u,k}$ and ${\bbD}_{u,k}^1$ in \eqref{eq: est11} follow directly from a duality argument, and the upper bound \eqref{eq: D, D^1 upper}. Thus we proved the proposition.
\end{proof}
\begin{remark}
For any bounded function $C^{-1}<p(y')<C$, the operators $p(y')\bbD_{u, k}$ and $\f{1}{p(y')}\bbD_{u,k}^1$ also satisfy the same properties in Proposition \eqref{prop: general-wave}. The wave operator $\bbD_{u, k}$ is normalized so that if $\left(\f{u'}{\th}\right)'=0$, then $\bbD_{u,k}=\bbD_{u,k}^1=\mathrm{Id}$.
 
Moreover, it is easy to check that
\ben\label{eq: formula D-D^1}
\bbD_{u,k}[\om](y')-\bbD^{1}_{u,k}[\om](y')=\f{\f{\th(y')}{u'(y')}\rho(y')}{\sqrt{J_1(y', k)^2+\pi^2J_2(y')^2}}\left[\left(\f{u'}{\th}\right)',\Pi_1\right][\om](y'). 
\een
which is a commutator.  
\end{remark}
\subsection{Wave operator in $(t, z, v)$ coordinates}
Recall that we construct the wave operator in $(t, \tilde{x}, \tilde{y})$ coordinates (we drop tildes and still use $(t, x, y)$ in the above calculations for convenience), which is from the nonlinear change of coordinates and then the linear inverse change of coordinates. We show the relationship between different coordinate systems in the following map. 
\beno
\underbrace{(t,x,y)}_{\text{original system}}\xmapsto{\text{nonlinear change of coordinates}} \underbrace{(t,z,v)}_{\text{working object}}\xrightleftharpoons[\text{linear change of coordinates}]{\text{inverse linear change of coordinate}}\underbrace{(t,\tilde{x},\tilde{y})}_{\text{wave operator}}.
\eeno
In this section, we give the representation formula of the wave operator in $(t, z, v)$ coordinates. 

Let us first rewrite ${\bbD}_{u,k}$ and ${\bbD}_{u,k}^1$ in different forms. We first define
\begin{align}
&b_1(y', k)=\f{J_1(y', k)}{\sqrt{J_1(y', k)^2+\pi^2J_2(y')^2}}\\
&b_2(y', k)=-\f{\rho(y')\f{\th(y')}{u'(y')}}{\sqrt{J_1(y', k)^2+\pi^2J_2(y')^2}}.
\end{align}
Recall that $\varphi_1(y)=\left(\f{u'}{\th}\right)'(y)$ defined in \eqref{eq:definition vaprhi} and $e(z, y', k)$ defined in \eqref{eq:e}. A direct calculation gives that
\begin{align*}
\Pi_1[\om](y')=-P.V.\int_0^1\f{e(z, y', k)}{u(y')-u(y)}\om(z)dz
\end{align*}
which gives
\beq\label{eq: D, D^1 new form}
\begin{aligned}
&\bbD_{u,k}(\om)(k,y)=b_1(k,y)\om(y)+\varphi_1(y)b_2(k,y)p.v.\int_0^1\f{e(k,y',y)}{u(y')-u(y)}\om(y')dy',\\
&\bbD_{u,k}^1(\om)(k,y)=b_1(k,y)\om(y)+b_2(k,y)\int_0^1\f{e(k,y',y)\varphi_1(y')}{u(y')-u(y)}\om(y')dy'.
\end{aligned}
\eeq

By the \eqref{eq: id1}, it is easy to obtain that the inverse of $\bbD_{u,k}$ exists and has the following formula
\begin{align}\label{eq: D_uk^{-1}}
\bbD_{u,k}^{-1}(\om)(k,y)&=b_1(k,y)\om(y)+\varphi_1(y)\int_0^1\f{e(k,y,y')}{u(y)-u(y')}b_2(k,y')\om(y')dy'.
\end{align}

We also define $B_1(v, k)$, $B_2(v, k)$ and $E(v, v', k)$ so that
\begin{align}
B_1(u(y), k)=b_1(y', k),\quad  
B_2(u(y), k)=b_2(y', k), \quad E(u(y), u(y'), k)=e(y, y', k). 
\end{align}
Let us now introduce the following operators associated with $\bbD_{u, k}$, $\bbD_{u, k}^{-1}$, and $\bbD_{u, k}^1$. For any Schwartz function $f(t, z, v)$ defined on $\mathbb{R}\times \mathbb{T}\times \mathrm{Ran}\, u$, we define
\begin{align}
\label{eq:wavein zv}&\bfD_{u,k}\big(\mathcal{F}_{1}{f}(t,k,\cdot)\big)(t,k,v)\\
\nonumber&\eqdef B_1(v, k)\mathcal{F}_{1}{f}(t,k,v)\\
\nonumber&\quad+\widetilde{\varphi_1}(v)B_2(v, k)\int_{u(0)}^{u(1)}E(k,v_1,v)\f{\mathcal{F}_{1}{f}(t,k,v_1)e^{-i(v_1-v)tk}}{v_1-v}(u^{-1})'(v_1)dv_1,\\
\label{eq:inversewavein zv}&\bfD_{u,k}^{-1}\big(\mathcal{F}_{1}{f}(t,k,\cdot)\big)(t,k,v)\\
\nonumber&\eqdef B_1(v, k)\mathcal{F}_{1}{f}(t,k,v)\\
\nonumber&\quad+\widetilde{\varphi_1}(v)\int_{u(0)}^{u(1)}E(k,v,v_1)B_2(k,v_1)\f{\mathcal{F}_{1}{f}(t,k,v_1)e^{i(v_1-v)tk}}{v-v_1}(u^{-1})'(v_1)dv_1.
\end{align}
and
\begin{align}
\label{eq:wD^1 in zv}&\bfD_{u,k}^1\big(\mathcal{F}_{1}{f}(t,k,\cdot)\big)(t,k,v)\\
\nonumber&\eqdef B_1(v, k)\mathcal{F}_{1}{f}(t,k,v)\\
\nonumber&\quad+B_2(v, k)\int_{u(0)}^{u(1)}E(k,v_1,v)\f{\widetilde{\varphi_1}(v_1)\mathcal{F}_{1}{f}(t,k,v_1)e^{-i(v_1-v)tk}}{v_1-v}(u^{-1})'(v)dv_1.
\end{align}
With these definitions, we have the following proposition. 
\begin{proposition}\label{prop:wave-coordinate}
For $k\neq 0$, let $\tilde{\Delta}_{t, k}^l=-k^2\widetilde{\varphi_4}+\widetilde{u'}(\pa_v-itk)\Big(\widetilde{\varphi_4}\widetilde{u'}(\pa_v-itk)\Big)$ and 
\beno
\mathbf{R}_{u,k}=v\mathrm{Id}-\widetilde{\varphi_1}(v)\big(\tilde{\Delta}_{t, k}^l\big)^{-1}
\eeno 
be a modified Rayleigh operator. 
Then it holds for any ${f}(t,z,v)$ and ${g}(t,z,v)$ that
\begin{align}
&\bfD_{u,k}\bfR_{u,k}=v\bfD_{u,k},\\
&\bfD_{u,k}\bfD_{u,k}^{-1}=\bfD_{u,k}^{-1}\bfD_{u,k}=\mathrm{Id},\ \bfD_{u,k}^{-1}=(\bfD_{u,k}^1)^*,\\
\label{eq:[pa_t,D]}
&[\pa_t,\bfD_{u,k}]\big(\mathcal{F}_{1}{{f}}(t,k,\cdot)\big)(t,k,v)
=[ikv,\bfD_{u,k}]\big(\mathcal{F}_{1}{{f}}(t,k,\cdot)\big)(t,k,v),\\
\label{eq: dual id}&\int\bfD_{u,k}\big(\mathcal{F}_{1}{f}(t,k,\cdot)\big)(t,k,v)\overline{\bfD_{u,k}^1\big(\mathcal{F}_{1}{g}(t,k,\cdot)\big)(t,k,v)}dv
=\int\mathcal{F}_{1}{f}(k,u)\overline{\mathcal{F}_{1}{g}(k,v)}dv,\\
\label{eq:L^2 bound}&
C^{-1}\leq \|\bfD_{u,k}\|_{L^2\to L^2}\leq C,\quad C^{-1}\leq \|\bfD_{u,k}^1\|_{L^2\to L^2}\leq C,\quad
C^{-1}\leq \|\bfD_{u,k}^{-1}\|_{L^2\to L^2}\leq C.
\end{align}
Let $\Om(t, z, v)$ be defined in \eqref{eq: Om, a, Pi, Psi}. Then applying the inverse linear change of coordinate \eqref{eq: inversecoordinate}, we get that
\begin{align*}
\bfD_{u,k}\big(\mathcal{F}_{1}{{\Om}}(t,k,\cdot)\big)(t,k,v)e^{ikz}\bigg|_{z=x-u(y)t,\ v=u(y)}
&=\bfD_{u,k}\big(\mathcal{F}_{1}{\Om}(t,k,\cdot)\big)(t,k,u(y))e^{ik(x-tu(y))}\\
&=\bbD_{u,k}\big(\mathcal{F}_{1}{w}(t,k,\cdot)\big)(t,k,y)e^{ikx},
\end{align*}
where $w(t,x,y)={\Om}\big(t,x-u(y)t,u(y)\big)$ is different from $\tilde{\om}(t,x,y)$ in the original coordinate system. 
\end{proposition}
\begin{remark}\label{Rmk:dual-v}
We have changed $\bbD_{u,k}^1$ in the $(z,v)$ coordinate slightly to ensure that \eqref{eq: dual id} holds. 
Let $\Om(t, z, v)$ be defined in \eqref{eq: Om, a, Pi, Psi}. It holds that
\beno
\bfD_{u,k}^1\big(\mathcal{F}_{1}{\Om}(t,k,\cdot)\big)(t,k,v)e^{ikz}\bigg|_{z=x-u(y)t,\ v=u(y)}
=\f{1}{u'(y)}\bbD_{u,k}^1\big(u'\mathcal{F}_{1}{\om}(t,k,\cdot)\big)(t,k,y)e^{ikx},
\eeno
where $w(t,x,y)={\Om}\big(t,x-u(y)t,u(y)\big)$. 
\end{remark}
The proposition follows directly from Proposition \ref{prop: general-wave} and the representation formulas \eqref{eq: D, D^1 new form} and \eqref{eq: D_uk^{-1}} by a simple change of coordinates. 

\begin{remark}\label{Rmk: compact support}
By the compact support property of $b_2$. It is easy to check $\bfD_{u,k}(\mathcal{F}_1\Om)$ has the same compact support as $\Om$. 
\end{remark}

\subsection{Gevrey regularity of the homogeneous solution}
To obtain the Fourier kernels of $\bfD_{u,k}$, $\bfD_{u,k}^{-1}$, and $\bfD_{u,k}^1$, we need to study the regularity of the coefficients $B_1(v, k)$, $B_2(v, k)$ and the kernel $E(v, v', k)$.  
Indeed, all these functions are related to the homogeneous solution $\phi_1(y, y', k)$ obtained in Proposition \ref{prop: homsol}. 

Let us define $\Phi(v, v', k)$ and $\Phi_1(v, v', k)$ to be such that
\ben\label{eq: Phi, Phi_1}
\Phi(u(y), u(y'), k)=\phi(y, y', k), \quad \Phi_1(u(y), u(y'), k)=\phi_1(y, y', k). 
\een
By Proposition \ref{prop: homsol} and \eqref{eq:phi_1}, we have
\beno
\Phi(v, v', k)=(v-v')\Phi_1(v, v', k),
\eeno
and
\beq\label{eq:Phi_1}
\begin{aligned}
\Phi_1(v, v', k)
&=1+\int_{v'}^v\f{k^2\tilde{\th}(z')(u^{-1})'(z')}{(z'-v')^2}\int_{v'}^{z'}\f{(z''-v')^2}{\tilde{\th}(z'')}(u^{-1})'(z'')\Phi_1(z'',v',k)dz''dz'\\
&\eqdef 1+k^2(\mathrm{T}_0\circ \mathrm{T}_{2,2})[\Phi_1],
\end{aligned}
\eeq
where $\tilde{\th}(v)=(\th\circ u^{-1})(v)$, and for $j_1\leq j_2+1$, $j_2=0, 1, 2,...$
\begin{align*}
&\mathrm{T}_0[f](v, v', k)=\int_{v'}^vf(v_1, v', k)dv_1\\
&\mathrm{T}_{j_1, j_2}[\Phi_1](z', v', k)=\f{\tilde{\th}(z')(u^{-1})'(z')}{(z'-v')^2}\int_{v'}^{z'}\f{(z''-v')^2}{\tilde{\th}(z'')}(u^{-1})'(z'')\Phi_1(z'',v',k)dz''. 
\end{align*} 
For any $k\neq 0$, let $M\geq 1$ be a large constant, then we define the following weighted Gevrey spaces:
\ben
\mathcal{G}^{M|k|,s}_{ph,\cosh}([a,b]^2)\eqdef \{f\in C^{\infty}([a,b]^2):~\|f\|_{\mathcal{G}^{M|k|,s}_{ph,\cosh}([a,b]^2)}<\infty\}, 
\een
with 
\beq
\|f\|_{\mathcal{G}^{M|k|,s}_{ph,\cosh}([a,b]^2)}\eqdef
\sup_{(v, v')\in [a,b]^2,\, m\geq m_1\geq 0}\f{|\pa_{v}^{m-m_1}(\pa_{v'}+\pa_{v})^{m_1}f(v, v')|}{\Gamma_s(m)M^m|k|^{m}\cosh M|k|(v-v')},
\eeq
where $\Gamma_s(m)=(m!)^{\f1s}(m+1)^{-2}$. 

\begin{lemma}\label{rmk: Gevrey1}
It holds that
\begin{align*}
\sup_{(v, v')\in [a,b]^2,\, m\geq m_1\geq 0}\f{|\pa_{v}^{m-m_1}\pa_{v'}^{m_1}f(v, v')|}{\Gamma_s(m)2^{m_1}M^m|k|^{m}\cosh M|k|(v-v')}\leq \|f\|_{\mathcal{G}^{M|k|,s}_{ph,\cosh}([a,b]^2)}. 
\end{align*}
\end{lemma}
\begin{proof}
We have 
\begin{align*}
&\sup_{(v, v')\in [a,b]^2,\, m\geq m_1\geq 0}\f{|\pa_{v}^{m-m_1}\pa_{v'}^{m_1}f(v, v')|}{\Gamma_s(m)2^{m_1}M^m|k|^{m}\cosh M|k|(v-v')}\\
&\leq \sup_{m_1\geq 0}2^{-m_1}\sum_{m_2=0}^{m_1}\f{m_1!}{m_2!(m_1-m_2)!}\sup_{(v, v')\in [a,b]^2,\, m\geq m_2\geq 0}\f{|\pa_{v}^{m-m_2}(\pa_{v'}+\pa_{v})^{m_2}f(v, v')|}{\Gamma_s(m)M^m|k|^{m}\cosh M|k|(v-v')}\\
&\leq \sup_{m_1\geq 0}2^{-m_1}\sum_{m_2=0}^{m_1}\f{m_1!}{m_2!(m_1-m_2)!}\|f\|_{\mathcal{G}^{M|k|,s}_{ph,\cosh}([a,b]^2)}\leq \|f\|_{\mathcal{G}^{M|k|,s}_{ph,\cosh}([a,b]^2)}. 
\end{align*}
Thus we proved the remark.
\end{proof}

\begin{remark}
It holds that
\begin{align*}
\|fg\|_{\mathcal{G}^{M|k|,s}_{ph,\cosh}([a,b]^2)}
\leq C\|f\|_{\mathcal{G}^{M|k|,s}_{ph,\cosh}([a,b]^2)}\|g\|_{\mathcal{G}^{M|k|,s}_{ph,1}([a,b]^2)}. 
\end{align*}
where the norm $\|\cdot\|_{\mathcal{G}^{M|k|,s}_{ph,1}([a,b]^2)}$ is defined in \eqref{eq:def-G_ph,1}.
\end{remark}

\begin{proposition}\label{prop: T}
Suppose $(u^{-1})'', \tilde{\th}\in \mathcal{G}^{K_u,s_0}_{ph,1}([u(0),u(1)])$, then there exists $M_0=M_0(K_u)\geq 1$, such that for any $M\geq M_0$, it holds that
\beno
\|\rmT_0\circ\rmT_{2,2}\|_{\mathcal{G}^{M|k|,s_0}_{ph,\cosh}([u(0),u(1)]^2)\to \mathcal{G}^{M|k|,s_0}_{ph,\cosh}([u(0),u(1)]^2)}\leq \f{C}{M^2k^2}. 
\eeno
\end{proposition}
\begin{proof}
In this proof, we use the following notations for convenience. Let $L_1(v)=\tilde{\th}(v)(u^{-1})'(v)$ and $L_2(v)=\f{(u^{-1})'(v)}{\tilde{\th}(v)}$, then 
\beno
\|L_1\|_{\mathcal{G}^{M|k|,s}_{ph,1}([u(0),u(1)]^2)}+\|L_2\|_{\mathcal{G}^{M|k|,s}_{ph,1}([u(0),u(1)]^2)}\leq C. 
\eeno
It is easy to check that
\begin{align*}
&(\pa_{v}+\pa_{v'})^{m}\rmT_0\circ\rmT_{2,2}f\\
&=\sum_{k_1+k_2\leq m}\f{m!}{(m-k_1-k_2)!k_1!k_2!}\\
&\qquad \times\int_{v'}^v\f{L_1^{(k_1+1)}(u_1)}{(u_1-v')^2}\int_{v'}^{u_1}(u_2-v')^2L_2^{(k_2+1)}(u_2)(\pa_{u_2}+\pa_{v'})^{m-k_1-k_2}f(k,u_2,v')d u_2 du_1,
\end{align*}
which gives us that
\begin{align*}
&\left|\f{(\pa_{v}+\pa_{v'})^{m}\rmT_0\circ\rmT_{2,2}f}{\Gamma_{s_0}(m)M^m|k|^m\cosh M|k|(v-v')}\right|\\
&\leq \sum_{k_1+k_2\leq m}\f{m!}{(m-k_1-k_2)!k_1!k_2!}
\Gamma_{s_0}(m-k_1-k_2)\Gamma_{s_0}(k_1)\Gamma_{s_0}(k_2)\\
&\qquad \times\int_{v'}^v\f{1}{(u_1-v')^2}\int_{v'}^{u_1}(u_2-v')^2\cosh M|k|(u_2-v')d u_2 du_1\\
&\qquad\times \sup_{0\leq m_1\leq m}\left|\f{(\pa_{v}+\pa_{v'})^{m_1}f}{\Gamma_{s_0}(m_1)M^{m_1}|k|^{m_1}\cosh M|k|(v-v')}\right|\|L_1\|_{\mathcal{G}^{M|k|,s_0}_{ph,1}([u(0),u(1)])}\|L_2\|_{\mathcal{G}^{M|k|,s_0}_{ph,1}([u(0),u(1)])}\\
&\leq \f{C}{M^2k^2}\sup_{0\leq m_1\leq m}\left|\f{(\pa_{v}+\pa_{v'})^{m_1}f}{\Gamma_{s_0}(m_1)M^{m_1}|k|^{m_1}\cosh M|k|(v-v')}\right|,
\end{align*}
which implies 
\beno
\sup_{m\geq 0}\left|\f{(\pa_{v}+\pa_{v'})^{m}\rmT_0\circ\rmT_{2,2}f}{\Gamma_{s_0}(m)M^m|k|^m\cosh M|k|(v-v')}\right|
\leq \f{C}{M^2k^2}\sup_{m\geq 0}\left|\f{(\pa_{v}+\pa_{v'})^{m}f}{\Gamma_{s_0}(m)M^{m}|k|^{m}\cosh M|k|(v-v')}\right|. 
\eeno
We also have 
\begin{align*}
&\pa_{v}\rmT_0\circ\rmT_{2,2}f(v, v')=\rmT_{2,2}f(v, v')\\
&\pa_{v}^2\rmT_0\circ\rmT_{2,2}f(v, v')=(v-v')L_1'(v)\rmT_{3,2}f-2\rmT_{3,2}f+L_1(v)L_2(v)f(v, v')
\end{align*}
{\bf Claim:} It holds for any $s\in (0,1)$ and $j\geq 0$ that 
\beq
\begin{aligned}
&\sup_{m_1,\,m\geq 0}\left\|\f{\pa_{v}^{m_1}(\pa_{v}+\pa_{v'})^{m-m_1}\Big(\f{1}{(v-v')^{j+1}}\int_{v'}^v(u'-v')^{j}f(u',v')du'\Big)}{\Gamma_{s}(m)M^m|k|^m\cosh M|k|(v-v')}\right\|_{L^{\infty}([u(0),u(1)]^2)}\\
&\qquad\leq C\sup_{m_1, m\geq 0}\left\|\f{(\pa_{v'}+\pa_{v})^{m-m_1}\pa_{v}^{m_1}f(v, v')}{\Gamma_{s}(m)M^m|k|^m\cosh M|k|(v-v')}\right\|_{L^{\infty}([u(0),u(1)]^2)}. 
\end{aligned}
\eeq
Let us admit the claim and finish the proof of the proposition first. Indeed, it is easy to check that
\begin{align*}
\sum_{m=0}^{2}\sum_{m_1=0}^{m}\left\|\f{\pa_{v}^{m_1}\pa_{v'}^{m-m_1}\rmT_0\circ\rmT_{2,2}f}{(M|k|)^m\cosh M|k|(v-v')}\right\|_{L^{\infty}([u(0),u(1)]^2)}
\leq \f{C}{M^2k^2}\sum_{m=0}^2\left\|\f{\pa_{v'}^mf}{\cosh M|k|(v-v')}\right\|_{L^{\infty}([u(0),u(1)]^2)}. 
\end{align*}
For $m_1\geq 2$, by using the claim, we have that 
\begin{align*}
&\left\|\f{(\pa_{v'}+\pa_{v})^{m-m_1}\pa_{v}^{m_1}\rmT_0\circ\rmT_{2,2}f}{\Gamma_{s_0}(m)(M|k|)^m\cosh M|k|(v-v')}\right\|_{L^{\infty}([u(0),u(1)]^2)}\\
&\leq \left\|\f{(\pa_{v'}+\pa_{v})^{m-m_1}\pa_{v}^{m_1-2}\Big((v-v')L_1'(v)\rmT_{3,2}f\Big)}{\Gamma_{s_0}(m)(M|k|)^m\cosh M|k|(v-v')}\right\|_{L^{\infty}([u(0),u(1)]^2)}\\
&\quad+2\left\|\f{(\pa_{v'}+\pa_{v})^{m-m_1}\pa_{v}^{m_1-2}\rmT_{3,2}f}{\Gamma_{s_0}(m)(M|k|)^m\cosh M|k|(v-v')}\right\|_{L^{\infty}([u(0),u(1)]^2)}\\
&\quad+\left\|\f{(\pa_{v'}+\pa_{v})^{m-m_1}\pa_{v}^{m_1-2}L_1(v)L_2(v)f(v, v'))}{\Gamma_{s_0}(m)(M|k|)^m\cosh M|k|(v-v')}\right\|_{L^{\infty}([u(0),u(1)]^2)}\\
&\leq \f{C}{M^2k^2}\sup_{m_2\leq m_1-2,\, m_3\leq m-m_1}\left\|\f{(\pa_{v'}+\pa_{v})^{m_3}\pa_{v}^{m_2}f(v, v')}{\Gamma_{s_0}(m_2+m_3)(M|k|)^{m_2+m_3}\cosh M|k|(v-v')}\right\|_{L^{\infty}([u(0),u(1)]^2)}. 
\end{align*}

Finally, let us prove the claim. Indeed we have
\begin{align*}
&\pa_{v}^{m_1}(\pa_{v}+\pa_{v'})^{m-m_1}\left(\f{1}{(v-v')^{j+1}}\int_{v'}^u(u'-v')^{j}f(u',v')du'\right)\\
&=\pa_{v}^{m_1}(\pa_{v}+\pa_{v'})^{m-m_1}\left(\int_0^1t^jf(v'+(v-v')t,v')dt\right)\\
&=\int_0^1t^jt^{m_1}\big(\pa_{v}^{m_1}(\pa_{v}+\pa_{v'})^{m-m_1}f\big)(v'+(v-v')t,v')dt,
\end{align*}
which gives that
\begin{align*}
&\left\|\f{\pa_{v}^{m_1}(\pa_{v}+\pa_{v'})^{m-m_1}\left(\f{1}{(v-v')^{j+1}}\int_{v'}^u(u'-v')^{j}f(u',v')du'\right)}{\cosh M|k|(v-v')}\right\|_{L^{\infty}[u(0),u(1)]^2}\\
&\leq \f{\int_0^1t^jt^{m_1}\cosh tM|k|(v-v')dt}{\cosh M|k|(v-v')}\left\|\f{\pa_{v}^{m_1}(\pa_{v}+\pa_{v'})^{m-m_1}f}{\cosh M|k|(v-v')}\right\|_{L^{\infty}[u(0),u(1)]^2}\\
&\leq C\left\|\f{\pa_{v}^{m_1}(\pa_{v}+\pa_{v'})^{m-m_1}f}{\cosh M|k|(v-v')}\right\|_{L^{\infty}[u(0),u(1)]^2}.
\end{align*}
Thus we proved the proposition. 
\end{proof}
By \eqref{eq:Phi_1}, Proposition \ref{prop: T} and Lemma \ref{rmk: Gevrey1}, we have the following corollary:
\begin{corol}\label{corol: regular Phi}
Suppose that $u, \th$ are the same as in Proposition \ref{prop: T} and $\Phi_1$ satisfies \eqref{eq:Phi_1}, then there exist $M_0$ and $C(k)>0$ such that for any $|k|\neq 0$, for all integers $m\geq m_1\geq 0$ and $M\geq M_0$, 
\beno
\left|\pa_v^{m_1}\pa_{v'}^{m-m_1}\Phi_1(v, v', k)\right|\leq C(k)\Gamma_{s_0}(m)(2M)^m|k|^m.
\eeno
Moreover there exists $C\geq 1$ such that
\beno
\left|\pa_{v}^{m_1}\pa_{v'}^{m-m_1}\Big(\f{\Phi_1(v, v', k)-1}{(v-v')^2}\Big)\right|\leq C(k)\Gamma_{s_0}(m)(CM)^m|k|^m. 
\eeno
\end{corol}
\begin{proof}
The second inequality follows directly from the first inequality and the fact that $\Phi_1(v,v', k)\big|_{v=v'}=1$ and $\pa_{v'}\Phi_1(v, v',k)\big|_{v=v'}=\pa_v\Phi_1(v, v', k)=0\big|_{v=v'}$. So we only need to prove the first inequality. 

By taking $M_0$ large enough, it follows from \eqref{eq:Phi_1} and Proposition \ref{prop: T} that
\begin{align*}
\|\Phi_1\|_{\mathcal{G}^{M|k|,s_0}_{ph,\cosh}([u(0),u(1)]^2)}
&\leq \|k^2\rmT_0\circ\rmT_{2,2}\Phi_1\|_{\mathcal{G}^{M|k|,s_0}_{ph,\cosh}([u(0),u(1)]^2)}+C\\
&\leq \f{C}{M^2}\|\Phi_1\|_{\mathcal{G}^{M|k|,s_0}_{ph,\cosh}([u(0),u(1)]^2)}+C,
\end{align*}
which gives us that
\beno
\|\Phi_1\|_{\mathcal{G}^{M|k|,s_0}_{ph,\cosh}([u(0),u(1)]^2)}\leq C. 
\eeno
Then the corollary follows from Lemma \ref{rmk: Gevrey1}. 
\end{proof}
\subsection{Gevrey regularity of the coefficients and Fourier kernels}
In order to estimate the Fourier kernels of $\bfD_{u, k}$, $\bfD_{u, k}^{-1}$, and $\bfD_{u, k}^1$, we need write the kernel $E(v, v', k)$ in a good form. Recall that in $(y, y')$ coordinates we have 
\beno
e(y,y', k)=
\left\{
\begin{aligned}
&\phi(y,y', k)\int_{0}^y\frac{\th(z)}{\phi(z,y',k )^2}dz\quad 0\leq y<y',\\
&\phi(y,y', k)\int_1^y\frac{\th(z)}{\phi(z,y', k)^2}dz\quad y'<y\leq 1. 
\end{aligned}
\right.
\eeno
 A direct calculation gives us that for $j=0,1$
\begin{align*}
&\phi(y, y', k)\int_j^y\f{\th(w)}{\phi(y, w, k)^2}dw\\
&=\phi(y, y', k)\int_j^y\f{\th(w)}{(u(w)-u(y'))^2}\Big(\f{1}{\phi_1(w, y', k)^2}-1\Big)dy'\\
&\quad+\phi(y, y', k)\int_j^y\f{\th(w)}{(u(w)-u(y'))^2}dw\\
&=\phi(y, y', k)\int_j^y\f{\th(z)}{(u(w)-u(y'))^2}\Big(\f{1}{\phi_1(w, y', k)^2}-1\Big)dw\\
&\quad+\phi(y, y', k)\int_{u(j)}^{u(y)}\f{\Big(\tilde{\th}(u^{-1})'\Big)(u_1)-\Big(\tilde{\th}(u^{-1})'\Big)(u(y'))-\Big(\tilde{\th}(u^{-1})'\Big)'(u(y'))(u_1-u(y'))}{(u_1-u(y'))^2}du_1\\
&\quad-\Big(\tilde{\th}(u^{-1})'\Big)(u(y'))\f{\phi(y, y', k)}{u(w)-u(y')}\bigg|_{w=j}^{w=y}
+\phi(y, y', k)\Big(\tilde{\th}(u^{-1})'\Big)'(u(y'))\ln |u(w)-u(y')|\bigg|_{w=j}^{w=y}.
\end{align*}
Thus we obtain
\beq\label{eq: decomposition-e}
\begin{aligned}
&\phi(y, y', k)\int_j^y\f{\th(w)}{\phi(y, w, k)^2}dw\\
&=-\Big(\tilde{\th}(u^{-1})'\Big)(u(y'))
+\phi(y, y', k)\Big(\tilde{\th}(u^{-1})'\Big)'(u(y'))\ln |u(y)-u(y')|
+\phi_j^{re}(y, y', k),
\end{aligned}
\eeq
where $\phi_j^{re}$ is the regular part
\begin{align*}
&\phi_j^{re}(y, y', k)\\
&=\phi(y, y', k)\int_j^y\f{\th(w)}{(u(w)-u(y'))^2}\Big(\f{1}{\phi_1(w, y', k)^2}-1\Big)dw\\
&\quad+\phi(y, y', k)\int_{u(j)}^{u(y)}\f{\Big(\tilde{\th}(u^{-1})'\Big)(u_1)-\Big(\tilde{\th}(u^{-1})'\Big)(u(y'))-\Big(\tilde{\th}(u^{-1})'\Big)'(u(y'))(u_1-u(y'))}{(u_1-u(y'))^2}du_1\\
&\quad+\Big(\tilde{\th}(u^{-1})'\Big)(u(y'))(1-\phi_1(y, y', k))
+\Big(\tilde{\th}(u^{-1})'\Big)(u(y'))\f{u(y)-u(y')}{u(j)-u(y')}\phi_1(y, y', k)\\
&\quad-\phi(y, y', k)\Big(\tilde{\th}(u^{-1})'\Big)'(u(y'))\ln |u(j)-u(y')|. 
\end{align*}
Note that the integral $\phi(y, y', k)\int_j^y\f{\th(w)}{\phi(w, y',k)^2}dw$ is well defined for $y<y'$ if $j=0$ and for $y>y'$ if $j=1$. However, the right-hand side of \eqref{eq: decomposition-e} can be defined for $y\neq y'$.

Therefore, we have
\beq\label{eq: new-e}
\begin{aligned}
e(y, y', k)=
&-\Big(\tilde{\th}(u^{-1})'\Big)(u(y'))
+\phi(y, y', k)\Big(\tilde{\th}(u^{-1})'\Big)'(u(y'))\ln |u(y)-u(y')|\\
&+\phi_0^{re}(y, y', k)(1_{\mathbb{R}^-}(y-y'))
+\phi_1^{re}(y, y', k)(1_{\mathbb{R}^+}(y-y'))
\end{aligned}
\eeq
and thus
\beq\label{eq: new-E}
\begin{aligned}
E(v, v', k)=
&-\Big(\tilde{\th}(u^{-1})'\Big)(v')
+\Phi(v, v', k)\Big(\tilde{\th}(u^{-1})'\Big)'(v')\ln |v-v'|\\
&+\Phi_0^{re}(v, v', k)(1_{\mathbb{R}^-}(v-v'))
+\Phi_1^{re}(v, v', k)(1_{\mathbb{R}^+}(v-v'))
\end{aligned}
\eeq
where for $j=0,1$, $\Phi_j^{re}(v, v', k)$ satisfies $\Phi_j^{re}(u(y), u(y'), k)=\phi_j^{re}(y, y', k)$, namely, 
\begin{align*}
&\Phi_j^{re}(v, v', k)\\
&=\Phi(v, v', k)\int_{u(j)}^v\f{\tilde{\th}(w)u^{-1}(w)}{(w-v')^2}\Big(\f{1}{\Phi_1(w, v', k)^2}-1\Big)dw\\
&\quad+\Phi(v, v', k)\int_{u(j)}^{v}\f{\Big(\tilde{\th}(u^{-1})'\Big)(u_1)-\Big(\tilde{\th}(u^{-1})'\Big)(v')-\Big(\tilde{\th}(u^{-1})'\Big)'(v')(u_1-v')}{(u_1-v')^2}du_1\\
&\quad+\Big(\tilde{\th}(u^{-1})'\Big)(v')(1-\Phi_1(v, v', k))
+\Big(\tilde{\th}(u^{-1})'\Big)(v')\f{v-v'}{u(j)-v'}\Phi_1(v, v', k)\\
&\quad-\Phi(v, v', k)\Big(\tilde{\th}(u^{-1})'\Big)'(v')\ln |u(j)-v'|. 
\end{align*}
We define
\beq\label{eq: Phi_j1^re}
\begin{aligned}
&\Phi_{j,1}^{re}(v, v', k)=\f{\Phi_j^{re}(v, v', k)}{v-v'}\\
&=\Phi_1(v, v', k)\int_{u(j)}^v\f{\tilde{\th}(w)u^{-1}(w)}{(w-v')^2}\Big(\f{1}{\Phi_1(w, v', k)^2}-1\Big)dw\\
&\quad+\Phi_1(v, v', k)\int_{u(j)}^{v}\f{\Big(\tilde{\th}(u^{-1})'\Big)(u_1)-\Big(\tilde{\th}(u^{-1})'\Big)(v')-\Big(\tilde{\th}(u^{-1})'\Big)'(v')(u_1-v')}{(u_1-v')^2}du_1\\
&\quad+\Big(\tilde{\th}(u^{-1})'\Big)(v')\f{(1-\Phi_1(v, v', k))}{v-v'}
+\Big(\tilde{\th}(u^{-1})'\Big)(v')\f{1}{u(j)-v'}\Phi_1(v, v', k)\\
&\quad-\Phi_1(v, v', k)\Big(\tilde{\th}(u^{-1})'\Big)'(v')\ln |u(j)-v'|. 
\end{aligned}
\eeq
We have the following proposition:
\begin{proposition}\label{Prop: Regularity E, B}
It holds that 
\beno
|\pa_v^{m_1}\pa_{v'}^{m-m_1}B_1(v, k)|+|\pa_v^{m_1}\pa_{v'}^{m-m_1}B_2(v, k)|\leq C(k)\Gamma_{s}(m)(M)^{m}.
\eeno
Recall that $\tchi_2=\chi_2\circ u^{-1}$ is a smooth function with compact support such that $\mathrm{supp}\, \chi_2\subset [\f{\kappa_0}{2},1-\f{\kappa_0}{2}]$ and satisfies \eqref{eq: chi}.
\beno
|\pa_v^{m_1}\pa_{v'}^{m-m_1}\tchi_2(v)\tchi_2(v')\Phi^{re}_j(v, v' k)|\leq C(k)\Gamma_{s}(m)(M)^{m}, \quad\text{for}\quad j=0,1,
\eeno
and
\beno
|\pa_v^{m_1}\pa_{v'}^{m-m_1}\tchi_2(v)\tchi_2(v')\Phi^{re}_{j,1}(v, v' k)|\leq C(k)\Gamma_{s}(m)(M)^{m}, \quad\text{for}\quad j=0,1.
\eeno
\end{proposition}
\begin{proof}
First, we write the numerator $J_1(y',k)$ of $b_1(y', k)$ in $v'$ coordinate 
\begin{align*}
\mathbf{J}_1(v',k)&\eqdef-(u(1)-v')\f{\th(0)}{u'(0)}-(v'-u(0))\f{\th(1)}{u'(1)}\\
&\quad+\tilde{\rho}(v')\int_{0}^1\f{\tilde{\th}(w)(u^{-1})'(w)}{(w-v')^2}\Big(\f{1}{\Phi_1(w,v',k)^2}-1\Big)dw\\
&\quad-\tilde{\rho}(v')\mathcal{H}\big[\big((\th\circ u^{-1})(u^{-1})'\big)'\chi_{[u(0),u(1)]}\big](v')
\end{align*}
where $\tilde{\rho}(v')=(v'-u(0))(u(1)-v')$. Notice that $\big((\th\circ u^{-1})(u^{-1})'\big)'$ has compact support. Thus by Remark \ref{Rmk: fourier-gevrey}, we have
\beno
\left|\mathcal{F}_{v\to \xi}[(\th\circ u^{-1})(u^{-1})'](\xi)\right|\leq Ce^{-\la_{u,\th}\langle\xi\rangle^{s}}
\eeno
which implies 
\beno
\left|\mathcal{F}_{v'\to \xi}\big[\mathcal{H}\big[\big((\th\circ u^{-1})(u^{-1})'\big)'\chi_{[u(0),u(1)]}\big](v')\big]\right|\leq Ce^{-\la_{u,\th}\langle\xi\rangle^{s}}.
\eeno
Thus by Lemma \ref{Lem: A1}, we have 
\beno
\left|\pa_{v'}^{m}\big[\mathcal{H}\big[\big((\th\circ u^{-1})(u^{-1})'\big)'\chi_{[u(0),u(1)]}\big](v')\right|\leq C\Gamma_{s}(m)(M)^{m}.
\eeno
The estimates of $B_1, B_2$ follow directly from \eqref{eq: est1}, Corollary \ref{corol: regular Phi}, Lemma \ref{eq: J_1^2+J_2^2}, and \eqref{eq: space-embed}. 

Now we prove the regularity of $\Phi_{j}^{re}$ and $\Phi_{j,1}^{re}$. Notice that with the cut-off function, the `bad' terms $\tchi_2(v)\tchi_2(v')\Big(\tilde{\th}(u^{-1})'\Big)(v')\f{v-v'}{u(j)-v'}\Phi_1(v, v', k)$
is regular. For the term that contains logarithmic function $\Phi(v, v', k)\Big(\tilde{\th}(u^{-1})'\Big)'(v')\ln |u(j)-v'|$, we notice that $\Big(\tilde{\th}(u^{-1})'\Big)'(v')$ has compact support. Thus by using \eqref{eq: est1}, Corollary \ref{corol: regular Phi}, Lemma \ref{eq: J_1^2+J_2^2}, and \eqref{eq: space-embed}, we prove the proposition. 
\end{proof}
Although the wave operator is nonlocal both in physical space and frequency space, the next proposition shows that the wave operator does not move frequencies a lot. Now we introduce the Fourier kernels and the estimates. 
\begin{proposition}\label{prop: kernel-wave-op}
Recall that $\tchi_2=\chi_2\circ u^{-1}$ is a smooth function with compact support such that $\mathrm{supp}\, \chi_2\subset [\f{\kappa_0}{2},1-\f{\kappa_0}{2}]$ and satisfies \eqref{eq: chi}. Then for any $0<|k|\leq k_{M}$, there exist $\mathcal{D}(t,k,\xi_1,\xi_2)$, $\mathcal{D}^1(t,k,\xi_1,\xi_2)$ and $\mathcal{D}^{-1}(t,k,\xi_1,\xi_2)$ such that 
\beno
\mathcal{F}_{2}\Big(\bfD_{u,k}\big(\tchi_2\mathcal{F}_{1}f(t,k,\cdot)\big)\Big)(t,k,\xi_1)=\int \mathcal{D}(t,k,\xi_1,\xi_2)\hat{f}_k(t,\xi_2)d\xi_2,
\eeno
and
\beno
\mathcal{F}_{2}\Big(\tchi_2\bfD_{u,k}^1\big(\tchi_2\mathcal{F}_{1}f(t,k,\cdot)\big)\Big)(t,k,\xi_1)=\int \mathcal{D}^1(t,k,\xi_1,\xi_2)\hat{f}_k(t,\xi_2)d\xi_2,
\eeno
and
\beno
\mathcal{F}_{2}\Big(\bfD_{u,k}^{-1}\big(\tchi_2\mathcal{F}_{1}f(t,k,\cdot)\big)\Big)(t,k,\xi_1)=\int \mathcal{D}^{-1}(t,k,\xi_1,\xi_2)\hat{f}_k(t,\xi_2)d\xi_2.
\eeno
Moreover, there exists $\la_{\mathcal{D}}=\la_{\mathcal{D}}(\la_0,\kappa_0,s_0,s_1, k_{M})$ independent of $t$ such that
\ben\label{eq: kernel est-DDD}
\left|\mathcal{D}(t,k,\xi_1,\xi_2)\right|+\left|\mathcal{D}^1(t,k,\xi_1,\xi_2)\right|+\left|\mathcal{D}^{-1}(t,k,\xi_1,\xi_2)\right|\lesssim e^{-\la_{\mathcal{D}}|\xi_1-\xi_2|^{s_0}}. 
\een
\end{proposition}
\begin{proof}
Note that by \eqref{eq:wavein zv}, \eqref{eq:inversewavein zv} and \eqref{eq:wD^1 in zv}, the wave operator $\bfD_{u,k}$ is similar to the operator $\bfD_{u,k}^1$ and the inverse wave operator $\bfD_{u,k}^{-1}$ in structure. We only present the proof for the wave operator $\bfD_{u,k}$. The estimates of $\bfD_{u,k}^1$ and $\bfD_{u,k}^{-1}$ are similar and we omit them. 

We divide the nonlocal part of the wave operator into four different types of integral operators whose kernels have different singularities:
\begin{align}\label{eq: bfD}
\bfD_{u,k}\big(\mathcal{F}_{1}\tilde{f}(t,k,\cdot)\big)(t,k,v')
=B_1(v, k)\mathcal{F}_{1}\tilde{f}(t,k,v)+\sum_{j=1}^4\Pi_j(\tilde{f}),
\end{align}
where
\begin{align*}
\Pi_1(\tilde{f})&=\widetilde{\varphi_1}(v')B_2(v', k)\int_{u(0)}^{u(1)}{\mathcal{F}_{1}\tilde{f}(u_1)e^{-i(u_1-v')tk}}(u^{-1})'(u_1)\Phi^{re}_{1,1}(u_1,v', k)du_1,\\
\Pi_2(\tilde{f})&=\widetilde{\varphi_1}(v')B_2(v', k)\int_{u(0)}^{u(1)}{\mathcal{F}_{1}\tilde{f}(u_1)e^{-i(u_1-v')tk}}(u^{-1})'(u_1)\Phi^{re}_{0,1}(u_1,v', k)(1_{\R^-}(u_1-v'))du_1\\
&\quad-\widetilde{\varphi_1}(v')B_2(v', k)\int_{u(0)}^{u(1)}{\mathcal{F}_{1}\tilde{f}(u_1)e^{-i(u_1-v')tk}}(u^{-1})'(u_1)\Phi^{re}_{1,1}(u_1,v', k)1_{\R^-}(u_1-v')du_1,\\
\Pi_3(\tilde{f})&=\widetilde{\varphi_1}(v')B_2(v', k)\int_{u(0)}^{u(1)}{\mathcal{F}_{1}\tilde{f}(u_1)e^{-i(u_1-v')tk}}(u^{-1})'(u_1)\\
&\qquad\qquad\qquad\qquad\qquad\qquad\qquad\times \Big(\tilde{\th}(u^{-1})'\Big)'(v')\Phi_1(u_1,v', k)\ln|u_1-v'|du_1,\\
\Pi_4(\tilde{f})&=-\Big(\tilde{\th}(u^{-1})'\Big)(v')\widetilde{\varphi_1}(v')B_2(v', k)\int_{u(0)}^{u(1)}\f{\mathcal{F}_{1}\tilde{f}(u_1)e^{-i(u_1-v')tk}}{u_1-v'}(u^{-1})'(u_1)du_1.
\end{align*} 
By Lemma \ref{lem: Fourier_type1}, we get that there exists $\mathcal{D}(t,k,\xi_1,\xi_2)$ such that
\beno
\mathcal{F}_{2}\Big(\bfD_{u,k}\big(\tchi_2\mathcal{F}_{1}f(t,k,\cdot)\big)\Big)(t,k,\xi_1)=\int \mathcal{D}(t,k,\xi_1,\xi_2)\hat{f}_k(t,\xi_2)d\xi_2. 
\eeno
The behavior of $\mathcal{D}(t,k,\xi_1,\xi_2)$ depends only on the regularity of $D_1$, $D_2$ and $E$. Therefore by Lemma \ref{lem: Fourier_type1}, Remark \ref{Rmk: fourier-gevrey}, Corollary \ref{corol: regular Phi}, Proposition \ref{Prop: Regularity E, B}, and Lemma \ref{Rmk: fo-g-2}, we obtain Proposition \ref{prop: kernel-wave-op}. 
\end{proof}
\subsection{Commutator}

In this section, let us study the difference between $\bfD_{u,k}^1$ and $\bfD_{u,k}$ and prove the following proposition.
\begin{proposition}\label{prop: commutator}
There exists $\mathcal{D}^{com}(t,k,\xi,\xi_1)$ such that
\beno
\mathcal{F}_2\Big(\tchi_2\bfD_{u,k}(\tchi_2\mathcal{F}_1{f})-\tchi_2\bfD_{u,k}^1(\tchi_2\mathcal{F}_1{f})\Big)(\xi)=\int \mathcal{D}^{com}(t,k,\xi,\xi_1)\hat{f}_k(t,\xi_1)d\xi_1.
\eeno
Moreover, there exists $\la_{\mathcal{D}}=\la_{\mathcal{D}}(\la_0,\th_0,s_0,s_1, k_{M})$ independent of $t$ such that
\beno
\big|\mathcal{D}^{com}(t,k,\xi,\xi_1)\big|\lesssim \min\left(\f{e^{-\la_{\mathcal{D}}|\xi-\xi_1|^{s_0}}}{1+|\xi-kt|^2}, \f{e^{-\la_{\mathcal{D}}|\xi-\xi_1|^{s_0}}}{1+|\xi_1-kt|^2}\right).
\eeno
\end{proposition}
\begin{proof}
We have by \eqref{eq: new-E} that 
\begin{align*}
&\tchi_2\bfD_{u,k}(\mathcal{F}_1f)-\tchi_2\bfD_{u,k}^1(\mathcal{F}_1f)\\
&=\tchi_2B_2(v, k)\int_{u(0)}^{u(1)}E(k,v_1,v)\f{\mathcal{F}_{1}{f}(v_1)e^{-i(v_1-v)tk}}{v_1-v}\Big[\widetilde{\varphi_1}(v)(u^{-1})'(v_1)-\widetilde{\varphi_1}(v_1)(u^{-1})'(v)\Big]dv_1\\
&=-\tchi_2B_2(k,v)\int_{u(0)}^{u(1)}\Big(\tilde{\th}(u^{-1})'\Big)(v)\f{\mathcal{F}_{1}{f}(v_1)e^{-i(v_1-v)tk}}{v_1-v}\Big[\widetilde{\varphi_1}(v)(u^{-1})'(v_1)-\widetilde{\varphi_1}(v_1)(u^{-1})'(v)\Big]dv_1\\
&\quad+\tchi_2B_2(k,v)\int_{u(0)}^{u(1)}\Phi_1(v_1, v, k)\Big(\tilde{\th}(u^{-1})'\Big)'(v)\ln |v_1-v|\mathcal{F}_{1}{f}(v_1)e^{-i(v_1-v)tk}\\
&\qquad\qquad\qquad\qquad \qquad\qquad\qquad \qquad\qquad\qquad 
\times \Big[\widetilde{\varphi_1}(v)(u^{-1})'(v_1)-\widetilde{\varphi_1}(v_1)(u^{-1})'(v)\Big]dv_1\\
&\quad+\tchi_2B_2(v, k)\int_{u(0)}^{u(1)}\Phi_{0,1}^{re}(v_1, v, k)(1_{\mathbb{R}^-}(v_1-v))\mathcal{F}_{1}{f}(v_1)e^{-i(v_1-v)tk}\\
&\qquad\qquad\qquad\qquad \qquad\qquad\qquad \qquad\qquad\qquad 
\times\Big[\widetilde{\varphi_1}(v)(u^{-1})'(v_1)-\widetilde{\varphi_1}(v_1)(u^{-1})'(v)\Big]dv_1\\
&\quad+\tchi_2B_2(v, k)\int_{u(0)}^{u(1)}\Phi_{1,1}^{re}(v_1, v, k)(1_{\mathbb{R}^+}(v_1-v))\mathcal{F}_{1}{f}(v_1)e^{-i(v_1-v)tk}\\
&\qquad\qquad\qquad\qquad \qquad\qquad\qquad \qquad\qquad\qquad 
\times\Big[\widetilde{\varphi_1}(v)(u^{-1})'(v_1)-\widetilde{\varphi_1}(v_1)(u^{-1})'(v)\Big]dv_1.
\end{align*}

Let $\mathcal{D}^{com}(t,k,\xi,\xi_1)$ be the Fourier kernel of the operator $\tchi_2\bfD_{u,k}(\tchi_2\mathcal{F}_1f)-\tchi_2\bfD_{u,k}^1(\tchi_2\mathcal{F}_1f)$, which means that
\beno
\mathcal{F}_2\Big(\tchi_2\bfD_{u,k}(\tchi_2\mathcal{F}_1f)-\tchi_2\bfD_{u,k}^1(\tchi_2\mathcal{F}_1f)\Big)(\xi)=\int \mathcal{D}^{com}(t,k,\xi,\xi_1)\hat{f}_k(t,\xi_1)d\xi_1.
\eeno
Then by the same argument as in the proof of Proposition \ref{prop: kernel-wave-op}, we get that there exists $\la_{\mathcal{D}}$ such that
\beno
\big|\mathcal{D}^{com}(t,k,\xi,\xi_1)\big|\lesssim e^{-\la_{\mathcal{D}}|\xi-\xi_1|^{s_0}}. 
\eeno
Let us also study the derivate $(\pa_v-itk)$ acting on $\tchi_2\bfD_{u,k}(\mathcal{F}_1f)-\tchi_2\bfD_{u,k}^1(\mathcal{F}_1f)$:
\begin{align*}
&(\pa_v-itk)\Big(\tchi_2\bfD_{u,k}(\mathcal{F}_1f)-\tchi_2\bfD_{u,k}^1(\mathcal{F}_1f)\Big)\\
&=\tchi_2B_2(v, k)\int_{u(0)}^{u(1)}\Phi_1(v_1, v, k)\Big(\tilde{\th}(u^{-1})'\Big)'(v)\ln |v_1-v|\mathcal{F}_{1}{f}(v_1)e^{-i(v_1-v)tk}\\
&\qquad\qquad\qquad\qquad \qquad\qquad\qquad \qquad\qquad\qquad 
\times \pa_{v}\Big[\widetilde{\varphi_1}(v)(u^{-1})'(v_1)-\widetilde{\varphi_1}(v_1)(u^{-1})'(v)\Big]dv_1\\
&\quad+\tchi_2B_2(v, k)\int_{u(0)}^{u(1)}\Phi_{0,1}^{re}(v_1, v, k)(1_{\mathbb{R}^-}(v_1-v))\mathcal{F}_{1}{f}(v_1)e^{-i(v_1-v)tk}\\
&\qquad\qquad\qquad\qquad \qquad\qquad\qquad \qquad\qquad\qquad 
\times \pa_{v}\Big[\widetilde{\varphi_1}(v)(u^{-1})'(v_1)-\widetilde{\varphi_1}(v_1)(u^{-1})'(v)\Big]dv_1\\
&\quad+\tchi_2B_2(v, k)\int_{u(0)}^{u(1)}\Phi_{1,1}^{re}(v_1, v, k)(1_{\mathbb{R}^+}(v_1-v))\mathcal{F}_{1}{f}(v_1)e^{-i(v_1-v)tk}\\
&\qquad\qquad\qquad\qquad \qquad\qquad\qquad \qquad\qquad\qquad 
\times \pa_{v}\Big[\widetilde{\varphi_1}(v)(u^{-1})'(v_1)-\widetilde{\varphi_1}(v_1)(u^{-1})'(v)\Big]dv_1\\
&\quad+\text{good terms}. 
\end{align*}
Here note that the derivative $\pa_v$ acting on $\ln|v_1-v|$ gives a good term:
\beno
\f{\widetilde{\varphi_1}(v)(u^{-1})'(v_1)-\widetilde{\varphi_1}(v_1)(u^{-1})'(v)}{v_1-v}\in \mathcal{G}_{ph,1}^{K_u,s}([u(0),u(1)]^2), 
\eeno
for some $K_u>0$.

Therefore, there exists $\mathcal{D}^{com,1}(t,k,\xi,\xi_1)$ such that
\beno
(i\xi-ikt)\mathcal{F}_2\Big(\tchi_2\bfD_{u,k}(\tchi_2\mathcal{F}_1f)-\tchi_2\bfD_{u,k}^1(\tchi_2\mathcal{F}_1f)\Big)(\xi)=\int \mathcal{D}^{com,1}(t,k,\xi,\xi_1)\hat{f}_k(t,\xi_1)d\xi_1.
\eeno
Moreover, Lemma \ref{Rmk: fo-g-2} gives us that there exists $\la_{\mathcal{D}}'$ such that 
\beno
|\mathcal{D}^{com,1}(t,k,\xi,\xi_1)|\lesssim e^{-\la_{\mathcal{D}}'|\xi-\xi_1|^{s_0}}, 
\eeno
which together with the fact that 
\beno
i(\xi-kt)\mathcal{D}^{com}(t,k,\xi,\xi_1)=\mathcal{D}^{com,1}(t,k,\xi,\xi_1),
\eeno
gives us that
\beno
\big|\mathcal{D}^{com}(t,k,\xi,\xi_1)\big|\lesssim \f{e^{-\la_{\mathcal{D}}'|\xi-\xi_1|^{s_0}}}{1+|\xi-kt|}. 
\eeno

We can repeat the above argument once more and get that 
\begin{align*}
&(\pa_v-itk)^2\Big(\tchi_2\bfD_{u,k}(\mathcal{F}_1f)-\tchi_2\bfD_{u,k}^1(\mathcal{F}_1f)\Big)\\
&=\tchi_2B_2(v, k)\int_{u(0)}^{u(1)}\Phi_1(v_1, v, k)\Big(\tilde{\th}(u^{-1})'\Big)'(v)\f{\mathcal{F}_{1}{f}(v_1)}{v_1-v}e^{-i(v_1-v)tk}\\
&\qquad\qquad\qquad\qquad \qquad\qquad\qquad \qquad\qquad\qquad 
\times \Big[\widetilde{\varphi_1}'(v)(u^{-1})'(v_1)-\widetilde{\varphi_1}(v_1)(u^{-1})''(v)\Big]dv_1\\
&\quad+\tchi_2B_2(v, k)\Phi_{0,1}^{re}(v, v, k)\mathcal{F}_{1}{f}(v)\Big[\widetilde{\varphi_1}'(v)(u^{-1})'(v)-\widetilde{\varphi_1}(v)(u^{-1})''(v)\Big]\\
&\quad+\tchi_2B_2(v, k)\Phi_{1,1}^{re}(v, v, k)\mathcal{F}_{1}{f}(v)\Big[\widetilde{\varphi_1}'(v)(u^{-1})'(v)-\widetilde{\varphi_1}(v)(u^{-1})''(v)\Big]\\
&\quad+\text{good terms}. 
\end{align*}
Therefore, there exists $\mathcal{D}^{com,2}(t,k,\xi,\xi_1)$ such that
\beno
(i\xi-ikt)^2\mathcal{F}_2\Big(\tchi_2\bfD_{u,k}(\tchi_2\mathcal{F}_1f)-\tchi_2\bfD_{u,k}^1(\tchi_2\mathcal{F}_1f)\Big)(\xi)=\int \mathcal{D}^{com,2}(t,k,\xi,\xi_1)\hat{f}_k(t,\xi_1)d\xi_1.
\eeno
Moreover, there exists $\la_{\mathcal{D}}'$ such that 
\beno
|\mathcal{D}^{com,2}(t,k,\xi,\xi_1)|\lesssim e^{-\la_{\mathcal{D}}'|\xi-\xi_1|^{s_0}}, 
\eeno
which together with the fact that 
\beno
-(\xi-kt)^2\mathcal{D}^{com}(t,k,\xi,\xi_1)=\mathcal{D}^{com,2}(t,k,\xi,\xi_1),
\eeno
gives us that
\beno
\big|\mathcal{D}^{com}(t,k,\xi,\xi_1)\big|\lesssim \f{e^{-\la_{\mathcal{D}}'|\xi-\xi_1|^{s_0}}}{1+|\xi-kt|^2}. 
\eeno
By using the fact that 
\beno
\f{1+|\xi_1-kt|^2}{1+|\xi-kt|^2}\lesssim \f{1+|\xi-kt|^2+|\xi-\xi_1|^2}{1+|\xi-kt|^2}\lesssim \langle\xi-\xi_1\rangle^2,
\eeno
we obtain Proposition \ref{prop: commutator} for some $\la_{\mathcal{D}}<\la_{\mathcal{D}}'$. 
\end{proof}

\section{The good system}\label{sec: good system}
Now we apply the wave operator on $\Om$ in \eqref{eq: Om}. Let us first introduce the new good unknown
\beq\label{eq: def-f}
\begin{aligned}
f(t,z,v)
&\eqdef P_0(\Om)(t,v)+P_{|k|\geq k_{M}}\Om(t,z,v)\\
&\quad+\sum_{0<|k|<k_{M}}\bfD_{u,k}(\mathcal{F}_1{\Om}(t,k,\cdot))(t,k,v)e^{izk},
\end{aligned}
\eeq
where 
\beno
P_{|k|\geq k_{M}}\Om(t,z,v)=\sum_{|k|\geq k_{M}}\mathcal{F}_1\Om(t,k,v)e^{ikz}.
\eeno
Let us give the following remark and determine $k_M$. 
\begin{remark}[Determination of $k_M$]\label{Rmk: determine k_M}
To make the choice of $k_M$ more precise, let us first introduce three constants:
\begin{enumerate}
\item The constant $C_{ell}\geq 1$ only depends on the background flow. It is the constant appearing in the linear elliptic estimate. 
\item The constant $C_{\varphi_1''}= 1+\|\widetilde{\varphi_1}\|_{\mathcal{G}^{s,\s+2}}$ only depends on the Gevrey norm of $\widetilde{\varphi_1}$. 
\item The constant $C_0\geq 1$ is a universal constant that depends on $s$ and $\s$. 
\end{enumerate}
We choose $k_M$ such that
\beq\label{eq: k_M determine}
k_{M}\geq 100C_0C_{u''}C_{ell}.
\eeq 
Since $C_{u''}$ and $C_{ell}$ only depend on the background flow, the constant $k_M$ is fixed once the background flow is given. 
\end{remark}

By Remark \ref{Rmk: compact support}, $f$ has the same compact support as $\Om$. 

Let us now deduce the equation of the new good unknown $f$. An easy calculation shows that
\begin{align*}
\pa_tf
&=\pa_tP_0(\Om)+\pa_tP_{k\geq k_M}\Om\\
&\quad+\pa_t\sum_{0<|k|<k_{M}}\bfD_{u,k}(\mathcal{F}_1{\Om}(t,k,\cdot))(t,k,v)e^{iz k}\\
&=-P_{0}\left(\rmU\cdot\na_{z,v}\Om\right)-P_{0}\left(\mathcal{N}_{\Om}[\Psi]\right)-P_0\left(\mathcal{N}_a[\Pi]\right)
+\udl{\varphi_1}\pa_{z}P_{k\geq k_M}\Psi\\
&\quad-P_{k\geq k_M}\Big(\rmU\cdot\na_{z,v}\Om\Big)
-P_{k\geq k_M}\left(\mathcal{N}_{\Om}[\Psi]\right)-P_{k\geq k_M}\left(\mathcal{N}_a[\Pi]\right)\\
&\quad+\sum_{0<|k|<k_{M}}[\pa_t,\bfD_{u,k}]\left(\mathcal{F}_{1}\Om(t,k,\cdot)\right)(t,k,v)e^{iz k}\\
&\quad+\sum_{0<|k|<k_{M}}\bfD_{u,k}\left(\mathcal{F}_{1}\Big(\udl{\varphi_1}\pa_{z}\big(\Psi\big)\Big)\right)(t,k,v)e^{iz k}\\
&\quad-\sum_{0<|k|<k_{M}}\bfD_{u,k}\left(\mathcal{F}_{1}\Big(\rmU\cdot\na_{z,v}\Om\Big)\right)(t,k,v)e^{iz k}\\
&\quad-\sum_{0<|k|<k_{M}}\bfD_{u,k}\left(\mathcal{F}_{1}\Big(\mathcal{N}_{\Om}[\Psi]\Big)\right)(t,k,v)e^{iz k}
-\sum_{0<|k|<k_{M}}\bfD_{u,k}\left(\mathcal{F}_{1}\Big(\mathcal{N}_{a}[\Pi]\Big)\right)(t,k,v)e^{iz k}.
\end{align*}
By the fact that,
\begin{align*}
&\bfD_{u,k}\left(\mathcal{F}_{1}\Big(\udl{\varphi_1}\pa_{z}\Psi\Big)\right)(t,k,v)\\
&=\bfD_{u,k}\left(\mathcal{F}_{1}\Big(\widetilde{\varphi_1}\pa_{z}\big(\big(\tilde{\Delta}_t^l\big)^{-1}\Om\big)\Big)\right)(t,k,v)
+\bfD_{u,k}\left(\mathcal{F}_{1}\Big(\udl{\varphi_1}\pa_{z}\Psi-\widetilde{\varphi_1}\pa_{z}\big(\big(\tilde{\Delta}_t^l\big)^{-1}\Om\big)\Big)\right)(t,k,v)\\
&=\bfD_{u,k}\left(\mathcal{F}_{1}\Big(\widetilde{\varphi_1}\pa_{z}\big(\big(\tilde{\Delta}_t^l\big)^{-1}\Om\big)\Big)\right)(t,k,v)
+\bfD_{u,k}\left(\mathcal{F}_{1}\Big(\varphi_1^{\d}\pa_{z}\Psi-\widetilde{\varphi_1}\pa_{z}\big(\big(\tilde{\Delta}_t^l\big)^{-1}\Om-\Psi\big)\Big)\right)(t,k,v)\\
&=-ikv\bfD_{u,k}\left(\mathcal{F}_{1}\Om(t,k,\cdot)\right)(t,k,v)
+\bfD_{u,k}\left(ik\mathcal{F}_{1}v\Om(t,k,\cdot)\right)(t,k,v)\\
&\quad+\bfD_{u,k}\left(\mathcal{F}_{1}\Big(\varphi^{\d}\pa_{z}\Psi-\widetilde{\varphi_1}\pa_{z}\big(\big(\tilde{\Delta}_t^l\big)^{-1}\Om-\Psi\big)\Big)\right)(t,k,v),
\end{align*}
we obtain by \eqref{eq:[pa_t,D]} that
\begin{align}
\pa_tf
\label{eq:f}&=-P_{0}\left(\rmU\cdot\na_{z,v}\Om\right)
-P_{0}\left(\mathcal{N}_{\Om}[\Psi]\right)-P_0\left(\mathcal{N}_a[\Pi]\right)
+\udl{\varphi_1}\pa_{z}P_{|k|\geq k_M}\Psi\\
\nonumber&\quad-P_{|k|\geq k_M}\Big(\rmU\cdot\na_{z,v}\Om\Big)
-P_{k\geq k_M}\left(\mathcal{N}_{\Om}[\Psi]\right)-P_{k\geq k_M}\left(\mathcal{N}_a[\Pi]\right)\\
\nonumber&\quad+\sum_{0<|k|<k_{M}}\bfD_{u,k}\left(\mathcal{F}_{1}\Big(\varphi_1^{\d}\pa_{z}\Psi-\widetilde{\varphi_1}\pa_{z}\big(\big(\tilde{\Delta}_t^l\big)^{-1}\Om-\Psi\big)\Big)\right)(t,k,v)e^{izk}\\
\nonumber&\quad-\sum_{0<|k|<k_{M}}\bfD_{u,k}\left(\mathcal{F}_{1}\Big(\rmU\cdot\na_{z,v}\Om\Big)\right)(t,k,v)e^{iz k}\\
\nonumber&\quad-\sum_{0<|k|<k_{M}}\bfD_{u,k}\left(\mathcal{F}_{1}\Big(\mathcal{N}_{\Om}[\Psi]\Big)\right)(t,k,v)e^{iz k}\\
\nonumber&\quad-\sum_{0<|k|<k_{M}}\bfD_{u,k}\left(\mathcal{F}_{1}\Big(\mathcal{N}_{a}[\Pi]\Big)\right)(t,k,v)e^{iz k}.
\end{align}
Here we recall the notations used in the previous calculations
\begin{align*}
\varphi_1^{\d}(t, v)&=\udl{\varphi_1}(t, v)-\widetilde{\varphi_1}(v)\\
\tilde{\Delta}_t^l&=\widetilde{\varphi_4}\pa_{zz}+\widetilde{u'}(\pa_v-t\pa_z)\Big(\widetilde{\varphi_4}\widetilde{u'}(\pa_v-t\pa_z)\Big). 
\end{align*}
We also have 
\begin{align*}
\pa_tv(t,y)=-\f{\chi_1(y)}{t^2}\int_0^tP_0(U^x)(t',y)dt'+\f{\chi_1(y)}{t}P_0(U^x)(t,y),
\end{align*}
and
\begin{align*}
\pa_{tt}v(t,y)=\pa_t\big(\udl{\pa_tv}(t,v(t,y))\big)=\pa_t\big(\udl{\pa_tv}\big)(t,v(t,y))+\pa_tv(t,y)\pa_v\big(\udl{\pa_tv}\big)(t,v(t,y))
\end{align*}
which gives that
\begin{align*}
\pa_{tt}v(t,y)&=\f{2\chi_1(y)}{t^3}\int_0^tP_0(U^x)(t',y)dt'-\f{2\chi_1(y)}{t^2}P_0(U^x)(t,y)+\f{\chi_1(y)}{t}P_0(\pa_tU^x)(t,y)\\
&=-\f{2}{t}\pa_tv(t,y)+\f{\chi_1(y)}{t}P_0(\pa_tU^x)(t,y).
\end{align*}
By using the fact that
\begin{align*}
\pa_tU^x+u(y)\pa_xU^x+u'(y)U^y+\th(y)\pa_xP+d\pa_xP+U^x\pa_xU^x+U^y\pa_yU^x=0,
\end{align*}
we obtain that
\beq\label{eq: v_t}
\begin{aligned}
\pa_t(\udl{\pa_tv})+\udl{\pa_tv}\pa_v(\udl{\pa_tv})+\f{2}{t}(\udl{\pa_tv})=
&-\f{\udl{\chi_1}}{t}\udl{\pa_yv}P_0\left(\na_{z,v}^{\bot}P_{\neq}(\Psi)\cdot\na_{z,v}\underline{U^x}\right)\\
&-\f{\udl{\chi_1}}{t}P_0\Big(P_{\neq}(a)\pa_zP_{\neq}(\Pi_{l}+\Pi_{n})\Big),
\end{aligned}
\eeq
where $\underline{U^x}(t,z(t,x,y),v(t,y))=U^x(t,x,y)$. 

We also have 
\begin{align*}
\pa_yv(t,y)=u'(y)-\f{1}{t}\int_0^tP_0(\om)(t',y)dt',
\end{align*}
then it holds that
\begin{align*}
\pa_t\Big(t\big(\pa_yv(t,y)-u'(y)\big)\Big)=-P_0(\om)(t,y). 
\end{align*}
Let $h(t,v(t,y))=\pa_yv(t,y)-u'(y)$, then $h(t,v)$ satisfies
\beq\label{eq: h}
\pa_th+\udl{\pa_tv}\pa_vh=\f{1}{t}\big(-P_0(\udl{\om})-h\big)\eqdef \bar{h}=\udl{\pa_yv}\pa_v\udl{\pa_tv}.
\eeq
We then get
\beq\label{eq: bar{h}}
\begin{aligned}
\pa_t\bar{h}
&=-\f{\bar{h}}{t}-\f{1}{t}(\pa_tP_0(\udl{\om})+\pa_th)\\
&=-\f{\bar{2h}}{t}-\udl{\pa_tv}\pa_v\bar{h}+\f{\udl{\pa_yv}}{t}\Big(P_0(\pa_z\Psi\pa_v\udl{\om}-\pa_v\Psi\pa_z\udl{\om})-P_0(\pa_z\Pi_{l}\pa_va-\pa_v\Pi_{l}\pa_za)\\
&\quad\quad\quad\quad \quad\quad\quad\quad\quad\quad\quad 
-P_0(\pa_z\Pi_{n}\pa_va-\pa_v\Pi_{n}\pa_za)\Big).
\end{aligned}
\eeq
We now introduce the cut-off function $\Upsilon_1(v), \Upsilon_2(v)\in \mathcal{G}^{M(\kappa_0),\f{s+1}{2}}_{ph,1}$ which satisfies 
\beno
&\mathrm{supp}\, \Upsilon_1(v)\subset \left[u(\kappa_0),u(1-\kappa_0)\right]\\
&\mathrm{supp}\, \Upsilon_2(v)\subset \left[u(\f{\kappa_0}{2}),u(1-\f{\kappa_0}{2})\right]
\eeno 
and
$
\Upsilon(v)\equiv 1,\  \text{for }v\in [u({\kappa_0}),u(1-\kappa_0)]. 
$
\subsection{The working system}
We then rewrite the system by considering the compact support. 
\begin{subequations}\label{eq: main system}
\beq\label{eq:f-1}
\left\{
\begin{aligned}
\pa_tf
&=-P_{0}\left(\Upsilon_2\rmU\cdot\na_{z,v}\Om\right)
-P_{0}\left(\mathcal{N}_{\Om}[\Upsilon_2\Psi]\right)\\
&\quad-P_0\left(\mathcal{N}_a[\Upsilon_2\Pi_{l}]\right)
-P_0\left(\mathcal{N}_a[\Upsilon_2\Pi_{n}]\right)\\
&\quad
-P_{|k|\geq k_M}\Big(\Upsilon_2\rmU\cdot\na_{z,v}\Om\Big)
+\udl{\varphi_1}\pa_{z}P_{|k|\geq k_M}(\Upsilon_2\Psi)
-P_{|k|\geq k_M}\left(\mathcal{N}_{\Om}[\Upsilon_2\Psi]\right)\\
&\quad
-P_{|k|\geq k_M}\left(\mathcal{N}_a[\Pi_{l}]\right)
-P_{|k|\geq k_M}\left(\mathcal{N}_a[\Pi_{n}]\right)\\
&\quad+\sum_{0<|k|<k_{M}}\Upsilon_2\bfD_{u,k}\left(\mathcal{F}_{1}\Big(\varphi_1^{\d}\pa_{z}(\Upsilon_2\Psi)-\widetilde{\varphi_1}\pa_{z}\big(\mathring{\mathcal{T}}_{1,D}^{-1}[\Om]-\Upsilon_2\Psi\big)\Big)\right)(t,k,v)e^{izk}\\
&\quad-\sum_{0<|k|<k_{M}}\Upsilon_2\bfD_{u,k}\left(\mathcal{F}_{1}\Big(\rmU\cdot\na_{z,v}\Om\Big)\right)(t,k,v)e^{iz k}\\
&\quad-\sum_{0<|k|<k_{M}}\Upsilon_2\bfD_{u,k}\left(\mathcal{F}_{1}\Big(\mathcal{N}_{\Om}[\Upsilon_2\Psi]\Big)\right)(t,k,v)e^{iz k}
-\Upsilon_2\bfD(\mathcal{N}_{a}[\Pi])\\
\pa_ta&=-\udl{\th'}\pa_z(\Upsilon_2\Psi)-\Upsilon_2\mathrm{U}\cdot\na_{z,v}a,\\
\mathrm{U}&=(0,\udl{\pa_tv})+\udl{\pa_yv}\na^{\bot}_{v,z}P_{\neq}(\Upsilon_2\Psi),\\
\Om&=\udl{\varphi_4}\pa_{zz}\Psi+\udl{\pa_yv}(\pa_v-t\pa_z)\Big(\udl{\varphi_4}\udl{\pa_yv}(\pa_v-t\pa_z)\Psi\Big).
\end{aligned}\right.
\eeq
where the operator $\mathring{\mathcal{T}}_{1,D}^{-1}$ is defined in \eqref{eq: inverse cutoff-op} and $\bfD(\mathcal{N}_{a}[\Pi])=\bfD(\mathcal{N}_{a}[\Pi_{l}])+\bfD(\mathcal{N}_{a}[\Pi_{n}])$ with 
\beq
\begin{aligned}
&\bfD(\mathcal{N}_{a}[\Pi_{l}])=\sum_{0<|k|<k_{M}}\bfD_{u,k}\left(\mathcal{F}_{1}\Big(\Upsilon_2\mathcal{N}_{a}[\Pi_{l}]\Big)\right)(t,k,v)e^{iz k}\\
&\bfD(\mathcal{N}_{a}[\Pi_{n}])=\sum_{0<|k|<k_{M}}\bfD_{u,k}\left(\mathcal{F}_{1}\Big(\Upsilon_2\mathcal{N}_{a}[\Pi_{n}]\Big)\right)(t,k,v)e^{iz k}.
\end{aligned}
\eeq
with $\mathcal{N}_a(\Pi)$ being defined in \eqref{eq: N_aPi} and 
\beq
\Pi_{l}=\Upsilon_2\Pi_{l,1}+\Upsilon_1\Pi_{l,2},\quad \Pi_{n}=\Upsilon_2\Pi_{n,1}+\Upsilon_1\Pi_{n,2}. 
\eeq
We can recover $\Om$ from $f$. 
\begin{align}\label{eq: Om-f}
\Om(t, z, v)=P_0(f)+\sum_{0<|k|<k_{M}}\bfD_{u, k}^{-1}\big[\mathcal{F}_1f\big](t, k, v)e^{izk}+P_{|k|\geq k_{M}}[f](t, z, v).
\end{align}
The pressure terms $\Pi$ are solved from the elliptic systems:
\beq\label{eq: Pressure-l1}
\left\{
\begin{aligned}
&\pa_z\Big((a+\udl{\th})\pa_z\Pi_{l,1}\Big)+\udl{\pa_yv}(\pa_v-t\pa_z)\Big((a+\udl{\th})\udl{\pa_yv}(\pa_v-t\pa_z)\Pi_{l,1}\Big)\\
&\quad\quad \quad \quad\quad \quad \quad\quad \quad 
=\udl{\varphi_6}\pa_{zz}(\Upsilon_2\Psi)+\udl{\varphi_7}P_{\neq}(\Upsilon_2\Psi),\\
&(\pa_v-t\pa_z)\Pi_{l,1}(t, z, u(0))=(\pa_v-t\pa_z)\Pi_{l,1}(t, z, u(1))=0,
\end{aligned}\right.
\eeq 
and
\beq\label{eq: Pressure-l2}
\left\{
\begin{aligned}
&\pa_z\Big((a+\udl{\th})\pa_z\Pi_{l,2}\Big)+\udl{\pa_yv}(\pa_v-t\pa_z)\Big((a+\udl{\th})\udl{\pa_yv}(\pa_v-t\pa_z)\Pi_{l,2}\Big)\\
&\quad\quad \quad \quad\quad \quad\quad\quad \quad  =\udl{\varphi_8}\udl{\pa_yv}(\pa_v-t\pa_z)P_{\neq}(\Upsilon_2\Psi),\\
&(\pa_v-t\pa_z)\Pi_{l,2}(t, z, u(0))=(\pa_v-t\pa_z)\Pi_{l,2}(t, z, u(1))=0,
\end{aligned}\right.
\eeq 
and
\beq\label{eq: Pressure-n1}
\left\{
\begin{aligned}
&\pa_z\Big((a+\udl{\th})\pa_z\Pi_n\Big)+\udl{\pa_yv}(\pa_v-t\pa_z)\Big((a+\udl{\th})\udl{\pa_yv}(\pa_v-t\pa_z)\Pi_n\Big)\\
&\quad\quad \quad \quad\quad 
=-2\udl{\chi_2}(\udl{\pa_yv}(\pa_v-t\pa_z)\pa_z(\Upsilon_2\Psi))^2-2\pa_{zz}(\Upsilon_2\Psi)\Big(\Om-\udl{\chi_2}\pa_{zz}(\Upsilon_2\Psi)\Big),\\
&\quad\quad \quad \quad\quad \quad
-2\udl{\chi_2'}\udl{\pa_yv}\Big[(\pa_z(\Upsilon_2\Psi))(\pa_v-t\pa_z)\pa_z(\Upsilon_2\Psi)+(\pa_v-t\pa_z)(\Upsilon_2\Psi)\pa_{zz}(\Upsilon_2\Psi)\Big]\\
&\quad\quad \quad \quad\quad \quad
-\f{1}{2}\udl{\chi_2''}\Big[(\pa_z(\Upsilon_2\Psi))^2+(\udl{\pa_yv})^2((\pa_v-t\pa_z)(\Upsilon_2\Psi))^2\Big],\\
&(\pa_v-t\pa_z)\Pi_n(t, z, u(0))=(\pa_v-t\pa_z)\Pi_n(t, z, u(1))=0.
\end{aligned}\right.
\eeq 
and
\beq\label{eq: Pressure-n-2}
\left\{
\begin{aligned}
&\pa_z\Big((a+\udl{\th})\pa_z\Pi_{n,2}\Big)+\udl{\pa_yv}(\pa_v-t\pa_z)\Big((a+\udl{\th})\udl{\pa_yv}(\pa_v-t\pa_z)\Pi_{n,2}\Big)\\
&\quad\quad \quad 
=
-\f{1}{2}\udl{\chi_2''}\Big[2(\udl{\pa_yv})^2((\pa_v-t\pa_z)P_{\neq}(\Upsilon_2\Psi))\pa_vP_0(\Upsilon_2\Psi)+(\udl{\pa_yv})^2(\pa_vP_0(\Upsilon_2\Psi))^2\Big],\\
&(\pa_v-t\pa_z)\Pi_{n,2}(t, z, u(0))=(\pa_v-t\pa_z)\Pi_{n,2}(t, z, u(1))=0.
\end{aligned}\right.
\eeq 
We have that the zero mode of the pressure solves
\beq\label{eq: pressure-zero-mode}
\begin{aligned}
&\pa_vP_0(\Pi\Upsilon_1)+(\udl{\varphi_4}\Upsilon_2)P_0(a)\pa_vP_0(\Pi\Upsilon_1)+(\udl{\varphi_4}\Upsilon_2)P_0\big[P_{\neq}(a)P_{\neq}(\pa_v-t\pa_z)P_{\neq}(\Pi\Upsilon_1)\big]\\
&=(\udl{\varphi_4}\Upsilon_2)P_0\Big(\pa_z(\Psi\Upsilon_2)\pa_z(\pa_v-t\pa_z)(\Psi\Upsilon_2)-\pa_{zz}(\Psi\Upsilon_2)(\pa_v-t\pa_z)(\Psi\Upsilon_2)\Big).
\end{aligned}
\eeq
We also list the equations related to the change of coordinates.
\beq\label{eq: v_t-2}
\begin{aligned}
\pa_t(\udl{\pa_tv})+\udl{\pa_tv}\pa_v(\udl{\pa_tv})+\f{2}{t}(\udl{\pa_tv})=
&-\f{\udl{\chi_1}}{t}\udl{\pa_yv}P_0\left(\na_{z,v}^{\bot}P_{\neq}(\Upsilon_2\Psi)\cdot\na_{z,v}(\Upsilon_2\underline{U^x})\right)\\
&-\f{\udl{\chi_1}}{t}P_0\Big(P_{\neq}(a)\pa_zP_{\neq}\big((\Upsilon_2\Pi_{l,1})+(\Upsilon_2\Pi_{l,2})+(\Upsilon_2\Pi_{n})\big)\Big),
\end{aligned}
\eeq
where 
\beno
\underline{U^x}(t,z,v)=-\underline{\pa_yv}(\pa_v-t\pa_z)(\Upsilon_2\Psi). 
\eeno 
For $j=1,2, ...,11$, we have
\beq\label{eq: varphi^d}
\pa_t\varphi_j^{\d}+\udl{\pa_tv}(t,v)\pa_v\varphi_j^{\d}=-\udl{\pa_tv}(t,v)\pa_v\widetilde{\varphi_j}(v)
\eeq
We also have that $h=\udl{\pa_yv}-\udl{u'}$ solves
\beq\label{eq: h-2}
\pa_th+\udl{\pa_tv}\pa_vh=\bar{h}.
\eeq
where $\bar{h}$ solves
\beq\label{eq: bar{h}-2}
\begin{aligned}
\pa_t\bar{h}
&=-\f{2\bar{h}}{t}-\udl{\pa_tv}\pa_v\bar{h}\\
&\quad+\f{\udl{\pa_yv}}{t}\Big(P_0(\pa_z(\Upsilon_2\Psi)\pa_v\udl{\om}-\pa_v(\Upsilon_2\Psi)\pa_z\udl{\om})-P_0(\pa_z(\Pi_{l})\pa_va-\pa_v(\Pi_{l})\pa_za)\\
&\quad\quad\quad\quad \quad
-P_0(\pa_z(\Upsilon_2\Pi_{n})\pa_va-\pa_v(\Upsilon_2\Pi_{n})\pa_za)\Big).
\end{aligned}
\eeq

\end{subequations}

\section{Energy functional and bootstrap proposition}\label{sec energy functional}
We will use the same multiplier $\rmA(t,\na)$ introduced in \cite{BM2015}. 
\beno
\rmA_k(t,\eta)=e^{\la(t)\langle k,\eta\rangle^s}\langle k,\eta\rangle^{\s}\rmJ_{k}(t,\eta).
\eeno
The index $\la(t)$ is the bulk Gevrey-$\f{1}{s}$ regularity and will be chosen to satisfy
\beno
&&\la(t)=\f34\la_0+\f14\la',\quad t\leq 1\\
&&\dot{\la}(t)=-\f{\d_{\la}}{\langle t\rangle^{2\tilde{q}}}(1+\la(t)),\quad t\in (1,\infty)
\eeno
where $\la_0, \la'$ are parameters that depend on the regularity of the background flow $u(y)$, density $\th(y)$ and chosen by the proof, $\d_{\la}\approx \la_0-\la'$ is a small parameter that ensures $\la(t)\geq \f{\la_0}{2}+\f{\la'}{2}$ for $t\geq 0$ and $\f12<\tilde{q}\leq \f{s}{8}+\f{7}{16}$ is a parameter chosen by the proof. 

The main multiplier for dealing with the Orr mechanism and the associated nonlinear growth is 
\beq\label{eq: J}
\rmJ_{k}(t,\eta)=\f{e^{\mu |\eta|^{\f12}}}{w_k(t,\eta)}+e^{\mu |k|^{\f12}},
\eeq
where 
\beq\label{eq: w_k(t,eta)}
w_k(t,\eta)=\left\{
\begin{aligned}
&w_k(t_{\rmE(\sqrt{\eta}),\eta},\eta)\quad t<t_{\rmE(\sqrt{\eta}),\eta}\\
&w_{\mathrm{NR}}(t,\eta)\quad t\in [t_{\rmE(\sqrt{\eta}),\eta},2\eta]\setminus \mathrm{I}_{k,\eta}\\
&w_{\mathrm{R}}(t,\eta)\quad t\in \mathrm{I}_{k,\eta}\\
&1\quad t\geq 2\eta.
\end{aligned}
\right.
\eeq
Here $(w_{\mathrm{R}}(t,\eta),w_{\mathrm{NR}}(t,\eta))$ is defined in the following way: 

Let $w_{\mathrm{NR}}$ be a non-decreasing function of time with $w_{\mathrm{NR}}(t,\eta)=1$ for $t\geq 2\eta$. For definiteness, we remark here that for $|\eta|\leq 1$, $w_{\mathrm{NR}}(t,\eta)=1$, which will be a consequence of the definition. For $k=1,2,3,...,\rmE(\sqrt{\eta})$, we define
\beq\label{eq: w_NR-R}
\begin{aligned}
w_{\mathrm{NR}}&(t,\eta)=\left(\f{k^2}{\eta}\Big(1+b_{k,\eta}\Big|t-\f{\eta}{k}\Big|\Big)\right)^{C\kappa_1}w_{\mathrm{NR}}(t_{k-1,\eta},\eta),\\
&\forall t\in \mathrm{I}^{\mathrm{R}}_{k,\eta}=\left[\f{\eta}{k},t_{k-1,\eta}\right],
\end{aligned}
\eeq
and
\beq\label{eq: w_NR-L}
\begin{aligned}
w_{\mathrm{NR}}&(t,\eta)=\left(1+a_{k,\eta}\Big|t-\f{\eta}{k}\Big|\right)^{-1-C\kappa_1}w_{\mathrm{NR}}\left(\f{\eta}{k},\eta\right),\\
&\forall t\in \mathrm{I}^{\mathrm{L}}_{k,\eta}=\left[t_{k,\eta},\f{\eta}{k}\right].
\end{aligned}
\eeq
The constant $b_{k,\eta}$ is chosen to ensure that $\f{k^2}{\eta}\Big(1+b_{k,\eta}\Big|t_{k-1,\eta}-\f{\eta}{k}\Big|\Big)=1$, hence for $k\geq 2$
\beq\label{eq: b_k,eta}
b_{k,\eta}=\f{2(k-1)}{k}\left(1-\f{k^2}{\eta}\right)
\eeq
and $b_{1,\eta}=1-\f{1}{\eta}$ and similarly, 
\beq\label{eq: a_k,eta}
a_{k,\eta}=\f{2(k+1)}{k}\left(1-\f{k^2}{\eta}\right).
\eeq
On each interval $\mathrm{I}_{k,\eta}$, we define $w_{\mathrm{R}}(t,\eta)$ by
\begin{align}\label{eq: w_R-R}
&w_{\mathrm{R}}(t,\eta)=\f{k^2}{\eta}\left(1+b_{k,\eta}\Big|t-\f{\eta}{k}\Big|\right)w_{\mathrm{NR}}(t,\eta),\quad
\forall t\in \mathrm{I}^{\mathrm{R}}_{k,\eta}=\left[\f{\eta}{k},t_{k-1,\eta}\right],\\
\label{eq: w_R-L}
&w_{\mathrm{R}}(t,\eta)=\f{k^2}{\eta}\left(1+a_{k,\eta}\Big|t-\f{\eta}{k}\Big|\right)w_{\mathrm{NR}}(t,\eta),\quad
\forall t\in \mathrm{I}^{\mathrm{L}}_{k,\eta}=\left[t_{k,\eta},\f{\eta}{k}\right].
\end{align}
We also define $\rmJ^{\mathrm{R}}(t,\eta)$ and $\rmA^{\mathrm{R}}(t,\eta)$ to assign resonant regularity at every critical time: 
\beq\label{eq: J^R,A^R}
\begin{aligned}
&\rmJ^{\mathrm{R}}(t,\eta)
=\left\{\begin{aligned}
&e^{\mu |\eta|^{\f12}}w_{\mathrm{R}}^{-1}(t_{\rmE(\sqrt{\eta}),\eta},\eta)\quad t<t_{\rmE(\sqrt{\eta}),\eta}\\
&e^{\mu |\eta|^{\f12}}w_{\mathrm{R}}^{-1}(t,\eta)\quad t\in [t_{\rmE(\sqrt{\eta}),\eta},2\eta]\\
&e^{\mu |\eta|^{\f12}}\quad t\geq 2\eta,
\end{aligned}\right.\\
&\rmA^{\mathrm{R}}(t,\eta)=e^{\la(t)\langle\eta\rangle^s}\langle\eta\rangle^{\s}\rmJ^{\mathrm{R}}(t,\eta).
\end{aligned}
\eeq

One may refer to \cite{BM2015} and Appendix \ref{sec: operator A} for more basic properties of the multiplier $\rmA(t,\na)$. 

We also introduce the Fourier multiplier $\rmA^*(t,\na)=\mathcal{F}^{-1}(\rmA^*_k(t, \eta)\widehat{f}_k(t,\eta))$ from \cite{ChenWeiZhangZhang2023}, 
\beno
\rmA^*_k(t, \eta)=\rmA_k(t,\eta)\rmB_k(t,\eta),\quad \rmB_k(t,\eta)=\Big(1+\f{k^2+|\eta|}{\langle t\rangle^2}\Big)^{\f12}. 
\eeno
The weight is well-constructed to capture the important regularity difference between the vorticity $f$ and the density $a$. 
\begin{remark}
It holds that
\begin{align*}
\f{\pa_t\rmA^*_k(t, \eta)}{\rmA^*_k(t, \eta)}
=\f{\pa_t\rmA_k(t, \eta)}{\rmA_k(t, \eta)}-\f{t(k^2+|\eta|)}{\langle t\rangle^4+(k^2+|\eta|)\langle t\rangle^2}\approx \f{\pa_t\rmA_k(t, \eta)}{\rmA_k(t, \eta)}. 
\end{align*}
\end{remark}
\begin{proof}
The remark follows directly from the fact that
\beno
\f{t(k^2+|\eta|)}{\langle t\rangle^4+(k^2+|\eta|)\langle t\rangle^2}\lesssim \f{\langle k\rangle^{\f12}}{\langle t\rangle^{\f32}}+ \f{\langle \eta\rangle^{\f12}}{\langle t\rangle^{2}}\lesssim \f{\langle k,\eta\rangle^s}{\langle t\rangle^{2\bar{q}}}. 
\eeno
\end{proof}

Let us mention some ideas for choosing the multiplier $\rmB$. It is used to solve the problem from the nonlinear interactions between the pressure and the density, which can be formally written as $\pa_z\Pi\pa_va$ and $\pa_v\Pi\pa_z a$. On one hand, if the pressure is in higher frequencies, namely, $\pa_z\Pi_{high}\pa_va_{low}$ and $\pa_v\Pi_{high}\pa_z a_{low}$, both terms behavior as the Reaction terms $\pa_z\Psi_{high}\pa_v\Om_{low}$ and $\pa_v\Psi_{high}\pa_z \Om_{low}$. They can be controlled easily. We only need to control  
\beno
\left\|\left\langle\f{\pa_v}{t\pa_z}\right\rangle^{-1}(\pa_z^2+(\pa_v-t\pa_z)^2)\left(\f{|\na|^{\f{s}{2}}}{\langle t\rangle^s}\rmA+\sqrt{\f{\pa_t w}{w}}\tilde{\rmA}\right)P_{\neq}(\Pi\Upsilon_2)\right\|_2^2
\eeno
by the $\mathcal{CK}$ terms. However, it is not enough to close the elliptic estimate for the pressure $\Pi$. Indeed, unlike the elliptic equation of the stream function, the elliptic error terms in the pressure equations are not just zero mode, see $\mathcal{S}_a[\Pi]$ and its estimates in Lemma \ref{lem: S_a}. Notice that $\Pi$ behaves better than the stream function $\Psi$, which allows us to replace $\left\langle\f{\pa_v}{t\pa_z}\right\rangle^{-1}\rmA(t,\na)$ by a stronger weight $\mathcal{M}_3$, see Propositions \ref{prop:pressure-zero mode}, \ref{Prop: non-zero-pessure-L}, and \ref{Prop: non-zero-pessure-N}, and section \ref{sec: pressure-zero} and section \ref{sec: ell-pressure} for more details. On the other hand, when the density $a$ is in higher frequencies, namely, $\pa_z\Pi_{low}\pa_va_{high}$ and $\pa_v\Pi_{low}\pa_z a_{high}$, we have derivative loss. Note that in similar terms that involve the stream function and vorticity, we use the commutator to gain a $1/2$ derivative. Here we put half more derivative in the multiplier of $a$ for short time, which allows us to control these terms. Then we should be careful in estimating the linear term $-\udl{\th'}\pa_z(\Upsilon_2\Psi)$ in the equation of $a$, because the added weight $\rmB$ will act on $\Psi$, see Proposition \ref{prop: I_a^l} and section \ref{sec: I_a^l} for more details. The other place where we take advantage of the operator $\rmB$ is the estimate of elliptic error term $\mathcal{S}_a(\Pi)$, see the proof of Lemma \ref{lem: S_a} for more details.

Let us now remark here that by the nice Fourier kernel estimates, the operator $\rmA$ commutates well with the wave operator. The following two remarks can be found in \cite{MasmoudiZhao2020}. 
\begin{remark}\label{Rmk: wave bdd-A}
There exists $C=C(k_M)\geq 1$ independent of $t$ such that for any $\la<\la_{\mathcal{D}}$,
\begin{align*}
{C}^{-1}\|\Om\|_{\mathcal{G}^{s,\la,\s}(\mathbb{T}\times [u(0),u(1)])}
\leq \|f\|_{\mathcal{G}^{s,\la,\s}(\mathbb{T}\times [u(0),u(1)])}
\leq C\|\Om\|_{\mathcal{G}^{s,\la,\s}(\mathbb{T}\times [u(0),u(1)])},
\end{align*}
and that there exists $C=C(k_M)\geq 1$ independent of $t$  such that for any $\la(t)<\la_{\mathcal{D}}$,
\beno
C^{-1}\|\rmA\Om\|_2^2\leq \|\rmA f\|_2^2\leq C\|\rmA\Om\|_2^2. 
\eeno
\end{remark}
\begin{remark}\label{Rmk:CKAOm}
Proposition \ref{prop: kernel-wave-op} gives us that for $\la_0$ and $\la'$ sufficiently small, 
\beno
-\dot{\la}(t)\left\||\na|^{s/2}\rmA \Om\right\|_2^2+\left\|\sqrt{\f{\pa_tw}{w}}\tilde{\rmA}\Om\right\|_2^2\lesssim \mathcal{CK}_{f},
\eeno
where 
\beno
\mathcal{CK}_f=\sum_{k\in\mathbb{Z}}\int_{\mathbb{R}}|\dot{\rmA}_k(t,\eta)|\rmA_k(t,\eta)|\widehat{f}_k(t,\eta)|^2d\eta.
\eeno
\end{remark}
\begin{remark}\label{Rmk: switch-large-zero}
The following fact will be frequently used: for any function $\udl{\varphi}$ of zero mode, we write $\udl{\varphi}=\widetilde{\varphi}+\varphi^{\d}$
\begin{align*}
\left\|\sqrt{\f{\pa_tw}{w}}\tilde{\rmA}(\udl{\varphi}P_{\neq}f)\right\|_{L^2}
&\lesssim \|\widetilde{\varphi}\|_{\mathcal{G}^{s,\s+2}}\mathcal{CK}_f+ \|\varphi^{\d}\|_{\mathcal{G}^{s,\s-6}}\mathcal{CK}_f+\left\|\sqrt{\f{\pa_tw}{w}}\rmA^{\rmR}\varphi^{\d}\right\|_2\|f\|_{\mathcal{G}^{s,\s-6}}\\
&\lesssim \mathcal{CK}_f+\ep\mathcal{CK}_{\varphi^{\d}},
\end{align*}
where
\beno
\mathcal{CK}_{\varphi^{\d}}=\int_{\mathbb{R}}\dot{\rmA}^{\rmR}(t,\eta)\rmA^{\rmR}(t,\eta)|\widehat{\varphi^{\d}}(t,\eta)|^2d\eta.
\eeno
\end{remark}
We give the following remark, related to the operator $\rmA$. 
\begin{remark}\label{Rmk: A-switch}
Let $\ell=k-l$ and $\zeta=\xi-\eta$, then 
\begin{enumerate}
\item if $\ell\neq 0$, $k\neq 0$, and $l\neq 0$, then
\begin{align*}
\rmA_k(\xi)\lesssim &\rmA_{l}(\eta)e^{c\la|\ell,\zeta|^s}\Big(1+\f{|\xi|}{k^2+(\xi-kt)^2}\mathbf{1}_{t\in {\rm{I}}_{k,\xi}\cap {\rm{I}}_{k,\eta}}\mathbf{1}_{c|\xi|\leq |\eta|\leq \f1c|\xi|}\Big)\mathbf{1}_{\rmL\rmH}\\
&+\rmA_{\ell}(\zeta)e^{c\la|l,\eta|^s}\Big(1+\f{|\xi|}{k^2+(\xi-kt)^2}\mathbf{1}_{t\in {\rm{I}}_{k,\xi}\cap {\rm{I}}_{k,\zeta}}\mathbf{1}_{c|\xi|\leq |\zeta|\leq \f1c|\xi|}\Big)\mathbf{1}_{\rmH\rmL}\\
&+e^{c\la|l,\eta|^s}e^{c\la|\ell,\xi|^s}\langle l,\eta\rangle^{\s/2+1}\langle \ell,\xi\rangle^{\s/2+1}\mathbf{1}_{\rmH\rmH},
\end{align*}
\item if $\ell=0$ and $k\neq 0$, then
\begin{align*}
\rmA_k(\xi)\lesssim &\rmA_{k}(\eta)e^{c\la|\zeta|^s}\mathbf{1}_{\rmL\rmH}
+\rmA^{\rmR}(\zeta)e^{c\la|k,\eta|^s}\mathbf{1}_{\rmH\rmL}\\
&+e^{c\la|k,\eta|^s}e^{c\la|\xi|^s}\langle k,\eta\rangle^{\s/2+1}\langle \xi\rangle^{\s/2+1}\mathbf{1}_{\rmH\rmH},
\end{align*}
\end{enumerate}
where $\rmL\rmH=\{4|\ell,\zeta|\leq |l,\eta|\}$, $\rmH\rmL=\{|\ell,\zeta|\geq 4|l,\eta|\}$, and $\rmH\rmH=\{\f14|\ell,\zeta|\leq |l,\eta|\leq 4|\ell,\zeta|\}$. 
\end{remark}

\subsection{Energy functionals}
Now we define the energy functionals:
\begin{align*}
\mathcal{E}_f&=\sum_{k\in\mathbb{Z}}\int_{\mathbb{R}}|\rmA_k(t,\eta)\widehat{f}_k(t,\eta)|^2d\eta,\quad 
\mathcal{CK}_f=\sum_{k\in\mathbb{Z}}\int_{\mathbb{R}}|\dot{\rmA}_k(t,\eta)|\rmA_k(t,\eta)|\widehat{f}_k(t,\eta)|^2d\eta,\\
\mathcal{E}_a&=\sum_{k\in\mathbb{Z}}\int_{\mathbb{R}}|\rmA_k^*(t,\eta)\widehat{f}_k(t,\eta)|^2d\eta,\quad
\mathcal{CK}_a=\sum_{k\in\mathbb{Z}}\int_{\mathbb{R}}|\dot{\rmA}_k^*(t,\eta)|\rmA_k^*(t,\eta)|\widehat{f}_k(t,\eta)|^2d\eta, \\
\mathcal{E}_{h}&=\int_{\mathbb{R}}|\rmA^{\rmR}(t,\eta)\widehat{h}(t,\eta)|^2d\eta,\quad
\mathcal{CK}_{h}=\int_{\mathbb{R}}|\dot{\rmA}^{\rmR}(t,\eta)|\rmA^{\rmR}(t,\eta)|\widehat{h}(t,\eta)|^2d\eta\\
\mathcal{E}_{\bar{h}}&=\int_{\mathbb{R}}\f{\langle t\rangle^{2+2s}}{\langle \eta\rangle^{2s}}|\rmA_0(t,\eta)\widehat{(\udl{\pa_yv}\pa_v\udl{\pa_tv})}(t,\eta)|^2d\eta,\\
\mathcal{CK}_{\bar{h}}&=\int_{\mathbb{R}}\f{\langle t\rangle^{2+2s}}{\langle \eta\rangle^{2s}}|\dot{\rmA}_0(t,\eta)|\rmA_0(t,\eta)|\widehat{(\udl{\pa_yv}\pa_v\udl{\pa_tv})}(t,\eta)|^2d\eta, \\
\mathcal{E}_{\varphi^{\d}_j}&=\int_{\mathbb{R}}|\rmA^{\rmR}(t,\eta)\widehat{\varphi^{\d}_j}(t,\eta)|^2d\eta,\quad
\mathcal{CK}_{\varphi^{\d}_j}=\int_{\mathbb{R}}\dot{\rmA}^{\rmR}(t,\eta)\rmA^{\rmR}(t,\eta)|\widehat{\varphi^{\d}_j}(t,\eta)|^2d\eta,\quad j=1,2,...,11.
\end{align*}
We also introduce an assistant functional:
\begin{align*}
\mathcal{E}_{\udl{\pa_tv}}&=\int_{\mathbb{R}}\langle t\rangle^{4-\rmK_{\rmD}\ep}e^{2\la(t)\langle \eta\rangle^s}\langle \eta\rangle^{2\s-12}|\widehat{\udl{\pa_tv}}(t,\eta)|^2d\eta.
\end{align*}

We define the total energy functional 
\begin{align*}
\mathcal{E}(t)=\mathcal{E}_f+\f{1}{\rmK_a}\mathcal{E}_a+\mathcal{E}_{\bar{h}}+\mathcal{E}_{\udl{\pa_tv}}+\f{1}{\rmK_v}\Big(\sum_{j=1}^{11}\mathcal{E}_{\varphi^{\d}_j}\Big)
\end{align*}
where $\rmK_{\rmD}\geq 1$, $\rmK_a\geq 1$ and $\rmK_v\geq 1$ are large constants determined in the proof. Indeed, $\rmK_{\rmD}$ depends on $s, \la, \la'$, $\rmK_a$ is mainly determined by the constant $C_{\varepsilon_0}$ in Proposition \ref{prop: I_a^l} and $\rmK_v$ is mainly determined by the Gevrey norms of $\widetilde{\varphi_j}$ and $\widetilde{\th}$. 

\subsection{Determine $k_M$}
In this section, we explain why the constant $k_{M}$ is determined as it is in Remark \ref{Rmk: determine k_M}. We define
\ben\label{eq: Linear nonlocal term-high in k}
I_{f,5}^{L}\eqdef\f1{2\pi}\sum_{|k|\geq k_{M}}\int_{\mathbb{R}}\rmA_k(t,\eta)^2\widehat{\Big(\widetilde{\varphi_1}\pa_{z}\mathring{\mathcal{T}}_{1,D}^{-1}(\Om)\Big)}_k(t,\eta)\overline{\widehat{\Om}_k(t, \eta)}d\eta.
\een
This is the linear part of $I_{f,5}$ which is defined in \eqref{eq: definition-I_fj}. We have the following Lemma. 
\begin{lemma}\label{lem: I_{f,5}^{L}}
It holds for all $k_M\geq 1$ that
\begin{align}\label{eq: estimate I_f5^L}
|I_{f,5}^{L}|\leq \f{C_{\varphi_1}}{\langle t\rangle^2}\|\rmA P_{|k|\geq k_M}f\|_{L^2}^2+\f{C_0C_{ell}C_{\varphi_1}}{k_M} \mathcal{CK}_f. 
\end{align}
Here both $C_{ell}>1$ and $C_{\varphi_1}>1$ are constants only depending on the background flow and density. More precisely, the constant $C_{ell}$ is from Proposition \ref{eq: elliptic estimate}, and the constant $C_{\varphi_1}$ is given by the Gevrey norm of $\widetilde{\varphi_1}$. 
\end{lemma}
\begin{proof}
We have
\begin{align*}
I_{f,5}^L
&=\f1{2\pi}\sum_{|k|\geq k_{M}}\int_{|\eta-\xi|\geq 8|k,\xi|}\rmA_k(t,\eta)^2\Big(\widehat{\widetilde{\varphi_1}}(\eta-\xi)ik\widehat{\mathring{\mathcal{T}}_{1,D}^{-1}(\Om)}_k(t,\xi)\Big)\overline{\widehat{\Om}_k(t, \eta)} d\xi d\eta\\
&\quad +\f1{2\pi}\sum_{|k|\geq k_{M}}\int_{|\eta-\xi|\leq \f18|k,\xi|}\rmA_k(t,\eta)^2\Big(\widehat{\widetilde{\varphi_1}}(\eta-\xi)ik\widehat{\mathring{\mathcal{T}}_{1,D}^{-1}(\Om)}_k(t,\xi)\Big)\overline{\widehat{\Om}_k(t, \eta)} d\xi d\eta\\
&\quad +\f1{2\pi}\sum_{|k|\geq k_{M}}\int_{\f18|k,\xi|\leq |\eta-\xi|\leq 8|k,\xi|}\rmA_k(t,\eta)^2\Big(\widehat{\widetilde{\varphi_1}}(\eta-\xi)ik\widehat{\mathring{\mathcal{T}}_{1,D}^{-1}(\Om)}_k(t,\xi)\Big)\overline{\widehat{\Om}_k(t, \eta)} d\xi d\eta\\
&=I_{f,5}^{L, HL}+I_{f,5}^{L,LH}+I_{f,5}^{L,HH}. 
\end{align*}
On the support of the integrand of $I_{f,5}^{L, HL}$, we have 
\beno
\rmA_k(t,\eta)\leq \rmA^{\rmR}(t, \eta-\xi)e^{c\la \langle k,\xi\rangle^{s}},
\eeno
which implies that 
\begin{align*}
|I_{f,5}^{L, HL}|
&\leq C\|\rmA^{\rmR}\widetilde{\varphi_1}\|_{L^2}\|P_{|k|\geq k_{M}}\mathring{\mathcal{T}}_{1,D}^{-1}(\Om)\|_{\mathcal{G}^{s,\la,4}}\|\rmA P_{|k|\geq k_{M}} \Om\|_{L^2}\\
&\leq C_{ell}\langle t\rangle^2\|\rmA^{\rmR}\widetilde{\varphi_1}\|_{L^2}\|P_{|k|\geq k_{M}}f\|_{\mathcal{G}^{s,\la,6}}\|\rmA P_{|k|\geq k_{M}} f\|_{L^2}.
\end{align*}
Here we use Proposition \ref{eq: elliptic estimate}
\begin{align*}
&\sum_{|k|\geq k_{M}}\|\langle k, \xi\rangle^4e^{c\la\langle k,\xi\rangle^s}\widehat{\mathring{\mathcal{T}}_{1,D}^{-1}(\Om)}_{k}(t,\xi)\|_{L^2}^2\\
&=\sum_{|k|\geq k_{M}}\left\|\langle k, \xi\rangle^4e^{c\la\langle k,\xi\rangle^s}\int \mathcal{G}_{D,1}(t, k, \xi,\eta)\widehat{\Om}_k(t,\eta)d\eta\right\|_{L^2}^2\\
&\leq \sum_{|k|\geq k_{M}}\left\|\int \f{e^{\la\langle k,\xi-\eta\rangle^s}e^{-\la_{\Delta}\langle \xi-\eta\rangle^s}}{1+k^2+(\eta-kt)^2}\langle k, \eta\rangle^4e^{c\la\langle k,\eta\rangle^s}\widehat{\Om}_k(t,\eta)d\eta\right\|_{L^2}^2\\
&\leq \sum_{|k|\geq k_{M}}\left\|\int \f{e^{-\f12\la_{\Delta}\langle \xi-\eta\rangle^s}}{\langle t\rangle^2}\langle k, \eta\rangle^6e^{c\la\langle k,\eta\rangle^s}\widehat{\Om}_k(t,\eta)d\eta\right\|_{L^2}^2
\leq \langle t\rangle^{-4} C_{ell}\|P_{|k|\geq k_{M}}f\|_{\mathcal{G}^{s,\la,6}}^2.
\end{align*}
A similar argument gives that 
\beno
|I_{f,5}^{L, HH}|
&\leq C_{ell}\langle t\rangle^2\|\rmA^{\rmR}\widetilde{\varphi_1}\|_{L^2}\|\rmA P_{|k|\geq k_{M}} f\|_{L^2}^2.
\eeno
To estimate $|I_{f,5}^{L, LH}|$, we have by Proposition \ref{eq: elliptic estimate} that,  
\begin{align*}
|I_{f,5}^{L, LH}|
&=\left|\f1{2\pi}\sum_{|k|\geq k_{M}}\int_{|\eta-\xi|\leq \f18|k,\xi|}\rmA_k(t,\eta)^2\Big(\widehat{\widetilde{\varphi_1}}(\eta-\xi)ik \mathcal{G}_{D,1}(t, \xi, \zeta)\widehat{\Om}_k(\zeta)\Big)\overline{\widehat{\Om}_k(t, \eta)} d\zeta d\xi d\eta\right|\\
&\leq C_{ell}\sum_{|k|\geq k_{M}}\left|\int_{|\eta-\xi|\leq \f18|k,\xi|}\rmA_k(t,\eta)^2\Big(e^{-\la_{u,\theta}\langle \eta-\xi\rangle^s}\f{|k|e^{-\la_{\Delta}\langle \zeta-\xi\rangle^s}}{k^2+(\zeta-kt)^2} \widehat{\Om}_k(\zeta)\Big)\overline{\widehat{\Om}_k(t, \eta)} d\zeta d\xi d\eta\right|\\
&\leq C_{ell}C_{\varphi_1}\sum_{|k|\geq k_{M}}\left|\int_{|\eta-\xi|\leq \f18|k,\xi|}\Big(e^{-\f34\la_{u,\theta}\langle \eta-\xi\rangle^s}\f{|k|e^{-\f34\la_{\Delta}\langle \zeta-\xi\rangle^s}}{k^2+(\zeta-kt)^2} \widehat{\rmA f}_k(\zeta)\Big)\overline{\widehat{\rmA f}_k(t, \eta)} d\zeta d\xi d\eta\right|.
\end{align*}
It is easy to check that for $k\geq k_M$
\beno
\f{1}{k^2+(\zeta-kt)^2}\lesssim \f{1}{k_M}\sqrt{\f{\dot{\rmA}_k(t,\eta)}{\rmA_{k}(t,\eta)}}\sqrt{\f{\dot{\rmA}_k(t,\zeta)}{\rmA_{k}(t,\zeta)}}\langle \zeta-\xi,\eta-\xi\rangle^3
\eeno
which implies that there is a universal constant $C_0$ such that 
\beno
|I_{f,5}^{L, LH}|\leq \f{C_0C_{ell}C_{\varphi_1}}{k_M}\mathcal{CK}_f.
\eeno
Thus we obtain \eqref{eq: estimate I_f5^L} and finish the proof. 
\end{proof}
One can regard $k_{M}$ as the threshold after which we no longer need to apply the wave operator. 

\subsection{The bootstrap argument}
The local wellposedness of the inhomogeneous incompressible Euler equation in the Gevrey class is classical. In this way, we may safely ignore the time interval $[0,1]$ by further restricting the size of the initial data. 

The goal is next to prove by a continuity argument that this energy $\mathcal{E}(t)$ (together with some related quantities) is uniformly bounded for all time if $\ep$ is sufficiently small. 
 
We define the following controls referred to in the sequel as the bootstrap hypotheses for $t\geq 1$ and some constant $C(k_{M})$,
\begin{itemize}
\item[(B1)]  $\mathcal{E}(t)\leq 10C(k_{M})\ep^2$; 
\item[(B2)] Compact support of $\Om$ and $a$: 
\beno
\mathrm{supp}\,\Om\subset \Big[2.5\kappa_0-10C(k_{M})\ep,1-2.5\kappa_0+10C(k_{M})\ep\Big],\\
\mathrm{supp}\, a \subset \Big[2.5\kappa_0-10C(k_{M})\ep,1-2.5\kappa_0+10C(k_{M})\ep\Big];
\eeno
\item[(B3)] `$\mathrm{CK}$' integral estimates
\begin{align*}
\int_{1}^t\bigg[\mathcal{CK}_f+\f{1}{\rmK_a}\mathcal{CK}_a+\mathcal{CK}_{\bar{h}}+\mathcal{CK}_{\udl{\pa_tv}}+\f{1}{\rmK_v}\Big(\sum_{j=1}^{11}\mathcal{CK}_{\varphi^{\d}_j}\Big)\bigg]d\tau
\leq 20C(k_{M})\ep^2,
\end{align*}
\end{itemize}
\begin{proposition}[Bootstrap]\label{prop: btsp}
Let $k_{M}$ satisfy \eqref{eq: k_M determine}. There exists an $\ep_0\in (0,\f12)$ depending only on $k_{M}$, $\la_{0}$, $\la'$, $s$ and $\s$ such that if $\ep<\ep_0$ and on $[1, T^*]$ the bootstrap hypotheses (B1)-(B3) hold, then for $\forall t\in [1, T^*]$,
\begin{itemize}
\item[1.]  $\mathcal{E}(t)\leq 8C(k_{M})\ep^2$, 
\item[2.] Compact support of $\Om$ and $a$: 
\beno
\mathrm{supp}\,\Om\subset \Big[2.5\kappa_0-8C(k_{M})\ep,1-2.5\kappa_0+8C(k_{M})\ep\Big],\\
\mathrm{supp}\, a \subset \Big[2.5\kappa_0-8C(k_{M})\ep,1-2.5\kappa_0+8C(k_{M})\ep\Big];
\eeno
\item[3.] and the $\mathrm{CK}$ controls satisfy:
\begin{align*}
\int_{1}^t\bigg[\mathcal{CK}_f+\f{1}{\rmK_a}\mathcal{CK}_a+\mathcal{CK}_{\bar{h}}+\mathcal{CK}_{\udl{\pa_tv}}+\f{1}{\rmK_v}\Big(\sum_{j=1}^{11}\mathcal{CK}_{\varphi^{\d}_j}\Big)\bigg]d\tau
\leq 18C(k_{M})\ep^2,
\end{align*}
\end{itemize}
from which it follows that $T^*=+\infty$. 
\end{proposition}
The remainder of this paper is devoted to the proof of Proposition \ref{prop: btsp}.

\section{Energy estimate}\label{sec: energy estimate}

\subsection{Energy estimate for $\mathcal{E}_f$}
A direct calculation gives
\begin{align}
&\nonumber\f12\f{d}{dt}\mathcal{E}_f+\mathcal{CK}_f\\
&\nonumber=-\f1{2\pi}\sum_{k=0,\, |k|\geq k_{M}}\int_{\mathbb{R}}\rmA_k(t,\eta)^2\widehat{\left(\Upsilon_2\rmU\cdot\na_{z,v}\Om\right)}_k(t,\eta)\overline{\widehat{\Om}_k(t,\eta)}d\eta\\
&\nonumber\quad-\f{1}{2\pi}\sum_{0<|k|<k_{M}}\int_{\mathbb{R}}\rmA_k(t,\eta)^2\mathcal{F}_2\Big(\bfD_{u,k}\left(\mathcal{F}_{1}\Big(\rmU\cdot\na_{z,v}\Om\Big)\right)\Big)_k(t,\eta)\overline{\widehat{f}_k(t, \eta)}d\eta\\
&\nonumber\quad
-\f1{2\pi}\sum_{k=0, \, |k|\geq k_{M}}\int_{\mathbb{R}}\rmA_k(t,\eta)^2\widehat{\left(\mathcal{N}_{\Om}[\Upsilon_2\Psi]\right)}_k(t,\eta)\overline{\widehat{\Om}_k(t, \eta)}d\eta\\
&\nonumber\quad-\f1{2\pi}\sum_{k=0, \, |k|\geq k_{M}}\int_{\mathbb{R}}\rmA_k(t,\eta)^2\widehat{\left(\mathcal{N}_a[\Pi]\right)}_{k}(t,\eta)\overline{\widehat{\Om}_k(t, \eta)}d\eta\\
&\nonumber\quad
+\f1{2\pi}\sum_{|k|\geq k_{M}}\int_{\mathbb{R}}\rmA_k(t,\eta)^2\widehat{\Big(\udl{\varphi_1}\pa_{z}(\Upsilon_2\Psi)\Big)}_k(t,\eta)\overline{\widehat{\Om}_k(t, \eta)}d\eta\\
&\nonumber\quad+\f{1}{2\pi}\sum_{0<|k|<k_{M}}\int_{\mathbb{R}}\rmA_k(t,\eta)^2\mathcal{F}_2\Big(\bfD_{u,k}\left(\mathcal{F}_{1}\Big(\varphi_1^{\d}\pa_{z}(\Upsilon_2\Psi)-\widetilde{\varphi_1}\pa_{z}\big(\mathring{\mathcal{T}}_{1,D}^{-1}[\Om]-\Upsilon_2\Psi\big)\Big)\right)\Big)_k(t,\eta)\overline{\widehat{f}_k(t, \eta)}d\eta\\
&\nonumber\quad-\f{1}{2\pi}\sum_{0<|k|<k_{M}}\int_{\mathbb{R}}\rmA_k(t,\eta)^2\mathcal{F}_2\Big(\bfD_{u,k}\left(\mathcal{F}_{1}\Big(\mathcal{N}_{\Om}[\Psi]\Big)\right)\Big)_{k}(t,\eta)\overline{\widehat{f}_k(t, \eta)}d\eta\\
&\label{eq: definition-I_fj}\quad-\f{1}{2\pi}\sum_{0<|k|<k_{M}}\int_{\mathbb{R}}\rmA_k(t,\eta)^2\widehat{\big(\bfD(\mathcal{N}_{a}[\Pi])\big)}_k(t,\eta)\overline{\widehat{f}_k(t, \eta)}d\eta
=\sum_{j=1}^{8}I_{f,j}.
\end{align}
We first treat $I_{f,2}$. The main idea is to recover the transport structure together with $I_{f,1}$. 
A usual way to recover the transport structure $\rmU\cdot\na_{z,v} D\Om$ is to study the commutator 
\beno
\big[D,\rmU\cdot\na_{z,v}\big]\Om.
\eeno
However, it is not good here, because the commutator will mix the information of different frequencies in $z$. The new idea is to use an important identity \eqref{eq: dual id}, which formally is that 
\ben\label{eq:DD^1formal}
\int D(f) D^1(g) dzdv=\int f g dzdv,
\een
where $D^1=(D^{-1})^*$ (see \eqref{eq:D^1dual} for more details).
More precisely, we need to control the following term 
\beno
\int \rmA D(\rmU\cdot\na_{z,v}\Om) \rmA D(\Om)dzdv. 
\eeno 
The idea is to remove the wave operator $D$ using \eqref{eq:DD^1formal} and recover the transport structure:
\begin{align*}
&\int \rmA D(\rmU\cdot\na_{z,v}\Om) \rmA D(\Om)dzdv\\
&=\int \rmA(\rmU\cdot\na_{z,v}\Om) \rmA(\Om)dzdv+\int [\rmA,D^1](\rmU\cdot\na_{z,v}\Om)  D\rmA(\Om)dzdv\\
&\quad+\int \rmA D^1(\rmU\cdot\na_{z,v}\Om)[\rmA, D](\Om)dzdv
+\int \rmA(D-D^1)(\rmU\cdot\na_{z,v}\Om)\rmA D(\Om)dzdv.
\end{align*}
With the previous brief idea in mind, now we present a precise calculation. 

\medskip

For $0<|k|<k_{M}$, let 
\beno
I_{f,2}^k=\int\rmA_k(\eta)\overline{\mathcal{F}_2\bfD_{u,k}(\mathcal{F}_{1}\Om)}_k(\eta)\rmA_k(\eta)\mathcal{F}_2\Big(\bfD_{u,k}\left(\mathcal{F}_{1}\Big(\rmU\cdot\na_{z,v}\Om\Big)\right)(t,k,\cdot)\Big)(\eta)d\eta
\eeno
then $I_{f,2}=\sum\limits_{0<|k|<k_{M}}I_{f,2}^k$ and
\begin{align*}
I_{f,2}^k&=\int\rmA_k(\eta)\overline{\mathcal{F}_2\bfD_{u,k}(\mathcal{F}_{1}\Om)}_k(\eta)\rmA_k(\eta)\mathcal{F}_2\Big(\bfD_{u,k}\left(\mathcal{F}_{1}\Big(\rmU\cdot\na_{z,v}\Om\Big)\right)(t,k,\cdot)\Big)(\eta)d\eta\\
&=\int\rmA_k(\eta)\overline{\mathcal{F}_2\bfD_{u,k}(\mathcal{F}_{1}\Om)}_k(\eta)\rmA_k(\eta)\mathcal{F}_2\Big(\tchi_2\bfD_{u,k}^1\left(\tchi_2\mathcal{F}_{1}\Big(\rmU\cdot\na_{z,v}\Om\Big)\right)(t,k,\cdot)\Big)(\eta)d\eta\\
&\quad+\int\rmA_k(\eta)\overline{\mathcal{F}_2\bfD_{u,k}(\mathcal{F}_{1}\Om)}_k(\eta)\rmA_k(\eta)
\mathcal{F}_2\Big[\big(\tchi_2\bfD_{u,k}-\tchi_2\bfD_{u,k}^1\big)\left(\tchi_2\mathcal{F}_{1}\Big(\rmU\cdot\na_{z,v}\Om\Big)\right)(t,k,\cdot)\Big](\eta)d\eta\\
&=\int\rmA_k(\eta)\overline{\widehat{\Om}}_k(\eta)\rmA_k(\eta)\reallywidehat{\big(\rmU\cdot\na_{z,v}\Om\big)}_k(\eta)d\eta\\
&\quad+\int(\rmA_k(\eta)-\rmA_k(\xi))\overline{\mathcal{D}(t,k,\eta,\xi)\hat{\Om}_k(\xi)}\rmA_k(\eta)\mathcal{F}_2\Big(\tchi_2\bfD_{u,k}^1\left(\tchi_2\mathcal{F}_{1}\Big(\rmU\cdot\na_{z,v}\Om\Big)\right)(t,k,\cdot)\Big)(\eta)d\xi d\eta\\
&\quad+\sum_{l}\int\overline{\mathcal{F}_2\bfD_{u,k}\big(\mathcal{F}_2^{-1}(\rmA_k(\eta)\widehat{\Om}_k)\big)}(\eta)\\
&\qquad\qquad\qquad\qquad
\times(\rmA_k(\eta)-\rmA_k(\xi))\mathcal{D}^1(t,k,\eta,\xi)\Big(\widehat{\rmU}_{k-l}(\xi-\xi_1)\cdot(il,i\xi_1)\widehat{\Om}_l(\xi_1)\Big)d\xi_1 d\xi d\eta\\
&\quad+\int\rmA_k(\eta)\bar{\hat{f}}_k(\eta)\rmA_k(\eta)\mathcal{F}_2\Big[\big(\tchi_2\bfD_{u,k}-\tchi_2\bfD_{u,k}^1\big)\left(\tchi_2\mathcal{F}_{1}\Big(\rmU\cdot\na_{z,v}\Om\Big)\right)(t,k,\cdot)\Big](\eta)d\eta\\
&=\int\rmA_k(\eta)\overline{\widehat{\Om}}_k(\eta)\rmA_k(\eta)\reallywidehat{\big(\rmU\cdot\na_{z,v}\Om\big)}_k(\eta)d\eta+I_{f,2,k}^{com1}+I_{f,2,k}^{com2}+I_{f,2,k}^{com3}
\end{align*}
and then
\ben\label{eq: Pi^com}
\begin{aligned}
I_{f,1}+I_{f,2}
=&-\f1{2\pi}\sum_{k\in\mathbb{Z}}\int_{\mathbb{R}}\rmA_k(t,\eta)^2\widehat{\left(\Upsilon\rmU\cdot\na_{z,v}\Om\right)}_k(t,\eta)\overline{\widehat{\Om}_k(t,\eta)}d\eta\\
&+\sum_{0<|k|<k_{M}}\big(I_{f,2,k}^{com1}+I_{f,2,k}^{com2}+I_{f,2,k}^{com3}\big)\eqdef I_{f}^{tr}+I_{f,2}^{com}.
\end{aligned}
\een
We have the following proposition for $I_f^{tr}$.
\begin{proposition}\label{prop: f-transport}
Under the bootstrap hypotheses,
\ben
\begin{aligned}
|I_f^{tr}|&\lesssim \ep\mathcal{CK}_f+\ep^2\mathcal{CK}_{h}+\ep^2\mathcal{CK}_{\varphi_9^{\d}}+\f{\ep^3}{\langle t\rangle^{2-\mathrm{K}_{\rm{D}}\ep/2}}\\
&\quad +\ep \left\|\left\langle\f{\pa_v}{t\pa_z}\right\rangle^{-1}(\pa_z^2+(\pa_v-t\pa_z)^2)\left(\f{|\na|^{\f{s}{2}}}{\langle t\rangle^s}\rmA+\sqrt{\f{\pa_t w}{w}}\tilde{\rmA}\right)P_{\neq}(\Psi\Upsilon_2)\right\|_2. 
\end{aligned}
\een
\end{proposition}
The proposition is proved in section \ref{sec: transport}. 

Note that although the wave operator in this paper is new, we still have the same estimates of the Fourier kernels of $\mathbf{D}_{u,k}$, $\mathbf{D}_{u,k}^1$, and $\tchi_2\bfD_{u,k}-\tchi_2\bfD_{u,k}^1$ as that in \cite{MasmoudiZhao2020}, see Proposition \ref{prop: kernel-wave-op} and Proposition \ref{prop: commutator}. 
Thus, the commutator term $I_{f,2}^{com}$ which involves the wave operator can be treated in the same way as \cite{MasmoudiZhao2020}. 
\begin{proposition}\label{prop:com}
Under the bootstrap hypotheses, 
\beno
\begin{aligned}
|I_{f,2}^{com}|\lesssim 
&\ep \mathcal{CK}_{f}
+\ep\mathcal{CK}_{h}+\ep\mathcal{CK}_{\varphi_9^{\d}}
+\f{\ep^3}{\langle t\rangle^{2-\rmK_{\rmD}\ep/2}}\\
&+\ep\left\|\left\langle\f{\pa_v}{t\pa_z}\right\rangle^{-1}(\pa_z^2+(\pa_v-t\pa_z)^2)\left(\f{|\na|^{\f{s}{2}}}{\langle t\rangle^s}\rmA+\sqrt{\f{\pa_t w}{w}}\tilde{\rmA}\right)P_{\neq}(\Psi\Upsilon_2)\right\|_2^2. 
\end{aligned}
\eeno
\end{proposition}
We refer to the estimate of $\Pi^{com}$ term in \cite{MasmoudiZhao2020} for the proof and more details. 

The last term in the previous two propositions will be controlled by the elliptic estimate. 
\begin{proposition}[Precision elliptic control] \label{prop:elliptic}
Under the bootstrap hypotheses,
\beq\label{eq:elliptic1}
\begin{aligned}
&\left\|\left\langle\f{\pa_v}{t\pa_z}\right\rangle^{-1}(\pa_z^2+(\pa_v-t\pa_z)^2)\left(\f{|\na|^{\f{s}{2}}}{\langle t\rangle^s}\rmA+\sqrt{\f{\pa_t w}{w}}\tilde{\rmA}\right)P_{|k|\geq k_{M}}(\Psi\Upsilon_2)\right\|_2^2\\
&\leq 2C_{ell}\Big[\mathcal{CK}_{f}
+C(k_M)\ep^2\Big(\mathcal{CK}_{h}+\mathcal{CK}_{\varphi_{9}^{\d}}+\mathcal{CK}_{\varphi_4^{\d}}\Big)\Big],
\end{aligned}
\eeq
The constant $C_{ell}$ is the same constant in Lemma \ref{lem: I_{f,5}^{L}}. \\
Generally, it holds that
\beq\label{eq:elliptic}
\begin{aligned}
&\left\|\left\langle\f{\pa_v}{t\pa_z}\right\rangle^{-1}(\pa_z^2+(\pa_v-t\pa_z)^2)\left(\f{|\na|^{\f{s}{2}}}{\langle t\rangle^s}\rmA+\sqrt{\f{\pa_t w}{w}}\tilde{\rmA}\right)P_{\neq}(\Psi\Upsilon_2)\right\|_2^2\\
&\lesssim_{k_{M}} \mathcal{CK}_{f}
+\ep^2\Big(\mathcal{CK}_{h}+\mathcal{CK}_{\varphi_{9}^{\d}}+\mathcal{CK}_{\varphi_4^{\d}}\Big).
\end{aligned}
\eeq
\end{proposition}
The basic idea of proving this proposition is to control the elliptic error terms, which is similar to previous papers \cite{BM2015, MasmoudiZhao2020}. However, a new Fourier analysis of the Green function is required, which is given in Appendix \ref{sec: ST equ}. The proof of this proposition is given in section \ref{sec: ell-stream}. 

The term $I_{f,3}$ and $I_{f,7}$ are arose from introducing the good unknown $\tom$. They are lower-order terms since there is no derivative loss. They are also different from the nonlocal error term $I_{f,6}^1$ or nonlocal small term $I_{f,7}$ that will be discussed later. In $I_{f,3}$ and $I_{f,7}$, the function $\pa_z(\Upsilon_2\Psi)$ interacts with the nonzero modes.
\begin{proposition}\label{prop: lower-order}
Under the bootstrap hypotheses, 
\beno
\begin{aligned}
|I_{f,3}|+|I_{f,7}|\lesssim 
&\ep \mathcal{CK}_{f}+\f{\ep^3}{\langle t\rangle^{2-\rmK_{\rmD}\ep/2}}\\
&+\ep\left\|\left\langle\f{\pa_v}{t\pa_z}\right\rangle^{-1}(\pa_z^2+(\pa_v-t\pa_z)^2)\left(\f{|\na|^{\f{s}{2}}}{\langle t\rangle^s}\rmA+\sqrt{\f{\pa_t w}{w}}\tilde{\rmA}\right)P_{\neq}(\Psi\Upsilon)\right\|_2^2. 
\end{aligned}
\eeno
\end{proposition}

The term $I_{f,6}$ contains two terms
\begin{align*}
I_{f,6}&=\f{1}{2\pi}\sum_{0<|k|<k_{M}}\int_{\mathbb{R}}\rmA_k(t,\eta)^2\mathcal{F}_2\Big(\bfD_{u,k}\left(\mathcal{F}_{1}\Big(\varphi_1^{\d}\pa_{z}(\Upsilon_2\Psi)\Big)\right)\Big)_k(t,\eta)\overline{\widehat{f}_k(t, \eta)}d\eta\\
&\quad+\f{1}{2\pi}\sum_{0<|k|<k_{M}}\int_{\mathbb{R}}\rmA_k(t,\eta)^2\mathcal{F}_2\Big(\bfD_{u,k}\left(\mathcal{F}_{1}\Big(\widetilde{\varphi_1}\pa_{z}\big(\mathring{\mathcal{T}}_{1,D}^{-1}[\Om]-\Upsilon_2\Psi\big)\Big)\right)\Big)_k(t,\eta)\overline{\widehat{f}_k(t, \eta)}d\eta\\
&\eqdef I_{f,6}^1+I_{f,6}^2. 
\end{align*}
The first term $I_{f,6}^1$ can be treated as the reaction term, which is similar to $\Pi^{u'',\ep_1}$ term in \cite{MasmoudiZhao2020}. The second term $I_{f,6}^2$ is relate to the elliptic error term $\mathring{\mathcal{T}}_{1,D}^{-1}[\mathcal{S}_1[\Psi]]$ (see \eqref{eq: diff-elliptic-1}), which is similar to $\Pi^{u'',\ep_2}$ term in \cite{MasmoudiZhao2020}.
\begin{proposition}\label{prop:nonlocal-ep}
Under the bootstrap hypotheses, 
\beno
\begin{aligned}
|I_{f,6}^1|+|I_{f,6}^2|\lesssim 
&\ep \mathcal{CK}_{f}+\f{\ep^3}{\langle t\rangle^{2-\rmK_{\rmD}\ep/2}}
+\ep\mathcal{CK}_{\varphi^{\d}_1}+\ep\mathcal{CK}_{h}+\ep\mathcal{CK}_{\varphi^{\d}_9}+\ep\mathcal{CK}_{\varphi^{\d}_4}\\
&+\ep\left\|\left\langle\f{\pa_v}{t\pa_z}\right\rangle^{-1}(\pa_z^2+(\pa_v-t\pa_z)^2)\left(\f{|\na|^{\f{s}{2}}}{\langle t\rangle^s}\rmA+\sqrt{\f{\pa_t w}{w}}\tilde{\rmA}\right)P_{\neq}(\Psi\Upsilon_2)\right\|_2^2. 
\end{aligned}
\eeno
\end{proposition}

The terms $I_{f,5}$ have the same structure and can be treated in the same way. Indeed, $I_{f,5}$ is similar to the term $\Pi^{u''}$ in \cite{MasmoudiZhao2020}. 
The key observation for this term is as follows: \\
1. we can treat it as the reaction term; \\
2. we can get an extra smallness if $|k|\geq k_{M}$ with $k_{M}$ sufficiently large.
\begin{proposition}\label{prop: nonlocal}
Under the bootstrap hypotheses 
\begin{align*}
|I_{f,5}|
\leq &\f{2C_0C_{\varphi_1}\mathcal{E}_f(t)}{\langle t\rangle^2}
+\f{1}{50}\mathcal{CK}_{f}
+\f{2C_0C_{ell}C_{\varphi_1}}{k_M}\left\|\left\langle\f{\pa_v}{t\pa_z}\right\rangle^{-1}\f{|\na|^{s/2}}{\langle t\rangle^{s}}\Delta_{L}P_{|k|\geq k_{M}}\rmA (\Psi\Upsilon_2)\right\|_2^2\\
&\quad +\f{2C_0C_{ell}C_{\varphi_1}}{k_M}\left\|1_{\rmR}\sqrt{\f{\pa_tw}{w}}\Delta_{L}\tilde{\rmA}P_{|k|\geq k_{M}}(\Psi\Upsilon_2)\right\|_2^2,
\end{align*}
where $1_{\rmR}=1_{t\in \mathrm{I}_{k,\eta}}$ for some $k\eta>0$.
\end{proposition} 
\begin{remark}
By applying \eqref{eq:elliptic1} and using \eqref{eq: k_M determine}, we get that
\beno
\begin{aligned}
|I_{f,5}|
\leq &\f{2C_0C_{\varphi_1}\mathcal{E}_f(t)}{\langle t\rangle^2}
+\f{1}{10}\mathcal{CK}_{f}
+C(k_{M})\ep^2\Big(\mathcal{CK}_{h}+\mathcal{CK}_{\varphi_9^{\d}}+\mathcal{CK}_{\varphi_4^{\d}}\Big).
\end{aligned}
\eeno
\end{remark}
\begin{proof}
Let us outline the main idea of the proof of Proposition \ref{prop: nonlocal}. As in the proof of Lemma \ref{lem: I_{f,5}^{L}}, we divide $I_f^5$ into high-low, low-high, and high-high interactions based on the frequencies. For the high-low and high-high interactions, we use the facts that 
\beno
&&\|P_{k\geq k_M}\Upsilon_2\Psi\|_{\mathcal{G}^{s,\s-6}}\leq \f{C}{\langle t\rangle^2}\|P_{k\geq k_M}\Om\|_{L^2}\leq \f{C}{\langle t\rangle^2}\mathcal{E}_f^{\f12}\\
&&\|\rmA^{\rmR}\udl{\varphi_1}\|_{L^2}\leq \|\rmA^{\rmR}\widetilde{\varphi_1}\|_{L^2}+\|\rmA^{\rmR}\varphi_1^{\d}\|_{L^2}\leq \f{9}{8}C_{\varphi_1},
\eeno
and obtain that 
\beno
|I^5_{f,\mathrm{HL}}|+|I^5_{f,\mathrm{HH}}|\leq \f{2C_0C_{\varphi_1}\mathcal{E}_f(t)}{\langle t\rangle^2}. 
\eeno
For the low-high interactions, we get that
\begin{align*}
|I^5_{f,\mathrm{LH}}|
&\lesssim \sum_{k\neq 0}\left|\int_{|\xi-\eta|\leq \f18|k,\xi|} \rmA \overline{\widehat{f}}_k(\eta)\rmA_{k}(t,\eta)\left(\widehat{\udl{\varphi_1}}(\eta-\xi)1_{|k|\geq k_{M}}|k|\widehat{(\Psi\Upsilon)}_k(\xi)\right) d\xi d\eta\right|\\
&\lesssim \sum_{k\neq 0}\bigg|\int_{|\xi-\eta|\leq \f18|k,\xi|}[1_{\mathrm{NR,NR}}+1_{\mathrm{NR,R}}+1_{\mathrm{R,NR}}+1_{\mathrm{R,R}}]1_{|k|\geq k_{M}}\\
&\qquad\qquad\qquad\times\rmA \overline{\widehat{f}}_k(\eta)\rmA_{k}(t,\eta)\left(\widehat{\udl{\varphi_1}}(\eta-\xi)|k|\widehat{(\Psi\Upsilon)}_k(\xi)\right) d\xi d\eta\bigg|\\
&=I^{5;\mathrm{NR,NR}}_{f,\mathrm{LH}}
+I^{5;\mathrm{NR,R}}_{f,\mathrm{LH}}
+I^{5;\mathrm{R,NR}}_{f,\mathrm{LH}}
+I^{5;\mathrm{R,R}}_{f,\mathrm{LH}}.
\end{align*}
For $I^{5;\mathrm{NR,NR}}_{f,\mathrm{LH}}$ and $I^{5;\mathrm{R,NR}}_{f,\mathrm{LH}}$, we use the fact that there is $C$ such that for all $k_M$
\beq\label{eq: 6.12modify}
\f{\langle t\rangle^{s}|k|}{\langle k,\xi\rangle^{s/2}}1_{t\notin \mathbf{I}_{k,\xi}} 1_{|k|\geq k_{M}}\leq \f{C}{k_{M}}  \left\langle\f{\xi}{tk}\right\rangle^{-1}\f{|k,\xi|^{s/2}}{\langle t\rangle^{s}}(k^2+(\xi-kt)^2),
\eeq
and obtain that 
\begin{align*}
|I^{5;\mathrm{NR,NR}}_{f,\mathrm{LH}}|+|I^{5;\mathrm{R,NR}}_{f,\mathrm{LH}}|
&\leq C \|\udl{\varphi_1}\|_{\mathcal{G}^{s,\la,\s-6}}\||\na|^{s/2}\rmA f\|_2
\left\|\f{|\pa_z|}{\langle\na\rangle^{s/2}}1_{\mathrm{NR}} 1_{|k|\geq k_{M}}P_{\neq}\rmA(\Psi\Upsilon)\right\|_2\\
&\leq \f{C_0/10}{k_{M}\langle t\rangle^{2s}}\|\udl{\varphi_1}\|_{\mathcal{G}^{s,\la,\s-6}} \||\na|^{s/2}\rmA f\|_2^2\\
&\quad+\f{C_0/10}{k_{M}}\|\udl{\varphi_1}\|_{\mathcal{G}^{s,\la,\s-6}} \left\|\left\langle\f{\pa_v}{t\pa_z}\right\rangle^{-1}\f{|\na|^{s/2}}{\langle t\rangle^{s}}\Delta_{L}P_{|k|\geq k_{M}}\rmA (\Psi\Upsilon)\right\|_2^2. 
\end{align*}
For $I^{5;\mathrm{NR,R}}_{f,\mathrm{LH}}$, we use the fact that 
\begin{align*}
&\sqrt{\f{w(t,\xi)}{\pa_tw(t,\xi)}}|k|\f{w_{\mathrm{R}}(t,\xi)}{w_{\mathrm{NR}}(t,\xi)}1_{|k|\geq k_{M}}1_{t\in \mathrm{I}_{k,\xi}}\\
&\approx \Big(1+\big|t-\f{\xi}{k}\big|\Big)^{\f12}\f{|k|^3\Big(1+\big|t-\f{\xi}{k}\big|\Big)}{\xi}1_{t\in \mathrm{I}_{k,\xi}}1_{|k|\geq k_{M}}
\approx \f{1}{k_{M}}\big(k^2+(\xi-kt)^2\big)\sqrt{\f{\pa_tw(t,\xi)}{w(t,\xi)}}.
\end{align*}
and obtain that
\begin{align*}
|I^{5;\mathrm{NR,R}}_{f,\mathrm{LH}}|
&\leq \f{C_0/10}{k_{M}} \|\udl{\varphi_1}\|_{\mathcal{G}^{s,\la,\s-6}}\left\|\sqrt{\f{\pa_tw}{w}}\tilde{\rmA} f\right\|_2^2+\f{C_0/10}{k_{M}} \|\udl{\varphi_1}\|_{\mathcal{G}^{s,\la,\s-6}}\left\|\f{|\na|^{s/2}}{\langle t\rangle^{s}}\rmA f\right\|_2^2\\
&\quad +\f{C_0/10}{k_{M}} \|\udl{\varphi_1}\|_{\mathcal{G}^{s,\la,\s-6}}\left\|\sqrt{\f{\pa_tw}{w}}\Delta_{L}\tilde{\rmA}P_{|k|\geq k_{M}}(\Psi\Upsilon)\right\|_2^2.
\end{align*}
For $I^{5;\mathrm{R,R}}_{f,\mathrm{LH}}$, we use the fact that 
\begin{align*}
&\sqrt{\f{w_{k}(t,\xi)}{\pa_tw_{k}(t,\xi)}}|k|1_{t\in \mathrm{I}_{k,\xi}}1_{|k|\geq k_{M}}
\lesssim |k|\sqrt{1+\Big|t-\f{\xi}{k}\Big|}1_{t\in \mathrm{I}_{k,\xi}}1_{|k|\geq k_{M}}\\
&\lesssim |k|\bigg(1+\Big|t-\f{\xi}{k}\Big|\bigg)\sqrt{\f{\pa_tw_{k}(t,\xi)}{w_{k}(t,\xi)}}1_{t\in \mathrm{I}_{k,\xi}}1_{|k|\geq k_{M}}
\lesssim \f{\big(k^2+(kt-\xi)^2\big)}{k_M}\sqrt{\f{\pa_tw_{k}(t,\xi)}{w_{k}(t,\xi)}}1_{t\in \mathrm{I}_{k,\xi}}1_{|k|\geq k_{M}},
\end{align*}
and obtain that 
\begin{align*}
|I^{5;\mathrm{R,R}}_{f,\mathrm{LH}}|
\leq \f{C_0/10}{k_{M}} \|\udl{\varphi_1}\|_{\mathcal{G}^{s,\la,\s-6}} \left\|\sqrt{\f{\pa_tw}{w}}\tilde{\rmA}f\right\|_2^2+\f{C_0/10}{k_{M}} \|\udl{\varphi_1}\|_{\mathcal{G}^{s,\la,\s-6}}
\left\|\sqrt{\f{\pa_tw}{w}}\Delta_{L}\tilde{\rmA}P_{|k|\geq k_{M}}1_{\rmR} (\Psi\Upsilon)\right\|_2^2.
\end{align*}
Here $C_0>1$ is a universal constant. By combining all the above estimates and using the bootstrap assumptions, we obtain the proposition.  

One may refer to the proof of $\Pi^{u''}$ in \cite{MasmoudiZhao2020} for more details. Indeed, the linear part of $I_{f,5}$ is treated in Lemma \ref{lem: I_{f,5}^{L}}, which is used to determine the $k_M$ before the bootstrap argument. 
\end{proof}

Now we control the pressure. We first estimate the zero mode. 

\begin{proposition}\label{prop:pressure-zero mode}
Under the bootstrap hypotheses,
\beno
\left\|t(t^2+|\pa_v|^2)(2-\pa_{v}^2)^{-1}\rmA_0P_0\big(\pa_v(\Pi\Upsilon_1)\big)\right\|_2\lesssim \ep
\eeno
and
\beno
\begin{aligned}
&\left\|t(t^2+|\pa_v|^2)(2-\pa_{v}^2)^{-1}\Big(\sqrt{\f{\pa_tw}{w}}+\f{\langle \pa_v\rangle^{\f{s}{2}}}{\langle t\rangle^s}\Big)\rmA_0P_0\big(\pa_v(\Pi\Upsilon_1)\big)\right\|_2^2\\
&\lesssim 
\f{\ep^4}{\langle t\rangle^{2s}}+\ep^2\mathcal{CK}_a+\ep^2\mathcal{CK}_{\varphi^{\d}_4}\\
&\quad+\ep^2\left\|\left\langle\f{\pa_v}{t\pa_z}\right\rangle^{-1}(\pa_z^2+(\pa_v-t\pa_z)^2)\left(\f{\langle\na\rangle^{\f{s}{2}}}{\langle t\rangle^s}\rmA+\sqrt{\f{\pa_t w}{w}}\tilde{\rmA}\right)P_{\neq}(\Psi\Upsilon)\right\|_2^2.
\end{aligned}
\eeno
\end{proposition}
Proposition \ref{prop:pressure-zero mode} is proved in section \ref{sec: pressure-zero}. 

The estimates of the non-zero modes of pressure are highly non-trivial. 
\begin{proposition}\label{Prop: non-zero-pessure-L}
Under the bootstrap hypotheses,
\begin{align*}
\left\|\langle\pa_z\rangle t(t+\langle \na\rangle)(1-\Delta)^{-1}\big((\pa_v-t\pa_z)^2+\pa_z^2\big)\rmA P_{\neq} \Pi_{\star}\right\|_2\lesssim \ep
\end{align*}
and
\begin{align*}
&\left\|\langle\pa_z\rangle t(t+\langle \na\rangle)(1-\Delta)^{-1}\big((\pa_v-t\pa_z)^2+\pa_z^2\big)\left(\f{\langle\na\rangle^{\f{s}{2}}}{\langle t\rangle^s}\rmA+\sqrt{\f{\pa_t w}{w}}\tilde{\rmA}\right) P_{\neq} \Pi_{\star}\right\|_2\\
&\lesssim \ep^2 \Big(\f{1}{\langle t\rangle^s}+\mathcal{CK}_{\varphi^{\d}_6}^{\f12}+\mathcal{CK}_{\varphi^{\d}_7}^{\f12}+\mathcal{CK}_{\th^{\d}}^{\f12}+\mathcal{CK}_{h}^{\f12}+\mathcal{CK}_{\varphi_8^{\d}}^{\f12}+\mathcal{CK}_{\varphi_9^{\d}}^{\f12}+\mathcal{CK}_a^{\f12}\Big)\\
&\quad+\ep\left\|\left\langle\f{\pa_v}{t\pa_z}\right\rangle^{-1}(\pa_z^2+(\pa_v-t\pa_z)^2)\left(\f{|\na|^{\f{s}{2}}}{\langle t\rangle^s}\rmA+\sqrt{\f{\pa_t w}{w}}\tilde{\rmA}\right)P_{\neq}(\Psi\Upsilon_2)\right\|_2\\
&\quad +\ep \left\|t(t^2+|\pa_v|^2)(2-\pa_{v}^2)^{-1}\Big(\sqrt{\f{\pa_tw}{w}}+\f{\langle \pa_v\rangle^{\f{s}{2}}}{\langle t\rangle^s}\Big)\rmA_0P_0\big(\pa_v(\Pi\Upsilon_1)\big)\right\|_2^2,
\end{align*}
where $\Pi_{\star}\in \{\Upsilon_2\Pi_{l,1}, \Upsilon_1\Pi_{l,2}\}$.
\end{proposition}
\begin{proposition}\label{Prop: non-zero-pessure-N}
Under the bootstrap hypotheses,
\begin{align*}
\left\|\langle\pa_z\rangle t(t+\langle \na\rangle)(1-\Delta)^{-1}\big((\pa_v-t\pa_z)^2+\pa_z^2\big)\rmA P_{\neq} \Pi_{\star}\right\|_2\lesssim \ep^2
\end{align*}
and
\begin{align*}
&\left\|\langle\pa_z\rangle t(t+\langle \na\rangle)(1-\Delta)^{-1}\big((\pa_v-t\pa_z)^2+\pa_z^2\big)\left(\f{\langle\na\rangle^{\f{s}{2}}}{\langle t\rangle^s}\rmA+\sqrt{\f{\pa_t w}{w}}\tilde{\rmA}\right) P_{\neq} \Pi_{\star}\right\|_2\\
&\lesssim  \ep\mathcal{CK}_f^{\f12}+\ep^2 \Big(\f{1}{\langle t\rangle^s}+\mathcal{CK}_{\varphi^{\d}_4}^{\f12}+\mathcal{CK}_{\varphi^{\d}_6}^{\f12}+\mathcal{CK}_{\varphi^{\d}_7}^{\f12}+\mathcal{CK}_{\th^{\d}}^{\f12}+\mathcal{CK}_{h}^{\f12}+\mathcal{CK}_{\varphi_8^{\d}}^{\f12}+\mathcal{CK}_{\varphi_9^{\d}}^{\f12}+\mathcal{CK}_a^{\f12}\Big)\\
&\quad+\left\|\left\langle\f{\pa_v}{t\pa_z}\right\rangle^{-1}(\pa_z^2+(\pa_v-t\pa_z)^2)\left(\f{|\na|^{\f{s}{2}}}{\langle t\rangle^s}\rmA+\sqrt{\f{\pa_t w}{w}}\tilde{\rmA}\right)P_{\neq}(\Psi\Upsilon_2)\right\|_2\\
&\quad +\ep \left\|t(t^2+|\pa_v|^2)(2-\pa_{v}^2)^{-1}\Big(\sqrt{\f{\pa_tw}{w}}+\f{\langle \pa_v\rangle^{\f{s}{2}}}{\langle t\rangle^s}\Big)\rmA_0P_0\big(\pa_v(\Pi\Upsilon_1)\big)\right\|_2^2,
\end{align*}
where $\Pi_{\star}\in \{\Upsilon_2\Pi_{n,1}, \Upsilon_1\Pi_{n,2}\}$.
\end{proposition}
There are four key ideas in estimates: 1, the pressure decomposition in section \ref{sec: step1} ensures the Fourier analysis works in the problem; 2, the new Fourier analysis of the Green function associated with the linearized equation is given in Appendix \ref{sec: ST equ}; 3, the well-designed weights $\rmB$, $\mathcal{M}_3$, $\mathcal{M}_4$, and $\mathcal{M}_5$ from \cite{ChenWeiZhangZhang2023} help control the elliptic error terms easily; 4, the iteration method helps us gain back the derivative loss in the equation of $\Pi_{l,2}$ and $\Pi_{n,2}$, see section \ref{sec: ell-pressure} for more details. These two propositions are proved in section \ref{sec: ell-pressure}.

As a direct corollary, we have the point-wise decay estimate for the pressure. 
\begin{corol}\label{cor: decay-pressure}
Under the bootstrap hypotheses,
\beno
\|\pa_vP_0(\Upsilon_2\Pi)\|_{\mathcal{G}^{s,\s-6}}\lesssim \f{\ep}{\langle t\rangle^3},
\eeno
and
\beno
\|P_{\neq}(\Pi_{\star})\|_{\mathcal{G}^{s,\s-6}}\lesssim \f{\ep}{\langle t\rangle^4},
\eeno
where $\Pi_{\star}\in \{\Upsilon_2\Pi_{l,1}, \Upsilon_2\Pi_{n,1}, \Upsilon_1\Pi_{l,2}, \Upsilon_1\Pi_{n,2}\}$. 
\end{corol}

The terms $I_{f,4}$ and $I_{f,8}$ are new, which are the interactions between the density $a$ and the pressure $\Pi$. If the pressure is in higher frequencies, these terms behave as the Reaction term in \cite{BM2015}. If the density $a$ is in higher frequencies, there is a derivative loss. So we need the multiplier $\rmB$ to gain half derivative. We have the following proposition for the estimates. 
\begin{proposition}\label{prop: I_f4,I_f8}
Under the bootstrap hypotheses,
\begin{align*}
|I_{f,4}|+|I_{f,8}|\lesssim &\f{\ep^3}{\langle t\rangle^3}
+\ep \mathcal{CK}_f+\ep\mathcal{CK}_a\\
&+\ep\left\|t(t^2+|\pa_v|^2)(2-\pa_{v}^2)^{-1}\Big(\sqrt{\f{\pa_tw}{w}}+\f{\langle \pa_v\rangle^{\f{s}{2}}}{\langle t\rangle^s}\Big)\rmA_0P_0\big(\pa_v(\Pi\Upsilon_1)\big)\right\|_2^2\\
&+\ep\left\|\langle\pa_z\rangle t(t+\langle \na\rangle)(1-\Delta)^{-1}\big((\pa_v-t\pa_z)^2+\pa_z^2\big)\left(\f{\langle\na\rangle^{\f{s}{2}}}{\langle t\rangle^s}\rmA+\sqrt{\f{\pa_t w}{w}}\tilde{\rmA}\right) P_{\neq} \Pi_{\star}\right\|_2^2. 
\end{align*}
where $\Pi_{\star}=\Upsilon_2\Pi_{l,1}+\Upsilon_2\Pi_{n,1}+ \Upsilon_1\Pi_{l,2}+ \Upsilon_1\Pi_{n,2}$. 
\end{proposition}
The proposition is proved in section \ref{sec: I_f4,8}. 

\subsection{Energy estimate of $\mathcal{E}_a$}
A direct calculation gives
\begin{align*}
\f12\f{d}{dt}\mathcal{E}_a+\mathcal{CK}_a
&=-\int \rmA^* a\rmA^*\left(\Upsilon_2\rmU\cdot\na_{z,v}a\right)dz dv
-\int \rmA^*a \rmA^*\big(\udl{\th'}\pa_z(\Upsilon_2\Psi)\big) dz dv\\
&\eqdef I_a^{tr}+I_a^{l}.
\end{align*}

We have the following proposition for the estimate of $I_a^{tr}$. 
\begin{proposition}\label{prop: a-transport}
Under the bootstrap hypotheses,
\ben
\begin{aligned}
|I_a^{tr}|&\lesssim \ep\mathcal{CK}_a+\ep^2\mathcal{CK}_{h}+\ep^2\mathcal{CK}_{\varphi_9^{\d}}+\f{\ep^3}{\langle t\rangle^{2-\mathrm{K}_{\rm{D}}\ep/2}}\\
&\quad +\ep \left\|\left\langle\f{\pa_v}{t\pa_z}\right\rangle^{-1}(\pa_z^2+(\pa_v-t\pa_z)^2)\left(\f{|\na|^{\f{s}{2}}}{\langle t\rangle^s}\rmA+\sqrt{\f{\pa_t w}{w}}\tilde{\rmA}\right)P_{\neq}(\Psi\Upsilon_2)\right\|_2. 
\end{aligned}
\een
\end{proposition}
The proposition is proved in section \ref{sec: transport}. 

The linear term force $I_a^l$ is new because of the non-constant background density $\th$. However, it is $U^y\pa_y\th$ which can be regarded as the Reaction term $\rmR^a_{\rmN}$ but without smallness. 
\begin{proposition}\label{prop: I_a^l}
Under the bootstrap hypotheses,
\begin{align*}
|I_a^l|&\leq \f{C\ep^2}{\langle t\rangle^2}+C\ep\mathcal{CK}_{\varphi^{\d}_{10}}+ \varepsilon_0\mathcal{CK}_{a}\\
&\quad +C_{\varepsilon_0}\left\|\left\langle\f{\pa_v}{t\pa_z}\right\rangle^{-1}(\pa_z^2+(\pa_v-t\pa_z)^2)\left(\f{|\na|^{\f{s}{2}}}{\langle t\rangle^s}\rmA+\sqrt{\f{\pa_t w}{w}}\tilde{\rmA}\right)P_{\neq}(\Psi\Upsilon_2)\right\|_2^2. 
\end{align*}
\end{proposition}
The proposition is proved in section \ref{sec: I_a^l}. 

\subsection{Energy estimates of the coordinate system and coefficients}
The coordinate system is easy to control, which is similar to {\it Proposition 2.5} in \cite{BM2015}. 
\begin{proposition}[Coordinate system controls]\label{prop: coordinate} Under the bootstrap hypotheses, for $\ep$ sufficiently small and $\rmK_v$ sufficiently large there is a $\rmK>0$ such that
\begin{align}
\label{(2.30a)}
&\mathcal{E}_{h}+\f14\int_1^t \mathcal{CK}_h(\tau)d\tau\leq \f{1}{2}\rmK_v\ep^2,\\
\label{(2.30b)}
&\mathcal{E}_{\bar{h}}+\int_{1}^t\mathcal{CK}_{\bar{h}}(\tau)+d\tau\leq C(k_{M})\ep^2+\rmK\ep^3,\\
\label{(2.30c)}
&\langle t\rangle^{4-\rmK_{\rmD}\ep}\|\udl{\pa_tv}\|_{\mathcal{G}^{s,\la(t),\s-6}}^2\leq C(k_{M})\ep^2+\rmK\ep^3.
\end{align}
\end{proposition}
\begin{proof}
We use to the estimate of $\udl{\om}$ in Corollary \ref{Cor: psi-new} to control the term $P_0(\pa_z(\Upsilon_2\Psi)\pa_v\udl{\om}-\pa_v(\Upsilon_2\Psi)\pa_z\udl{\om})$. The term which involves $P_0(\pa_z\Pi_{l}\pa_va-\pa_v\Pi_{l}\pa_za)$ and $P_0(\pa_z\Pi_{n}\pa_va-\pa_v\Pi_{n}\pa_za)$ can be controlled by using the same idea of the proof of Proposition \ref{prop: I_f4,I_f8} in section \ref{sec: I_f4,8}. 
\end{proof}

\begin{proposition}\label{prop: coefficients} 
Under the bootstrap hypotheses, for $\ep$ sufficiently small and $\rmK_v$ sufficiently large there is a $\rmK>0$ such that for $j=1,2,...,11,$
\begin{align}
&\mathcal{E}_{\varphi^{\d}_j}+\f14\int_1^t \mathcal{CK}_{\varphi^{\d}_j}(\tau)d\tau\leq \f{1}{2}\rmK_v\ep^2.
\end{align}
\end{proposition}

The energy estimates for the coefficients $\varphi^{\d}_j$ (j=1,...,11) is easier than $h$. They have the same transport structure as the equation of $h$:
\beq
(\pa_t+\udl{\pa_tv}\pa_v)\mathcal{U}=\mathcal{F},
\eeq
where $\mathcal{F}$ represents one of the forcing terms $\udl{\pa_tv}\pa_v\widetilde{\varphi_j}$. Due to the fact that $\widetilde{\varphi_j}$ are smoother given functions compared to the solutions, $\mathcal{F}$ has the same behavior as $\udl{\pa_tv}$ which is better than the force term $\bar{h}=\udl{\pa_yv}\pa_v\udl{\pa_tv}$ in the equation of $h$ \eqref{eq: h}. By following the estimate of $h$ in {\it section 8.2} of \cite{BM2015}, one can obtain the estimate of $\mathcal{E}_{\varphi^{\d}_j}(t)$ easily. 

\subsection{Proof of the main theorem}
Now, we prove the compact support result in Proposition \ref{prop: btsp} and conclude the proof. Under the bootstrap assumption, we have
\beno
\Delta\psi(t,x,y)=\om(t,x,y)=\udl{\om}(t,x-tv(t,y),v(t,y)), \quad 
\psi(x,0)=\psi(x,1)=0. 
\eeno
Taking the Fourier transform in $x$, we get
\ben\label{eq:psi}
\mathcal{F}_1\psi(t,k,y)=\int_0^1G_k(y,z)\udl{\om}(t,k,v(t,z))e^{-iktv(t,z)}dz,
\een
and
\ben\label{eq:u}
\pa_y\mathcal{F}_1\psi(t,k,y)=\int_0^1\pa_yG_k(y,z)\udl{\om}(t,k,v(t,z))e^{-iktv(t,z)}dz,
\een
where 
\beno
G_k(y_1,y_2)=\f{1}{k\sinh k}\times\left\{\begin{aligned}
&\sinh(k(1-y_2))\sinh ky_1,\quad y_1\leq y_2,\\
&\sinh(ky_2)\sinh(k(1-y_1)),\quad y_1\geq y_2,
\end{aligned}\right.
\eeno
and
\beno
G_0(y_1,y_2)=\left\{\begin{aligned}
&(1-y_2)y_1,\quad y_1\leq y_2,\\
&y_2(1-y_1),\quad y_1\geq y_2,
\end{aligned}\right.
\eeno

Integrating by parts in $z$ in the identities \eqref{eq:psi}(twice) and \eqref{eq:u} (once), we get that
\ben\label{eq: inviscid-damping}
|U^y(t,x,y)|\lesssim \f{\ep}{\langle t\rangle^2},\quad |U^x-<U^x>|\lesssim \f{\ep}{\langle t\rangle}. 
\een
Then let us define
\beno
\left\{\begin{aligned}
&\f{d(X^1,X^2)}{dt}(t,x,y)=(u(y)+U^x(t,X^1,X^2),U^y(t,X^1,X^2)),\\
&(X^1,X^2)\big|_{t=0}=(x,y).\\
\end{aligned}\right.
\eeno
Then $\f{d}{dt}\om(t,X^1,X^2)=(u''\pa_x\psi)(t,X_1,X_2)$ and $|X^2(t,x,y)-y|\leq C\ep$, thus the for $y\notin [4\kappa_0-C\ep, 1-4\kappa_0+C\ep]$, $\f{d}{dt}\om(t,X^1,X^2)=0$. 
Note that we only need that $u''$ to be equal to $0$ close to the boundary. Thus the support of $\om(t)$ will always be away from the boundary.  

This ends the proof of the bootstrap Proposition \ref{prop: btsp}. \qed

Let us now go back to the original coordinate system and prove Theorem \ref{Thm: main}. Applying the same method in {\it section 2.4} of \cite{BM2015}, we get for $\om_1(t,z,y)=\udl{\om}(t,z,v)=\om(t,x,y)$, $P_1(t, z, y)=P(t, x, y)=\Pi(t, z, v)$, $d_1(t, z, y)=d(t, x, y)=a(t, z, v)$ and $\psi_1(t,z,y)=\Psi(t,z,v)=\psi(t,x,y)$, 
\beno
\pa_t\om_1+\na^{\bot}_{z,y}P_{\neq}\psi_1\cdot\na_{z,y}\om_1-u''\pa_z\psi_1-\th'(y)\pa_zP_1+\pa_yP_1\pa_zd_1-\pa_zP_1\pa_ya=0,
\eeno
and
\beno
\pa_td_1+\th'\pa_z\psi_1+\na^{\bot}_{z,y}P_{\neq}\psi_1\cdot\na_{z,y}d_1=0
\eeno
Then we can define
\beno
\begin{aligned}
f_{\infty}=&\om_1(1)-\int_{1}^{\infty}\na^{\bot}_{z,y}P_{\neq}\psi_1(s)\cdot\na_{z,y}\om_1(s)ds+\int_{1}^{\infty}u''\pa_z\psi_1(s)ds\\
&+\int_1^{\infty}\th'(y)\pa_zP_1(s)ds-\int_1^{\infty}\pa_yP_1\pa_zd_1(s)ds+\int_1^{\infty}\pa_zP_1\pa_ya(s)ds,
\end{aligned}
\eeno
and
\beno
\begin{aligned}
d_{\infty}=&d_1(1)-\int_1^{\infty}\th'\pa_z\psi_1(s)ds-\int_1^{\infty}\na^{\bot}_{z,y}P_{\neq}\psi_1\cdot\na_{z,y}d_1(s)ds. 
\end{aligned}
\eeno
Therefore there exists $f_{\infty}, d_{\infty}$ such that, 
\beno
\left\|\om(t,x+tu(y)+\Phi(t,y)\chi_1(y),y)-f_{\infty}(x,y)\right\|_{\mathcal{G}^{s,\la_{\infty}}}\lesssim \f{\ep}{\langle t\rangle}.
\eeno
\beno
\left\|d(t,x+tu(y)+\Phi(t,y)\chi_1(y),y)-d_{\infty}(x,y)\right\|_{\mathcal{G}^{s,\la_{\infty}}}\lesssim \f{\ep}{\langle t\rangle}.
\eeno
Note that $\om(t,x,y)=0$ for $y\notin [2\kappa_0,1-2\kappa_0]$ so is $f_{\infty}$, thus we can replace $\Phi(t,y)\chi_1(y)$ by $\Phi(t,y)$ which gives \eqref{eq: Scattering om} and \eqref{eq: Scattering den}. 

Now, we prove \eqref{eq: inviscid damping}. We have 
\beno
-\pa_yP_0(U^x)=P_0(\om),
\quad
\pa_tP_0(U^x)+P_0(U^y\om)=0,
\eeno
which implies that $P_0(U^x)(t,y)=C_U$ for any $t\geq 1$ and $y\in [0,2\kappa_0]\cup [1-2\kappa_0,1]$. 
Let $\tilde{U}^x(t,z,y)=U^x(t,x,y)$, then $\Big(P_0(\tilde{U}^x)(t,y)-C_{U}\Big)$ has the same compact support as $\om_1$ and satisfies
\begin{align*}
\pa_t\Big(P_0(\tilde{U}^x)(t,y)-C_{U}\Big)
&=-P_0(\na^{\bot}_{z,y}P_{\neq }\psi_1\cdot\na_{z,y}\tilde{U}^x)\chi_1(y)\\
&=-P_0(-\pa_yP_{\neq }(\Upsilon_2\psi_1)\pa_z(\pa_y-t\pa_yv\pa_z)(\Upsilon_2\psi_1)+\pa_zP_{\neq }(\Upsilon_2\psi_1)\pa_y\pa_z(\Upsilon_2\psi_1)). 
\end{align*}
Then \eqref{eq: inviscid damping} follows from the fact that 
\ben\label{eq:decay-psi_1}
\|\Upsilon_2\psi_1\|_{\mathcal{G}^{s,\la_{\infty}',2}}\lesssim \f{\ep}{\langle t\rangle^2}. 
\een
This ends the proof of Theorem \ref{Thm: main}. \qed

\section{Elliptic estimate for steam function}\label{sec: ell-stream}
In this section, we study the following elliptic equation
\beq\label{eq: steam function}
\left\{\begin{aligned}
&\udl{\varphi_4}\pa_{zz}\Psi+\udl{\pa_yv}(\pa_v-t\pa_z)\Big(\udl{\varphi_4}\udl{\pa_yv}(\pa_v-t\pa_z)\Psi\Big)=\Om,\\
&\Psi(t, z, u(0))=\Psi(t, z, u(1))=0,
\end{aligned}\right.
\eeq
We rewrite the equation by separating its linear part and the nonlinear part. 
\beq\label{eq: decomposition stream}
\begin{aligned}
\Om&=\udl{\varphi_4}\pa_{zz}\Psi+\udl{\pa_yv}(\pa_v-t\pa_z)\Big(\udl{\varphi_4}\udl{\pa_yv}(\pa_v-t\pa_z)\Psi\Big)\\
&=\mathcal{T}_1[\Psi]
+\varphi^{\d}_4\pa_{zz}\Psi
+\udl{\pa_yv}(\pa_v-t\pa_z)\Big(\udl{\varphi_4}(\udl{\pa_yv}-\widetilde{u'})(\pa_v-t\pa_z)\Psi\Big)\\
&\quad+\udl{\pa_yv}(\pa_v-t\pa_z)\Big(\varphi_4^{\d}\widetilde{u'}(\pa_v-t\pa_z)\Psi\Big)
+(\udl{\pa_yv}-\widetilde{u'})(\pa_v-t\pa_z)\Big(\widetilde{\varphi_4}\widetilde{u'}(\pa_v-t\pa_z)\Psi\Big). 
\end{aligned}
\eeq
where 
\ben\label{eq: T_1}
\mathcal{T}_1[\Psi]=\widetilde{\varphi_2}\pa_{zz}\Psi+\widetilde{u'}(\pa_v-t\pa_z)\Big(\widetilde{\varphi_2}\widetilde{u'}(\pa_v-t\pa_z)\Psi\Big).
\een
It is easy to see that $\mathcal{T}_1$ is an elliptic second-order differential operator. We define $\mathcal{T}_{1,D}^{-1}$ to be the inverse operator with Dirichlet boundary conditions, namely for any function $f(t, z, v)$, $\mathcal{T}_1\big[\mathcal{T}_{1,D}^{-1}[f]\big]=f$ with 
\beno
\mathcal{T}_{1,D}^{-1}[f](t, z, 0)=\mathcal{T}_{1,D}^{-1}[f](t, z, 1)
\eeno
We also define 
\ben\label{eq: inverse cutoff-op}
\mathring{\mathcal{T}}_{1,D}^{-1}=\Upsilon_2\mathcal{T}_{1,D}^{-1}. 
\een
\begin{proposition}\label{prop: T_1}
Let $\Om$ be compactly supported. Then there is $\mathcal{G}_{D,1}(t, k, \xi, \eta)$ such that 
\beno
\widehat{(\mathring{\mathcal{T}}_{1,D}^{-1}[\Om])}(t, k, \xi)=\int_{\mathbb{R}}\mathcal{G}_{D,1}(t, k, \xi, \eta)\widehat{\Om}(t, k, \eta)d\eta,
\eeno
with estimate
\beno
|\mathcal{G}_{D,1}(t, k, \xi, \eta)|\leq C\min\left\{\f{e^{-\la_{\Delta}\langle \xi-\eta\rangle^s}}{1+k^2+(\xi-kt)^2}, \f{e^{-\la_{\Delta}\langle \xi-\eta\rangle^s}}{1+k^2+(\eta-kt)^2}\right\}.
\eeno
\end{proposition}
The proposition is proved in Appendix \ref{sec: SL-D}. 

By applying ${\mathcal{T}}_{1,D}^{-1}$ on equation \eqref{eq: decomposition stream} and multiplying the cut-off function $\Upsilon_2$, we get that
\begin{align}\label{eq: diff-elliptic-1}
\mathring{\mathcal{T}}_{1,D}^{-1}[\Om]-\Upsilon_2\Psi
=\mathring{\mathcal{T}}_{1,D}^{-1}[\mathcal{S}_1[\Psi]]
\end{align}
where 
\beq
\begin{aligned}
\mathcal{S}_1[\Psi]=
&\varphi^{\d}_4\pa_{zz}\Psi
+\udl{\pa_yv}(\pa_v-t\pa_z)\Big(\udl{\varphi_4}(\udl{\pa_yv}-\widetilde{u'})(\pa_v-t\pa_z)\Psi\Big)\\
&+\udl{\pa_yv}(\pa_v-t\pa_z)\Big(\varphi_4^{\d}\widetilde{u'}(\pa_v-t\pa_z)\Psi\Big)
+(\udl{\pa_yv}-\widetilde{u'})(\pa_v-t\pa_z)\Big(\widetilde{\varphi_4}\widetilde{u'}(\pa_v-t\pa_z)\Psi\Big)\\
=&\varphi^{\d}_4\pa_{zz}\Psi+h_1(\pa_v-t\pa_z)^2(\Upsilon_2\Psi)+h_2(\pa_v-t\pa_z)(\Upsilon_2\Psi)
\end{aligned}
\eeq
has compact support with $\mathrm{supp}\,\mathcal{S}_1[\Psi]\subset [\kappa_0,1-\kappa_0]$ with 
\begin{align*}
&h_1=\udl{\pa_yv}\udl{\varphi_4}(\udl{\pa_yv}-\widetilde{u'})+\udl{\pa_yv}\varphi_4^{\d}\widetilde{u'}
+(\udl{\pa_yv}-\widetilde{u'})\widetilde{\varphi_4}\widetilde{u'}\\
&h_2=\udl{\pa_yv}\pa_v\big(\udl{\varphi_4}(\udl{\pa_yv}-\widetilde{u'})\big)
+\udl{\pa_yv}\pa_v\big(\varphi_4^{\d}\widetilde{u'}\big)
+(\udl{\pa_yv}-\widetilde{u'})\pa_v(\widetilde{\varphi_4}\widetilde{u'}). 
\end{align*}

\subsection{Lossy estimate}
We have the following lemmas.
\begin{lemma}\label{lem: lossy-elliptic u}
It holds for all $0\leq s_1\leq \s-1$ that 
\ben\label{eq: lossy-elliptic u1}
\|\pa_{z}^2\Upsilon_2P_{\neq}\Psi\|_{\mathcal{G}^{s,\la(t),s_1-2}}\lesssim \f{\|\Om\|_{\mathcal{G}^{s,\la(t),s_1}}}{\langle t\rangle^2}.
\een
\end{lemma}
With the estimate of the Fourier kernel of $\mathring{\mathcal{T}}_{1,D}^{-1}$, the lossy estimate is simple. We refer to the proof of {\it Lemma 3.4} in \cite{MasmoudiZhao2020}. 

\subsection{Precision elliptic control}
The precision elliptic control is slightly more complicated than that in \cite{MasmoudiZhao2020}. We have more terms here. 
The idea is to show that $\mathcal{S}_1[\Psi]$ is small. We have the following proposition. 
\begin{proposition}\label{prop: estimate diff-1}
Under the bootstrap hypotheses,
\ben\label{eq: error D-Psi_1}
\begin{aligned}
&\left\|\left\langle\f{\pa_v}{t\pa_z}\right\rangle^{-1}(\pa_z^2+(\pa_v-t\pa_z)^2)\left(\f{|\na|^{\f{s}{2}}}{\langle t\rangle^s}\rmA+\sqrt{\f{\pa_t w}{w}}\tilde{\rmA}\right)P_{\star}\big(\mathring{\mathcal{T}}_{1,D}^{-1}[\mathcal{S}_1[\Psi]]\big)\right\|_2^2\\
&\lesssim \ep \left\|\left\langle\f{\pa_v}{t\pa_z}\right\rangle^{-1}(\pa_z^2+(\pa_v-t\pa_z)^2)\left(\f{|\na|^{\f{s}{2}}}{\langle t\rangle^s}\rmA+\sqrt{\f{\pa_t w}{w}}\tilde{\rmA}\right)P_{\star}(\Psi\Upsilon_2)\right\|_2^2\\
&\quad +\ep \mathcal{CK}_{h}+\ep\mathcal{CK}_{\varphi_{9}^{\d}}+\ep\mathcal{CK}_{\varphi_4^{\d}},
\end{aligned}
\een
and
\ben\label{eq: error Psi_1}
\begin{aligned}
&\left\|\left\langle\f{\pa_v}{t\pa_z}\right\rangle^{-1}(\pa_z^2+(\pa_v-t\pa_z)^2)P_{\star}\big(\mathring{\mathcal{T}}_{1,D}^{-1}[\mathcal{S}_1[\Psi]]\big)\right\|_2^2\\
&\lesssim \ep \left\|\left\langle\f{\pa_v}{t\pa_z}\right\rangle^{-1}(\pa_z^2+(\pa_v-t\pa_z)^2)P_{\star}(\Psi\Upsilon_2)\right\|_2^2+\ep^2,
\end{aligned}
\een
where $P_{\star}\in \big\{P_{0<|k|<k_{M}}, P_{|k|\geq k_{M}}\big\}$ is the projecction. 
\end{proposition}

\begin{proof}
By applying Proposition \ref{prop: T_1}, we have
\begin{align*}
&\left\|\left\langle\f{\pa_v}{t\pa_z}\right\rangle^{-1}(\pa_z^2+(\pa_v-t\pa_z)^2)\left(\f{|\na|^{\f{s}{2}}}{\langle t\rangle^s}\rmA+\sqrt{\f{\pa_t w}{w}}\tilde{\rmA}\right)P_{\star}\big(\mathring{\mathcal{T}}_{1,D}^{-1}[\mathcal{S}_1[\Psi]]\big)\right\|_2\\
&\lesssim \left\|\mathcal{M}_1\big([\mathcal{S}_1[\Psi]]\big)\right\|_2+\left\|\mathcal{M}_2\big([\mathcal{S}_1[\Psi]]\big)\right\|_2
\end{align*}
and
\begin{align*}
&\left\|\left\langle\f{\pa_v}{t\pa_z}\right\rangle^{-1}(\pa_z^2+(\pa_v-t\pa_z)^2)P_{\star}\big(\mathring{\mathcal{T}}_{1,D}^{-1}[\mathcal{S}_1[\Psi]]\big)\right\|_2
\lesssim \left\|\mathcal{M}_0\big([\mathcal{S}_1[\Psi]]\big)\right\|_2
\end{align*}
where
\beno
&&\mathcal{M}_0(t,\na)=\left\langle\f{\pa_v}{t\pa_z}\right\rangle^{-1}\rmA P_{\star},\\
&&\mathcal{M}_1(t,\na)=\left\langle\f{\pa_v}{t\pa_z}\right\rangle^{-1}\f{|\na|^{s/2}}{\langle t\rangle^{s}}\rmA P_{\star},\\
&&\mathcal{M}_2(t,\na)=\left\langle\f{\pa_v}{t\pa_z}\right\rangle^{-1}\sqrt{\f{\pa_tw}{w}}\tilde{\rmA} P_{\star}.
\eeno
Let 
\ben\label{eq: T^1,T^2}
\mathrm{T}^1=h_1(\pa_v-t\pa_z)^2(\Psi\Upsilon_2),\quad
\mathrm{T}^2=h_2(\pa_v-t\pa_z)(\Psi\Upsilon_2),\quad
\mathrm{T}^3=\varphi^{\d}_4\pa_{zz}(\Upsilon_2\Psi).
\een
Divide each via a paraproduct decomposition in the $v$ variable only 
\begin{align}
\label{eq:T_1}\mathrm{T}^1&=\sum_{\rmM\geq 8}(h_1)_{\rmM}(\pa_v-t\pa_z)^2(\Psi\Upsilon)_{<\rmM/8}
+\sum_{\rmM\geq 8}(h_1)_{<\rmM/8}(\pa_v-t\pa_z)^2(\Psi\Upsilon)_{\rmM}\\
\nonumber
&\quad+\sum_{\rmM\in \mathcal{D}}\sum_{\f18\rmM\leq \rmM'\leq 8\rmM}(h_1)_{\rmM'}(\pa_v-t\pa_z)^2(\Psi\Upsilon)_{\rmM}\\
\nonumber
&=\mathrm{T}^1_{\mathrm{HL}}+\mathrm{T}^1_{\mathrm{LH}}+\mathrm{T}^1_{\mathcal{R}},\\
\label{eq:T_2}\mathrm{T}^2&=\sum_{\rmM\geq 8}(h_2)_{M}(\pa_v-t\pa_z)(\Psi\Upsilon)_{<\rmM/8}
+\sum_{\rmM\geq 8}(h_2)_{<\rmM/8}(\pa_v-t\pa_z)(\Psi\Upsilon)_{\rmM}\\
\nonumber&\quad+\sum_{\rmM\in \mathcal{D}}\sum_{\f18\rmM\leq \rmM'\leq 8\rmM}(h_2)_{\rmM'}(\pa_v-t\pa_z)(\Psi\Upsilon)_{\rmM}\\
\nonumber&=\mathrm{T}^2_{\mathrm{HL}}+\mathrm{T}^2_{\mathrm{LH}}+\mathrm{T}^2_{\mathcal{R}},\\
\label{eq:T_3}\mathrm{T}^3&=\sum_{\rmM\geq 8}(\varphi^{\d}_4)_{\rmM}(\pa_z)^2(\Psi\Upsilon)_{<\rmM/8}
+\sum_{\rmM\geq 8}(\varphi^{\d}_4)_{<\rmM/8}(\pa_z)^2(\Psi\Upsilon)_{\rmM}\\
\nonumber
&\quad+\sum_{\rmM\in \mathcal{D}}\sum_{\f18\rmM\leq \rmM'\leq 8\rmM}(\varphi^{\d}_4)_{\rmM'}(\pa_z)^2(\Psi\Upsilon)_{\rmM}\\
\nonumber
&=\mathrm{T}^3_{\mathrm{HL}}+\mathrm{T}^3_{\mathrm{LH}}+\mathrm{T}^3_{\mathcal{R}}.
\end{align}
Note that $\mathrm{T}^1=h_1(\pa_v-t\pa_z)^2(\Psi\Upsilon)$ has a similar form as the term $\mathrm{T}^1=(1-(v')^2)(\pa_v-t\pa_z)^2\phi$ in \cite{BM2015}. The support of the functions is away from the boundary, which guarantees that the Fourier analysis works well. Similarly, the term $\mathrm{T}^2$ has the same structure. The new term $\mathrm{T}^3$ here is even better than $\mathrm{T}^1$. 
The key point in {\it proof of Proposition 2.4, section 4.2.1 --- 4.2.3} in \cite{BM2015} is that the Fourier multipliers $\mathcal{M}_0$, $\mathcal{M}_1$ and $\mathcal{M}_2$ have some algebraic type properties. It is the same here. Thus we have 
\begin{align*}
&\sum_{j=1}^3\left\|\mathcal{M}_0\mathrm{T}^j_{\mathrm{LH}}\right\|_2^2\lesssim \ep^2\left\|\mathcal{M}_0\Delta_{L}P_{\star}(\Psi\Upsilon)\right\|_2^2\\
&\sum_{j=1}^3\left\|\mathcal{M}_1\mathrm{T}^j_{\mathrm{LH}}\right\|_2^2\lesssim \ep^2\left\|\mathcal{M}_1\Delta_{L}P_{\star}(\Psi\Upsilon)\right\|_2^2\\
&\sum_{j=1}^3\left\|\mathcal{M}_2\mathrm{T}^j_{\mathrm{LH}}\right\|_2^2
\lesssim \ep^2\left\|\mathcal{M}_1\Delta_{L}P_{\star}(\Psi\Upsilon)\right\|_2^2+\ep^2\left\|\mathcal{M}_2\Delta_{L}P_{\star}(\Psi\Upsilon)\right\|_2^2,
\end{align*}
and
\begin{align*}
&\sum_{j=1}^3\left\|\mathcal{M}_0\mathrm{T}^j_{\mathrm{HL}}\right\|_2^2\lesssim \ep^2\left\|\mathcal{M}_0\Delta_{L}P_{\star}(\Psi\Upsilon)\right\|_2^2+\ep^4,\\
&\sum_{j=1}^3\left\|\mathcal{M}_1\mathrm{T}^j_{\mathrm{HL}}\right\|_2^2
\lesssim \ep^2\left\|\mathcal{M}_1\Delta_{L}P_{\star}(\Psi\Upsilon)\right\|_2^2+\ep^2\mathrm{CK}_{h}^2+\ep^2\mathrm{CK}_{\varphi^{\d}_4}^2+\ep^2\mathrm{CK}_{\varphi^{\d}_9}^2,\\
&\sum_{j=1}^3\left\|\mathcal{M}_2\mathrm{T}^j_{\mathrm{HL}}\right\|_2^2
\lesssim \ep^2\left\|\mathcal{M}_2\Delta_{L}P_{\star}(\Psi\Upsilon)\right\|_2^2+\ep^2\left\|\mathcal{M}_1\Delta_{L}P_{\star}(\Psi\Upsilon)\right\|_2^2\\
&\qquad\qquad\qquad\qquad
+\ep^2\mathrm{CK}_{h}^1+\ep^2\mathrm{CK}_{\varphi^{\d}_4}^1+\ep^2\mathrm{CK}_{\varphi^{\d}_9}^2,
\end{align*}
and
\begin{align*}
&\sum_{j=1}^3\left\|\mathcal{M}_0\mathrm{T}^j_{\mathcal{R}}\right\|_2^2\lesssim \ep^2\left\|\mathcal{M}_0\Delta_{L}P_{\star}(\Psi\Upsilon)\right\|_2^2,\\
&\sum_{j=1}^3\left\|\mathcal{M}_1\mathrm{T}_{\mathcal{R}}^j\right\|_2^2\lesssim \ep^2\left\|\mathcal{M}_1\Delta_{L}P_{\star}(\Psi\Upsilon)\right\|_2^2,\\
&\sum_{j=1}^3\left\|\mathcal{M}_2\mathrm{T}_{\mathcal{R}}^j\right\|_2^2\lesssim \ep^2\left\|\mathcal{M}_1\Delta_{L}P_{\star}(\Psi\Upsilon)\right\|_2^2+\ep^2\left\|\mathcal{M}_2\Delta_{L}P_{\star}(\Psi\Upsilon)\right\|_2^2,
\end{align*}
where $\Delta_{L}=\pa_{z}^2+(\pa_v-t\pa_z)^2$. 
Thus we prove the proposition. 
\end{proof}
Now we prove Proposition \ref{prop:elliptic}. 
\begin{proof}
We consider the case $|k|\geq k_{M}$ first. In this case, $P_kf=P_{k}\Om$. By using \eqref{eq: diff-elliptic-1}, and Proposition \ref{prop: estimate diff-1}, we have 
\begin{align*}
&\left\|\left\langle\f{\pa_v}{t\pa_z}\right\rangle^{-1}(\pa_z^2+(\pa_v-t\pa_z)^2)\left(\f{|\na|^{\f{s}{2}}}{\langle t\rangle^s}\rmA+\sqrt{\f{\pa_t w}{w}}\tilde{\rmA}\right)P_{|k|\geq k_{M}}(\Psi\Upsilon_2)\right\|_2^2\\
&\leq C_{ell}\mathcal{CK}_f+\left\|\left\langle\f{\pa_v}{t\pa_z}\right\rangle^{-1}(\pa_z^2+(\pa_v-t\pa_z)^2)\left(\f{|\na|^{\f{s}{2}}}{\langle t\rangle^s}\rmA+\sqrt{\f{\pa_t w}{w}}\tilde{\rmA}\right)P_{|k|\geq k_{M}}\big(\mathring{\mathcal{T}}_{1,D}^{-1}[\mathcal{S}_1[\Psi]]\big)\right\|_2^2\\
&\leq C_{ell}\mathcal{CK}_f+\ep \left\|\left\langle\f{\pa_v}{t\pa_z}\right\rangle^{-1}(\pa_z^2+(\pa_v-t\pa_z)^2)\left(\f{|\na|^{\f{s}{2}}}{\langle t\rangle^s}\rmA+\sqrt{\f{\pa_t w}{w}}\tilde{\rmA}\right)P_{|k|\geq k_{M}}(\Psi\Upsilon_2)\right\|_2^2\\
&\quad +\ep \mathcal{CK}_{h}+\ep\mathcal{CK}_{\varphi_{9}^{\d}}+\ep\mathcal{CK}_{\varphi_4^{\d}}
\end{align*}
which gives \eqref{eq:elliptic1}. 

For the case $0<|k|<k_{M}$, we have $P_kf(t, z, v)=\bfD_{u,k}[\mathcal{F}_1\Om(t, k,\cdot)](t, k, v)e^{izk}$ and thus 
\beno
P_k\Om(t, z, v)=\bfD_{u,k}^{-1}[\mathcal{F}_1f(t, k,\cdot)](t, k, v)e^{izk}.
\eeno 
By following the same argument, we have 
\begin{align*}
&\left\|\left\langle\f{\pa_v}{t\pa_z}\right\rangle^{-1}(\pa_z^2+(\pa_v-t\pa_z)^2)\left(\f{|\na|^{\f{s}{2}}}{\langle t\rangle^s}\rmA+\sqrt{\f{\pa_t w}{w}}\tilde{\rmA}\right)P_{|k|\geq k_{M}}(\Psi\Upsilon_2)\right\|_2^2\\
&\leq \left\|\left\langle\f{\pa_v}{t\pa_z}\right\rangle^{-1}\left(\f{|\na|^{\f{s}{2}}}{\langle t\rangle^s}\rmA+\sqrt{\f{\pa_t w}{w}}\tilde{\rmA}\right)P_{|k|\leq k_{M}}(\Om)\right\|_2^2+\ep \mathcal{CK}_{h}+\ep\mathcal{CK}_{\varphi_{9}^{\d}}+\ep\mathcal{CK}_{\varphi_4^{\d}}.
\end{align*}
To control the first term by $\mathcal{CK}_f$, we use Remark \ref{Rmk:CKAOm}. Thus we prove the proposition. 
\end{proof}
We conclude this section by introducing the following corollary, which will be used in the estimates of pressure terms. 
\begin{corol}\label{Cor: psi-new}
Under the bootstrap hypotheses,
\beno
 \left\|\left\langle\f{\pa_v}{t\pa_z}\right\rangle^{-1}(\pa_z^2+(\pa_v-t\pa_z)^2)\rmA P_{\neq}(\Psi\Upsilon_2)\right\|_2\lesssim \ep,
\eeno
and
\begin{align*}
&\left\|\langle\pa_v\rangle\rmA_0P_0(\Upsilon_2\Psi)\right\|_2\lesssim \ep,\\
&\left\|\Big(\sqrt{\f{\pa_tw}{w}}+\f{\langle \pa_v\rangle^{\f{s}{2}}}{\langle t\rangle^s}\Big)\langle\pa_v\rangle\rmA_0P_0(\Upsilon_2\Psi)\right\|_2\lesssim 
\mathcal{CK}_f^{\f12}+\ep\Big(\mathcal{CK}_h^{\f12}+\mathcal{CK}_{\varphi^{\d}_4}^{\f12}+\mathcal{CK}_{\varphi^{\d}_9}^{\f12}\Big). 
\end{align*}
Moreover, we have estimate for $\udl{\om}(t, z, v)=\udl{\th}\Om+\udl{\varphi}_2\udl{\pa_yv}(\pa_v-t\pa_z)\Psi$. It holds that
\beno
\|\rmA\udl{\om}\|_{L^2}\lesssim \ep ,
\eeno
and
\begin{align*}
&\left\|\Big(\sqrt{\f{\pa_tw}{w}}\rmA+\f{\langle \na\rangle^{\f{s}{2}}}{\langle t\rangle^s}\tilde{\rmA}\Big)\udl{\om}\right\|_2\lesssim 
\mathcal{CK}_f^{\f12}+\ep\Big(\mathcal{CK}_h^{\f12}+\mathcal{CK}_{\varphi^{\d}_2}^{\f12}+\mathcal{CK}_{\varphi^{\d}_4}^{\f12}+\mathcal{CK}_{\varphi^{\d}_9}^{\f12}\Big).
\end{align*}
\end{corol}
\begin{proof}
The estimates for nonzero modes and the estimates for $\udl{\om}$ in the corollary follow directly from \eqref{eq: diff-elliptic-1}, \eqref{eq: error Psi_1} and Remark \ref{Rmk: wave bdd-A} and the estimate of zero mode of $\Upsilon_2\Psi$.

For the zero mode, similarly, we have the decomposition 
\beq\label{eq: decomposition stream-P_0}
\begin{aligned}
P_0(\Om)
&=P_0(\mathcal{T}_1[\Psi])+
h_1\pa_{vv}P_0(\Upsilon_2\Psi)+h_2\pa_vP_0(\Upsilon_2\Psi), 
\end{aligned}
\eeq
where 
\ben\label{eq: T_1-P_0}
P_0(\mathcal{T}_1[\Psi])=\widetilde{u'}\pa_v\Big(\widetilde{\varphi_2}\widetilde{u'}\pa_vP_0(\Psi)\Big).
\een
We then have 
\begin{align*}
\langle\pa_v\rangle\rmA_0P_0(\Upsilon_2\Psi)
=\langle\pa_v\rangle\mathring{\mathcal{T}}_{1,D}^{-1}\Big[P_0(\Om)-h_1\pa_{vv}P_0(\Upsilon_2\Psi)-h_2\pa_vP_0(\Upsilon_2\Psi)\Big].
\end{align*}
Notice that Proposition \ref{prop: T_1} holds also for $k=0$. 
Note that for $|\eta|\geq 2|\xi-\eta|$
\begin{align*}
\f{\rmA_0(\xi)}{\rmA^{\rmR}(\eta)}\lesssim e^{c\la\langle \xi-\eta\rangle^s},\quad \text{for some}\quad 0<c<1. 
\end{align*}
Thus we conclude that 
\begin{align*}
\left\|\langle\pa_v\rangle\rmA_0P_0(\Upsilon_2\Psi)\right\|_2
&\lesssim \left\|\langle\pa_v\rangle^{-1}\rmA_0P_0\Om\right\|_2
+\left\|\langle\pa_v\rangle^{-1}\rmA_0(h_1\pa_{vv}P_0(\Upsilon_2\Psi))\right\|_2\\
&\quad+\left\|\langle\pa_v\rangle^{-1}\rmA_0(h_2\pa_vP_0(\Upsilon_2\Psi))\right\|_{L^2}\\
&\lesssim \|\rmA_0P_0\Om\|_2+\ep \left\|\langle\pa_v\rangle\rmA_0P_0(\Upsilon_2\Psi)\right\|_2\\
&\quad+\Big(\|\rmA^{\rmR}h_1\|_{L^2}+\|\rmA^{\rmR}\langle\pa_v\rangle^{-1}h_2\|_{L^2}\Big)\|P_0(\Upsilon_2\Psi)\|_{\mathcal{G}^{s,\s-6}}.
\end{align*}
Similarly, we have 
\begin{align*}
&\left\|\Big(\sqrt{\f{\pa_tw}{w}}+\f{\langle \pa_v\rangle^{\f{s}{2}}}{\langle t\rangle^s}\Big)\langle\pa_v\rangle\rmA_0P_0(\Upsilon_2\Psi)\right\|_2\\
&\lesssim \mathcal{CK}_f^{\f12}
+\left\|\langle\pa_v\rangle^{-1}\rmA_0\Big(\sqrt{\f{\pa_tw}{w}}+\f{\langle \pa_v\rangle^{\f{s}{2}}}{\langle t\rangle^s}\Big)(h_1\pa_{vv}P_0(\Upsilon_2\Psi))\right\|_2\\
&\quad+\left\|\langle\pa_v\rangle^{-1}\rmA_0\Big(\sqrt{\f{\pa_tw}{w}}+\f{\langle \pa_v\rangle^{\f{s}{2}}}{\langle t\rangle^s}\Big)(h_2\pa_vP_0(\Upsilon_2\Psi))\right\|_{L^2}\\
&\lesssim \mathcal{CK}_f^{\f12}+\ep \left\|\Big(\sqrt{\f{\pa_tw}{w}}+\f{\langle \pa_v\rangle^{\f{s}{2}}}{\langle t\rangle^s}\Big)\langle\pa_v\rangle\rmA_0P_0(\Upsilon_2\Psi)\right\|_2\\
&\quad+\Big(\mathcal{CK}_h^{\f12}+\mathcal{CK}_{\varphi^{\d}_4}^{\f12}+\mathcal{CK}_{\varphi^{\d}_9}^{\f12}\Big)\|P_0(\Upsilon_2\Psi)\|_{\mathcal{G}^{s,\s-6}},
\end{align*}
which gives the corollary. 
\end{proof}
The Corollary \ref{Cor: psi-new} gives the estimates of $\udl{\om}(t, z, v)$ which is the same as $\Om$. The estimates of the zero mode of the stream function $P_0(\Upsilon_2\Psi)$ are used in the estimate of $A_4=\udl{\chi_2'}\udl{\pa_yv}\pa_vP_{0}(\Upsilon_2\Psi)\pa_{zz}(\Upsilon_2\Psi)$ in the estimates of the pressure, see Lemma \ref{Lem: Xi_4}. 

\section{Zero mode of the pressure}\label{sec: pressure-zero}
In this section, we estimate the zero mode of the pressure. Unlike the whole space case, we will not use the elliptic equation to study the zero mode of the pressure. Instead, we use \eqref{eq: pressure-zero-mode}. 
Let us define
\begin{align*}
&\mathcal{M}_{00}(t, \xi)=\f{t(t^2+|\xi|^2)}{2+\xi^2}\rmA_0(\xi), \\
&\mathcal{M}_{01}(t, \xi)=\f{t(t^2+|\xi|^2)}{2+\xi^2}\Big(\sqrt{\f{\pa_tw}{w}}+\f{\langle \xi\rangle^{\f{s}{2}}}{\langle t\rangle^{s}}\Big)\rmA_0(\xi)
\end{align*}
By applying $\mathcal{M}_{00}$ and $\mathcal{M}_{01}$ on \eqref{eq: pressure-zero-mode}, we have for $j=0,1$
\beq\label{eq: P_0 pressure-eq}
\begin{aligned}
&\mathcal{M}_{0j}\pa_vP_0(\Pi\Upsilon_1)\\
=&-\mathcal{M}_{0j}\Big((\udl{\varphi_4}\Upsilon_2)P_0(a)\pa_vP_0(\Pi\Upsilon_1)\Big)
-\mathcal{M}_{0j}\Big((\udl{\varphi_4}\Upsilon_2)P_0\big[P_{\neq}(a)P_{\neq}(\pa_v-t\pa_z)P_{\neq}(\Pi\Upsilon_1)\big]\Big)\\
&+\mathcal{M}_{0j}\Big[(\udl{\varphi_4}\Upsilon_2)P_0\Big(\pa_z(\Psi\Upsilon_2)\pa_z(\pa_v-t\pa_z)(\Psi\Upsilon_2)-\pa_{zz}(\Psi\Upsilon_2)(\pa_v-t\pa_z)(\Psi\Upsilon_2)\Big)\Big)\Big].
\end{aligned}
\eeq
We first show that the first two terms in \eqref{eq: P_0 pressure-eq} are small.
\begin{proposition}\label{prop: 0pressure-error}
Under the bootstrap hypotheses,
\begin{align*}
&\left\|\mathcal{M}_{00}\Big((\udl{\varphi_4}\Upsilon_2)P_0(a)\pa_vP_0(\Pi\Upsilon_1)\Big)\right\|_2
\lesssim \ep \|\mathcal{M}_{00}\pa_vP_0(\Pi\Upsilon_1)\|_{L^2},\\
&\left\|\mathcal{M}_{00}\langle\pa_v\rangle\Big((\udl{\varphi_4}\Upsilon_2)P_0\big[P_{\neq}(a)P_{\neq}(\pa_v-t\pa_z)P_{\neq}(\Pi\Upsilon_1)\big]\Big)\right\|_2
\lesssim \ep^2,\\
&\left\|\mathcal{M}_{01}\Big((\udl{\varphi_4}\Upsilon_2)P_0(a)\pa_vP_0(\Pi\Upsilon_1)\Big)\right\|_2
\lesssim\ep \|\mathcal{M}_{01}\pa_vP_0(\Pi\Upsilon_1)\|_{L^2}+\f{\ep^2}{\langle t\rangle^s}+\ep \mathcal{CK}_a^{\f12}\\
&\left\|\mathcal{M}_{01}\Big((\udl{\varphi_4}\Upsilon_2)P_0\big[P_{\neq}(a)P_{\neq}(\pa_v-t\pa_z)P_{\neq}(\Pi\Upsilon_1)\big]\Big)\right\|_2\\
&\lesssim \f{\ep^2}{\langle t\rangle^s}+\ep \big(\mathcal{CK}_a^{\f12}+\mathcal{CK}_{\varphi^{\d}_4}^{\f12}\big)\\
&\quad+\ep \left\|\Big(\sqrt{\f{\pa_tw}{w}}+\f{\langle\na\rangle^{\f{s}{2}}}{\langle t\rangle^s}\Big){\langle \pa_z\rangle t(1+\langle \pa_z\rangle+|\pa_v|)}(1-\Delta)^{-1}\big((\pa_v-t\pa_z)^2+\pa_z^2\big)\rmA P_{\neq}(\Pi\Upsilon_1)\right\|_{L^2}. 
\end{align*}
\end{proposition}
\begin{proof}
It is easy to check that
\beno
\langle t\rangle^3\|\pa_vP_0(\Pi\Upsilon_1)\|_{\mathcal{G}^{s,\s-6}}\lesssim \|\mathcal{M}_{00}\pa_vP_0(\Pi\Upsilon_1)\|_{L^2}. 
\eeno
We have 
\begin{align*}
&\left\|\mathcal{M}_{00}\Big((\udl{\varphi_4}\Upsilon_2)P_0(a)\pa_vP_0(\Pi\Upsilon_1)\Big)\right\|_2\\
&\lesssim \|\rmA^{\rmR}(\udl{\varphi_4}\Upsilon_2)\|_{L^2}\|\rmA^*a\|_{L^2}\langle t\rangle^3\|\pa_vP_0(\Pi\Upsilon_1)\|_{\mathcal{G}^{s,\s-6}}+\ep \|\mathcal{M}_{00}\pa_vP_0(\Pi\Upsilon_1)\|_{L^2}\\
&\lesssim \ep \|\mathcal{M}_{00}\pa_vP_0(\Pi\Upsilon_1)\|_{L^2}. 
\end{align*}
and
\begin{align*}
&\left\|\mathcal{M}_{01}\Big((\udl{\varphi_4}\Upsilon_2)P_0(a)\pa_vP_0(\Pi\Upsilon_1)\Big)\right\|_2\\
&\lesssim \|(\udl{\varphi_4}\Upsilon_2)\|_{\mathcal{G}^{s,\s-6}}\mathcal{CK}_a^{\f12}\langle t\rangle^3\|\pa_vP_0(\Pi\Upsilon_1)\|_{\mathcal{G}^{s,\s-6}}+\ep \|\mathcal{M}_{01}\pa_vP_0(\Pi\Upsilon_1)\|_{L^2}\\
&\quad +\ep\Big(\f{1}{\langle t\rangle^s}+\mathcal{CK}_{\varphi^{\d}_4}^{\f12}\Big)\langle t\rangle^2\|\pa_vP_0(\Pi\Upsilon_1)\|_{\mathcal{G}^{s,\s-6}}\\
&\lesssim \ep \|\mathcal{M}_{01}\pa_vP_0(\Pi\Upsilon_1)\|_{L^2}+\f{\ep^2}{\langle t\rangle^s}+\ep \mathcal{CK}_a^{\f12}. 
\end{align*}
We also have for $|\xi-\eta|\leq \f12|\eta|$
\begin{align*}
&\f{t(t^2+|\xi|^2)}{2+\xi^2}\f{2+k^2+\eta^2}{\langle k\rangle t(t+\langle k\rangle+|\eta|)}
\lesssim\f{(t+|\xi|)(t+\langle k\rangle+|\xi|)}{\langle k\rangle (t+\langle k\rangle+|\eta|)}\f{2+k^2+\eta^2}{2+\xi^2}\\
&\lesssim \langle k,\xi-\eta\rangle^3\Big(|\eta-kt|+|k|\Big)\mathbf{1}_{|\eta-kt|\geq \f12|kt|}+\f{|\eta|}{|k|}\mathbf{1}_{|\eta-kt|\leq \f12|kt|, t\in {\rm{I}}_{k,\eta}}+
|\eta-kt|\mathbf{1}_{|\eta-kt|\leq \f12|kt|, t\notin {\rm{I}}_{k,\eta}}
\end{align*}
and by \eqref{EQ 3.32}, we get
\begin{align}\label{eq: A_0/A_k}
\f{\rmA_0(\xi)}{\rmA_k(\eta)}\lesssim \left(\f{|k|(|k|+|\eta-kt|)}{|\xi|}\mathbf{1}_{t\in {\rm{I}}_{k,\eta}}+1\right)e^{c\la|\eta-\xi|^s}e^{-c\la|k|^s}.
\end{align}
The case $|\xi-\eta|\geq \f12|\eta|$ is simple, in which case, we switch $\rmA_0(\xi)$ to $\rmA_{-k}^{*}(\xi-\eta)$. Thus we have 
\begin{align*}
&\left\|\mathcal{M}_{00}\langle\pa_v\rangle\Big((\udl{\varphi_4}\Upsilon_2)P_0\big[P_{\neq}(a)P_{\neq}(\pa_v-t\pa_z)P_{\neq}(\Pi\Upsilon_1)\big]\Big)\right\|_2\\
&\lesssim \|\rmA^{\rmR}(\udl{\varphi_4}\Upsilon_2)\|_{L^2}\|\rmA^*a\|_{L^2}\langle t\rangle^4\|P_{\neq}(\Pi\Upsilon_1)\|_{\mathcal{G}^{s,\s-6}}\\
&\quad+\ep \left\|{\langle \pa_z\rangle t(1+\langle \pa_z\rangle+|\pa_v|)}(1-\Delta)^{-1}\big((\pa_v-t\pa_z)^2+\pa_z^2\big)\rmA P_{\neq}(\Pi\Upsilon_1)\right\|_{L^2}\lesssim \ep^2
\end{align*}
and by Lemma \ref{Lem 3.4}
\begin{align*}
&\left\|\mathcal{M}_{01}\Big((\udl{\varphi_4}\Upsilon_2)P_0\big[P_{\neq}(a)P_{\neq}(\pa_v-t\pa_z)P_{\neq}(\Pi\Upsilon_1)\big]\Big)\right\|_2\\
&\lesssim \|(\udl{\varphi_4}\Upsilon_2)\|_{\mathcal{G}^s}\mathcal{CK}_a^{\f12}\langle t\rangle^4\|P_{\neq}(\Pi\Upsilon_1)\|_{\mathcal{G}^{s,\s-6}}\\
&\quad \ep \left\|\Big(\sqrt{\f{\pa_tw}{w}}+\f{\langle\na\rangle^{\f{s}{2}}}{\langle t\rangle^s}\Big){\langle \pa_z\rangle t(1+\langle \pa_z\rangle+|\pa_v|)}(1-\Delta)^{-1}\big((\pa_v-t\pa_z)^2+\pa_z^2\big)\rmA P_{\neq}(\Pi\Upsilon_1)\right\|_{L^2}\\
&\quad +\ep\Big(\f{1}{\langle t\rangle^s}+\mathcal{CK}_{\varphi^{\d}_4}^{\f12}\Big)\langle t\rangle^4\|P_{\neq}(\Pi\Upsilon_1)\|_{\mathcal{G}^{s,\s-6}}\\
&\lesssim \f{\ep^2}{\langle t\rangle^s}+\ep \big(\mathcal{CK}_a^{\f12}+\mathcal{CK}_{\varphi^{\d}_4}^{\f12}\big)\\
&\quad+\ep \left\|\Big(\sqrt{\f{\pa_tw}{w}}+\f{\langle\na\rangle^{\f{s}{2}}}{\langle t\rangle^s}\Big){\langle \pa_z\rangle t(1+\langle \pa_z\rangle+|\pa_v|)}(1-\Delta)^{-1}\big((\pa_v-t\pa_z)^2+\pa_z^2\big)\rmA P_{\neq}(\Pi\Upsilon_1)\right\|_{L^2}. 
\end{align*}
Thus, we prove the proposition. 
\end{proof}
We now prove Proposition \ref{prop:pressure-zero mode}. 
\begin{proof}
By Proposition \ref{prop: 0pressure-error}, we only need to control the last term in \eqref{eq: P_0 pressure-eq}. Indeed, we only need to consider 
\begin{align*}
I_j=\mathcal{M}_{0j}\Big[(\udl{\varphi_4}\Upsilon_2)P_0\Big(\pa_{zz}(\Psi\Upsilon_2)(\pa_v-t\pa_z)(\Psi\Upsilon_2)\Big)\Big)\Big]
\end{align*}
due to the projection $P_0$. 
By using the same idea in the proof of Proposition \ref{prop:pressure-zero mode}, we have for $|\xi-\eta|\leq \f{1}{2}|\eta|$
\begin{align*}
&\f{t(t^2+|\xi|^2)}{2+\xi^2}\left\langle \f{\eta}{kt}\right\rangle\approx \f{t(t^2+|\xi|^2)}{2+\xi^2}\f{|\eta|+|kt|}{|kt|}\\
&\lesssim 
\Big(|\eta-kt|+|k|\Big)\mathbf{1}_{|\eta-kt|\geq \f12|kt|}+\f{|\eta|}{|k|}\mathbf{1}_{|\eta-kt|\leq \f12|kt|, t\in {\rm{I}}_{k,\eta}}+
|\eta-kt|\mathbf{1}_{|\eta-kt|\leq \f12|kt|, t\notin {\rm{I}}_{k,\eta}}.
\end{align*}
Then by \eqref{eq: A_0/A_k}, 
\begin{align*}
\|I_0\|_{L^2}
&\lesssim \|\rmA^{\rmR}(\udl{\varphi_4}\Upsilon_2)\|_{L^2}\langle t\rangle^4\|P_{\neq}(\Psi\Upsilon_2)\|_{\mathcal{G}^{s,\s-6}}^2\\
&\quad+\|(\udl{\varphi_4}\Upsilon_2)\|_{\mathcal{G}^{s,\s-6}}
\langle t\rangle^2\|P_{\neq}(\Psi\Upsilon_2)\|_{\mathcal{G}^{s,\s-6}} \left\|\left\langle\f{\pa_v}{t\pa_z}\right\rangle^{-1}(\pa_z^2+(\pa_v-t\pa_z)^2)\rmA P_{\neq}(\Psi\Upsilon_2)\right\|_2\\
&\lesssim \ep^2,
\end{align*}
and by Lemma \ref{Lem 3.4}
\begin{align*}
\|I_1\|_{L^2}
&\lesssim \Big(\f{1}{\langle t\rangle^s}+\mathcal{CK}_a^{\f12}+\mathcal{CK}_{\varphi^{\d}_4}^{\f12}\Big)\langle t\rangle^4\|P_{\neq}(\Psi\Upsilon_2)\|_{\mathcal{G}^{s,\s-6}}^2\\
&\quad+\|(\udl{\varphi_4}\Upsilon_2)\|_{\mathcal{G}^{s,\s-6}}
\langle t\rangle^2\|P_{\neq}(\Psi\Upsilon_2)\|_{\mathcal{G}^{s,\s-6}} \\
&\quad\quad\times \left\|\left\langle\f{\pa_v}{t\pa_z}\right\rangle^{-1}(\pa_z^2+(\pa_v-t\pa_z)^2)\left(\f{|\na|^{\f{s}{2}}}{\langle t\rangle^s}\rmA+\sqrt{\f{\pa_t w}{w}}\tilde{\rmA}\right)P_{\neq}(\Psi\Upsilon_2)\right\|_2\\
&\lesssim \f{\ep^2}{\langle t\rangle^s}+\ep^2\mathcal{CK}_a^{\f12}+\ep^2\mathcal{CK}_{\varphi^{\d}_4}^{\f12}\\
&\quad+\left\|\left\langle\f{\pa_v}{t\pa_z}\right\rangle^{-1}(\pa_z^2+(\pa_v-t\pa_z)^2)\left(\f{|\na|^{\f{s}{2}}}{\langle t\rangle^s}\rmA+\sqrt{\f{\pa_t w}{w}}\tilde{\rmA}\right)P_{\neq}(\Psi\Upsilon_2)\right\|_2,
\end{align*}
which gives Proposition \ref{prop:pressure-zero mode}. 
\end{proof}

\section{Elliptic estimate for the pressure}\label{sec: ell-pressure}
In this section, we study the following elliptic equation
\beq\label{eq: pressure}
\left\{\begin{aligned}
&\pa_z\Big((a+\udl{\th})\pa_z\Pi\Big)+\udl{\pa_yv}(\pa_v-t\pa_z)\Big((a+\udl{\th})\udl{\pa_yv}(\pa_v-t\pa_z)\Pi\Big)=\Xi,\\
&(\pa_v-t\pa_z)\Pi(t, z, u(0))=(\pa_v-t\pa_z)\Pi(t, z, u(1))=0,
\end{aligned}
\right.
\eeq
where $\Pi\in \{\Pi_{l,1}, \Pi_{l,2}, \Pi_{n,1},\Pi_{n,2}\}$ and $\Xi$ represents the right-hand side of equations \eqref{eq: Pressure} respectively. 

We rewrite the equation by separating its linear part and the nonlinear part. 
\begin{align*}
\Xi=\mathcal{T}_2[\Pi]
&+\pa_z\Big((a+\th^{\d})\pa_z\Pi\Big)
+(\udl{\pa_yv}-\widetilde{u'})(\pa_v-t\pa_z)\Big((a+\udl{\th})\udl{\pa_yv}(\pa_v-t\pa_z)\Pi\Big)\\
&+\widetilde{u'}(\pa_v-t\pa_z)\Big((a+\th^{\d})\udl{\pa_yv}(\pa_v-t\pa_z)\Pi\Big)\\
&+\widetilde{u'}(\pa_v-t\pa_z)\Big(\widetilde{\th}(\udl{\pa_yv}-\widetilde{u'})(\pa_v-t\pa_z)\Pi\Big),
\end{align*}
where 
\beq\label{eq: T_2-def}
\mathcal{T}_2[\Pi]=\widetilde{\th}\pa_{zz}\Pi+\widetilde{u'}(\pa_v-t\pa_z)\Big(\widetilde{\th}\widetilde{u'}(\pa_v-t\pa_z)\Pi\Big)
\eeq
Similar to $\mathcal{T}_1$, we define $\mathcal{T}_{N}^{-1}$ to be the inverse operator of $\mathcal{T}_2$ with Neumann boundary conditions, namely for any function $f(t, z, v)$, $\mathcal{T}_2\big[\mathcal{T}_{N}^{-1}[f]\big]=f$ with 
\beno
\pa_v\mathcal{T}_{N}^{-1}[f](t, z, 0)=\pa_v\mathcal{T}_{N}^{-1}[f](t, z, 1)
\eeno
We also define 
\ben
\mathring{\mathcal{T}}_{j, N}^{-1}=\Upsilon_j\mathcal{T}_{N}^{-1}, \quad j=1,2.
\een
\begin{proposition}\label{eq: T_2}
Let $\Xi$ be compactly supported. Then for $k\neq 0$ there is $\mathcal{G}_{N, j}(t, k, \xi, \eta)$ such that for $j=1,2,$
\beno
\widehat{(\mathring{\mathcal{T}}_{j,N}^{-1}[\Xi])}(t, k, \xi)=\int_{\mathbb{R}}\mathcal{G}_{N, j}(t, k, \xi, \eta)\widehat{\Xi}(t, k, \eta)d\eta,
\eeno
with estimate
\beno
|\mathcal{G}_{N, j}(t, k, \xi, \eta)|\leq C\min\left\{\f{e^{-\la_{\Delta}\langle \xi-\eta\rangle^s}}{1+k^2+(\xi-kt)^2}, \f{e^{-\la_{\Delta}\langle \xi-\eta\rangle^s}}{1+k^2+(\eta-kt)^2}\right\}.
\eeno
\end{proposition}
The proposition is proved in Appendix \ref{sec: SL-N}. 

Next, we rewrite the pressure equation by separating interactions with zero-mode and interactions with $a$. 
\beq\label{eq: elliptic pressure-1}
\begin{aligned}
\Xi=\mathcal{T}_2[\Pi]&+\mathcal{S}_2[\Pi]+\mathcal{S}_a[\Pi],
\end{aligned}
\eeq
where the `pressure error' terms are
\begin{align*}
\mathcal{S}_2[\Pi]&=\th^{\d}\pa_{zz}(\Upsilon_2\Pi)
+g_1(\pa_v-t\pa_z)^2(\Upsilon_2\Pi)+g_2(\pa_v-t\pa_z)(\Upsilon_2\Pi),\\
\mathcal{S}_a[\Pi]&=\pa_z\big(a\pa_z(\Upsilon_2\Pi)\big)+g_3(\pa_v-t\pa_z)\big(a(\pa_v-t\pa_z)(\Upsilon_2\Pi)\big)
+g_4a(\pa_v-t\pa_z)(\Upsilon_2\Pi).
\end{align*}
with 
\begin{align*}
g_1&=(\udl{\pa_yv}-\widetilde{u'})\udl{\th}\udl{\pa_yv}+\tilde{u'}\th^{\d}\udl{\pa_yv}+\widetilde{u'}\tilde{\th}(\udl{\pa_yv}-\widetilde{u'})\\
g_2&=(\udl{\pa_yv}-\widetilde{u'})\pa_v(\udl{\th}\udl{\pa_yv})+\tilde{u'}\pa_v(\th^{\d}\udl{\pa_yv})+\widetilde{u'}\pa_v(\tilde{\th}(\udl{\pa_yv}-\widetilde{u'}))\\
g_3&=(\udl{\pa_yv}-\widetilde{u'})\udl{\pa_yv}+\tilde{u'}\udl{\pa_yv}+\widetilde{u'}(\udl{\pa_yv}-\widetilde{u'})\\
g_4&=(\udl{\pa_yv}-\widetilde{u'})\pa_v(\udl{\pa_yv})+\tilde{u'}\pa_v(\udl{\pa_yv})+\widetilde{u'}\pa_v((\udl{\pa_yv}-\widetilde{u'})). 
\end{align*}
By acting $\mathcal{T}_{N}^{-1}$ on \eqref{eq: elliptic pressure-1} and multiplying the cut-off function, we get that
\begin{align}\label{eq: pre-eq-error}
\mathring{\mathcal{T}}_{j, N}^{-1}[P_{\neq}\Xi]-\Upsilon_jP_{\neq}\Pi
=\mathring{\mathcal{T}}_{j, N}^{-1}\big[P_{\neq}\mathcal{S}_2[\Pi]+P_{\neq}\mathcal{S}_a[\Pi]\big].
\end{align}

We introduce three multipliers, which are constructed in \cite{ChenWeiZhangZhang2023}: for $k\in \mathbb{Z}\setminus\{0\},\ \xi\in\mathbb{R}$, let
\begin{align*}
\mathcal{M}_3(t,k,\xi)&=\f{\langle k\rangle t(t+\langle k\rangle+|\xi|)}{1+k^2+\xi^2}\rmA_k(\xi),\\
\mathcal{M}_4(t,k,\xi)&=\f{\langle k\rangle t(t+\langle k\rangle+|\xi|)}{1+k^2+\xi^2}\f{|k,\xi|^{\f{s}{2}}}{\langle t\rangle^s}\rmA_k(\xi),\\
\mathcal{M}_5(t,k,\xi)&=\f{\langle k\rangle t(t+\langle k\rangle+|\xi|)}{1+k^2+\xi^2}\sqrt{\f{\pa_tw}{w}}\tilde{\rmA}_k(\xi).
\end{align*}
It is easy to check that 
\ben\label{eq: inequality-11}
\f{\langle k\rangle t(t+\langle k\rangle+|\xi|)}{1+k^2+\xi^2}=\f{|kt|}{|\xi|+|kt|}\f{(|\xi|+|kt|)(t+\langle k\rangle+|\xi|)}{k^2+\xi^2+1}\geq \f{|kt|}{|\xi|+|kt|}\approx \left\langle\f{\xi}{kt}\right\rangle^{-1},
\een
which implies $\mathcal{M}_4(t, k, \xi)\geq \mathcal{M}_1(t, k, \xi)$ and $\mathcal{M}_5(t, k, \xi)\geq \mathcal{M}_2(t, k, \xi)$. 

We first control the two error terms $\mathring{\mathcal{T}}_{j, N}^{-1}\big[P_{\neq}\mathcal{S}_2[\Pi]\big]$ and $\mathring{\mathcal{T}}_{j, N}^{-1}\big[P_{\neq}\mathcal{S}_a[\Pi]\big]$
\begin{lemma}\label{lem: S_2}
Under the bootstrap hypotheses,
\begin{align*}
\left\|\mathcal{M}_3(\pa_z^2+(\pa_v-t\pa_z)^2)\mathring{\mathcal{T}}_{j, N}^{-1}P_{\neq}\big[\mathcal{S}_2[\Pi]\big]\right\|_2
&\lesssim \ep\left\|\mathcal{M}_3(\pa_z^2+(\pa_v-t\pa_z)^2)P_{\neq}(\Upsilon_2\Pi)\right\|_2\\
\left\|(\mathcal{M}_4+\mathcal{M}_5)(\pa_z^2+(\pa_v-t\pa_z)^2)\mathring{\mathcal{T}}_{j, N}^{-1}P_{\neq}\big[\mathcal{S}_2[\Pi]\big]\right\|_2
&\lesssim \ep\big(\mathcal{CK}_{\th^{\d}}^{\f12}+\mathcal{CK}_{h}^{\f12}+\mathcal{CK}_{\varphi_9^{\d}}^{\f12}\big)\\
&\quad +\ep\|(\mathcal{M}_4+\mathcal{M}_5)(\pa_z^2+(\pa_v-t\pa_z)^2)P_{\neq}(\Upsilon_2\Pi)\|_{2}. 
\end{align*}
\end{lemma}
\begin{lemma}\label{lem: S_a}
Under the bootstrap hypotheses,
\begin{align*}
\left\|\mathcal{M}_3(\pa_z^2+(\pa_v-t\pa_z)^2)\mathring{\mathcal{T}}_{j, N}^{-1}P_{\neq}\big[\mathcal{S}_a[\Pi]\big]\right\|_2
&\lesssim \ep\left\|\mathcal{M}_3(\pa_z^2+(\pa_v-t\pa_z)^2)P_{\neq}(\Upsilon_2\Pi)\right\|_2+\ep^2,\\
\left\|(\mathcal{M}_4+\mathcal{M}_5)(\pa_z^2+(\pa_v-t\pa_z)^2)\mathring{\mathcal{T}}_{j, N}^{-1}P_{\neq}\big[\mathcal{S}_a[\Pi]\big]\right\|_2
&\lesssim \ep\big(\mathcal{CK}_{\th^{\d}}^{\f12}+\mathcal{CK}_{h}^{\f12}+\mathcal{CK}_{\varphi_9^{\d}}^{\f12}+\mathcal{CK}_a^{\f12}+\f{\ep}{\langle t\rangle^s}\big)\\
&\quad +\ep\|(\mathcal{M}_4+\mathcal{M}_5)(\pa_z^2+(\pa_v-t\pa_z)^2)P_{\neq}(\Upsilon_2\Pi)\|_{2}\\
&\quad +\ep \|\mathcal{M}_{01}\pa_vP_0(\Upsilon_2\Pi)\|_{L^2}.
\end{align*}
\end{lemma}
By applying Proposition \ref{eq: T_2}, we have 
\begin{align*}
&\left\|\mathcal{M}_3(\pa_z^2+(\pa_v-t\pa_z)^2)\mathring{\mathcal{T}}_{j, N}^{-1}P_{\neq}\big[\mathcal{S}_2[\Pi]+\mathcal{S}_a[\Pi]\big]\right\|_2\\
&\lesssim \|\mathcal{M}_3P_{\neq}\mathcal{S}_2[\Pi]\|_2+\|\mathcal{M}_3P_{\neq}\mathcal{S}_a[\Pi]\|_2,\\
&\left\|(\mathcal{M}_4+\mathcal{M}_5)(\pa_z^2+(\pa_v-t\pa_z)^2)\mathring{\mathcal{T}}_{j, N}^{-1}P_{\neq}\big[\mathcal{S}_2[\Pi]+\mathcal{S}_a[\Pi]\big]\right\|_2\\
&\lesssim \|(M_4+M_5)P_{\neq}\mathcal{S}_2[\Pi]\|_2+\|(M_4+M_5)P_{\neq}\mathcal{S}_a[\Pi]\|_2. 
\end{align*}
The idea of the proof of these two lemmas is the same as the proofs of {\it Lemma 3.2, Lemma 3.3, and Lemma 3.4} of \cite{ChenWeiZhangZhang2023}. 
\begin{proof}
\no{\bf Estimate for $\mathcal{S}_2[\Pi]$: } 
By fact that
\begin{align*}
&\f{\langle k\rangle t(t+\langle k\rangle+|\xi|)}{1+k^2+\xi^2}\lesssim_{c} \f{\langle k\rangle t(t+\langle k\rangle+|\eta|)}{1+k^2+\eta^2} e^{c\min\{\langle\xi-\eta\rangle^{\f12},\langle k, \eta\rangle^{\f12}\}},\\
&\f{|k,\xi|^{\f{s}{2}}}{\langle t\rangle^s}\lesssim \f{|k,\eta|^{\f{s}{2}}}{\langle t\rangle^s}+\f{|\eta-\xi|^{\f{s}{2}}}{\langle t\rangle^s},
\end{align*}
and Remark \ref{Rmk: A-switch}, we get
\begin{align*}
\mathcal{M}_3(t, k, \xi)&\lesssim \rmA^{\rmR}(\xi-\eta)\f{\langle k\rangle t(t+\langle k\rangle+|\eta|)}{1+k^2+\eta^2}e^{c\la|k,\eta|^s}+ e^{c\la|\xi-\eta|^s}\mathcal{M}_3(t, k, \eta)\\
\mathcal{M}_4(t, k, \xi)&\lesssim \f{|\eta-\xi|^{\f{s}{2}}}{\langle t\rangle^s}\rmA^{\rmR}(\xi-\eta)\f{\langle k\rangle t(t+\langle k\rangle+|\eta|)}{1+k^2+\eta^2}e^{c\la|k,\eta|^s}+ e^{c\la|\xi-\eta|^s}\mathcal{M}_4(t, k, \eta).
\end{align*}
By Lemma \ref{Lem 3.4}, we have
\begin{align*}
\mathcal{M}_5(t, k, \xi)
&\lesssim \Big(\f{|\eta-\xi|^{\f{s}{2}}}{\langle t\rangle^s}+\sqrt{\f{\pa_tw}{w}}\Big)\rmA^{\rmR}(\xi-\eta)\mathcal{M}_3(t, k, \eta)e^{c\la|k,\eta|^s}\\
&\quad+ e^{c\la|\xi-\eta|^s}\Big(\mathcal{M}_4(t, k, \eta)+\mathcal{M}_5(t, k, \eta)\Big).
\end{align*}
Thus we have
\begin{align*}
&\left\|\mathcal{M}_3P_{\neq}[\th^{\d}\pa_{zz}(\Upsilon_2\Pi)+g_1(\pa_v-t\pa_z)^2(\Upsilon_2\Pi)]\right\|_{L^2}
\lesssim \ep\left\|\mathcal{M}_3(\pa_z^2+(\pa_v-t\pa_z)^2)P_{\neq}(\Upsilon_2\Pi)\right\|_2, \\
&\left\|\big(\mathcal{M}_4+\mathcal{M}_5\big)P_{\neq}[\th^{\d}\pa_{zz}(\Upsilon_2\Pi)+g_1(\pa_v-t\pa_z)^2(\Upsilon_2\Pi)]\right\|_{L^2}\\
&\lesssim \big(\mathcal{CK}_{\th^{\d}}^{\f12}+\mathcal{CK}_{h}^{\f12}+\mathcal{CK}_{\varphi_9^{\d}}^{\f12}\big)\left\|\mathcal{M}_3(\pa_z^2+(\pa_v-t\pa_z)^2)P_{\neq}(\Upsilon_2\Pi)\right\|_2\\
&\quad +\ep\|(\mathcal{M}_4+\mathcal{M}_5)(\pa_z^2+(\pa_v-t\pa_z)^2)P_{\neq}(\Upsilon_2\Pi)\|_{2}.
\end{align*}
For the $g_2(\pa_v-t\pa_z)\Pi$ term, since there is one derivative loss in $g_2$, we need to use the decay of $\Pi$ and the multiplier $\f{|k|t(t+|k,\eta|)}{k^2+\eta^2}$ to gain the regularity, which plays a similar role to the multiplier $\left\langle\f{\eta}{kt}\right\rangle^{-1}$ in $\mathcal{M}_1$ and $\mathcal{M}_2$. 
We have for $|\xi-\eta|\geq 2|k,\eta|$ (the case $g_2$ is in higher frequencies)
\beq\label{eq: g_2-g_3}
\begin{aligned}
&\f{\langle k\rangle t(t+\langle k\rangle+|\xi|)}{1+k^2+\xi^2}\f{1+k^2+\eta^2}{\langle k\rangle t(t+\langle k\rangle+|\eta|)}
\lesssim \f{\langle k,\eta\rangle^2}{1+k^2+\xi^2}\Big(1+\f{|\eta-\xi|}{(t+\langle k\rangle+|\eta|)}\Big)\lesssim \f{\langle k,\eta\rangle^2}{\langle \eta-\xi\rangle},
\end{aligned}
\eeq
and for $|\xi-\eta|\leq 2|k,\eta|$
\begin{align*}
\f{\langle k\rangle t(t+\langle k\rangle+|\xi|)}{1+k^2+\xi^2}\f{1+k^2+\eta^2}{\langle k\rangle t(t+\langle k\rangle+|\eta|)}
\lesssim \Big(1+\f{|\xi-\eta||\xi+\eta|}{1+k^2+\xi^2}\Big)\Big(1+\f{|\xi-\eta|}{t+\langle k\rangle+|\eta|}\Big)\lesssim \langle\xi-\eta\rangle^3.
\end{align*}
Then we have by Remark \ref{Rmk: A-switch} and Lemma \ref{Lem 3.4}, 
\begin{align*}
&\left\|\mathcal{M}_3[g_2(\pa_v-t\pa_z)P_{\neq}(\Upsilon_2\Pi)]\right\|_{L^2}
\lesssim \ep\left\|\mathcal{M}_3(\pa_z^2+(\pa_v-t\pa_z)^2)P_{\neq}(\Upsilon_2\Pi)\right\|_2, \\
&\left\|\big(\mathcal{M}_4+\mathcal{M}_5\big)P_{\neq}[g_2(\pa_v-t\pa_z)(\Upsilon_2\Pi)]\right\|_{L^2}\\
&\lesssim \big(\mathcal{CK}_{\th^{\d}}^{\f12}+\mathcal{CK}_{h}^{\f12}+\mathcal{CK}_{\varphi_9^{\d}}^{\f12}\big)\left\|\mathcal{M}_3(\pa_z^2+(\pa_v-t\pa_z)^2)P_{\neq}(\Upsilon_2\Pi)\right\|_2\\
&\quad +\ep\|(\mathcal{M}_4+\mathcal{M}_5)(\pa_z^2+(\pa_v-t\pa_z)^2)P_{\neq}(\Upsilon_2\Pi)\|_{2}.
\end{align*}
This gives the lemma. 
\end{proof}

\begin{proof}
{\bf Estimate for $\mathcal{S}_a[\Pi]$:} 
For the estimate of $\mathcal{S}_a[\Pi]$, we take advantage of the divergence form. Note that switching from $\rmA_k(\xi)$ to $\rmA_l(\eta)$, we will lose $\f{\xi}{k^2+(\xi-kt)^2}$ in the case $t\in {\rm{I}}_{k,\xi}\cap {\rm{I}}_{k,\eta},\, k\neq l$. By taking out one derivative $(\pa_z, \pa_v-t\pa_z)$, in the worst case $t\in {\rm{I}}_{k,\xi}\cap {\rm{I}}_{k,\eta},\, k\neq l$, we multiply a small fact $\f{\xi-kt}{\eta-lt}$. 

We first consider the first two terms in $\mathcal{S}_a[\Pi]$ when $\Pi$ is non-zero modes, namely, $\pa_z(a\pa_{z}P_{\neq}(\Upsilon_2\Pi))$ and $g_3(\pa_v-t\pa_z)(a(\pa_v-t\pa_z)P_{\neq}(\Upsilon_2\Pi))$. Since $g_3$ is the zero-mode, using the idea of the proof of $\mathcal{S}_2[\Pi]$, we have 
\begin{align*}
&\left\|\mathcal{M}_3P_{\neq}\big[g_3(\pa_v-t\pa_z)(a(\pa_v-t\pa_z)(\Upsilon_2\Pi))\big]\right\|_{L^2}
\lesssim \|\mathcal{M}_3P_{\neq}[(\pa_v-t\pa_z)(a(\pa_v-t\pa_z)P_{\neq}(\Upsilon_2\Pi))]\|_{L^2}\\
&\left\|\big(\mathcal{M}_4+\mathcal{M}_5\big)P_{\neq}[g_3(\pa_v-t\pa_z)(a(\pa_v-t\pa_z)P_{\neq}(\Upsilon_2\Pi))]\right\|_{L^2}\\
&\lesssim \left\|\big(\mathcal{M}_4+\mathcal{M}_5\big)P_{\neq}[(\pa_v-t\pa_z)(a(\pa_v-t\pa_z)P_{\neq}(\Upsilon_2\Pi))]\right\|_{L^2}\\
&\quad+\Big(\f{1}{\langle t\rangle^s}+\mathcal{CK}_{h}^{\f12}+\mathcal{CK}_{\varphi_9^{\d}}^{\f12}\Big)\|\mathcal{M}_3P_{\neq}[(\pa_v-t\pa_z)(a(\pa_v-t\pa_z)P_{\neq}(\Upsilon_2\Pi))]\|_{L^2}
\end{align*}
Then we treat the last term in $\mathcal{S}_a[\Pi]$, namely, $g_4a(\pa_v-t\pa_z)(\Upsilon_2\Pi)$. If $g_4$ is in lower frequencies, the treatment is easier than $g_3(\pa_v-t\pa_z)(a(\pa_v-t\pa_z)(\Upsilon_2\Pi))$. The difficulty happens when $g_4$ is in higher frequencies. In such a case, we follow the idea of the treatment of $g_2(\pa_v-t\pa_z)P_{\neq}(\Upsilon_2\Pi)$ term, then by \eqref{eq: g_2-g_3} we have
\begin{align*}
&\left\|\mathcal{M}_3\Big[g_4P_{\neq}\Big(a(\pa_v-t\pa_z)(\Upsilon_2\Pi)\Big)\Big]\right\|_{L^2}\\
&\lesssim \left\|\mathcal{M}_3(|\pa_z|+|\pa_v-t\pa_z|)P_{\neq}\Big(a(\pa_v-t\pa_z)P_{\neq}(\Upsilon_2\Pi)\Big)\Big]\right\|_{L^2}\\
&\quad+\left\|\mathcal{M}_3(|\pa_z|+|\pa_v-t\pa_z|)P_{\neq}\Big(a\pa_vP_{0}(\Upsilon_2\Pi)\Big)\Big]\right\|_{L^2}\\
&\left\|\big(\mathcal{M}_4+\mathcal{M}_5\big)P_{\neq}\Big[g_4\Big(a(\pa_v-t\pa_z)(\Upsilon_2\Pi)\Big)\Big]\right\|_{L^2}\\
&\lesssim \Big(\f{1}{\langle t\rangle^s}+\mathcal{CK}_{h}^{\f12}+\mathcal{CK}_{\varphi_9^{\d}}^{\f12}\Big)\left\|\mathcal{M}_3(|\pa_z|+|\pa_v-t\pa_z|)P_{\neq}\Big(a(\pa_v-t\pa_z)P_{\neq}(\Upsilon_2\Pi)\Big)\Big]\right\|_{L^2}\\
&\quad+\Big(\f{1}{\langle t\rangle^s}+\mathcal{CK}_{h}^{\f12}+\mathcal{CK}_{\varphi_9^{\d}}^{\f12}\Big)\left\|\mathcal{M}_3(|\pa_z|+|\pa_v-t\pa_z|)P_{\neq}\Big(a\pa_vP_{0}(\Upsilon_2\Pi)\Big)\Big]\right\|_{L^2}\\
&\quad +\ep\|(\mathcal{M}_4+\mathcal{M}_5)(|\pa_z|+|\pa_v-t\pa_z|)P_{\neq}(a(\pa_v-t\pa_z)(\Upsilon_2\Pi))\|_{2}.
\end{align*}
We have by Lemma \ref{Lem 3.6} and Remark \ref{Rmk: A-switch} that for $|k-l,\eta-\xi|\leq |l,\eta|$, $k\neq 0, l\neq 0$
\begin{align*}
&|k|\mathcal{M}_3(t,k,\xi)\left\langle t-\f{\xi}{k}\right\rangle\\
&\lesssim \mathcal{M}_3(t,l,\eta)(|l|+|\eta-lt|)\f{|k|+|\xi-kt|}{|l|+|\eta-lt|}\f{\langle k\rangle (t+\langle k\rangle+|\xi|)}{\langle l\rangle (t+\langle l\rangle+|\eta|)}\f{1+l^2+\eta^2}{1+k^2+\xi^2}e^{c\la |k-l,\eta-\xi|^s}\\
&\quad +\rmA_l(\eta)\f{\langle k\rangle t(t+\langle k\rangle+|\xi|)}{1+k^2+\xi^2}\f{|\xi|(|k|+|\xi-kt|)}{k^2+(\xi-kt)^2}\mathbf{1}_{t\in {\rm{I}}_{k,\xi}\cap {\rm{I}}_{k,\eta},\, k\neq l}\\
&\lesssim \mathcal{M}_3(t,l,\eta)(|l|+|\eta-lt|)e^{c\la |k-l,\eta-\xi|}\langle k-l,\eta-\xi\rangle^5+\rmA_l(\eta)\f{{\xi}/{k}}{{\eta}/{l}}|l|\left\langle t-\f{\eta}{l}\right\rangle e^{c\la |k-l,\eta-\xi|}\\
&\lesssim e^{c\la |k-l,\eta-\xi|}\mathcal{M}_3(t,l,\eta)(|l|+|\eta-lt|)
\end{align*}
For $|k-l,\eta-\xi|\geq |l,\eta|$, $k\neq 0, l\neq 0$, $k\neq l$, we have by Remark \ref{Rmk: A-switch}
\begin{align*}
&|k|\mathcal{M}_3(t,k,\xi)\left\langle t-\f{\xi}{k}\right\rangle\\
&\lesssim \f{\langle k\rangle t(t+\langle k\rangle+|\xi|)}{1+k^2+\xi^2}(|k|+|\xi-kt|)\rmA_{k-l}(\xi-\eta)e^{c\la|l,\eta|^{s}}\\
&\quad+\f{\langle k\rangle t(t+\langle k\rangle+|\xi|)}{1+k^2+\xi^2}\rmA_{k-l}(\xi-\eta)\f{|\xi|(|k|+|\xi-kt|)}{k^2+|\xi-kt|^2}e^{c\la|l,\eta|^{s}}\mathbf{1}_{t\in {\rm{I}}_{k,\xi}\cap {\rm{I}}_{k,\xi-\eta}, k\neq k-l}\\
&\lesssim \rmA_{k-l}^{*}(\xi-\eta)\langle t\rangle^3 e^{c\la|l,\eta|^{s}}.
\end{align*}
Similarly, by Lemma \ref{Lem 3.4}, we have
\begin{align*}
&|k|(\mathcal{M}_{4}(t,k,\xi)+\mathcal{M}_{5}(t,k,\xi))\left\langle t-\f{\xi}{k}\right\rangle\\
&\lesssim e^{c\la |k-l,\eta-\xi|}(\mathcal{M}_4(t,l,\eta)+\mathcal{M}_5(t,l,\eta))(|l|+|\eta-lt|)\\
&\quad+\Big(\sqrt{\f{\pa_tw}{w}}+\f{|k-l,\xi-\eta|^{\f{s}{2}}}{\langle t\rangle^s}\Big)\rmA_{k-l}^{*}(\xi-\eta)\langle t\rangle^2 e^{c\la|l,\eta|^{s}}.
\end{align*}
Thus, we get 
\begin{align*}
& \left\|\mathcal{M}_3P_{\neq}[(\pa_v-t\pa_z)(a(\pa_v-t\pa_z)P_{\neq}((\Upsilon_2\Pi)))]\right\|_{L^2}\\
 &\lesssim \ep \left\|\mathcal{M}_3\big((\pa_v-t\pa_z)^2+\pa_z^2\big)P_{\neq}(\Upsilon_2\Pi)\right\|_2
 +\ep \langle t\rangle^4\|P_{\neq}(\Upsilon_2\Pi)\|_{\mathcal{G}^{s,\s-6}}\\
&\left\|(\mathcal{M}_4+\mathcal{M}_5)P_{\neq}[(\pa_v-t\pa_z)(a(\pa_v-t\pa_z)P_{\neq}(\Upsilon_2\Pi))]\right\|_{L^2}\\
&\lesssim \ep \left\|(\mathcal{M}_4+\mathcal{M}_5)\big((\pa_v-t\pa_z)^2+\pa_z^2\big)P_{\neq}(\Upsilon_2\Pi)\right\|_{L^2}
+\langle t\rangle^4\|P_{\neq}(\Upsilon_2\Pi)\|_{\mathcal{G}^{s,\s-6}}\mathcal{CK}_a^{\f12}. 
\end{align*}

Next, we consider the second term in $\mathcal{S}_a[\Pi]$ when $\Pi$ is the zero mode, namely, $g_3 (\pa_v-t\pa_z) (P_{\neq}(a)\pa_vP_{0}(\Upsilon_2\Pi))$. Again, we first get rid of $g_3$: 
\begin{align*}
&\left\|\mathcal{M}_3\big[g_3(\pa_v-t\pa_z)(P_{\neq}a\pa_vP_0(\Upsilon_2\Pi))\big]\right\|_{L^2}
\lesssim \|\mathcal{M}_3[(\pa_v-t\pa_z)(P_{\neq}a\pa_vP_0(\Upsilon_2\Pi))]\|_{L^2}\\
&\left\|\big(\mathcal{M}_4+\mathcal{M}_5\big)[g_3(\pa_v-t\pa_z)(P_{\neq}a\pa_vP_{0}(\Upsilon_2\Pi))]\right\|_{L^2}\\
&\lesssim \left\|\big(\mathcal{M}_4+\mathcal{M}_5\big)P_{\neq}[(\pa_v-t\pa_z)(a\pa_vP_0(\Upsilon_2\Pi))]\right\|_{L^2}\\
&\quad+\Big(\f{1}{\langle t\rangle^{s}}+\mathcal{CK}_{h}^{\f12}+\mathcal{CK}_{\varphi_9^{\d}}^{\f12}\Big)\|\mathcal{M}_3[(\pa_v-t\pa_z)(P_{\neq}a\pa_vP_0(\Upsilon_2\Pi))]\|_{L^2}.
\end{align*}
We have for $|\xi-\eta|\leq 2|k,\eta|$, (which is associated with the case $a$ is in higher frequencies,)
\begin{align*}
\f{\langle k\rangle t(t+\langle k\rangle+|\xi|)}{1+k^2+\xi^2}\f{|\xi-kt|}{\rmB_k(\eta)}
&\lesssim \f{\langle k\rangle t(t+\langle k\rangle+|\xi|)}{1+k^2+\xi^2}\f{t|\xi-kt|}{t+|\xi|^{\f12}+|k|}\langle \xi-\eta\rangle^{\f12}\\
&\lesssim \langle t\rangle^3\langle \xi-\eta\rangle^{\f12},
\end{align*}
and for $|\xi-\eta|\geq 2|k,\eta|$,
\begin{align*}
&\f{\langle k\rangle t(t+\langle k\rangle+|\xi|)}{1+k^2+\xi^2}{|\xi-kt|}\Big(\f{t(t^2+(\xi-\eta)^2)}{2+(\xi-\eta)^2}\Big)^{-1}
\lesssim \langle k,\eta\rangle^6\f{|\xi-kt|}{t+|\xi|}.
\end{align*}
Thus we obtain by Remark \ref{Rmk: A-switch} that
\begin{align*}
\|\mathcal{M}_3[(\pa_v-t\pa_z)(P_{\neq}a\pa_vP_0(\Upsilon_2\Pi))]\|_{L^2}
&\lesssim \|\rmA^*a\|_{L^2}\langle t\rangle^3\|\pa_vP_0(\Upsilon_2\Pi)\|_{\mathcal{G}^{s,\s-6}}\\
&\quad+\|a\|_{\mathcal{G}^{s,\s-6}}\|\mathcal{M}_{00}\pa_vP_0(\Upsilon_2\Pi)\|_{L^2}\lesssim \ep^2
\end{align*}
and by Lemma \ref{Lem 3.4}
\begin{align*}
& \left\|\big(\mathcal{M}_4+\mathcal{M}_5\big)P_{\neq}[(\pa_v-t\pa_z)(a\pa_vP_0(\Upsilon_2\Pi))]\right\|_{L^2}\\
& \lesssim \ep\mathcal{CK}_a^{\f12}+\ep \|\mathcal{M}_{01}\pa_vP_0(\Upsilon_2\Pi)\|_{L^2}. 
\end{align*}
Thus we proved the lemma. 
\end{proof}
Now we introduce two lemmas with weaker weight. 
\begin{lemma}\label{lem: S_2-1}
Under the bootstrap hypotheses,
\begin{align*}
&\left\|\mathcal{M}_3(-\pa_z^2-(\pa_v-t\pa_z)^2)^{\f12}\mathring{\mathcal{T}}_{j, N}^{-1}P_{\neq}\big[\mathcal{S}_2[\Pi]\big]\right\|_2\\
&\lesssim \ep\left\|\mathcal{M}_3(-\pa_z^2-(\pa_v-t\pa_z)^2)^{\f12}P_{\neq}(\Upsilon_2\Pi)\right\|_2\\
&\left\|(\mathcal{M}_4+\mathcal{M}_5)(-\pa_z^2-(\pa_v-t\pa_z)^2)^{\f12}\mathring{\mathcal{T}}_{j, N}^{-1}P_{\neq}\big[\mathcal{S}_2[\Pi]\big]\right\|_2\\
&\lesssim \ep\big(\mathcal{CK}_{\th^{\d}}^{\f12}+\mathcal{CK}_{h}^{\f12}+\mathcal{CK}_{\varphi_9^{\d}}^{\f12}\big)
+\ep\|(\mathcal{M}_4+\mathcal{M}_5)(-\pa_z^2-(\pa_v-t\pa_z)^2)^{\f12}P_{\neq}(\Upsilon_2\Pi)\|_{2}. 
\end{align*}
\end{lemma}
\begin{lemma}\label{lem: S_a-1}
Under the bootstrap hypotheses,
\begin{align*}
&\left\|\mathcal{M}_3(-\pa_z^2-(\pa_v-t\pa_z)^2)^{\f12}\mathring{\mathcal{T}}_{j, N}^{-1}P_{\neq}\big[\mathcal{S}_a[\Pi]\big]\right\|_2\\
&\lesssim \ep\left\|\mathcal{M}_3(\pa_z^2+(\pa_v-t\pa_z)^2)P_{\neq}(\Upsilon_1\Pi)\right\|_2+\ep^2,\\
&\left\|(\mathcal{M}_4+\mathcal{M}_5)(-\pa_z^2-(\pa_v-t\pa_z)^2)^{\f12}\mathring{\mathcal{T}}_{j, N}^{-1}P_{\neq}\big[\mathcal{S}_a[\Pi]\big]\right\|_2\\
&\lesssim \ep\big(\mathcal{CK}_{\th^{\d}}^{\f12}+\mathcal{CK}_{h}^{\f12}+\mathcal{CK}_{\varphi_9^{\d}}^{\f12}+\mathcal{CK}_a^{\f12}+\f{\ep}{\langle t\rangle^s}\big)\\
&\quad +\ep\left\|(\mathcal{M}_4+\mathcal{M}_5)(\pa_z^2+(\pa_v-t\pa_z)^2)P_{\neq}(\Upsilon_1\Pi)\right\|_{2}\\
&\quad +\ep \|\mathcal{M}_{01}\pa_vP_0(\Upsilon_2\Pi)\|_{L^2}.
\end{align*}
\end{lemma}
Noticing that, by applying Proposition \ref{eq: T_2}, we only need to control 
\beno
 \left\|\mathcal{M}_3(-\pa_z^2-(\pa_v-t\pa_z)^2)^{-\f12}P_{\neq}\Big(\mathcal{S}_2[\Pi]+\mathcal{S}_a[\Pi]\Big)\right\|_2
\eeno
and
\beno
 \left\|(\mathcal{M}_4+\mathcal{M}_5)(-\pa_z^2-(\pa_v-t\pa_z)^2)^{-\f12}P_{\neq}\Big(\mathcal{S}_2[\Pi]+\mathcal{S}_a[\Pi]\Big)\right\|_2.
\eeno
Notice that when the gaining $(k^2+(\xi-kt)^2)^{-\f12}$ is weak, namely, $|\xi|\approx |kt|$, the weight $\f{\langle k\rangle t(t+\langle k\rangle+|\xi|)}{1+k^2+\xi^2}\approx 1$. 
The above two lemmas can be proved by using the same idea in the proof of Lemma \ref{lem: S_2} and Lemma \ref{lem: S_a}, and following the proofs step by step. 
\begin{proof}
We have for $|\xi-\eta|\leq 2|k,\eta|$ and $k\neq 0$ that
\begin{align*}
\f{\langle k\rangle t(t+\langle k\rangle+|\xi|)}{1+k^2+\xi^2}\f{\rmA_k(\xi)}{(k^2+(\xi-kt)^2)^{\f12}}\lesssim \f{\langle k\rangle t(t+\langle k\rangle+|\eta|)}{1+k^2+\xi^2}\f{\langle \xi-\eta\rangle^4}{(k^2+(\eta-kt)^2)^{\f12}}\rmA_k(\eta)e^{c\la\langle \xi-\eta\rangle^s}
\end{align*}
and for $|\xi-\eta|\geq 2|k,\eta|$ and $k\neq 0$, we have \eqref{eq: g_2-g_3}. Then we have by Remark \ref{Rmk: A-switch} and Lemma \ref{Lem 3.4}, 
\begin{align*}
&\left\|\mathcal{M}_3(-\pa_z^2-(\pa_v-t\pa_z)^2)^{-\f12}[g_2(\pa_v-t\pa_z)P_{\neq}(\Upsilon_2\Pi)]\right\|_{L^2}\\
&\lesssim \ep\left\|\mathcal{M}_3(-\pa_z^2-(\pa_v-t\pa_z)^2)^{\f12}P_{\neq}(\Upsilon_2\Pi)\right\|_2, \\
&\left\|\big(\mathcal{M}_4+\mathcal{M}_5\big)(-\pa_z^2-(\pa_v-t\pa_z)^2)^{-\f12}P_{\neq}[g_2(\pa_v-t\pa_z)(\Upsilon_2\Pi)]\right\|_{L^2}\\
&\lesssim \big(\mathcal{CK}_{\th^{\d}}^{\f12}+\mathcal{CK}_{h}^{\f12}+\mathcal{CK}_{\varphi_9^{\d}}^{\f12}\big)\left\|\mathcal{M}_3(\pa_z^2+(\pa_v-t\pa_z)^2)P_{\neq}(\Upsilon_2\Pi)\right\|_2\\
&\quad +\ep\|(\mathcal{M}_4+\mathcal{M}_5)(-\pa_z^2-(\pa_v-t\pa_z)^2)^{\f12}P_{\neq}(\Upsilon_2\Pi)\|_{2}.
\end{align*}
This gives Lemma \ref{lem: S_2-1}. 
\end{proof}
\begin{proof}
By using the same idea of the proof of Lemma \ref{lem: S_a}, we only need to estimate 
\begin{align*}
\|\mathcal{M}_3P_{\neq}[(a(\pa_v-t\pa_z)P_{\neq}(\Upsilon_2\Pi))]\|_{L^2},\quad \|\mathcal{M}_3P_{\neq}[(a\pa_zP_{\neq}(\Upsilon_2\Pi))]\|_{L^2}
\end{align*}
and
\begin{align*}
\left\|\big(\mathcal{M}_4+\mathcal{M}_5\big)P_{\neq}[(a(\pa_v-t\pa_z)P_{\neq}(\Upsilon_2\Pi))]\right\|_{L^2}.
\end{align*}
Considering the support of $a$, we have $a(\pa_v-t\pa_z)P_{\neq}(\Upsilon_2\Pi)=a(\pa_v-t\pa_z)P_{\neq}(\Upsilon_1\Pi)$. Thus we use $\Upsilon_1\Pi$ and estimate
\begin{align*}
\|\mathcal{M}_3P_{\neq}[(a(\pa_v-t\pa_z)P_{\neq}(\Upsilon_1\Pi))]\|_{L^2},\quad \|\mathcal{M}_3P_{\neq}[(a\pa_zP_{\neq}(\Upsilon_1\Pi))]\|_{L^2}
\end{align*}
and
\begin{align*}
\left\|\big(\mathcal{M}_4+\mathcal{M}_5\big)P_{\neq}[(a(\pa_v-t\pa_z)P_{\neq}(\Upsilon_1\Pi))]\right\|_{L^2}.
\end{align*}
We have by Lemma \ref{Lem 3.6} and Remark \ref{Rmk: A-switch} that for $|k-l,\eta-\xi|\leq |l,\eta|$, $k\neq 0, l\neq 0$
\begin{align*}
&\mathcal{M}_3(t,k,\xi)(|l|+|\eta-lt|)\\
&\lesssim \mathcal{M}_3(t,l,\eta)(|l|+|\eta-lt|)e^{\la |k-l,\eta-\xi|^s}\\
&\quad +\rmA_l(\eta)\f{\langle k\rangle t(t+\langle k\rangle+|\xi|)}{1+k^2+\xi^2}\f{|\xi|(|l|+|\eta-lt|)}{k^2+(\xi-kt)^2}\mathbf{1}_{t\in {\rm{I}}_{k,\xi}\cap {\rm{I}}_{k,\eta},\, k\neq l}\\
&\lesssim \mathcal{M}_3(t,l,\eta)e^{c\la |k-l,\eta-\xi|}\langle k-l,\eta-\xi\rangle^5+\rmA_l(\eta)\f{{\xi}/{k}}{{\eta}/{l}}|l|^2\left\langle t-\f{\eta}{l}\right\rangle^2 e^{c\la |k-l,\eta-\xi|}\\
&\lesssim e^{c\la |k-l,\eta-\xi|}\mathcal{M}_3(t,l,\eta)(|l|+|\eta-lt|)^2.
\end{align*}
The lemma follows from Remark \ref{Rmk: A-switch} and Lemma \ref{Lem 3.4}. 
\end{proof}

\subsection{Estimate of $\Pi_{l,1}$}
In this section, we obtain the estimate for $P_{\neq}(\Pi_{l,1})$. 
We only need to estimate $\mathcal{M}_3(\pa_z^2+(\pa_v-t\pa_z)^2)\mathring{\mathcal{T}}_{j,N}^{-1}[P_{\neq} \Xi_1]$,
$\mathcal{M}_4(\pa_z^2+(\pa_v-t\pa_z)^2)\mathring{\mathcal{T}}_{j,N}^{-1}[P_{\neq} \Xi_1]$, and
$\mathcal{M}_5(\pa_z^2+(\pa_v-t\pa_z)^2)\mathring{\mathcal{T}}_{j,N}^{-1}[P_{\neq} \Xi_1]$ where 
\beno
\Xi_1=\udl{\varphi_6}\pa_{zz}(\Upsilon_2\Psi)+\udl{\varphi_7}P_{\neq}(\Upsilon_2\Psi)
\eeno
By Proposition \ref{eq: T_2}, we have
\beq\label{eq: Xi_1}
\begin{aligned}
\left\|\mathcal{M}_3(\pa_z^2+(\pa_v-t\pa_z)^2)\mathring{\mathcal{T}}_{j,N}^{-1}[P_{\neq} \Xi_1]\right\|_2\lesssim \left\|\mathcal{M}_3[P_{\neq} \Xi_1]\right\|_2,\\
\left\|\mathcal{M}_4(\pa_z^2+(\pa_v-t\pa_z)^2)\mathring{\mathcal{T}}_{j,N}^{-1}[P_{\neq} \Xi_1]\right\|_2\lesssim \left\|\mathcal{M}_4[P_{\neq} \Xi_1]\right\|_2,\\
\left\|\mathcal{M}_5(\pa_z^2+(\pa_v-t\pa_z)^2)\mathring{\mathcal{T}}_{j,N}^{-1}[P_{\neq} \Xi_1]\right\|_2\lesssim \left\|\mathcal{M}_5[P_{\neq} \Xi_1]\right\|_2. 
\end{aligned}
\eeq
By using
\ben\label{eq: inequality-1}
(|t|+|k|+|\xi|)|k|\lesssim \langle k,\xi\rangle(|\xi-kt|+|k|),\quad \text{for}\ k\neq 0,
\een
and Remark \ref{Rmk: A-switch}, we then have for $2|k,\eta|\geq |\xi-\eta|$, $k\neq 0$,
\begin{align*}
&|k|^2\f{\langle k\rangle t(t+\langle k\rangle+|\xi|)}{1+k^2+\xi^2}\rmA_k(\xi)
\lesssim |k^2|\rmA_k(\eta)\langle \xi-\eta\rangle^3\f{\langle k\rangle t(t+\langle k\rangle+|\eta|)}{1+k^2+\eta^2}e^{c\la|\xi-\eta|^s}\\
&\lesssim \f{t|k|^2| k, \eta|(| k|+ |tk-\eta|)}{1+k^2+\eta^2}\f{|\eta|+|kt|}{|kt|}\left\langle\f{\eta}{kt}\right\rangle^{-1}\rmA_k(\eta)e^{c\la|\xi-\eta|^s}\\
&\lesssim (k^2+(\eta-kt)^2)\left\langle\f{\eta}{kt}\right\rangle^{-1}\rmA_k(\eta)e^{c\la|\xi-\eta|^s}. 
\end{align*}
and for $2|k,\eta|\leq |\xi-\eta|$, $k\neq 0$,
\begin{align*}
&|k|\f{\langle k\rangle t(t+\langle k\rangle+|\xi|)}{1+k^2+\xi^2}\rmA_k(\xi)
\lesssim \rmA^{\rmR}(\xi-\eta)e^{\la|k,\eta|^s}\langle t\rangle^2\langle k\rangle^2
\end{align*}
Thus we get that 
\begin{align}\label{eq: M_3Xi_1}
\left\|\mathcal{M}_3[P_{\neq} \Xi_1]\right\|_2
\lesssim \langle t\rangle^2\|P_{\neq}(\Upsilon_2\Psi)\|_{\mathcal{G}^{s,\s-6}}+ \left\|\left\langle\f{\pa_v}{t\pa_z}\right\rangle^{-1}(\pa_z^2+(\pa_v-t\pa_z)^2)\rmA P_{\neq}(\Psi\Upsilon_2)\right\|_2\lesssim \ep.
\end{align}
By Lemma \ref{Lem 3.4}, we have
\begin{align*}
\left\|\big(\mathcal{M}_4+\mathcal{M}_5\big)[P_{\neq} \Xi_1]\right\|_2
&\lesssim \Big(\f{1}{\langle t\rangle^s}+\mathcal{CK}_{\varphi^{\d}_6}^{\f12}+\mathcal{CK}_{\varphi^{\d}_7}^{\f12}\Big)\langle t\rangle^2\|P_{\neq}(\Upsilon_2\Psi)\|_{\mathcal{G}^{s,\s-6}}\\
&\quad+\left\|\left\langle\f{\pa_v}{t\pa_z}\right\rangle^{-1}(\pa_z^2+(\pa_v-t\pa_z)^2)\left(\f{|\na|^{\f{s}{2}}}{\langle t\rangle^s}\rmA+\sqrt{\f{\pa_t w}{w}}\tilde{\rmA}\right)P_{\neq}(\Psi\Upsilon_2)\right\|_2.
\end{align*}
Now we prove Proposition \ref{Prop: non-zero-pessure-L} for $\Pi_{\star}=\Upsilon_2\Pi_{l,1}$. 
\begin{proof}
By \eqref{eq: pre-eq-error}, Lemma \ref{lem: S_2}, Lemma \ref{lem: S_a}, \eqref{eq: Xi_1}, and \eqref{eq: M_3Xi_1}, we get that
\begin{align*}
&\left\|\mathcal{M}_3(\pa_z^2+(\pa_v-t\pa_z)^2)P_{\neq}(\Upsilon_2\Pi_{l,1})\right\|_2\\
&\lesssim 
\left\|\mathcal{M}_3(\pa_z^2+(\pa_v-t\pa_z)^2)\mathring{\mathcal{T}}_{2, N}^{-1}[P_{\neq}\Xi_1]\right\|_2
+\left\|\mathcal{M}_3(\pa_z^2+(\pa_v-t\pa_z)^2)\mathring{\mathcal{T}}_{j, N}^{-1}\big[P_{\neq}\mathcal{S}_2[\Pi_{l,1}]\big]\right\|_2\\
&\quad+\left\|\mathcal{M}_3(\pa_z^2+(\pa_v-t\pa_z)^2)\mathring{\mathcal{T}}_{2, N}^{-1}\big[P_{\neq}\mathcal{S}_a[\Pi_{l,1}]\big]\right\|_2\\
&\lesssim \ep+\ep\left\|\mathcal{M}_3(\pa_z^2+(\pa_v-t\pa_z)^2)P_{\neq}(\Upsilon_2\Pi_{l,1})\right\|_2
+\ep^2
\end{align*}
and
\begin{align*}
&\left\|(\mathcal{M}_4+\mathcal{M}_5)(\pa_z^2+(\pa_v-t\pa_z)^2)P_{\neq}(\Upsilon_2\Pi_{l,1})\right\|_2\\
&\lesssim 
\left\|(\mathcal{M}_4+\mathcal{M}_5)(\pa_z^2+(\pa_v-t\pa_z)^2)\mathring{\mathcal{T}}_{2, N}^{-1}[P_{\neq}\Xi_1]\right\|_2\\
&\quad+\left\|(\mathcal{M}_4+\mathcal{M}_5)(\pa_z^2+(\pa_v-t\pa_z)^2)\mathring{\mathcal{T}}_{j, N}^{-1}\big[P_{\neq}\mathcal{S}_2[\Pi_{l,1}]\big]\right\|_2\\
&\quad+\left\|(\mathcal{M}_4+\mathcal{M}_5)(\pa_z^2+(\pa_v-t\pa_z)^2)\mathring{\mathcal{T}}_{2, N}^{-1}\big[P_{\neq}\mathcal{S}_a[\Pi_{l,1}]\big]\right\|_2\\
&\lesssim \ep \Big(\f{1}{\langle t\rangle^s}+\mathcal{CK}_{\varphi^{\d}_6}^{\f12}+\mathcal{CK}_{\varphi^{\d}_7}^{\f12}+\mathcal{CK}_{\th^{\d}}^{\f12}+\mathcal{CK}_{h}^{\f12}+\mathcal{CK}_{\varphi_9^{\d}}^{\f12}+\mathcal{CK}_a^{\f12}\Big)\\
&\quad+\left\|\left\langle\f{\pa_v}{t\pa_z}\right\rangle^{-1}(\pa_z^2+(\pa_v-t\pa_z)^2)\left(\f{|\na|^{\f{s}{2}}}{\langle t\rangle^s}\rmA+\sqrt{\f{\pa_t w}{w}}\tilde{\rmA}\right)P_{\neq}(\Psi\Upsilon_2)\right\|_2\\
&\quad +\ep\|(\mathcal{M}_4+\mathcal{M}_5)(\pa_z^2+(\pa_v-t\pa_z)^2)P_{\neq}(\Upsilon_2\Pi_{l,1})\|_{2}\\
&\quad +\ep \|\mathcal{M}_{01}\pa_vP_0(\Upsilon_2\Pi)\|_{L^2},
\end{align*}
which gives 
\begin{align*}
&\left\|\mathcal{M}_3(\pa_z^2+(\pa_v-t\pa_z)^2)P_{\neq}(\Upsilon_2\Pi_{l,1})\right\|_2\lesssim \ep,\\
\end{align*}
and
\begin{align*}
&\left\|(\mathcal{M}_4+\mathcal{M}_5)(\pa_z^2+(\pa_v-t\pa_z)^2)P_{\neq}(\Upsilon_2\Pi_{l,1})\right\|_2\\
&\lesssim \ep \Big(\f{1}{\langle t\rangle^s}+\mathcal{CK}_{\varphi^{\d}_6}^{\f12}+\mathcal{CK}_{\varphi^{\d}_7}^{\f12}+\mathcal{CK}_{\th^{\d}}^{\f12}+\mathcal{CK}_{h}^{\f12}+\mathcal{CK}_{\varphi_9^{\d}}^{\f12}+\mathcal{CK}_a^{\f12}\Big)\\
&\quad+\left\|\left\langle\f{\pa_v}{t\pa_z}\right\rangle^{-1}(\pa_z^2+(\pa_v-t\pa_z)^2)\left(\f{|\na|^{\f{s}{2}}}{\langle t\rangle^s}\rmA+\sqrt{\f{\pa_t w}{w}}\tilde{\rmA}\right)P_{\neq}(\Psi\Upsilon_2)\right\|_2\\
&\quad +\ep \|\mathcal{M}_{01}\pa_vP_0(\Upsilon_2\Pi)\|_{L^2}.
\end{align*}
Thus we prove the estimates for $P_{\neq}(\Upsilon_2\Pi_{l,1})$. 
\end{proof}
\subsection{Estimate of $\Pi_{l,2}$}
The $\Pi_{l,2}$ term arose from the boundary correction. It seems worse than $\Pi_{l,1}$. However, on the support of $\udl{\varphi_8}$, $\Upsilon_1=0$. Thus to estimate $\Pi_{l,2}$, we use two steps. First, we get the estimate of $\Pi_{l,2}$ using equation \eqref{eq: Pressure-l2}. Second, we obtain an improved estimate for $\Upsilon_1\Pi_{l,2}$, namely, we study the following equation: 
\beq\label{eq: Pressure-l2-2}
\left\{
\begin{aligned}
&\pa_z\Big((a+\udl{\th})\pa_z(\Upsilon_1\Pi_{l,2})\Big)+\udl{\pa_yv}(\pa_v-t\pa_z)\Big((a+\udl{\th})\udl{\pa_yv}(\pa_v-t\pa_z)(\Upsilon_1\Pi_{l,2})\Big)\\
&\quad\quad \quad \quad\quad 
=2(\udl{\pa_yv})^2(\pa_v\Upsilon_1)\Big(\udl{\th}(\pa_v-t\pa_z)(\Upsilon_2\Pi_{l,2})\Big)
+(\udl{\pa_yv})^2\Big(\udl{\th}(\pa_v^2\Upsilon_1)(\Upsilon_2\Pi_{l,2})\Big),\\
&(\pa_v-t\pa_z)\Pi_{l,2}(t, z, u(0))=(\pa_v-t\pa_z)\Pi_{l,2}(t, z, u(1))=0.
\end{aligned}\right.
\eeq 
Note that on the support of $\pa_v\Upsilon_1$ or $\pa_v^2\Upsilon_1$, both $\udl{\th}$ and $\udl{\pa_yv}$ are constant value.

We show the idea in the following inequalities: 
\begin{align*}
&\text{Step 1. }\quad \|\Upsilon_2\Pi_{l,2}\|_{Y}\lesssim \ep  \|\Upsilon_2\Pi_{l,2}\|_{Y}+\ep \|\Upsilon_1\Pi_{l,2}\|_{X}+\text{good terms},\\
&\text{Step 2. }\quad \|\Upsilon_1\Pi_{l,2}\|_{X}\lesssim \ep \|\Upsilon_1\Pi_{l,2}\|_{X}+\|\Upsilon_2\Pi_{l,2}\|_{Y}+\text{good terms}. 
\end{align*}

Let us first estimate $P_{\neq} (\Upsilon_2\Pi_{l,2})$. 

\begin{lemma}\label{lem: P_{l,2}-1}
Under the bootstrap hypotheses,
\begin{align*}
\left\|\mathcal{M}_3(-\pa_z^2-(\pa_v-t\pa_z)^2)^{\f12}P_{\neq}(\Upsilon_2\Pi_{l,2})\right\|_{L^2}\lesssim \ep+\ep\left\|\mathcal{M}_3(\pa_z^2+(\pa_v-t\pa_z)^2)P_{\neq}(\Upsilon_1\Pi_{l,2})\right\|_2,
\end{align*}
and
\begin{align*}
&\left\|(\mathcal{M}_4+\mathcal{M}_5)(-\pa_z^2-(\pa_v-t\pa_z)^2)^{\f12}P_{\neq}(\Upsilon_2\Pi_{l,2})\right\|_{L^2}\\
&\lesssim +\ep\left\|(\mathcal{M}_4+\mathcal{M}_5)(\pa_z^2+(\pa_v-t\pa_z)^2)P_{\neq}(\Upsilon_1\Pi_{l,2})\right\|_{2}\\
&\quad+\ep\Big(\mathcal{CK}_{\th^{\d}}^{\f12}+\mathcal{CK}_{h}^{\f12}+\mathcal{CK}_{\varphi_8^{\d}}^{\f12}+\mathcal{CK}_{\varphi_9^{\d}}^{\f12}+\mathcal{CK}_a^{\f12}+\f{1}{\langle t\rangle^s}\Big)\\
&\quad +\ep \|\mathcal{M}_{01}\pa_vP_0(\Upsilon_2\Pi_{l,2})\|_{L^2}\\
&\quad+\left\|\left\langle\f{\pa_v}{t\pa_z}\right\rangle^{-1}(\pa_z^2+(\pa_v-t\pa_z)^2)\left(\f{|\na|^{\f{s}{2}}}{\langle t\rangle^s}\rmA+\sqrt{\f{\pa_t w}{w}}\tilde{\rmA}\right)P_{|k|\geq k_{M}}(\Psi\Upsilon_2)\right\|_2.
\end{align*}
Moreover, it holds that
\begin{align*}
\langle t\rangle^3\|P_{\neq}(\Upsilon_2\Pi_{l,2})\|_{\mathcal{G}^{s,\s-6}}
&+\langle t\rangle^2\left\|(-\pa_z^2-(\pa_v-t\pa_z)^2)^{\f12}P_{\neq}(\Upsilon_2\Pi_{l,2})\right\|_{\mathcal{G}^{s,\s-6}}\\
&\lesssim \ep+\ep\left\|\mathcal{M}_3(\pa_z^2+(\pa_v-t\pa_z)^2)P_{\neq}(\Upsilon_1\Pi_{l,2})\right\|_2.
\end{align*}
\end{lemma}

\begin{proof}
By \eqref{eq: pre-eq-error} Lemma \ref{lem: S_2-1}, and Lemma \ref{lem: S_a-1}, we have
\begin{align*}
&\left\|\mathcal{M}_3(-\pa_z^2-(\pa_v-t\pa_z)^2)^{\f12}P_{\neq}(\Upsilon_2\Pi_{l,2})\right\|_{L^2}\\
&\lesssim \left\|\mathcal{M}_3(-\pa_z^2-(\pa_v-t\pa_z)^2)^{\f12}\mathring{\mathcal{T}}_{2,N}^{-1}[P_{\neq} \Xi_2]\right\|_{L^2}\\
&\quad+\left\|\mathcal{M}_3(-\pa_z^2-(\pa_v-t\pa_z)^2)^{\f12}\mathring{\mathcal{T}}_{j, N}^{-1}P_{\neq}\big[\mathcal{S}_2[\Pi_{l,2}]\big]\right\|_2\\
&\quad+\left\|\mathcal{M}_3(-\pa_z^2-(\pa_v-t\pa_z)^2)^{\f12}\mathring{\mathcal{T}}_{j, N}^{-1}P_{\neq}\big[\mathcal{S}_a[\Pi_{l,2}]\big]\right\|_2\\
&\lesssim \ep^2+\left\|\mathcal{M}_3(-\pa_z^2-(\pa_v-t\pa_z)^2)^{-\f12}[P_{\neq} \Xi_2]\right\|_{L^2}+\ep\left\|\mathcal{M}_3(-\pa_z^2-(\pa_v-t\pa_z)^2)^{\f12}P_{\neq}(\Upsilon_2\Pi_{l,2})\right\|_{L^2}\\
&\quad{+\ep\left\|\mathcal{M}_3(\pa_z^2+(\pa_v-t\pa_z)^2)P_{\neq}(\Upsilon_1\Pi_{l,2})\right\|_2},
\end{align*}
and
\begin{align*}
&\left\|(\mathcal{M}_4+\mathcal{M}_5)(-\pa_z^2-(\pa_v-t\pa_z)^2)^{\f12}P_{\neq}(\Upsilon_2\Pi_{l,2})\right\|_{L^2}\\
&\lesssim \left\|(\mathcal{M}_4+\mathcal{M}_5)(-\pa_z^2-(\pa_v-t\pa_z)^2)^{\f12}\mathring{\mathcal{T}}_{2,N}^{-1}[P_{\neq} \Xi_2]\right\|_{L^2}\\
&\quad+\left\|(\mathcal{M}_4+\mathcal{M}_5)(-\pa_z^2-(\pa_v-t\pa_z)^2)^{\f12}\mathring{\mathcal{T}}_{j, N}^{-1}P_{\neq}\big[\mathcal{S}_2[\Pi_{l,2}]\big]\right\|_2\\
&\quad+\left\|(\mathcal{M}_4+\mathcal{M}_5)(-\pa_z^2-(\pa_v-t\pa_z)^2)^{\f12}\mathring{\mathcal{T}}_{j, N}^{-1}P_{\neq}\big[\mathcal{S}_a[\Pi_{l,2}]\big]\right\|_2\\
&\lesssim 
\ep\|(\mathcal{M}_4+\mathcal{M}_5)(-\pa_z^2-(\pa_v-t\pa_z)^2)^{\f12}P_{\neq}(\Upsilon_2\Pi_{l,2})\|_{2}\\
&\quad{+\ep\|(\mathcal{M}_4+\mathcal{M}_5)(\pa_z^2+(\pa_v-t\pa_z)^2)P_{\neq}(\Upsilon_1\Pi_{l,2})\|_{2}}\\
&\quad+\ep\Big(\mathcal{CK}_{\th^{\d}}^{\f12}+\mathcal{CK}_{h}^{\f12}+\mathcal{CK}_{\varphi_9^{\d}}^{\f12}+\mathcal{CK}_a^{\f12}+\f{\ep}{\langle t\rangle^s}\Big)\\
&\quad +\left\|(\mathcal{M}_4+\mathcal{M}_5)(-\pa_z^2-(\pa_v-t\pa_z)^2)^{-\f12}[P_{\neq} \Xi_2]\right\|_{L^2}\\
&\quad +\ep \|\mathcal{M}_{01}\pa_vP_0(\Upsilon_2\Pi_{l,2})\|_{L^2},
\end{align*}
Now we control $\left\|\mathcal{M}_3(-\pa_z^2-(\pa_v-t\pa_z)^2)^{-\f12}[P_{\neq} \Xi_2]\right\|_{L^2}$ and $\left\|(\mathcal{M}_4+\mathcal{M}_5)(-\pa_z^2-(\pa_v-t\pa_z)^2)^{-\f12}[P_{\neq} \Xi_2]\right\|_{L^2}$. where 
\begin{align*}
\Xi_2=\udl{\varphi_8}\udl{\pa_yv}(\pa_v-t\pa_z)P_{\neq}(\Upsilon_2\Psi)
\end{align*}
By \eqref{eq: inequality-1}, we have for $|\xi-\eta|\leq 2|k,\eta|$
\begin{align*}
&\f{\langle k\rangle t(t+\langle k\rangle+|\xi|)}{1+k^2+\xi^2}\f{(\eta-kt)}{|k|+|\xi-kt|}\\
&\lesssim \f{(\eta-kt)}{1+|k|+|\xi|}\f{|\eta|+|kt|}{k}\left\langle\f{\eta}{kt}\right\rangle^{-1}\lesssim \langle \eta-\xi\rangle (k^2+(\eta-kt))^2\left\langle\f{\eta}{kt}\right\rangle^{-1}
\end{align*}
and for $|\xi-\eta|\geq 2|k,\eta|$
\begin{align*}
&\f{\langle k\rangle t(t+\langle k\rangle+|\xi|)}{1+k^2+\xi^2}\f{(\eta-kt)}{|k|+|\xi-kt|}\lesssim \f{(\eta-kt)t}{1+|k|+|\xi|}\lesssim \langle t\rangle^2|k,\eta|. 
\end{align*}
Thus we get by Remark \ref{Rmk: A-switch} and Lemma \ref{Lem 3.4} that 
\begin{align*}
&\left\|\mathcal{M}_3(-\pa_z^2-(\pa_v-t\pa_z)^2)^{-\f12}[P_{\neq} \Xi_2]\right\|_{L^2}\\
&\lesssim  \left\|\left\langle\f{\pa_v}{t\pa_z}\right\rangle^{-1}(\pa_z^2+(\pa_v-t\pa_z)^2)\rmA P_{\neq}(\Psi\Upsilon_2)\right\|_2\\
&\quad +(1+\|\rmA^{\rmR}(\varphi_8^{\d})\|_2)(1+\|\rmA^{\rmR}h\|_2+\|\rmA^{\rmR}\varphi_9^{\d}\|_2)\langle t\rangle^2\|P_{\neq}(\Psi\Upsilon_2)\|\lesssim \ep,
\end{align*}
and
\begin{align*}
&\left\|(\mathcal{M}_4+\mathcal{M}_5)(-\pa_z^2-(\pa_v-t\pa_z)^2)^{-\f12}[P_{\neq} \Xi_2]\right\|_{L^2}\\
&\lesssim \left\|\left\langle\f{\pa_v}{t\pa_z}\right\rangle^{-1}(\pa_z^2+(\pa_v-t\pa_z)^2)\left(\f{|\na|^{\f{s}{2}}}{\langle t\rangle^s}\rmA+\sqrt{\f{\pa_t w}{w}}\tilde{\rmA}\right)P_{|k|\geq k_{M}}(\Psi\Upsilon_2)\right\|_2\\
&\quad + \f{\ep}{\langle t\rangle^s}+\ep\mathcal{CK}_{\varphi^{\d}_8}^{\f12}+\ep\mathcal{CK}_{\varphi^{\d}_9}^{\f12}+\ep\mathcal{CK}_h^{\f12}. 
\end{align*}
Thus we proved the lemma. 
\end{proof}
Now we use \eqref{eq: Pressure-l2-2} to prove Proposition \ref{Prop: non-zero-pessure-L} for $\Pi_{\star}=\Upsilon_1\Pi_{l,2}$. 
\begin{proof}
By \eqref{eq: pre-eq-error}, Lemma \ref{lem: S_2}, Lemma \ref{lem: S_a}, \eqref{eq: Xi_1}, and \eqref{eq: M_3Xi_1}, we get that
\begin{align*}
&\left\|\mathcal{M}_3(\pa_z^2+(\pa_v-t\pa_z)^2)P_{\neq}(\Upsilon_1\Pi_{l,2})\right\|_2\\
&\lesssim 
\left\|\mathcal{M}_3[P_{\neq}\Xi_3]\right\|_2
+\left\|\mathcal{M}_3(\pa_z^2+(\pa_v-t\pa_z)^2)\mathring{\mathcal{T}}_{1, N}^{-1}\big[P_{\neq}\mathcal{S}_2[\Pi_{l,2}]\big]\right\|_2\\
&\quad+\left\|\mathcal{M}_3(\pa_z^2+(\pa_v-t\pa_z)^2)\mathring{\mathcal{T}}_{1, N}^{-1}\big[P_{\neq}\mathcal{S}_a[\Pi_{l,2}]\big]\right\|_2.
\end{align*}
and
\begin{align*}
&\left\|(\mathcal{M}_4+\mathcal{M}_5)(\pa_z^2+(\pa_v-t\pa_z)^2)P_{\neq}(\Upsilon_1\Pi_{l,2})\right\|_2\\
&\lesssim 
\left\|(\mathcal{M}_4+\mathcal{M}_5)[P_{\neq}\Xi_3]\right\|_2\\
&\quad+\left\|(\mathcal{M}_4+\mathcal{M}_5)(\pa_z^2+(\pa_v-t\pa_z)^2)\mathring{\mathcal{T}}_{1, N}^{-1}\big[P_{\neq}\mathcal{S}_2[\Pi_{l,2}]\big]\right\|_2\\
&\quad+\left\|(\mathcal{M}_4+\mathcal{M}_5)(\pa_z^2+(\pa_v-t\pa_z)^2)\mathring{\mathcal{T}}_{1, N}^{-1}\big[P_{\neq}\mathcal{S}_a[\Pi_{l,2}]\big]\right\|_2,
\end{align*}
with 
\beno
\Xi_3=2(\udl{\pa_yv})^2(\pa_v\Upsilon_1)\Big(\udl{\th}(\pa_v-t\pa_z)(\Upsilon_2\Pi_{l,2})\Big)
+(\udl{\pa_yv})^2\Big(\udl{\th}(\pa_v^2\Upsilon_1)(\Upsilon_2\Pi_{l,2})\Big)
\eeno
We have for $|\xi-\eta|\leq 2|k,\eta|$
\begin{align*}
\f{\langle k\rangle t(t+\langle k\rangle+|\xi|)}{1+k^2+\xi^2}
\lesssim \f{\langle k\rangle t(t+\langle k\rangle+|\eta|)}{1+k^2+\eta^2}\langle \xi-\eta \rangle^3,
\end{align*}
and for $|\xi-\eta|\geq 2|k,\eta|$
\begin{align*}
\f{\langle k\rangle t(t+\langle k\rangle+|\xi|)}{1+k^2+\xi^2}
\lesssim \langle t\rangle^2.
\end{align*}
Thus we have by Remark \ref{Rmk: A-switch} that
\begin{align*}
\left\|\mathcal{M}_3[P_{\neq}\Xi_3]\right\|_2
&\lesssim \left\|\mathcal{M}_3(-\pa_z^2-(\pa_v-t\pa_z)^2)^{\f12}P_{\neq}(\Upsilon_2\Pi_{l,2})\right\|_{L^2}\\
&\quad +\|\rmA^{\rmR}(\udl{\pa_yv}\Upsilon_1)\|_2^2\|\rmA^{\rmR}\udl{\th}\|_{L^2}\langle t\rangle^2\|(-\pa_z^2-(\pa_v-t\pa_z)^2)^{\f12}P_{\neq}(\Upsilon_2\Pi_{l,2})\|_{\mathcal{G}^{s,\s-6}}\lesssim \ep. 
\end{align*}
and by Lemma \ref{Lem 3.4}
\begin{align*}
&\left\|(\mathcal{M}_4+\mathcal{M}_5)[P_{\neq}\Xi_3]\right\|_2\\
&\lesssim \left\|(\mathcal{M}_4+\mathcal{M}_5)(-\pa_z^2-(\pa_v-t\pa_z)^2)^{\f12}P_{\neq}(\Upsilon_2\Pi_{l,2})\right\|_{L^2}\\
&\quad +\Big(\f{1}{\langle t\rangle^s}+\mathcal{CK}_{h}^{\f12}+\mathcal{CK}_{\th}^{\f12}\Big)\langle t\rangle^2\|(-\pa_z^2-(\pa_v-t\pa_z)^2)^{\f12}P_{\neq}(\Upsilon_2\Pi_{l,2})\|_{\mathcal{G}^{s,\s-6}}
\end{align*}
By using Lemma \ref{lem: P_{l,2}-1}, we proved the estimates for $P_{\neq}(\Upsilon_1\Pi_{l,2})$. 
\end{proof}

\subsection{Estimate of $\Pi_{n,1}$}
In this section, we treat the nonlinear pressure. Similarly, we first estimate 
\begin{align*}
\left\|\mathcal{M}_3[P_{\neq}\Xi_4]\right\|_2,\quad\text{and}\quad \left\|(\mathcal{M}_4+\mathcal{M}_5)[P_{\neq}\Xi_4]\right\|_2
\end{align*}
where 
\begin{align*}
\Xi_4&=-2\udl{\chi_2}(\udl{\pa_yv}(\pa_v-t\pa_z)\pa_z(\Upsilon_2\Psi))^2-2\pa_{zz}(\Upsilon_2\Psi)\Big(\Om-\udl{\chi_2}\pa_{zz}(\Upsilon_2\Psi)\Big),\\
&\quad
-2\udl{\chi_2'}\udl{\pa_yv}\Big[(\pa_z(\Upsilon_2\Psi))(\pa_v-t\pa_z)\pa_z(\Upsilon_2\Psi)+(\pa_v-t\pa_z)(\Upsilon_2\Psi)\pa_{zz}(\Upsilon_2\Psi)\Big]\\
&\quad
-\f{1}{2}\udl{\chi_2''}\Big[(\pa_z(\Upsilon_2\Psi))^2+(\udl{\pa_yv})^2((\pa_v-t\pa_z)P_{\neq}(\Upsilon_2\Psi))^2\Big]. 
\end{align*}
We have the following lemma:
\begin{lemma}\label{Lem: Xi_4}
Under the bootstrap hypotheses,
\beno
\left\|\mathcal{M}_3[P_{\neq}\Xi_4]\right\|_2\lesssim \ep^2,
\eeno
and
\begin{align*}
\left\|(\mathcal{M}_4+\mathcal{M}_5)[P_{\neq}\Xi_4]\right\|_2
&\lesssim \f{\ep^2}{\langle t\rangle^2}+\ep\mathcal{CK}_f^{\f12}
+\ep^2\mathcal{CK}_{h}^{\f12}+\ep^2\mathcal{CK}_{\varphi^{\d}_4}^{\f12}+\ep^2\mathcal{CK}_{\varphi^{\d}_9}^{\f12}\\
&\quad+\left\|\left\langle\f{\pa_v}{t\pa_z}\right\rangle^{-1}(\pa_z^2+(\pa_v-t\pa_z)^2)\left(\f{|\na|^{\f{s}{2}}}{\langle t\rangle^s}\rmA+\sqrt{\f{\pa_t w}{w}}\tilde{\rmA}\right)P_{\neq}(\Psi\Upsilon_2)\right\|_2. 
\end{align*}
\end{lemma}
\begin{proof}
We separate terms in $\Xi_4$ into four classes:
\begin{align*}
A_1&=\Big\{P_{\neq}\big(\udl{\chi_2}(\udl{\pa_yv}(\pa_v-t\pa_z)\pa_z(\Upsilon_2\Psi))^2\big), \ P_{\neq}\big(\udl{\chi_2''}(\udl{\pa_yv})^2((\pa_v-t\pa_z)P_{\neq}(\Upsilon_2\Psi))^2\big)\Big\}\\
A_2&=\Big\{P_{\neq}\big(\pa_{zz}(\Upsilon_2\Psi)P_{\neq}\Om\big),\ 
P_{\neq}\big(\udl{\chi_2}(\pa_{zz}(\Upsilon_2\Psi))^2\big),\ 
P_{\neq}\big(\udl{\chi_2''}(\pa_{z}(\Upsilon_2\Psi))^2\big), \\ 
&\quad \quad P_{\neq}\big(\udl{\chi_2'}\udl{\pa_yv}(\pa_z(\Upsilon_2\Psi))(\pa_v-t\pa_z)\pa_z(\Upsilon_2\Psi)\big),\ 
P_{\neq}\big(\udl{\chi_2'}\udl{\pa_yv}(\pa_v-t\pa_z)P_{\neq}(\Upsilon_2\Psi)\pa_{zz}(\Upsilon_2\Psi)\big)\Big\}\\
A_3&=\Big\{\pa_{zz}(\Upsilon_2\Psi)P_{0}\Om\Big\}\\
A_4&=\Big\{\udl{\chi_2'}\udl{\pa_yv}\pa_vP_{0}(\Upsilon_2\Psi)\pa_{zz}(\Upsilon_2\Psi)\Big\}
\end{align*}
The first term in each class is the typical term. 
By using \eqref{eq: inequality-1}, we have for $|k-l,\xi-\eta|\leq |l,\eta|$ that
\begin{align*}
\f{\langle k\rangle t(t+\langle k\rangle+|\xi|)}{1+k^2+\xi^2}(|\eta-lt|)|l|
&\lesssim \langle k-l,\xi-\eta\rangle^{4}\f{\langle l\rangle t(t+\langle l\rangle+|\eta|)}{1+l^2+\eta^2}(|\eta-lt|)|l|\\
&\lesssim \langle k-l,\xi-\eta\rangle^{4}\left\langle\f{\eta}{lt}\right\rangle^{-1}(l^2+(\eta-lt)^2)
\end{align*}
which together with Remark \ref{Rmk: A-switch} and Lemma \ref{Lem 3.4}, gives that for $f\in A_1$, $\left\|\mathcal{M}_3f\right\|_2\lesssim \ep^2$, and 
\begin{align*}
\left\|(\mathcal{M}_4+\mathcal{M}_5)f\right\|_2\lesssim &\ep^2\mathcal{CK}_h^{\f12}+\ep^2\mathcal{CK}_{\varphi^{\d}_9}^{\f12}+\f{\ep^2}{\langle t\rangle^s}\\
&+\ep\left\|\left\langle\f{\pa_v}{t\pa_z}\right\rangle^{-1}(\pa_z^2+(\pa_v-t\pa_z)^2)\left(\f{|\na|^{\f{s}{2}}}{\langle t\rangle^s}\rmA+\sqrt{\f{\pa_t w}{w}}\tilde{\rmA}\right)P_{\neq}(\Psi\Upsilon_2)\right\|_2. 
\end{align*}
We have $k, l\neq 0$ and $k\neq l$,
\begin{align*}
&(k-l)^2\f{\langle k\rangle t(t+\langle k\rangle+|\xi|)}{1+k^2+\xi^2}\rmA_k(\xi)\\
&\lesssim \langle t\rangle^2 e^{c\la|k-l,\xi-\eta|^s}\rmA_l(\eta)\\
&\quad+\langle l,\eta\rangle^4\left\langle \f{\xi-\eta}{(k-l)t} \right\rangle^{-1}((k-l)^2+(\xi-\eta-(k-l)t)^2)\rmA_{k-l}(\xi-\eta)e^{c\la |l,\eta|^s},
\end{align*}
which together with Remark \ref{Rmk: A-switch} and Lemma \ref{Lem 3.4} gives for $f\in A_2$, $\left\|\mathcal{M}_3f\right\|_2\lesssim \ep^2$, and 
\begin{align*}
\left\|(\mathcal{M}_4+\mathcal{M}_5)f\right\|_2\lesssim &\ep^2\mathcal{CK}_h^{\f12}+\ep^2\mathcal{CK}_{\varphi^{\d}_9}^{\f12}+\ep\mathcal{CK}_f^{\f12}+\f{\ep^2}{\langle t\rangle^s}\\
&+\ep\left\|\left\langle\f{\pa_v}{t\pa_z}\right\rangle^{-1}(\pa_z^2+(\pa_v-t\pa_z)^2)\left(\f{|\na|^{\f{s}{2}}}{\langle t\rangle^s}\rmA+\sqrt{\f{\pa_t w}{w}}\tilde{\rmA}\right)P_{\neq}(\Psi\Upsilon_2)\right\|_2. 
\end{align*}
By Remark \ref{Rmk: A-switch} and \eqref{eq: inequality-1}, we have for $k\neq 0$
\begin{align*}
&k^2\f{\langle k\rangle t(t+\langle k\rangle+|\xi|)}{1+k^2+\xi^2}\rmA_k(\xi)\\
&\lesssim \langle \xi-\eta\rangle^{3}\f{(|k|+|\eta-kt|)tk^2}{|k|+|\eta|}\f{\eta+kt}{kt}\left\langle\f{\eta}{kt}\right\rangle^{-1}\rmA_k(\eta)e^{c\la|\xi-\eta|^s}\\
&\quad+\langle k,\eta\rangle^{4}\f{(|k|+|\eta-kt|)tk^2}{|k|+|\eta|}\f{\eta}{(k^2+(\eta-kt)^2)}\rmA_0(\xi-\eta)e^{c\la|k,\eta|^s}\\
&\quad+\langle k,\eta\rangle^{4}\langle t\rangle^2\rmA_0(\xi-\eta)e^{c\la|k,\eta|^s}\\
&\lesssim \langle k,\eta\rangle^{5}\langle t\rangle^2\rmA_0(\xi-\eta)e^{c\la|k,\eta|^s}+\langle \xi-\eta\rangle^{3}(|k|^2+|\eta-kt|^2)\left\langle\f{\eta}{kt}\right\rangle^{-1}\rmA_k(\eta)e^{c\la|\xi-\eta|^s}
\end{align*}
which gives that for $f\in A_3$, $\left\|\mathcal{M}_3f\right\|_2\lesssim \ep^2$. By Lemma \ref{Lem 3.4}, it holds that
\begin{align*}
\left\|(\mathcal{M}_4+\mathcal{M}_5)f\right\|_2\lesssim &\ep\mathcal{CK}_f^{\f12}+\f{\ep^2}{\langle t\rangle^s}\\
&+\ep\left\|\left\langle\f{\pa_v}{t\pa_z}\right\rangle^{-1}(\pa_z^2+(\pa_v-t\pa_z)^2)\left(\f{|\na|^{\f{s}{2}}}{\langle t\rangle^s}\rmA+\sqrt{\f{\pa_t w}{w}}\tilde{\rmA}\right)P_{\neq}(\Psi\Upsilon_2)\right\|_2. 
\end{align*}
By Remark \ref{Rmk: A-switch} and \eqref{eq: inequality-1}, we have for 
\begin{align*}
&|\xi-\eta|\f{\langle k\rangle t(t+\langle k\rangle+|\xi|)}{1+k^2+\xi^2}\rmA_k(\xi)k^2\\
&\lesssim \langle\xi-\eta\rangle^{4}\f{t(\langle k\rangle+|\eta-kt|)}{1+|k|+|\eta|}\f{\eta+kt}{kt}\left\langle\f{\eta}{kt}\right\rangle^{-1}\rmA_k(\xi)k^2e^{c\la|\xi-\eta|^s}\\
&\quad+\langle t\rangle^2\langle k,\eta\rangle^{4}\f{t+|\xi-\eta|}{t \langle \xi-\eta\rangle(|k|+|\xi|)}|\xi-\eta|\rmA_0(\xi-\eta)e^{c\la|k,\eta|^s}\f{\xi}{k^2+(\xi-kt)^2}\\
&\lesssim \langle \xi-\eta\rangle^{4}(|k|^2+|\eta-kt|^2)\left\langle\f{\eta}{kt}\right\rangle^{-1}\rmA_k(\eta)e^{c\la|\xi-\eta|^s}\\
&\quad+\langle t\rangle^2\langle k,\eta\rangle^{4}|\xi-\eta|\rmA_0(\xi-\eta)e^{c\la|k,\eta|^s},
\end{align*}
which together with Corollary \ref{Cor: psi-new} and Lemma \ref{Lem 3.4} gives for $f\in A_4$, $\left\|\mathcal{M}_3f\right\|_2\lesssim \ep^2$, and
\begin{align*}
\left\|(\mathcal{M}_4+\mathcal{M}_5)f\right\|_2\lesssim &\ep\mathcal{CK}_f^{\f12}+\f{\ep^2}{\langle t\rangle^s}
+\ep^2\Big(\mathcal{CK}_h^{\f12}+\mathcal{CK}_{\varphi^{\d}_4}^{\f12}+\mathcal{CK}_{\varphi^{\d}_9}^{\f12}\Big)\\
&+\ep\left\|\left\langle\f{\pa_v}{t\pa_z}\right\rangle^{-1}(\pa_z^2+(\pa_v-t\pa_z)^2)\left(\f{|\na|^{\f{s}{2}}}{\langle t\rangle^s}\rmA+\sqrt{\f{\pa_t w}{w}}\tilde{\rmA}\right)P_{\neq}(\Psi\Upsilon_2)\right\|_2. 
\end{align*}
Thus we prove the lemma. 
\end{proof}

\subsection{Estimate of $\Pi_{n,2}$}
The strategy of estimating $\Pi_{n,2}$ is similar to the estimate of $\Pi_{l,2}$. We first get estimate for $P_{\neq} (\Upsilon_2\Pi_{n,2})$, where we use Corollary \ref{Cor: psi-new}. Then we study the elliptic equation for $P_{\neq} (\Upsilon_1\Pi_{n,2})$: 
\beq\label{eq: Pressure-n2-2}
\left\{
\begin{aligned}
&\pa_z\Big((a+\udl{\th})\pa_z(\Upsilon_1\Pi_{n,2})\Big)+\udl{\pa_yv}(\pa_v-t\pa_z)\Big((a+\udl{\th})\udl{\pa_yv}(\pa_v-t\pa_z)(\Upsilon_1\Pi_{n,2})\Big)\\
&\quad\quad \quad \quad\quad 
=2(\udl{\pa_yv})^2(\pa_v\Upsilon_1)\Big(\udl{\th}(\pa_v-t\pa_z)(\Upsilon_2\Pi_{n,2})\Big)
+(\udl{\pa_yv})^2\Big(\udl{\th}(\pa_v^2\Upsilon_1)(\Upsilon_2\Pi_{n,2})\Big),\\
&(\pa_v-t\pa_z)\Pi_{n,2}(t, z, u(0))=(\pa_v-t\pa_z)\Pi_{n,2}(t, z, u(1))=0.
\end{aligned}\right.
\eeq 
The proof is similar. We omit the details. 

\section{Transport terms}\label{sec: transport}
The transport structure in the equation of $a$ and $f$ are similar. Here we only present the proof for $a$. 
We rewrite $I_a^{tr}$ as 
\begin{align*}
I_a^{tr}&=-\f{1}{2}\int \na\cdot (\Upsilon\rmU) |\rmA^*a|^2dzdv
+\int \rmA^* a[\rmA^*(\Upsilon\rmU\cdot \na a)-\Upsilon\rmU\cdot\na\rmA^*a]dzdv\\
&=-\f{1}{2}\int \na\cdot (\Upsilon\rmU) |\rmA^*a|^2dzdv+\f{1}{2\pi}\sum_{\rmN\in \bfD}(\mathrm{T}_{\rmN}^a+\mathrm{R}_{\rmN}^a)+\f{1}{2\pi}\mathcal{R}^a,
\end{align*}
where $\rm{T}$, $\rm{R}$, and $\mathcal{R}$ are given by using the paraproduct decomposition:
\beq\label{eq: paraproduct decomposition}
\begin{aligned}
&\mathrm{T}_{\rmN}^a=2\pi \int \rmA^* a \left[\rmA^*\left((\Upsilon\mathrm{U})_{<\rmN/8}\cdot\na_{z,v}a_{\rmN}\right)-\mathrm{U}_{<\rmN/8}\cdot\na_{z,v}\rmA^*a _{\rmN}\right] dx,\\
&\mathrm{R}_{\rmN}^a=2\pi \int \rmA^*a \left[\rmA^*\left((\Upsilon\mathrm{U})_{\rmN}\cdot\na_{z,v}a_{<\rmN/8}\right)-\mathrm{U}_{\rmN}\cdot\na_{z,v}\rmA^* a_{<\rmN/8}\right] dx,\\
&\mathcal{R}^a=2\pi\sum_{\rmN\in \bf{D}}\sum_{\f18\rmN\leq \rmN'\leq 8\rmN}\int \rmA^* a \left[\rmA^*\left((\Upsilon\mathrm{U})_{\rmN}\cdot\na_{z,v}a_{\rmN'}\right)-(\Upsilon\mathrm{U})_{\rmN}\cdot\na_{z,v}\rmA^* a_{\rmN'}\right] dx.
\end{aligned}
\eeq
Similarly, we have the decomposition for $I^{tr}_f$ as in \eqref{eq: paraproduct decomposition} by replacing $\rmA^*$ by $\rmA$ and replacing $a$ by $f$. We denote them by $\mathrm{T}_{\rmN}^f$, $\mathrm{R}_{\rmN}^f$, and $\mathcal{R}^f$ respectively, also see {\it (2.26)} of \cite{BM2015}. 

By Sobolev embedding, $\s>5$ and the bootstrap hypotheses, we have 
\ben\label{eq: 2.25}
\left|\int \na\cdot(\Upsilon\mathrm{U})|\rmA\Om|^2dx\right|
\leq \|\na(\Upsilon\mathrm{U})\|_{\infty}\|\rmA\Om\|_2^2\lesssim \f{\ep}{\langle t\rangle^{2-\rmK_{\rmD}\ep/2}}\|\rmA\Om\|_2^2\lesssim \f{\ep^3}{\langle t\rangle^{2-\rmK_{\rmD}\ep/2}}.
\een

\subsection{Treatment of $\mathrm{T}_{\rmN}^a$}
We rewrite 
\beq\label{eq: T^a_N}
\begin{aligned}
\mathrm{T}_{\rmN}^a
&=\sum_{k=0}\int \rmA^*_k(\eta)\overline{\widehat{a}_k(\eta)}i\widehat{\udl{\pa_tv}}(\eta-\xi)_{<\f{\rmN}{8}}\xi \rmA^*_k(\xi)\widehat{a}_{k}(\xi)_{\rmN}\left(\sqrt{\f{{\langle t\rangle^2}+{k^2+|\eta|}}{{\langle t\rangle^2}+{k^2+|\xi|}}}-1\right)\f{\rmA_k(\eta)}{\rmA_k(\xi)}d\xi d\eta\\
&\quad+\sum_{k\neq 0}\sum_{l\neq k}\int \rmA^*_k(\eta)\overline{\widehat{a}_k(\eta)}\widehat{\udl{\pa_yv}\Upsilon_2^2\Psi}_{k-l}(\eta-\xi)_{<\f{\rmN}{8}}(k\xi-l\eta) \rmA^*_l(\xi)\widehat{a}_{l}(\xi)_{\rmN}\\
&\quad\quad\quad\quad\quad\quad\quad\quad\quad\quad\quad\quad\quad\quad\quad
\times\left(\sqrt{\f{{\langle t\rangle^2}+{k^2+|\eta|}}{{\langle t\rangle^2}+{l^2+|\xi|}}}-1\right)\f{\rmA_k(\eta)}{\rmA_l(\xi)}d\xi d\eta\\
&\quad+\sum_{k=0}\int \rmA^*_k(\eta)\overline{\widehat{a}_k(\eta)}i\widehat{\udl{\pa_tv}}(\eta-\xi)_{<\f{\rmN}{8}}\xi \rmA^*_k(\xi)\widehat{a}_{k}(\xi)_{\rmN}\Big(\f{\rmA_k(\eta)}{\rmA_l(\xi)}-1\Big)d\xi d\eta\\
&\quad+\sum_{k\neq 0}\sum_{l\neq k}\int \rmA^*_k(\eta)\overline{\widehat{a}_k(\eta)}\widehat{\udl{\pa_yv}\Upsilon_2^2\Psi}_{k-l}(\eta-\xi)_{<\f{\rmN}{8}}(k\xi-l\eta) \rmA^*_l(\xi)\widehat{a}_{l}(\xi)_{\rmN}\Big(\f{\rmA_k(\eta)}{\rmA_l(\xi)}-1\Big)d\xi d\eta\\
&\eqdef \mathrm{T}_{\rmN, 1}^a+\mathrm{T}_{\rmN,2}^a+\mathrm{T}_{\rmN,3}^a+\mathrm{T}_{\rmN,4}^a.
\end{aligned}
\eeq
For $\rmT^f_{N}$ term, we have an even simpler decomposition and there are no $\mathrm{T}_{\rmN, 1}^f$ and $\mathrm{T}_{\rmN,2}^f$. 

By fact that $|l,\xi|\approx |k, \eta|$, $|k\xi-l\eta|\leq |k-l,\eta-\xi||l,\xi|$, and
\begin{align*}
&|l,\xi|\left(\sqrt{\f{{\langle t\rangle^2}+{k^2+|\eta|}}{{\langle t\rangle^2}+{l^2+|\xi|}}}-1\right)\f{\rmA_k(\eta)}{\rmA_l(\xi)}\\
&\lesssim\f{(|k-l|+|\eta-\xi|)|l,\xi|}{\sqrt{{\langle t\rangle^2}+{l^2+|\xi|}}}\f{\rmA_k(\eta)}{\rmA_l(\xi)}
\lesssim e^{11\mu|k-l,\eta-\xi|}|l,\xi|^{s}\langle t\rangle^{2-2s},
\end{align*}
we obtain that
\beno
\begin{aligned}
&\sum_{\rmN}\mathrm{T}_{\rmN, 1}^a
\lesssim \langle t\rangle^{2-2s+2\tilde{q}}\|\udl{\pa_tv}\|_{\mathcal{G}^{\s-6}}\mathcal{CK}_a\lesssim \mathcal{E}_{\udl{\pa_tv}}^{\f12}\mathcal{CK}_a\lesssim \ep\mathcal{CK}_a,\\
&\sum_{\rmN}\mathrm{T}_{\rmN, 2}^a
\lesssim \langle t\rangle^{2-2s+2\tilde{q}}\|\udl{\pa_yv}\Upsilon_2\|_{\mathcal{G}^{\s-6}}\|\Upsilon_2\Psi\|_{\mathcal{G}^{\s-6}}\mathcal{CK}_a\lesssim \ep\mathcal{CK}_a.
\end{aligned}
\eeno
As in \cite{BM2015}, we write
\begin{align*}
\f{\rmA_k(\eta)}{\rmA_l(\xi)}-1
&=[e^{\la|k,\eta|^s-\la|l,\xi|^s}-1]
+e^{\la|k,\eta|^s-\la|l,\xi|^s}\Big[\f{\rmJ_{k}(\eta)}{\rmJ_{l}(\xi)}-1\Big]\f{\langle k, \eta\rangle^{\s}}{\langle l,\xi\rangle^{\s}}\\
&\quad
+e^{\la|k,\eta|^s-\la|l,\xi|^s}\Big[\f{\langle k, \eta\rangle^{\s}}{\langle l, \xi\rangle^{\s}}-1\Big]\\
&=C^{\rmA}_1+C^{\rmA}_2+C^{\rmA}_3.
\end{align*}
It is easy to check that
\beno
||l,\xi|C^{\rmA}_1|\lesssim |l,\xi|^{s}e^{c\la|k-l,\eta-\xi|^s},\quad 
||l,\xi|C^{\rmA}_3|\lesssim e^{c\la|k-l,\eta-\xi|^s}.
\eeno
By Lemma \ref{Lem 3.6}, Lemma \ref{Lem 3.7}, we have by following the steps in {\it section 5} of \cite{BM2015}. 
\begin{align*}
|C_2^{\rmA}|l,\xi||
\lesssim e^{c\la|k-l,\eta-\xi|^s}e^{20\mu|k-l,\eta-\xi|^{\f12}}
\Big(|l,\xi|^{s}+t^2\sqrt{\f{\pa_tw_{k}(t,\eta)}{w_{k}(t,\eta)}}\sqrt{\f{\pa_tw_{l}(t,\xi)}{w_{l}(t,\xi)}}\mathbf{1}_{k\neq l}\Big),
\end{align*}
which gives that
\beno
\begin{aligned}
&\sum_{\rmN}\mathrm{T}_{\rmN, 3}^a
\lesssim \langle t\rangle^{2\tilde{q}}\|\udl{\pa_tv}\|_{\mathcal{G}^{\s-6}}\mathcal{CK}_a\lesssim \mathcal{E}_{\udl{\pa_tv}}^{\f12}\mathcal{CK}_a\lesssim \ep\mathcal{CK}_a,\\
&\sum_{\rmN}\mathrm{T}_{\rmN, 4}^a
\lesssim \langle t\rangle^{2}\|\udl{\pa_yv}\Upsilon_2\|_{\mathcal{G}^{\s-6}}\|\Upsilon_2\Psi\|_{\mathcal{G}^{\s-6}}\mathcal{CK}_a\lesssim \ep\mathcal{CK}_a.
\end{aligned}
\eeno

\subsection{Treatment of $\rmR_{\rmN}^a$}
We rewrite
\begin{align*}
&\rmR_{\rmN}^a
=\sum_{k}\int \rmA^*_k(\eta)^2\overline{\widehat{a}_{k}(\eta)}i(\eta-\xi)
\widehat{\udl{\pa_tv}}(\xi)_{\rmN}\widehat{a}_k(\eta-\xi)_{<\rmN/8}d\xi d\eta\\
&+\sum_{\rmM\geq 8}\sum_{k\neq 0, l\neq 0}\int \rmA^*_k(\eta)^2\overline{\widehat{a}_{k}(\eta)}i(k\zeta-l\eta)
\Big(\widehat{\udl{\pa_yv}\Upsilon_2}(\xi-\zeta)_{<\rmM/8}\widehat{\Upsilon_2\Psi}_{l}(\zeta)_{\rmM}\Big)_{\rmN}\widehat{a}_{k-l}(\eta-\xi)_{<\rmN/8}d\xi d\eta d\zeta\\
&+\sum_{\rmM\geq 8}\sum_{k\neq 0, l\neq 0}\int \rmA^*_k(\eta)^2\overline{\widehat{a}_{k}(\eta)}i(k\zeta-l\eta)
\Big(\widehat{\udl{\pa_yv}\Upsilon_2}(\xi-\zeta)_{\rmM}\widehat{\Upsilon_2\Psi}_{l}(\zeta)_{<\rmM/8}\Big)_{\rmN}
\widehat{a}_{k-l}(\eta-\xi)_{<\rmN/8}d\xi d\eta d\zeta\\
&+\sum_{\rmM}\sum_{\rmM/8 \leq \rmM'\leq 8\rmM}\sum_{k\neq 0, l\neq 0}\int \rmA^*_k(\eta)^2\overline{\widehat{a}_{k}(\eta)}i(k\zeta-l\eta)
\Big(\widehat{\udl{\pa_yv}\Upsilon_2}(\xi-\zeta)_{\rmM}\widehat{\Upsilon_2\Psi}_{l}(\zeta)_{\rmM'}\Big)_{\rmN}\\
&\qquad\qquad\qquad\qquad\qquad\qquad\qquad\qquad\qquad\qquad\qquad\qquad\qquad
\times\widehat{a}_{k-l}(\eta-\xi)_{<\rmN/8}d\xi d\eta d\zeta\\
&-\sum_{k,l}\int \rmA^*_k(\eta)\overline{\widehat{a}_{k}(\eta)}
\Big(\widehat{(\Upsilon_2^2\rmU)}_{l}(\xi)\Big)_{\rmN}\cdot \rmA^*_{k-l}(\eta-\xi)\widehat{\na a}_{k-l}(\eta-\xi)_{<\rmN/8}d\xi d\eta\\
&=\rmR_{\rmN,1}^a+\rmR_{\rmN,2}^a+\rmR_{\rmN,3}^a+\rmR_{\rmN,4}^a+\rmR_{\rmN,5}^a.
\end{align*}
\subsubsection{Treatment of $\rmR_{\rmN,1}^a$}
For $\rmR_{\rmN,1}^a$, we have
\begin{align*}
\sqrt{\f{{\langle t\rangle^2}+{k^2+|\eta|}}{{\langle t\rangle^2}+{k^2+|\eta-\xi|}}}\lesssim \sqrt{|\eta|}\mathbf{1}_{t+|k|\leq \sqrt{\eta}}+\mathbf{1}_{t+|k|\geq \sqrt{\eta}}
\end{align*}
which together with Lemma \ref{Lem 3.6} gives that 
\begin{align*}
|\rmR_{\rmN,1}^a|
&\lesssim \sum_{k}\int \rmA^*_k(\eta)|\overline{\widehat{a}_{k}(\eta)}|\f{\rmA_k(\eta)}{\rmA_0(\xi)}\Big(\sqrt{|\eta|}\mathbf{1}_{t+|k|\leq \sqrt{\eta}}+\mathbf{1}_{t+|k|\geq \sqrt{\eta}}\Big)
\rmA_0(\xi)|\widehat{\udl{\pa_tv}}(\xi)_{\rmN}||\widehat{\na a}_k(\eta-\xi)_{<\rmN/8}|d\xi d\eta\\
&\lesssim \|a\|_{\mathcal{G}^{\s-6}}\|\rmA_0|\pa_v|^{\f12}(\udl{\pa_tv})_{\rmN}\|_{L^2}\|\rmA^* a_{\sim\rmN}\|_{L^2}
+\langle t\rangle\|a\|_{\mathcal{G}^{\s-6}}\|\rmA_0(\udl{\pa_tv})_{\rmN}\|_{L^2}\|\rmA^* a_{\sim\rmN}\|_{L^2},
\end{align*}
which together with the 
\ben
\sum_{\rmN}|\rmR_{\rmN,1}^a| \lesssim \f{\ep^3}{\langle t\rangle^{2s}}.
\een
We have for $|k-l, \eta-\xi|\leq \f{6}{32}|l,\xi|$ that
\beq\label{eq: A* to lower}
\begin{aligned}
&\f{\rmB_k(t,\eta)}{\rmB_{k-l}(t,\eta-\xi)}=\sqrt{\f{{\langle t\rangle^2}+{k^2+|\eta|}}{{\langle t\rangle^2}+{(k-l)^2+|\eta-\xi|}}}\\
&\lesssim \Big(1
+\f{|\zeta|^{\f12}}{\langle t\rangle}\mathbf{1}_{\langle t\rangle \leq \f{1}{8}|\zeta|^{\f12},\ |l|^2 \leq \f{1}{8}|\zeta|}
+\f{|l|}{\langle t\rangle}\mathbf{1}_{\langle t\rangle \leq \f{1}{8}|l|,\ |l|^2 \geq 8|\zeta|}\Big)\langle k-l,\eta-\xi\rangle\langle \zeta-\xi\rangle
\end{aligned}
\eeq
\subsubsection{Treatment of $\rmR_{\rmN, 2}^a$}
For $\rmR_{\rmN, 2}^a$, we have
\beno
|\xi-\zeta|\leq \f{6}{32}|l,\zeta|,\ |k-l,\eta-\xi|\leq \f{6}{32}|l,\xi|
\eeno
which gives $|l,\zeta|\approx |l,\xi|\approx |k,\eta|$. We then have
\begin{align*}
&\f{\rmA^*_k(\eta)}{\rmA_l(\zeta)\sqrt{{\langle t\rangle^2}+{(k-l)^2+|\eta-\xi|}}}\\
&\lesssim \f{\rmA_k(\eta)}{\rmA_l(\zeta)}\sqrt{\f{{\langle t\rangle^2}+{k^2+|\eta|}}{{\langle t\rangle^2}+{(k-l)^2+|\eta-\xi|}}}\\
&\lesssim \f{\rmA_k(\eta)}{\rmA_l(\zeta)}\langle k-l,\eta-\xi\rangle\langle \zeta-\xi\rangle
\Big(1
+\f{|\zeta|^{\f12}}{\langle t\rangle}\mathbf{1}_{\langle t\rangle \leq \f{1}{8}|\zeta|^{\f12},\ |l|^2 \leq \f{1}{8}|\zeta|}
+\f{|l|}{\langle t\rangle}\mathbf{1}_{\langle t\rangle \leq \f{1}{8}|l|,\ |l|^2 \geq 8|\zeta|}\Big)\\
&\eqdef C_4^{\rmA}+C_5^{\rmA}+C_6^{\rmA}.
\end{align*}
We then have by following the argument in {\it section 6} or \cite{BM2015} that
\beq\label{eq: est C_4}
\begin{aligned}
|l,\zeta|C_4^{\rmA}
\lesssim &e^{c\la |k-l,\eta-\xi|^s}e^{c\la|\xi-\zeta|^s}e^{20\mu |k-l,\eta-\xi|^{\f12}}e^{20\mu |\zeta-\xi|^{\f12}}\\
&\times \bigg(|k,\eta|^{s/2}|l,\zeta|^{1-s/2}\mathbf{1}_{t\notin \mathbf{I}_{l,\zeta}}\\
&\quad+\sqrt{\f{\pa_tw_k(\eta)}{w_{k}(\eta)}}\sqrt{\f{\pa_tw_l(\zeta)}{w_{l}(\zeta)}}\f{|\zeta|^2}{l^2\langle \zeta/t-l\rangle^2}\mathbf{1}_{2|\eta|^{\f12}\leq t\leq 2|\eta|}\mathbf{1}_{t\notin \mathbf{I}_{l,\zeta}}\\
&\quad+\Big[\sqrt{\f{\pa_tw_k(\eta)}{w_{k}(\eta)}}+\f{|k,\eta|^{s/2}}{\langle t\rangle^s}\Big]\f{w_{\rmR}(\zeta)}{w_{\rmN\rmR}(\zeta)}\sqrt{\f{w_l(\zeta)}{\pa_tw_l(\zeta)}}|l,\zeta| \mathbf{1}_{t\in \mathbf{I}_{l,\zeta}}\\
&\quad+\sqrt{\f{\pa_tw_k(\eta)}{w_{k}(\eta)}}\sqrt{\f{w_l(\zeta)}{\pa_tw_l(\zeta)}}|l|\mathbf{1}_{t\in \mathbf{I}_{l,\zeta}}\bigg)\\
&\lesssim e^{c\la |k-l,\eta-\xi|^s}e^{c\la|\xi-\zeta|^s}\Big[\sqrt{\f{\pa_tw_k(\eta)}{w_{k}(\eta)}}+\f{|k,\eta|^{s/2}}{\langle t\rangle^s}\Big]\\
&\quad\quad\quad
\times\left\langle \f{\zeta}{l t}\right\rangle^{-1}\Big[\f{|l,\zeta|^{\f{s}{2}}}{\langle t\rangle^s}+\sqrt{\f{\pa_tw_l(\zeta)}{w_{l}(\zeta)}}\Big](l^2+(\zeta-lt)^2)\mathbf{1}_{l\neq 0}. 
\end{aligned}
\eeq
The last inequality follows directly from {\it Lemma 6.1} in \cite{BM2015}. 
We also have by Lemma \ref{Lem 3.6} that
\begin{align*}
|l,\zeta|C_5^{\rmA}&\lesssim e^{c\la |k-l,\eta-\xi|^s}e^{c\la|\xi-\zeta|^s}|\zeta|^{\f32}\langle t\rangle^{-1}\mathbf{1}_{\langle t\rangle \leq \f{1}{8}|\zeta|^{\f12},\ |lt|\leq \f12|\zeta|}
\\
&\lesssim e^{c\la |k-l,\eta-\xi|^s}e^{c\la|\xi-\zeta|^s}\f{1}{\langle \zeta\rangle^{\f12}\langle t\rangle}(l^2+(\zeta-lt)^2)\mathbf{1}_{\langle t\rangle \leq \f{1}{8}|\zeta|^{\f12}}\\
&\lesssim e^{c\la |k-l,\eta-\xi|^s}e^{c\la|\xi-\zeta|^s}\f{|\zeta|^{\f12}}{|l|\langle t\rangle^2}(l^2+(\zeta-lt)^2)\left\langle \f{\zeta}{l t}\right\rangle^{-1}\\
&\lesssim e^{c\la |k-l,\eta-\xi|^s}e^{c\la|\xi-\zeta|^s}\f{|\eta|^{\f{s}{2}}|\zeta|^{\f{s}{2}}}{\langle t\rangle^2}(l^2+(\zeta-lt)^2)\left\langle \f{\zeta}{l t}\right\rangle^{-1}
\end{align*}
and
\begin{align*}
|l,\zeta|C_6^{\rmA}&\lesssim e^{c\la |k-l,\eta-\xi|^s}e^{c\la|\xi-\zeta|^s}\f{|l|}{\langle t\rangle}\Big(\mathbf{1}_{\langle t\rangle \leq \f{1}{8}|l|,\ |l|^2 \geq 8|\zeta|,\ |lt|\geq 2|\zeta|}
+\mathbf{1}_{\langle t\rangle \leq \f{1}{8}|l|,\ |l|^2 \geq 8|\zeta|,\ |lt|\leq 2|\zeta|}\Big)
\\
&\lesssim e^{c\la |k-l,\eta-\xi|^s}e^{c\la|\xi-\zeta|^s}\f{1}{\langle t\rangle^2}(l^2+(\zeta-lt)^2)\left\langle \f{\zeta}{l t}\right\rangle^{-1}.
\end{align*}
Therefore, we have
\begin{align*}
\sum_{\rmN}|\rmR^a_{\rmN,2}|
\lesssim &\mathcal{CK}_a^{\f12}\|\udl{\pa_yv}\Upsilon_2\|_{\mathcal{G}^{\s-6}}\|\rmB a\|_{\mathcal{G}^{\s-6}}\\
&\times\left\|\left\langle\f{\pa_v}{t\pa_z}\right\rangle^{-1}(\pa_z^2+(\pa_v-t\pa_z)^2)\left(\f{|\na|^{\f{s}{2}}}{\langle t\rangle^s}\rmA+\sqrt{\f{\pa_t w}{w}}\tilde{\rmA}\right)P_{\neq}(\Psi\Upsilon_2)\right\|_2.
\end{align*}
For the term $\rmR^f_{\rmN,2}$, the estimate is simpler where the estimates of $C_5^{\rmA}$ and $C_6^{\rmA}$ are not needed. 
\subsubsection{Treatment of $\rmR_{\rmN, 3}^a$}
The treatment of $\rmR_{\rmN, 3}^a$ is similar to the treatment of {\it $\rmR^{\ep,1}_{\rmN;\rmH\rmL}$ in section 6.2.1} of \cite{BM2015}. Here we present the key steps and main estimates.  In support of the integrand function, we have 
\beno
|\xi-\zeta|\approx |l,\xi|\approx |k,\eta|.
\eeno
We have
\ben\label{eq: B-to lower}
\begin{aligned}
&\f{\rmB_k(t,\eta)}{\rmB_{k-l}(t,\eta-\xi)}=\sqrt{\f{{\langle t\rangle^2}+{k^2+|\eta|}}{{\langle t\rangle^2}+{(k-l)^2+|\eta-\xi|}}}\\
&\lesssim \Big(1
+\f{|\zeta-\xi|^{\f12}}{\langle t\rangle}\mathbf{1}_{\langle t\rangle \leq \f{1}{8}|\zeta-\xi|^{\f12},\ |l|^2 \leq \f{1}{8}|\zeta-\xi|}
+\f{|l|}{\langle t\rangle}\mathbf{1}_{\langle t\rangle \leq \f{1}{8}|l|,\ |l|^2 \geq 8|\zeta-\xi|}\Big). 
\end{aligned}
\een
We will divide the integral based on the relative size of $l$ and $\xi$. For the case $16|l|\geq |\xi|$, we have on the support of the integrand, 
\beno
|k-l,\eta-\xi|\leq \f{3}{16}|l,\xi|,\quad |\xi-\zeta|\lesssim |l|,
\eeno
then by \eqref{eq: B-to lower}, we have
\begin{align*}
\f{|l,\zeta|\rmA^*_k(\eta)}{\rmA_l(\xi)\rmB_{k-l}(t,\eta-\xi)}
&=|l,\zeta|\f{\rmB_k(t,\eta)}{\rmB_{k-l}(t,\eta-\xi)}\f{\rmA_k(\eta)}{\rmA_l(\xi)}\\
&\lesssim |l|\Big(1
+\f{|l|}{\langle t\rangle}\mathbf{1}_{\langle t\rangle \leq \f{1}{8}|l|,\ |l|^2 \geq 8|\zeta-\xi|}\Big)\f{\rmA_k(\eta)}{\rmA_l(\xi)}
\eqdef C_7^{\rmA}+C_8^{\rmA}.
\end{align*}
We then get
\begin{align}
\label{eq: C7}C_7^{\rmA}&\lesssim e^{c\la|k-l,\eta-\xi|^{s}}e^{c\la|\zeta-\xi|^s}|k,\eta|^{\f{s}{2}}|l|^{1-\f{s}{2}}\mathbf{1}_{t\notin \mathbf{I}_{l,\xi}, l\neq 0}\\
\nonumber&\lesssim e^{c\la|k-l,\eta-\xi|^{s}}e^{c\la|\zeta-\xi|^s}\f{|k,\eta|^{\f{s}{2}}}{\langle t\rangle^s}\left\langle \f{\zeta}{l t}\right\rangle^{-1}\f{|l,\zeta|^{\f{s}{2}}}{\langle t\rangle^s}(l^2+(\zeta-lt)^2)\mathbf{1}_{l\neq 0},\\
C_8^{\rmA}&\lesssim e^{c\la|k-l,\eta-\xi|^{s}}e^{c\la|\zeta-\xi|^s}\f{|l|^2}{\langle t\rangle^2}\mathbf{1}_{\langle t\rangle \leq \f{1}{8}|l|,\ |l|^2 \geq 8|\zeta-\xi|}\\
\nonumber&\lesssim e^{c\la|k-l,\eta-\xi|^{s}}e^{c\la|\zeta-\xi|^s}\f{1}{\langle t\rangle^2}(l^2+(\zeta-lt)^2)\left\langle \f{\zeta}{l t}\right\rangle^{-1}\mathbf{1}_{l\neq 0}.
\end{align}
For the case $16|l|<|\xi|$, we have on the support of the integrand,
\beno
|k-l,\eta-\xi|\leq \f{3}{16}|l,\xi|,\quad |l,\zeta|\leq \f{67}{100}|\xi-\zeta|,
\eeno
and then 
\begin{align*}
&\f{|l,\zeta|\rmA^*_k(\eta)}{\rmA^{\rmR}(\xi-\zeta)\rmB_{k-l}(t,\eta-\xi)}
=|l,\zeta|\f{\rmB_k(t,\eta)}{\rmB_{k-l}(t,\eta-\xi)}\f{\rmA_k(\eta)}{\rmA^{\rmR}(\xi-\zeta)}\\
&\lesssim |l,\zeta|\Big(1
+\f{|\zeta-\xi|^{\f12}}{\langle t\rangle}\mathbf{1}_{\langle t\rangle \leq \f{1}{8}|\zeta-\xi|^{\f12},\ |l|^2 \leq \f{1}{8}|\zeta-\xi|}
+\f{|l|}{\langle t\rangle}\mathbf{1}_{\langle t\rangle \leq \f{1}{8}|l|,\ |l|^2 \geq 8|\zeta-\xi|}\Big)\f{\rmA_k(\eta)}{\rmA^{\rmR}(\xi-\zeta)}\\
&\lesssim \Big(1
+\f{|\zeta-\xi|^{\f12}}{\langle t\rangle}\mathbf{1}_{\langle t\rangle \leq \f{1}{8}|\zeta-\xi|^{\f12},\ |l|^2 \leq \f{1}{8}|\zeta-\xi|}\Big)
e^{c\la|k-l,\eta-\xi|}e^{c\la|l,\zeta|}
\eqdef C^{\rmA}_9+C^{\rmA}_{10}.
\end{align*}
We get that
\begin{align}
\label{eq: C9}&C^{\rmA}_9\lesssim e^{c\la|k-l,\eta-\xi|}e^{c\la|l,\zeta|}\\
&C^{\rmA}_{10}\lesssim e^{c\la|k-l,\eta-\xi|}e^{c\la|l,\zeta|}\f{|k,\eta|^{\f{s}{2}}|\zeta-\xi|^{\f{s}{2}}}{\langle t\rangle}.
\end{align}
Combining all the estimates, we arrive at
\begin{align*}
\sum_{\rmN}|\rmR^a_{\rmN,3}|
\lesssim &\mathcal{CK}_a^{\f12}\|\udl{\pa_yv}\Upsilon_2\|_{\mathcal{G}^{\s-6}}\|\rmB a\|_{\mathcal{G}^{\s-6}}\\
&\times\left\|\left\langle\f{\pa_v}{t\pa_z}\right\rangle^{-1}(\pa_z^2+(\pa_v-t\pa_z)^2)\left(\f{|\na|^{\f{s}{2}}}{\langle t\rangle^s}\rmA+\sqrt{\f{\pa_t w}{w}}\tilde{\rmA}\right)P_{\neq}(\Psi\Upsilon_2)\right\|_2\\
&+\left\|\f{\langle \pa_v\rangle^{\f{s}{2}}}{\langle t\rangle^{s}}\rmA^{*}a\right\|_{L^2}\left\|\f{\langle \pa_v\rangle^{\f{s}{2}}}{\langle t\rangle^{s}}\rmA^{\rmR}(\udl{\pa_yv}\Upsilon_2)\right\|_{L^2}\langle t\rangle^{2s-1}\|\Upsilon_2\Psi\|_{\mathcal{G}^{\s-6}}\|\rmB a\|_{\mathcal{G}^{\s-6}}\\
&+\f{1}{\langle t\rangle^2}\left\|\rmA^{*}a\right\|_{L^2}\left\|\rmA^{\rmR}(\udl{\pa_yv}\Upsilon_2)\right\|_{L^2}\langle t\rangle^{2}\|\Upsilon_2\Psi\|_{\mathcal{G}^{\s-6}}\|\rmB a\|_{\mathcal{G}^{\s-6}}\\
\lesssim &\ep \mathcal{CK}_a^{\f12}\left\|\left\langle\f{\pa_v}{t\pa_z}\right\rangle^{-1}(\pa_z^2+(\pa_v-t\pa_z)^2)\left(\f{|\na|^{\f{s}{2}}}{\langle t\rangle^s}\rmA+\sqrt{\f{\pa_t w}{w}}\tilde{\rmA}\right)P_{\neq}(\Psi\Upsilon_2)\right\|_2\\
&+\ep^2 \mathcal{CK}_a^{\f12}\left(\f{1}{\langle t\rangle^s}+\mathcal{CK}_{h}^{\f12}+\mathcal{CK}_{\varphi^{\d}_9}^{\f12}\right)+\f{\ep^3}{\langle t\rangle^2}. 
\end{align*}

\subsubsection{Treatment of $\rmR_{\rmN,4}^a$}
The treatment of $\rmR_{\rmN,4}^a$ is similar to that of {\it $\rmR_{\rmN, \rmH\rmL}^{\ep,1}$ in section 6.2.1} of \cite{BM2015}. 
We divide it into two cases $|l|> 100|\zeta|$ and $|l|\leq 100|\zeta|$. 

For the first case $|l|> 100|\zeta|$, on the support of the integrand, we have
\beno
|k-l,\eta-\xi|\leq \f{6}{32}|l,\xi|,\quad |\xi-\zeta|\leq 24|\zeta|\leq \f{24}{100}|l,\zeta|,
\eeno
which gives 
\beno
|\xi|\lesssim |\zeta|\lesssim |l|,\quad |l,\xi|\approx |l|\approx |k,\eta|
\eeno
\beno
\begin{aligned}
&|l,\zeta|\f{\rmA^*_k(\eta)}{\rmA_l(\zeta)\rmB_{k-l}(t,\eta-\xi)}\lesssim |l|\f{\rmA_k(\eta)}{\rmA_l(\zeta)}\f{\rmB_k(t,\eta)}{\rmB_{k-l}(t,\eta-\xi)}=\f{\rmA_k(\eta)}{\rmA_l(\zeta)}\sqrt{\f{{\langle t\rangle^2}+{k^2+|\eta|}}{{\langle t\rangle^2}+{(k-l)^2+|\eta-\xi|}}}\\
&\lesssim |l|^{1-\f{s}{2}}|k,\eta|^{\f{s}{2}}\Big(1
+\f{|l|}{\langle t\rangle}\mathbf{1}_{|l|\geq 100t}\Big)\mathbf{1}_{t\notin\mathbf{I}_{l,\zeta}}e^{\la|k-l,\eta-\xi|^s}e^{\la |\xi-\zeta|^{s}}\\
&\lesssim \f{|k,\eta|^{\f{s}{2}}}{\langle t\rangle^s}\left\langle \f{\zeta}{l t}\right\rangle^{-1}\f{|l,\zeta|^{\f{s}{2}}}{\langle t\rangle^s}(l^2+(\zeta-lt)^2)\mathbf{1}_{l\neq 0}e^{\la|k-l,\eta-\xi|^s}e^{\la |\xi-\zeta|^{s}}.
\end{aligned}
\eeno

For the second case $|l|\leq 100|\zeta|$, on the support of the integrand, we have
\beno
|k-l,\eta-\xi|\leq \f{3}{16}|l,\xi|,\quad |\xi-\zeta|\leq 24|\zeta|\leq 24|l,\zeta|,\quad |l,\zeta|\leq 101|\zeta|\leq 2424|\xi-\zeta|
\eeno
We have 
\beno
e^{\la|k,\eta|^s}\leq e^{\la|l,\xi|^s+c\la|k-l,\eta-\xi|^s}\leq e^{c\la|l,\zeta|^s+c\la|\xi-\zeta|^s+c\la|k-l,\eta-\xi|^s}
\eeno
and 
\beno
\rmJ_{k}(\eta)\lesssim e^{2\mu|k,\eta|^{\f12}}\leq e^{2\mu|l,\zeta|^{\f12}+2\mu|\xi-\zeta|^{\f12}+2\mu|k-l,\eta-\xi|^{\f12}}
\eeno
and
\beno
\f{\rmB_k(t,\eta)}{\rmB_{k-l}(t,\eta-\xi)}
\lesssim \langle k,\eta\rangle\lesssim 1+|l,\zeta|+|\xi-\zeta|+|k-l,\eta-\xi|
\eeno
Thus we have 
\ben
\sum_{\rmN\geq 8}|\rmR_{\rmN,4}^a|\lesssim \f{\ep^3}{\langle t\rangle^2}+\ep\mathcal{CK}_a+\ep\left\|\left\langle\f{\pa_v}{t\pa_z}\right\rangle^{-1}(\pa_z^2+(\pa_v-t\pa_z)^2)\f{|\na|^{\f{s}{2}}}{\langle t\rangle^s}\rmA P_{\neq}(\Psi\Upsilon_2)\right\|_2^2. 
\een

\subsubsection{Treatment of $\rmR_{\rmN, 5}^a$}
The treatment of $\rmR_{\rmN, 5}^a$ is simple. We have
\begin{align*}
|\rmR_{\rmN, 5}^a|\lesssim \|\rmA^*a_{\sim \rmN}\|_{L^2}\|(\Upsilon_2\rmU)_{\rmN}\|_{H^{\s-6}}\|\rmA^*a_{<\rmN/8}\|_{L^2},
\end{align*}
which gives that 
\beno
\sum_{\rmN\geq 8}|\rmR_{\rmN, 5}^a|\lesssim \f{\ep^3}{\langle t\rangle^{2-\rmK_{\rm{D}}\ep/2}}. 
\eeno

\subsection{Treatment of $\mathcal{R}^a$}
The treatment of $\mathcal{R}^a$ is similar to that of $\rmR_{\rmN,4}^a$, one can follow {\it section 7} of \cite{BM2015} and obtain that
\beno
|\mathcal{R}^a|\lesssim \f{\ep^3}{\langle t\rangle^{2-\rmK_{\rm{D}}\ep/2}}. 
\eeno

\section{The lower order terms in $I_f$}\label{sec: low-order}
In this section, we estimate $I_{f,3}$ and $I_{f,9}$. 
Recall that
\begin{align*}
I_{f,3}
&=-\f{1}{2\pi}\sum_{|k|\geq k_{M}}\int_{\mathbb{R}}\rmA_k(t,\eta)^2\widehat{\Big(\udl{\varphi_2}\pa_z\Psi\Om\Big)}_{k}(t,\eta)\overline{\widehat{f}_k(t, \eta)}d\eta\\
&\quad-\f{1}{2\pi}\sum_{|k|\geq k_{M}}\int_{\mathbb{R}}\rmA_k(t,\eta)^2\mathcal{F}\Big(\udl{\varphi_3}\pa_z\Psi\udl{\pa_yv}(\pa_v-t\pa_z)\Psi\Big)_{k}(t,\eta)\overline{\widehat{f}_k(t, \eta)}d\eta\\
&=I_{f,3}^1+I_{f,3}^2,
\end{align*}
and
\begin{align*}
I_{f,9}
&=-\f{1}{2\pi}\sum_{0<|k|<k_{M}}\int_{\mathbb{R}}\rmA_k(t,\eta)^2\mathcal{F}_2\Big(\bfD_{u,k}\left(\mathcal{F}_{1}\Big(\udl{\varphi_2}\pa_z\Psi\Om\Big)\right)\Big)_{k}(t,\eta)\overline{\widehat{f}_k(t, \eta)}d\eta\\
&\quad-\f{1}{2\pi}\sum_{0<|k|<k_{M}}\int_{\mathbb{R}}\rmA_k(t,\eta)^2\mathcal{F}_2\Big(\bfD_{u,k}\left(\mathcal{F}_{1}\Big(\udl{\varphi_3}\pa_z\Psi\udl{\pa_yv}(\pa_v-t\pa_z)\Psi\Big)\right)\Big)_{k}(t,\eta)\overline{\widehat{f}_k(t, \eta)}d\eta\\
&=I_{f,9}^1+I_{f,9}^2.
\end{align*}
The estimates are similar to each other. Here we present the treatment of the most technical term $I_{f,9}^1$ which involves the wave operator. By Proposition \ref{prop: kernel-wave-op}, we have 
\begin{align*}
|I_{f,9}^1|
&\lesssim \left|\sum_{0<|k|<k_{M}}\sum_{l\neq 0}\int_{\mathbb{R}^4}\rmA_k(\eta)^2\mathcal{D}(t, k, \eta,\xi_1)\widehat{\udl{\varphi_2}}(\xi_1-\xi_2)il\widehat{\Upsilon_2\Psi}_l(\xi_3)\widehat{\Om}_{k-l}(\xi_2-\xi_3)\overline{\widehat{f}_k(\eta)}d\eta d\xi_1 d\xi_2d\xi_3\right|\\
&\lesssim \sum_{0<|k|<k_{M}}\sum_{l\neq 0}\int_{\mathbb{R}^4}\rmA_k(\eta)^2e^{-\la_{\mathcal{D}}|\eta-\xi_1|^s}|\widehat{\udl{\varphi_2}}(\xi_1-\xi_2)||l\widehat{\Upsilon_2\Psi}_l(\xi_3)||\widehat{\Om}_{k-l}(\xi_2-\xi_3)||\overline{\widehat{f}_k(\eta)}|d\eta d\xi_1 d\xi_2d\xi_3\\
&\lesssim \sum_{0<|k|<k_{M}}\sum_{l\neq 0}\int_{\mathbb{R}^4}e^{-c\la_{\mathcal{D}}|\eta-\xi_1|^s}\rmA_k(\xi_1)|\widehat{\udl{\varphi_2}}(\xi_1-\xi_2)||l\widehat{\Upsilon_2\Psi}_l(\xi_3)||\widehat{\Om}_{k-l}(\xi_2-\xi_3)||\overline{\widehat{\rmA f}_k(\eta)}|d\eta d\xi_1 d\xi_2d\xi_3\\
&= \sum_{\rmN\geq 8}\sum_{0<|k|<k_{M}}\sum_{l\neq 0}\int_{\mathbb{R}^4}e^{-c\la_{\mathcal{D}}|\eta-\xi_1|^s}\rmA_k(\xi_1)\Big(|\widehat{\udl{\varphi_2}}(\xi_1-\xi_2)||l\widehat{\Upsilon_2\Psi}_l(\xi_3)|\Big)_{<\rmN/8}\\
&\qquad\qquad\qquad\qquad\qquad\qquad
\times |\widehat{\Om}_{k-l}(\xi_2-\xi_3)_{\rmN}||\overline{\widehat{\rmA f}_k(\eta)}|d\eta d\xi_1 d\xi_2d\xi_3\\
&\quad+\sum_{\rmN\geq 8}\sum_{\rmM\geq 8}\sum_{0<|k|<k_{M}}\sum_{l\neq 0}\int_{\mathbb{R}^4}e^{-c\la_{\mathcal{D}}|\eta-\xi_1|^s}\rmA_k(\xi_1)\Big(|\widehat{\udl{\varphi_2}}(\xi_1-\xi_2)_{\rmM}||l\widehat{\Upsilon_2\Psi}_l(\xi_3)_{<\rmM/8}|\Big)_{\rmN}\\
&\qquad\qquad\qquad\qquad\qquad\qquad
\times |\widehat{\Om}_{k-l}(\xi_2-\xi_3)_{<\rmN/8}||\overline{\widehat{\rmA f}_k(\eta)}|d\eta d\xi_1 d\xi_2d\xi_3\\
&\quad+\sum_{\rmN\geq 8}\sum_{\rmM\geq 8}\sum_{0<|k|<k_{M}}\sum_{l\neq 0}\int_{\mathbb{R}^4}e^{-c\la_{\mathcal{D}}|\eta-\xi_1|^s}\rmA_k(\xi_1)\Big(|\widehat{\udl{\varphi_2}}(\xi_1-\xi_2)_{<\rmM/8}||l\widehat{\Upsilon_2\Psi}_l(\xi_3)_{\rmM}|\Big)_{\rmN}\\
&\qquad\qquad\qquad\qquad\qquad\qquad
\times |\widehat{\Om}_{k-l}(\xi_2-\xi_3)_{<\rmN/8}||\overline{\widehat{\rmA f}_k(\eta)}|d\eta d\xi_1 d\xi_2d\xi_3\\
&\quad+\sum_{\rmN\geq 8}\sum_{\rmM}\sum_{\f18\rmM\leq \rmM'\leq 8\rmM}\sum_{0<|k|<k_{M}}\sum_{l\neq 0}\int_{\mathbb{R}^4}e^{-c\la_{\mathcal{D}}|\eta-\xi_1|^s}\rmA_k(\xi_1)\Big(|\widehat{\udl{\varphi_2}}(\xi_1-\xi_2)_{\rmM'}||l\widehat{\Upsilon_2\Psi}_l(\xi_3)_{\rmM}|\Big)_{\rmN}\\
&\qquad\qquad\qquad\qquad\qquad\qquad
\times |\widehat{\Om}_{k-l}(\xi_2-\xi_3)_{<\rmN/8}||\overline{\widehat{\rmA f}_k(\eta)}|d\eta d\xi_1 d\xi_2d\xi_3\\
&\quad+\sum_{\rmN}\sum_{\f18\rmN\leq \rmN'\leq 8\rmN}\sum_{0<|k|<k_{M}}\sum_{l\neq 0}\int_{\mathbb{R}^4}e^{-c\la_{\mathcal{D}}|\eta-\xi_1|^s}\rmA_k(\xi_1)\Big(|\widehat{\udl{\varphi_2}}(\xi_1-\xi_2)||l\widehat{\Upsilon_2\Psi}_l(\xi_3)|\Big)_{\rmN'}\\
&\qquad\qquad\qquad\qquad\qquad\qquad
\times |\widehat{\Om}_{k-l}(\xi_2-\xi_3)_{\rmN}||\overline{\widehat{\rmA f}_k(\eta)}|d\eta d\xi_1 d\xi_2d\xi_3\\
&\eqdef I_{f,9, \rmL\rmH}^1+I_{f,9, \rmH\rmL,\rmH\rmL}^1
+I_{f,9, \rmH\rmL,\rmL\rmH}^1+I_{f,9, \rmH\rmL,\rmH\rmH}^1
+I_{f,9, \rmH\rmH}^1. 
\end{align*}
\subsection{The low-high interactions}
For the low-high interaction, the estimate is easy. We have by Lemma \ref{Lem 3.6}
\begin{align*}
\f{\rmA_k(\xi_1)}{\rmA_{k-l}(\xi_2-\xi_3)}\lesssim e^{c\la |\xi_1-\xi_2|+c\la|l,\xi_3|}\Big(1+\f{\xi_1}{k^2}\mathbf{1}_{t\in \mathbf{I}_{k,\xi_1}}\Big)\lesssim e^{c\la |\xi_1-\xi_2|+c\la|l,\xi_3|}\langle k-l, \xi_2-\xi_3\rangle^{\f12}\langle t\rangle^{\f12},
\end{align*}
which gives that
\ben
\begin{aligned}
I_{f,9, \rmL\rmH}^1
&\lesssim_{k_{M}} \|\udl{\varphi_2}\|_{\mathcal{G}^{\s-6}}\langle t\rangle^2\|\Upsilon_2\Psi\|_{\mathcal{G}^{\s-6}}\left\|\f{|\na|^{\f{s}{2}}}{\langle t\rangle^{\f34}}\rmA \Om\right\|_2\mathcal{CK}_f^{\f12}
\lesssim_{k_{M}} \ep \mathcal{CK}_f. 
\end{aligned}
\een

\subsection{The high-low interactions}
The treatment of $I_{f,9, \rmH\rmL,\rmH\rmL}^1$ is similar to the Reaction term $\mathrm{R}_{\rmN, 3}^a$, but simpler. We will divide the integral based on the relative size of $l$ and $\xi_3$. For the case $16|l|\geq |\xi_3|$, we have on the support of the integrand, 
\beno
|l,\xi_3|\leq \f{3}{16}|\xi_1-\xi_2|,\quad |k-l,\xi_2-\xi_3|\leq \f{3}{16}|l,\xi_1-\xi_2+\xi_3|\lesssim |\xi_1-\xi_2|.
\eeno
Then by \eqref{eq: C7}, we have 
\begin{align*}
|l|\f{\rmA_k(\xi_1)}{\rmA_l(\xi_3)}
&\lesssim 
e^{c\la |k-l,\xi_2-\xi_3|^s+c\la|\xi_1-\xi_2|^s}
\f{|k,\xi_1|^{\f{s}{2}}}{\langle t\rangle^s}\left\langle \f{\xi_3}{l t}\right\rangle^{-1}\f{|l,\xi_3|^{\f{s}{2}}}{\langle t\rangle^s}(l^2+(\xi_3-lt)^2)\mathbf{1}_{l\neq 0}\\
&\lesssim 
\langle \xi_1-\eta\rangle e^{c\la |k-l,\xi_2-\xi_3|^s+c\la|\xi_1-\xi_2|^s}\f{|k,\eta|^{\f{s}{2}}}{\langle t\rangle^s}\left\langle \f{\xi_3}{l t}\right\rangle^{-1}\f{|l,\xi_3|^{\f{s}{2}}}{\langle t\rangle^s}(l^2+(\xi_3-lt)^2)\mathbf{1}_{l\neq 0}.
\end{align*}
Here $\langle \xi_1-\eta\rangle$ will be absorbed by $e^{-c\la_{\mathcal{D}}|\xi_1-\eta|^s}$. 

For the case $16|l|<|\xi_3|$, we have on the support of the integrand,
\beno
|\xi_3|\approx |l,\xi_3|\leq \f{3}{16}|\xi_1-\xi_2|,\quad |k-l,\xi_2-\xi_3|\leq \f{3}{16}|l,\xi_1-\xi_2+\xi_3|\lesssim |\xi_1-\xi_2|.
\eeno
Then by \eqref{eq: C9}, we have 
\begin{align*}
|l|\f{\rmA_k(\xi_1)}{\rmA^{\rmR}(\xi_1-\xi_2)}
\lesssim e^{c\la |k-l,\xi_2-\xi_3|^s+c\la|\xi_3|^s}. 
\end{align*}
We then get that 
\beno
I_{f,9, \rmH\rmL,\rmH\rmL}^1\lesssim \ep \mathcal{CK}_f^{\f12}\left\|\left\langle\f{\pa_v}{t\pa_z}\right\rangle^{-1}(\pa_z^2+(\pa_v-t\pa_z)^2)\left(\f{|\na|^{\f{s}{2}}}{\langle t\rangle^s}\rmA+\sqrt{\f{\pa_t w}{w}}\tilde{\rmA}\right)P_{\neq}(\Psi\Upsilon_2)\right\|_2+\f{\ep^3}{\langle t\rangle^2}. 
\eeno

We then treat $I_{f,9, \rmH\rmL,\rmL\rmH}^1$, which is similar to the Reaction term $\mathrm{R}_{\rmN,2}^a$, but simpler, because the derivative loss is only in $z$. 
On the support of the integrand, we have that
\beno
|\xi_1-\xi_2|\leq \f{6}{32}|l,\xi_3|,\quad |k-l,\xi_2-\xi_3|\leq \f{6}{32}|l,\xi_1-\xi_2+\xi_3|
\eeno
which gives $|l,\xi_1-\xi_2+\xi_3|\approx |l,\xi_3|\approx |k,\xi_1|$. 
We then have by \eqref{eq: est C_4}, 
\beno
\begin{aligned}
|l|\f{\rmA_k(\xi_1)}{\rmA_l(\xi_3)}&\lesssim e^{c\la |k-l,\xi_2-\xi_3|^s+c\la|\xi_1-\xi_2|^s}\Big[\sqrt{\f{\pa_tw_k(\xi_1)}{w_{k}(\xi_1)}}+\f{|k,\xi_1|^{s/2}}{\langle t\rangle^s}\Big]\\
&\quad\quad\quad
\times\left\langle \f{\xi_3}{l t}\right\rangle^{-1}\Big[\f{|l,\xi_3|^{\f{s}{2}}}{\langle t\rangle^s}+\sqrt{\f{\pa_tw_l(\xi_3)}{w_{l}(\xi_3)}}\Big](l^2+(\xi_3-lt)^2)\mathbf{1}_{l\neq 0}\\
&\lesssim \langle \xi_1-\eta\rangle e^{c\la |k-l,\xi_2-\xi_3|^s+c\la|\xi_1-\xi_2|^s}\Big[\sqrt{\f{\pa_tw_k(\eta)}{w_{k}(\eta)}}+\f{|k,\eta|^{s/2}}{\langle t\rangle^s}\Big]\\
&\quad\quad\quad
\times\left\langle \f{\xi_3}{l t}\right\rangle^{-1}\Big[\f{|l,\xi_3|^{\f{s}{2}}}{\langle t\rangle^s}+\sqrt{\f{\pa_tw_l(\xi_3)}{w_{l}(\xi_3)}}\Big](l^2+(\xi_3-lt)^2)\mathbf{1}_{l\neq 0},
\end{aligned}
\eeno
which gives that
\begin{align*}
I_{f,9, \rmH\rmL,\rmL\rmH}^1
\lesssim &\mathcal{CK}_f^{\f12}\|\udl{\varphi_2}\|_{\mathcal{G}^{\s-6}}\|\Om\|_{\mathcal{G}^{\s-6}}\\
&\times\left\|\left\langle\f{\pa_v}{t\pa_z}\right\rangle^{-1}(\pa_z^2+(\pa_v-t\pa_z)^2)\left(\f{|\na|^{\f{s}{2}}}{\langle t\rangle^s}\rmA+\sqrt{\f{\pa_t w}{w}}\tilde{\rmA}\right)P_{\neq}(\Psi\Upsilon_2)\right\|_2\\
\lesssim &\ep \mathcal{CK}_f+\ep \left\|\left\langle\f{\pa_v}{t\pa_z}\right\rangle^{-1}(\pa_z^2+(\pa_v-t\pa_z)^2)\left(\f{|\na|^{\f{s}{2}}}{\langle t\rangle^s}\rmA+\sqrt{\f{\pa_t w}{w}}\tilde{\rmA}\right)P_{\neq}(\Psi\Upsilon_2)\right\|_2^2.
\end{align*}

The treatment of $I_{f,9, \rmH\rmL,\rmH\rmH}^1$ is similar to that of $\rmR_{\rmN,4}^a$. We have
\beno
I_{f,9, \rmH\rmL,\rmH\rmH}^1\lesssim \f{\ep^3}{\langle t\rangle^2}+\ep\mathcal{CK}_f+\ep\left\|\left\langle\f{\pa_v}{t\pa_z}\right\rangle^{-1}(\pa_z^2+(\pa_v-t\pa_z)^2)\f{|\na|^{\f{s}{2}}}{\langle t\rangle^s}\rmA P_{\neq}(\Psi\Upsilon_2)\right\|_2^2. 
\eeno

\subsection{The high-high interactions}
The treatment of $I_{f,9, \rmH\rmH}^1$ is even simpler than that of $\mathcal{R}^a$. We have  
\beno
I_{f,9, \rmH\rmH}^1\lesssim \f{\ep^3}{\langle t\rangle^{2-\rmK_{\rmD}\ep/2}}. 
\eeno

\section{Linear force term in $I_a$}\label{sec: I_a^l}
In this section, we estimate $I_a^l$ and prove Proposition \ref{prop: I_a^l}. We write 
\begin{align*}
I_a^l&=-\f{1}{2\pi}\sum_{k\neq 0}\int \rmA_k^*(\eta)^2\overline{\widehat{a}_k(\eta)}\widehat{\udl{\th'}}(\eta-\xi)ik \widehat{\Upsilon_2\Psi}_{k}(\xi)d\xi d\eta\\
&=-\f{1}{2\pi}\sum_{\rmN\geq 8}\sum_{k\neq 0}\int \rmA_k^*(\eta)^2\overline{\widehat{a}_k(\eta)}\widehat{\udl{\th'}}(\eta-\xi)_{<\rmN/8}ik \widehat{\Upsilon_2\Psi}_{k}(\xi)_{\rmN}d\xi d\eta\\
&\quad-\f{1}{2\pi}\sum_{\rmN\geq 8}\sum_{k\neq 0}\int \rmA_k^*(\eta)^2\overline{\widehat{a}_k(\eta)}\widehat{\udl{\th'}}(\eta-\xi)_{\rmN}ik \widehat{\Upsilon_2\Psi}_{k}(\xi)_{<\rmN/8}d\xi d\eta\\
&\quad-\f{1}{2\pi}\sum_{\rmN}\sum_{\f18\rmN\leq \rmN'\leq 8\rmN}\sum_{k\neq 0}\int \rmA_k^*(\eta)^2\overline{\widehat{a}_k(\eta)}\widehat{\udl{\th'}}(\eta-\xi)_{\rmN'}ik \widehat{\Upsilon_2\Psi}_{k}(\xi)_{\rmN}d\xi d\eta\\
&\eqdef I_{a,\rmL\rmH}^l+I_{a,\rmH\rmL}^l+I_{a,\rmH\rmH}^l.
\end{align*} 
For $I_{a,\rmL\rmH}^l$, on the support of the integrand, we have
\beno
|\eta-\xi|\leq \f{3}{16}|k,\xi|
\eeno
and then by Lemma \ref{Lem 3.6}, it holds that
\beno
|k|\f{\rmA^*_k(\eta)}{\rmA_k(\xi)}
\lesssim |k| e^{c\la|\eta-\xi|^s}\sqrt{1+\f{k^2+|\xi|}{\langle t\rangle^2}}=C_{11}^{\rmA}.
\eeno
We consider the following four cases:
\begin{itemize}
\item[1.] if $t\leq 2\sqrt{|\xi|}, \, |k|\leq \f{1}{4}\sqrt{|\xi|}$, then \beno
C_{11}^{\rmA}\leq |k|e^{c\la|\eta-\xi|^s}\f{\sqrt{|\xi|}}{\langle t\rangle}\lesssim e^{c\la|\eta-\xi|^s}\left\langle\f{\xi}{kt}\right\rangle^{-1}\f{\xi^2}{k\langle t\rangle^2}\lesssim e^{c\la|\eta-\xi|^s}\left\langle\f{\xi}{kt}\right\rangle^{-1}\f{(\xi-kt)^2+k^2}{\langle t\rangle^2}.
\eeno
\item[2.] if $t\leq 2\sqrt{|\xi|},\, |k|\geq \f{1}{4}\sqrt{|\xi|}\geq \f18t$, then
\begin{align*}
C_{11}^{\rmA}
&\leq |k|e^{c\la|\eta-\xi|^s}\f{|k|}{\langle t\rangle}\lesssim e^{c\la|\eta-\xi|^s}\left\langle\f{\xi}{kt}\right\rangle^{-1}\f{|k||\xi|}{\langle t\rangle^2}\mathbf{1}_{|kt|\leq 2|\xi|}
+e^{c\la|\eta-\xi|^s}\left\langle\f{\xi}{kt}\right\rangle^{-1}\f{|k|^2}{\langle t\rangle}\mathbf{1}_{|kt|\geq 2|\xi|}\\
&\lesssim e^{c\la|\eta-\xi|^s}\Big((\xi-kt)^2+k^2\Big)\Big(\left\langle\f{\xi}{kt}\right\rangle^{-1}\f{|\xi|^{\f12}}{\langle t\rangle^2}\mathbf{1}_{|k|\approx \sqrt{|\xi|}}+\left\langle\f{\xi}{kt}\right\rangle^{-1}\f{1}{\langle t\rangle^3}\Big)\\
&\lesssim e^{c\la|\eta-\xi|^s}\Big((\xi-kt)^2+k^2\Big)\left\langle\f{\xi}{kt}\right\rangle^{-1}\f{|\xi|^{\f{s}{2}}|k,\eta|^{\f{s}{2}}}{\langle t\rangle^2}.
\end{align*}
\item[3.] if $2|\xi|>t\geq 2\sqrt{|\xi|}$, then $t\in \mathrm{I}_{l,\xi}\cap \mathrm{I}_{j,\eta}$ for some $l, j$, and 
\begin{align*}
C_{11}^{\rmA}
&\lesssim |k|^2e^{c\la|\eta-\xi|^s}\sqrt{\f{\pa_tw_l(\xi)}{w_l(\xi)}}\sqrt{\f{\pa_tw_j(\eta)}{w_j(\eta)}}\Big(1+|\f{\xi}{l}-t|\Big)^{\f12}\Big(1+|\f{\eta}{j}-t|\Big)^{\f12}\\
&\lesssim |k|^2e^{c\la|\eta-\xi|^s}\sqrt{\f{\pa_tw_l(\xi)}{w_l(\xi)}}\sqrt{\f{\pa_tw_j(\eta)}{w_j(\eta)}}\Big(1+|\f{\xi}{k}-t|\Big)\left\langle\f{\xi}{kt}\right\rangle^{-1}\mathbf{1}_{|k|\gtrsim |l|\approx |j|}\\
&\quad +|k|^2e^{c\la|\eta-\xi|^s}\sqrt{\f{\pa_tw_l(\xi)}{w_l(\xi)}}\sqrt{\f{\pa_tw_j(\eta)}{w_j(\eta)}}\Big(1+|\f{\xi}{k}-t|\Big)\mathbf{1}_{|kt|\leq \f12|\xi|}\\
&\lesssim e^{c\la|\eta-\xi|^s}\sqrt{\f{\pa_tw_l(\xi)}{w_l(\xi)}}\sqrt{\f{\pa_tw_j(\eta)}{w_j(\eta)}}\big(k^2+|\xi-tk|^2\big)\left\langle\f{\xi}{kt}\right\rangle^{-1}. 
\end{align*}
\item[4.] if $t\geq 2|\xi|$, then
\begin{align*}
C_{11}^{\rmA}
&\lesssim k^2e^{c\la|\eta-\xi|^s}\f{k^2+|\xi-tk|^2}{k^2t^2}\left\langle\f{\xi}{kt}\right\rangle^{-1}. 
\end{align*}
\end{itemize}
Thus, we have
\ben
\begin{aligned}
I_{a,\rmL\rmH}^l
&\lesssim \mathcal{CK}_{a}^{\f12}\left\|\left\langle\f{\pa_v}{t\pa_z}\right\rangle^{-1}(\pa_z^2+(\pa_v-t\pa_z)^2)\left(\f{|\na|^{\f{s}{2}}}{\langle t\rangle^s}\rmA+\sqrt{\f{\pa_t w}{w}}\tilde{\rmA}\right)P_{\neq}(\Psi\Upsilon_2)\right\|_2\\
&\leq \varepsilon_0\mathcal{CK}_{a}+C_{\varepsilon_0}\left\|\left\langle\f{\pa_v}{t\pa_z}\right\rangle^{-1}(\pa_z^2+(\pa_v-t\pa_z)^2)\left(\f{|\na|^{\f{s}{2}}}{\langle t\rangle^s}\rmA+\sqrt{\f{\pa_t w}{w}}\tilde{\rmA}\right)P_{\neq}(\Psi\Upsilon_2)\right\|_2^2. 
\end{aligned}
\een

For $I_{a,\rmH\rmL}^l$, we consider two cases $16|k|\leq |\xi|$ and $16|k|>|\xi|$. 
For the first case $16|k|\leq |\xi|$, on the support of the integrand, we have
\beno
|\xi|\approx |k,\xi|\leq \f{3}{16}|\eta-\xi|\approx |k,\eta|\approx |\eta|.
\eeno
Then we have
\begin{align*}
\f{\rmA_k^*(\eta)}{\rmA^{\rmR}(\eta-\xi)}\lesssim e^{c\la|k,\xi|^s}\Big(1+\f{\langle\eta\rangle^{\f12}}{\langle t\rangle}\Big)
\lesssim e^{c\la|k,\xi|^s}\Big(1+\f{|k,\eta|^{\f{s}{2}}|\eta-\xi|^{\f{s}{2}}}{\langle t\rangle}\Big).
\end{align*}
For the second case $16|k|>|\xi|$, on the support of the integrand, we have 
\beno
|k|\approx |k,\xi|\leq \f{3}{16}|\eta-\xi|\approx |k,\eta|.
\eeno
Then we have 
\begin{align*}
\f{\rmA_k^*(\eta)}{\rmA_k(\xi)}
&\lesssim ke^{c\la|\eta-\xi|}\Big(1+\f{|\xi|+k^2}{\langle t\rangle^2}\Big)^{\f12}\\
&\lesssim |k|e^{c\la|\eta-\xi|}\Big(1+\f{|k|}{\langle t\rangle}\Big)\mathbf{1}_{\f{|k|}{16}\leq |\xi|<16|k|}
+|k|e^{c\la|\eta-\xi|}\Big(1+\f{|k|}{\langle t\rangle}+\f{\sqrt{|\xi|}}{\langle t\rangle}\Big)\mathbf{1}_{|\xi|<\f{1}{16}|k|}\\
&\lesssim e^{c\la|\eta-\xi|}\left\langle\f{\xi}{tk}\right\rangle^{-1}(k^2+(\xi-kt)^2)\Big(\f{|k|^{\f12}}{\langle t\rangle^{\f32}}+\f{|\xi|^{s}}{\langle t\rangle^{2s}}\Big).
\end{align*}
Therefore, we have
\begin{align*}
|I_{a,\rmH\rmL}^l|
&\lesssim \f{\ep^2}{\langle t\rangle^2}+\ep\mathcal{CK}_a^{\f12}\left\|\f{|\na|^{\f{s}{2}}}{\langle t\rangle}\rmA^{\rmR}\udl{\th'}\right\|_{L^2}\\
&\quad+\mathcal{CK}_a^{\f12}\left\|\left\langle\f{\pa_v}{t\pa_z}\right\rangle^{-1}(\pa_z^2+(\pa_v-t\pa_z)^2)\left(\f{|\na|^{\f{s}{2}}}{\langle t\rangle^s}\rmA+\sqrt{\f{\pa_t w}{w}}\tilde{\rmA}\right)P_{\neq}(\Psi\Upsilon_2)\right\|_2\\
&\lesssim \f{\ep^2}{\langle t\rangle^2}
+\ep\mathcal{CK}_a^{\f12}\Big(\f{1}{\langle t\rangle}+\mathcal{CK}_{\varphi^{\d}_{10}}^{\f12}\Big)\\
&\quad+\mathcal{CK}_a^{\f12}\left\|\left\langle\f{\pa_v}{t\pa_z}\right\rangle^{-1}(\pa_z^2+(\pa_v-t\pa_z)^2)\left(\f{|\na|^{\f{s}{2}}}{\langle t\rangle^s}\rmA+\sqrt{\f{\pa_t w}{w}}\tilde{\rmA}\right)P_{\neq}(\Psi\Upsilon_2)\right\|_2\\
&\lesssim \f{\ep^2}{\langle t\rangle^2}+\ep\mathcal{CK}_{\varphi^{\d}_{10}}+  \varepsilon_0\mathcal{CK}_{a}\\
&\quad +C_{\varepsilon_0}\left\|\left\langle\f{\pa_v}{t\pa_z}\right\rangle^{-1}(\pa_z^2+(\pa_v-t\pa_z)^2)\left(\f{|\na|^{\f{s}{2}}}{\langle t\rangle^s}\rmA+\sqrt{\f{\pa_t w}{w}}\tilde{\rmA}\right)P_{\neq}(\Psi\Upsilon_2)\right\|_2^2. 
\end{align*}
The treatment of $I_{a,\rmH\rmH}^l$ is simple. 
We have
\beno
|I_{a,\rmH\rmH}^l|\lesssim \f{\ep^2}{\langle t\rangle^2}.
\eeno

\section{The nonlinear interactions with the pressure in $I_f$}\label{sec: I_f4,8}
In this section, we treat the $I_{f,4}$ and $I_{f,8}$. We have by \eqref{eq: N_aPi} and Proposition \ref{prop: kernel-wave-op} that
\begin{align*}
I_{f,4}&=-\f{1}{2\pi}\sum_{l\neq 0}\int_{\mathbb{R}^3}\rmA_0(\eta)\Big[\mathbf{h}_1(\xi')(l(\xi+\zeta))\widehat{(\Pi_{\star})}_l(\xi)\widehat{a}_{-l}(\zeta)\Big]\rmA_0(\eta)\overline{\widehat{f}_0(\eta)}d\xi d\zeta d\eta\\
&\quad -\f{1}{2\pi}\sum_{l\neq 0}\int_{\mathbb{R}^3}\rmA_0(\eta)\Big[\mathbf{h}_2(\xi')(il)\widehat{(\Pi_{\star})}_l(\xi)\widehat{a}_{-l}(\zeta)\Big]\rmA_0(\eta)\overline{\widehat{f}_0(\eta)}d\xi d\zeta d\eta\\
&\quad -\f{1}{2\pi}\sum_{k\geq k_{M}}\sum_{l}\int_{\mathbb{R}^3}\rmA_k(\eta)\Big[\mathbf{h}_1(\xi')(-k\xi+l(\xi+\zeta))\widehat{(\Pi_{\star})}_l(\xi)\widehat{a}_{k-l}(\zeta)\Big]\rmA_k(\eta)\overline{\widehat{f}_k(\eta)}d\xi d\zeta d\eta\\
&\quad -\f{1}{2\pi}\sum_{k\geq k_{M}}\sum_{l\neq 0}\int_{\mathbb{R}^3}\rmA_k(\eta)\Big[\mathbf{h}_2(\xi')(il)\widehat{(\Pi_{\star})}_l(\xi)\widehat{a}_{k-l}(\zeta)\Big]\rmA_k(\eta)\overline{\widehat{f}_k(\eta)}d\xi d\zeta d\eta\\
I_{f,8}&=-\f{1}{2\pi}\sum_{0<k< k_{M}}\sum_{l}\int_{\mathbb{R}^4}\rmA_k(\eta')\mathcal{D}(t, k, \eta', \eta)\\
&\qquad\qquad\qquad\qquad\qquad  
\times \Big[\mathbf{h}_1(\xi')(-k\xi+l(\xi+\zeta))\widehat{(\Pi_{\star})}_l(\xi)\widehat{a}_{k-l}(\zeta)\Big]\overline{\widehat{\rmA f}_k(\eta')}d\xi d\zeta d\eta d\eta'\\
&\quad-\f{1}{2\pi}\sum_{0<k< k_{M}}\sum_{l\neq 0}\int_{\mathbb{R}^4}\rmA_k(\eta')\mathcal{D}(t, k, \eta', \eta)\\
&\qquad\qquad\qquad\qquad\qquad  
\times \Big[\mathbf{h}_2(\xi')(il)\widehat{(\Pi_{\star})}_l(\xi)\widehat{a}_{k-l}(\zeta)\Big]\overline{\widehat{\rmA f}_k(\eta')}d\xi d\zeta d\eta d\eta'\\
&=I_{f,8}^1+I_{f,8}^2,
\end{align*}
where $\xi'=\eta-\xi-\zeta$, 
\begin{align*}
\mathbf{h}_1(\xi')=\int_{\mathbb{R}}\widehat{(\udl{\varphi_4}\Upsilon_2)}(\xi'')\widehat{(\Upsilon_2\udl{\pa_yv})}(\xi'-\xi'')d\xi'',\qquad
\mathbf{h}_2(\xi')=\widehat{\udl{\varphi_5}}(\xi'). 
\end{align*}
and $\Pi_{\star}\in \big\{P_{\neq}(\Upsilon_2\Pi_{l,1}), \, P_{\neq}(\Upsilon_1\Pi_{l,2}),\, P_{\neq}(\Upsilon_2\Pi_{n,1}), \, P_{\neq}(\Upsilon_1\Pi_{n,2}),\, P_0(\Upsilon_2\Pi)\big\}$. 

We only present the estimates for $I_{f,8}^1$ which involves the wave operator, other terms are similar or even easier. 
We have by Proposition \ref{prop: kernel-wave-op} that 
\begin{align*}
|I_{f,8}^1|&\lesssim \sum_{0<k< k_{M}}\sum_{l}\int_{\mathbb{R}^4}e^{-c\la_{\mathcal{D}}|\eta-\eta'|^{s_0}}\rmA_k(\eta)
\Big(\mathbf{1}_{|\xi'|\geq 4|k,\eta|}+\mathbf{1}_{4|\xi'|\leq |k,\eta|}+\mathbf{1}_{\f14|k,\eta|\leq |\xi'|\leq 4|k,\eta|}\Big)
\\
&\qquad\qquad\qquad\qquad\qquad  
\times \Big[|\mathbf{h}_1(\xi')||-k\xi+l(\xi+\zeta)||\widehat{(\Pi_{\star})}_l(\xi)||\widehat{a}_{k-l}(\zeta)|\Big]|\overline{\widehat{\rmA f}_k(\eta')}|d\xi d\zeta d\eta d\eta'\\
&=I_{f,8}^{\rmH\rmL}+I_{f,8}^{\rmL\rmH}+I_{f,8}^{\rmH\rmH},
\end{align*}
Then we have by Remark \ref{Rmk: A-switch} and Corollary \ref{cor: decay-pressure} and fact that
\beno
 \|\rmA^{\rmR}\mathbf{h}_1\|_2\lesssim (1+\|\rmA^{\rmR}\varphi^{\d}_4\|_2)(1+\|\rmA^{\rmR}h\|_2+\|\rmA^{\rmR}\varphi_9^{\d}\|_{2})\lesssim 1
\eeno 
that
\begin{align*}
I_{f,8}^{\rmH\rmL}+I_{f,8}^{\rmH\rmH}
&\lesssim \|\rmA^{\rmR}\mathbf{h}_1\|_2\|P_{\neq}\Pi_{\star}\|_{\mathcal{G}^{s,\s-6}}\|a\|_{\mathcal{G}^{s,\s-6}}\|\rmA f\|_2\\
&\quad
+\|\rmA^{\rmR}\mathbf{h}_1\|_2\|P_{0}\pa_v\Pi\|_{\mathcal{G}^{s,\s-6}}\|a\|_{\mathcal{G}^{s,\s-6}}\|\rmA f\|_2\lesssim \f{\ep^3}{\langle t\rangle^3}.
\end{align*}
For $I_{f,8}^{\rmL\rmH}$, by Remark \ref{Rmk: A-switch}, we have 
\begin{align*}
I_{f,8}^{\rmL\rmH}&\lesssim \sum_{0<k< k_{M}}\sum_{l}\int_{\mathbb{R}^4}e^{-c\la_{\mathcal{D}}|\eta-\eta'|^{s_0}}e^{c\la|\xi'|^s}|\mathbf{h}_1(\xi')|
\Big(\mathbf{1}_{4|\xi'|\leq |k,\eta|}\Big)\rmA_k(\xi+\zeta)
\\
&\qquad\qquad\qquad\qquad\qquad  
\times \Big[|-k\xi+l(\xi+\zeta)||\widehat{(\Pi_{\star})}_l(\xi)||\widehat{a}_{k-l}(\zeta)|\Big]|\overline{\widehat{\rmA f}_k(\eta')}|d\xi d\zeta d\eta d\eta'\\
&=\sum_{0<k< k_{M}}\int_{\mathbb{R}^4}e^{-c\la_{\mathcal{D}}|\eta-\eta'|^{s_0}}e^{c\la|\xi'|^s}|\mathbf{h}_1(\xi')|
\mathbf{1}_{4|\xi'|\leq |k,\eta|}\rmA_k(\xi+\zeta)\\
&\qquad\qquad\qquad\qquad\qquad  
\times
\Big(\mathbf{1}_{|\xi|\geq 4|k,\zeta|}+\mathbf{1}_{4|\xi|\leq |k,\zeta|}+\mathbf{1}_{\f14|k,\zeta|\leq |\xi|\leq 4|k,\zeta|}\Big)
\\
&\qquad\qquad\qquad\qquad\qquad  
\times \Big[|-k\xi||\widehat{P_0(\Pi)}(\xi)||\widehat{a}_{k}(\zeta)|\Big]|\overline{\widehat{\rmA f}_k(\eta')}|d\xi d\zeta d\eta d\eta'\\
&\quad +\sum_{0<k< k_{M}}\sum_{l\neq 0}\int_{\mathbb{R}^4}e^{-c\la_{\mathcal{D}}|\eta-\eta'|^{s_0}}e^{c\la|\xi'|^s}|\mathbf{h}_1(\xi')|
\mathbf{1}_{4|\xi'|\leq |k,\eta|}\rmA_k(\xi+\zeta)
\\
&\qquad\qquad\qquad\qquad\qquad  
\times\Big(\mathbf{1}_{|l,\xi|\geq 4|k,\zeta|}+\mathbf{1}_{4|l,\xi|\leq |k,\zeta|}+\mathbf{1}_{\f14|k,\zeta|\leq |l,\xi|\leq 4|k,\zeta|}\Big)
\\
&\qquad\qquad\qquad\qquad\qquad  
\times \Big[|-k\xi+l(\xi+\zeta)||\widehat{(\Pi_{\star})}_l(\xi)||\widehat{a}_{k-l}(\zeta)|\Big]|\overline{\widehat{\rmA f}_k(\eta')}|d\xi d\zeta d\eta d\eta'\\
&=I_{f,8,0}^{\rmH\rmL}+I_{f,8,0}^{\rmL\rmH}+I_{f,8,0}^{\rmH\rmH} +I_{f,8,\neq}^{\rmH\rmL}+I_{f,8,\neq}^{\rmL\rmH}+I_{f,8,\neq}^{\rmH\rmH}. 
\end{align*}
The terms $I_{f,8,0}^{\rmH\rmH}$ and $I_{f,8,\neq}^{\rmH\rmH}$ are easy. We have
\begin{align*}
I_{f,8,0}^{\rmH\rmH}\lesssim \f{\ep^3}{\langle t\rangle^3},\quad
I_{f,8,\neq}^{\rmH\rmH}\lesssim \f{\ep^3}{\langle t\rangle^4}. 
\end{align*}
For the high-low interactions in $I_{f,8,0}^{\rmH\rmL}$, we have
\begin{align*}
\rmA_k(\xi+\zeta)\lesssim& \rmA_0(\xi)\Big(1+\f{\xi+\zeta}{k^2+(\xi+\zeta-kt)^2}\mathbf{1}_{t\in {\rm{I}}_{k,\xi+\zeta}}\Big)e^{c\la|k,\zeta|^s}\\
\lesssim& \rmA_0(\xi)\f{\langle \xi\rangle^{\f{s}{2}}}{\langle t\rangle^s}\f{t(t^2+\xi^2)}{2+\xi^2}\f{\langle\xi\rangle^{2-\f{s}{2}}}{t^{1-s}(t^2+\xi^2)}e^{c\la|k,\zeta|^s}\\
&+\rmA_0(\xi)\f{t(t^2+\xi^2)}{2+\xi^2}\f{\langle\xi\rangle^{3}/k^2}{t(t^2+\xi^2)}\f{\Big(\langle k,\zeta\rangle^{2}+\langle\eta-\eta'\rangle^{2}+\langle\xi'\rangle^{2}\Big)}{(1+(\eta'/k-t)^2)^{\f12}(1+(\xi/k-t)^2)^{\f12}}\mathbf{1}_{t\in {\rm{I}}_{k,\xi+\zeta}}e^{c\la|k,\zeta|^s}\\
\lesssim& \rmA_0(\xi)\f{\langle \xi\rangle^{\f{s}{2}}}{\langle t\rangle^s}\f{t(t^2+\xi^2)}{2+\xi^2}\f{\langle k,\eta'\rangle^{\f{s}{2}}}{t}\Big(\langle k,\zeta\rangle^{\f{s}{2}}+\langle\eta-\eta'\rangle^{\f{s}{2}}+\langle\xi'\rangle^{\f{s}{2}}\Big)e^{c\la|k,\zeta|^s}\\
&+\rmA_0(\xi)\f{t(t^2+\xi^2)}{2+\xi^2}\sqrt{\f{\pa_tw(\xi)}{w(\xi)}}\sqrt{\f{\pa_tw_k(\eta')}{w_k(\eta')}}\Big(\langle k,\zeta\rangle^{2}+\langle\eta-\eta'\rangle^{2}+\langle\xi'\rangle^{2}\Big)e^{c\la|k,\zeta|^s},
\end{align*}
which gives that
\begin{align*}
I_{f,8,0}^{\rmH\rmL}\lesssim \ep \mathcal{CK}_f+\ep\left\|t(t^2+|\pa_v|^2)(2-\pa_{v}^2)^{-1}\Big(\sqrt{\f{\pa_tw}{w}}+\f{\langle \pa_v\rangle^{\f{s}{2}}}{\langle t\rangle^s}\Big)\rmA_0P_0\big(\pa_v(\Pi\Upsilon_1)\big)\right\|_2^2.
\end{align*}
The high-low interactions in $I_{f,8,\neq}^{\rmH\rmL}$ is similar to the Reaction term $\rmR_{\rmN}^a$, by using \eqref{eq: est C_4}, we have
\begin{align*}
I_{f,8,\neq}^{\rmH\rmL}\lesssim \ep \mathcal{CK}_f+\ep\left\|\left\langle\f{\pa_v}{t\pa_z}\right\rangle^{-1}(\pa_z^2+(\pa_v-t\pa_z)^2)\left(\f{|\na|^{\f{s}{2}}}{\langle t\rangle^s}\rmA+\sqrt{\f{\pa_t w}{w}}\tilde{\rmA}\right)P_{\neq}(\Pi_{\star})\right\|_2^2. 
\end{align*}
Then by \eqref{eq: inequality-11}, we have
\begin{align*}
I_{f,8,\neq}^{\rmH\rmL}\lesssim \ep \mathcal{CK}_f+\ep\left\|\langle\pa_z\rangle t(t+\langle \na\rangle)(1-\Delta)^{-1}\big((\pa_v-t\pa_z)^2+\pa_z^2\big)\left(\f{\langle\na\rangle^{\f{s}{2}}}{\langle t\rangle^s}\rmA+\sqrt{\f{\pa_t w}{w}}\tilde{\rmA}\right) P_{\neq} \Pi_{\star}\right\|_2^2. 
\end{align*}

To treat the low-high interactions in $I_{f,8,0}^{\rmL\rmH}$ and $I_{f,8,\neq}^{\rmL\rmH}$, we take advantage of the multiplier $\rmB$. 
We have 
\begin{align*}
|k|\rmA_k(\xi+\zeta)\lesssim |k|\rmA_k^{*}(\zeta)e^{c\la|\xi|^s}\f{t}{t+|k|+|\zeta|^{\f12}}\lesssim \rmA_k^{*}(\zeta)e^{c\la|\xi|^s}\langle t\rangle,
\end{align*}
which gives that
\begin{align*}
I_{f,8,0}^{\rmL\rmH}\lesssim \|\rmA^*a\|_2\langle t\rangle\|\pa_vP_0(\Upsilon_2\Pi)\|_{\mathcal{G}^{s,\s-6}}\|\rmA f\|_{L^2}\lesssim \f{\ep^3}{\langle t\rangle^2}. 
\end{align*}
We also have 
\begin{align*}
&|-k\xi+l(\xi+\zeta)|\rmA_k(\xi+\zeta)\\
&\lesssim \langle l,\xi\rangle \langle k-l,\zeta\rangle \rmA_{k-l}^{*}(\zeta)e^{c\la|l,\xi|^s}\f{t}{t+|k-l|+|\zeta|^{\f12}}\Big(1+\f{\xi+\zeta}{k^2+(\xi+\zeta-kt)^2}\mathbf{1}_{t\in {\rm{I}}_{k,\xi+\zeta}\cap {\rm{I}}_{k,\zeta}\, l\neq 0}\Big)\\
&\lesssim \langle t\rangle^{2+2s}\f{|\zeta|^{\f{s}{2}}|k,\eta'|^{\f{s}{2}}}{\langle t\rangle^{2s}} \rmA_{k-l}^{*}(\zeta)e^{c\la|l,\xi|^s},
\end{align*}
which gives that
\begin{align*}
I_{f,8,\neq}^{\rmL\rmH}\lesssim \ep\mathcal{CK}_a+\ep\mathcal{CK}_f.
\end{align*}

\begin{appendix}
\section{Gevrey spaces}
The physical space characterization of Gevrey functions is also used in this paper. We start with a characterization of the Gevrey spaces on the physical side. 
See {\it Lemma A.1} in \cite{IonescuJia2020cmp} for the elementary proof. 
\begin{lemma}[\cite{IonescuJia2020cmp}]\label{Lem: A1}
Suppose that $d=1,2$, $0<s<1$, $K>1$ and $g\in C^{\infty}(\R^d)$ with $\mathrm{supp}\, g\subset [a,b]^d$ satisfies the bounds
\ben
|D^{\al}g(x)|\leq K^m(m+1)^{m/s}, \quad x\in \R^d
\een
for all integers $m\geq 0$ and multi-indices $\al$ with $|\al|=m$. Then
\ben
|\hat{g}(\xi)|\lesssim_{K,s} Le^{-\mu |\xi|^s},
\een
for all $\xi\in \R^d$ and some $\mu=\mu(K,s)>0$. 

Conversely, assume that, for some $\mu>0$ and $s\in (0,1)$,
\ben
|\hat{g}(\xi)|\leq Le^{-\mu |\xi|^s},
\een
for all $\xi\in \R^d$. Then there is $K>1$ depending on $s$ and $\mu$ such that 
 \ben
|D^{\al}g(x)|\lesssim_{\mu,s} K^m(m+1)^{m/s}, \quad x\in \R^d
\een
for all integers $m\geq 0$ and multi-indices $\al$ with $|\al|=m$.
\end{lemma}

For $x\in [a,b]^d$ and parameters $s\in (0,1)$ and $M\geq 1$, we define the spaces
\beq
\mathcal{G}_{ph}^{M,s}([a,b]^d)\eqdef
\left\{g: [a,b]^d\to \mathbb{C}:~\|g\|_{\mathcal{G}_{ph}^{M,s}([a,b]^d)}<\infty\right\}. 
\eeq
where 
\beq
\|g\|_{\mathcal{G}_{ph}^{M,s}([a,b]^d)}\eqdef \sup_{x\in [a,b]^d,\, m\geq 0,\, |\al|\leq m}\f{|D^{\al}g(x)|}{(m+1)^{m/s}M^{m}}. 
\eeq
Here `{\it ph}' represents the physical side. 

We define the spaces
\beq\label{eq:def-G_ph,1}
\mathcal{G}_{ph,1}^{M,s}([a,b]^d)\eqdef
\left\{g: [a,b]^d\to \mathbb{C}:~\|g\|_{\mathcal{G}_{ph,1}^{M,s}([a,b]^d)}<\infty\right\}. 
\eeq
where 
\beq
\|g\|_{\mathcal{G}_{ph,1}^{M,s}([a,b]^d)}\eqdef \sup_{x\in [a,b]^d,\, m\geq 0,\, |\al|\leq m}\f{|D^{\al}g(x)|}{\Gamma_s(m)M^{m}},
\eeq
with $\Gamma_s(m)=(m!)^{\f1s}(m+1)^{-2}$, also see \cite{Yamanaka} for more details. 

By Stirling's formula $N!\sim \sqrt{2\pi N}(N/e)^{N}$, it is easy to check that there exist $K_1<K_2$ such that 
\beno
K_1^{m}(m+1)^{m/s}\lesssim \Gamma_s(m)\lesssim  K_2^{m}(m+1)^{m/s},
\eeno
which implies
\ben\label{eq: space-embed}
\mathcal{G}_{ph}^{K_1M,s}([a,b]^d)\subset \mathcal{G}_{ph,1}^{M,s}([a,b]^d)\subset \mathcal{G}_{ph}^{K_2M,s}([a,b]^d). 
\een
We also have
\begin{equation}\label{eq:3.9}
\sum_{j=0}^{k}\f{k!}{j!(k-j)!}\Gamma_s(j)\Gamma_s(k-j)<\Gamma_s(k),
\end{equation}
which implies
\ben
\|fg\|_{\mathcal{G}_{ph,1}^{M,s}([a,b]^d)}\leq \|f\|_{\mathcal{G}_{ph,1}^{M,s}([a,b]^d)}\|g\|_{\mathcal{G}_{ph,1}^{M,s}([a,b]^d)}. 
\een

\begin{remark}\label{Rmk: fourier-gevrey}
Suppose $g\in \mathcal{G}_{ph,1}^{M,s}(\R^d)$ with $\mathrm{supp}\, g\subset [a,b]^d$, then for all integers $m\geq 0$ and multi-indices $\al$ with $|\al|=m$. Then
\ben
|\hat{g}(\xi)|\lesssim_{K,s} Le^{-\mu |\xi|^s},
\een
for all $\xi\in \R^d$ and some $\mu=\mu(K,s)>0$. 
\end{remark}

We also introduce the composition lemma in \cite{Yamanaka}. 
\begin{lemma}[\cite{Yamanaka}]\label{eq: composition}
Let $I$, $J$ be real open intervals and let $f:J\to \R$ be a $C^{\infty}$-function such that $f'\in \mathcal{G}_{ph,1}^{L,s}(J)$ and $g: I\to J$ be a $C^{\infty}$-function such that $g'\in \mathcal{G}_{ph,1}^{M,s}(I)$. Let $N$ be a real constant such that
\beno
N\geq \max(M,L\|g'\|_{\mathcal{G}_{ph,1}^{M,s}(I)}).
\eeno
Then the derivative $(f\circ g)'$ of the composite function $f\circ g$ belongs to $\mathcal{G}_{ph,1}^{N,s}(I)$ and satisfies
\beno
\|(f\circ g)'\|_{\mathcal{G}_{ph,1}^{N,s}(I)}\leq \f{N}{L}\|f'\|_{\mathcal{G}_{ph,1}^{L,s}(J)}.
\eeno
\end{lemma}
We also introduce the estimate of the inverse function in the Gevrey class. 
\begin{lemma}[\cite{Yamanaka}]\label{eq: inverse-gevrey}
Let $I$ and $J$ be real open intervals. Let $f:I\to J$ be a $C^{\infty}$-surjection such that
\beno
|f'(x)|\geq \f{1}{A},\quad x\in I
\eeno
for some $A>0$ and such that $f''\in \mathcal{G}_{ph,1}^{L,s}(I)$. Then $f$ has a $C^{\infty}$-inverse $f^{-1}:J\to I$ such that $(f^{-1})'\in  \mathcal{G}_{ph,1}^{M,s}(J)$ for some $M>0$. 
\end{lemma}

\section{The Fourier transform of the integral operator}
In this section, we make some preparations to study the Gevrey regularity of the nonlocal part of the wave operator. Indeed, we will write the nonlocal part into the following four types of integral operators:
\begin{align*}
&\Pi_m(F)(v')=\int_{\R}K(v, v')\mu_m(v-v')e^{-ikt(v-v')}F(v)dv,\quad m=1,2,3,4,\\
&\Pi_m^*(F)(v)=\int_{\R}K(v, v')\mu_m(v-v')e^{-ikt(v-v')}F(v')dv',\quad m=1,2,3,4,
\end{align*}
where $K(v,v')$ represents a smooth kernel with compact support which may vary from one line to the other, and  $\mu_1\in C^{\infty}(\R)$ with compact support such that 
\beno
\mu_1(u)=\left\{\begin{aligned}
&1,|u|\leq u(1)-u(0),\\
&0,|u|\geq 2(u(1)-u(0)),
\end{aligned}\right.
\eeno
$\mu_2(u)=\mu_1(u)\mathbf{1}_{\R^-}(u)$ and $\mu_3(u)=\mu_1(u)\ln|u|$, $\mu_4(u)=p.v.\f{1}{u}$. Thus it holds that
\ben\label{eq:fourier-mu}
|\widehat{\mu_1}(\xi)|+|\widehat{\mu_2}(\xi)|+|\widehat{\mu_3}(\xi)|+|\widehat{\mu_4}(\xi)|\lesssim 1
\een

It is easy to check that $\Pi_m^*, \, m=1,2,3,4$ are the dual operators of $\Pi_m,\, m=1,2,3,4$. 

For a smooth kernel with compact support defined on $\R^2$, let $\widehat{K}$ be the Fourier transform in the first variable and $\widehat{\widehat{K}}$ be the Fourier transform in both variables. 
\begin{lemma}\label{lem: Fourier_type1}
Suppose $K(v, v')\in C^{\infty}(\R^2)$ with compact support, then it holds for $m=1,2,3,4$ that
\beno
\widehat{\Pi_m(F)}(\eta)=\f{1}{2\pi}\int_{\R^2}\hat{F}(\xi)\widehat{\widehat{K}}(-\xi-\xi', \eta+\xi')\widehat{\mu_m}(\xi'+kt)d\xi'd\xi,
\eeno
and
\beno
\widehat{\Pi_m^*(F)}(\xi)=\f{1}{2\pi}\int_{\R^2}\hat{F}(\eta)\widehat{\widehat{K}}(-\xi-\xi', \eta+\xi')\widehat{\mu_m}(\xi'+kt)d\xi'd\eta.
\eeno
\end{lemma}
\begin{proof}
Let $K_{L}(v-v',v')=K(v, v')$, then we have that 
\begin{align*}
\Pi_m(F)(v')&=\int_{\R}K_L(v-v',v')\mu_m(v-v')e^{-ikt(v-v')}F(u)dv\\
&=\f{1}{2\pi}\int_{\R}K_L(v-v',v')\mu_m(v-v')e^{-ikt(v-v')}\int_{\R}\hat{F}(\xi)e^{i\xi v}d\xi dv\\
&=\f{1}{2\pi}\int_{\R}\hat{F}(\xi)e^{iv'\xi}\int_{\R}K_L(v-v',v')\mu_m(v-v')e^{-ikt(v-v')}e^{i\xi (v-v')} dvd\xi\\
&=\f{1}{2\pi}\int_{\R}\hat{F}(\xi)e^{iv'\xi}\mathcal{F}_1\big(K_L(\cdot, v')\mu_m\big)(-\xi+kt)d\xi\\
&=\f{1}{2\pi}\int_{\R}\int_{\R}\hat{F}(\xi)e^{iv'\xi}\widehat{K_L}(-\xi-\xi', v')\widehat{\mu_m}(\xi'+kt)d\xi'd\xi.
\end{align*}

Thus we get that
\begin{align*}
\widehat{\Pi_m(F)}(\eta)=\f{1}{2\pi}\int_{\R^2}\hat{F}(\xi)\widehat{\widehat{K_L}}(-\xi-\xi', \eta-\xi)\widehat{\mu_m}(\xi'+kt)d\xi'd\xi.
\end{align*}
We also have
\begin{align*}
\widehat{\widehat{K_L}}(\xi, \eta)
&=\int_{\R^2}K_L(w,v')e^{-iw\xi-ic\eta}dwdv'\\
&=\int_{\R^2}K_L(v-v',v')e^{-i(v-v')\xi-iv'\eta}dvdv'\\
&=\int_{\R^2}K(v,v')e^{-i(v-v')\xi-iv'\eta}dvdv'
=\widehat{\widehat{K}}(\xi,\eta-\xi),
\end{align*}
which gives the first identity and the second identity can be obtained by the same argument. 
\end{proof}
\begin{lemma}\label{Rmk: fo-g-2}
Suppose $K(v, v')\in \mathcal{G}_{ph,1}^{M,s_0}(\R^2)$ with compact support in $[u(0),u(1)]^2$, then there exists $\la=\la(M,s_0)$, such that for $m=1,2,3,4,$
\ben\label{eq: Kernel-Fourier-G}
\left|\int\widehat{\widehat{K}}(-\xi-\xi', \eta+\xi')\widehat{\mu_m}(\xi'+kt)d\xi'\right|\lesssim e^{-\la|\eta-\xi|^{s_0}}. 
\een
\end{lemma}
\begin{proof}
This lemma follows directly from \eqref{eq:fourier-mu}, Remark \ref{Rmk: fourier-gevrey} and the fact that
\beno
\int_{\R}e^{-\la'(|\xi+\xi'|^2+|\eta+\xi'|^2)^{\f{s_0}{2}}}d\xi'
\lesssim \int_{\R}\f{e^{-\la|\xi-\eta|^{s_0}}}{1+|\xi+\xi'|^2+|\eta+\xi'|^2}d\xi'\lesssim e^{-\la|\xi-\eta|^{s_0}},
\eeno
holds for $0<\la<\la'$. 
\end{proof}
\begin{remark}\label{Rmk: improved kernel regularity}
We also point out that if the regularity assumption on $K(v, v')$ is replaced by 
\ben\label{eq: new-regularity-assu}
\sup_{m\geq 0}\f{\|(\pa_v+\pa_{v'})^{m}(1-\pa_{vv}-\pa_{v'v'})K(v, v')\|_{L^1}}{\Gamma_s(m)M^m}\leq C,
\een
then \eqref{eq: Kernel-Fourier-G} also holds. Indeed, the weaker regularity assumption \eqref{eq: new-regularity-assu} implies that there is $\la$ such that
\ben
|\widehat{\widehat{K}}(\xi,\eta)|\lesssim \f{e^{-\la|\xi+\eta|^s}}{1+|\xi|^2+|\eta|^2}.
\een
Then \eqref{eq: Kernel-Fourier-G} follows directly from \eqref{eq:fourier-mu}. 
\end{remark}

\section{Sturm-Liouville Equation}\label{sec: ST equ}
In this section, we consider two elliptic equations:
\beq\label{eq: steam function-ST}
\left\{\begin{aligned}
&\widetilde{\varphi_2}\pa_{zz}\Psi+\widetilde{u'}(\pa_v-t\pa_z)\Big(\widetilde{\varphi_2}\widetilde{u'}(\pa_v-t\pa_z)\Psi\Big)=\Om,\\
&\Psi(t, z, u(0))=\Psi(t, z, u(1))=0,
\end{aligned}\right.
\eeq
and
\beq\label{eq: pressure-lin}
\left\{\begin{aligned}
&\widetilde{\th}\pa_{zz}\Pi+\widetilde{u'}(\pa_v-t\pa_z)\Big(\widetilde{\th}\widetilde{u'}(\pa_v-t\pa_z)\Pi\Big)=\Xi\Upsilon,\\
&(\pa_v-t\pa_z)\Pi(t, z, u(0))=(\pa_v-t\pa_z)\Pi(t, z, u(1))=0.
\end{aligned}
\right.
\eeq
For the first equation \eqref{eq: steam function-ST}, if we let
\beq\label{eq: Q-Psi}
\mathfrak{Q}=\mathcal{F}_1[\Psi](t,k,v)e^{-itkv}
\eeq 
then $\mathfrak{Q}$ solves 
\beno
\left\{\begin{aligned}
&\pa_v(h_1(v)\pa_v\mathfrak{Q})-k^2h_2(v)\mathfrak{Q}=\mathfrak{F}(v, k)\\
&\mathfrak{Q}(\al, k)=\mathfrak{Q}(\b, k)=0
\end{aligned}\right.
\eeno
where 
\beq\label{eq: coefficients}
\begin{aligned}
&\al=u(0),\  \b=u(1),\ \mathfrak{F}(v, k)=\widetilde{u'}(v)^{-1}\mathcal{F}_1[\Om]e^{-iktv},\\ &h_1(v)=\widetilde{\varphi_2}\widetilde{u'}(v),\  h_2(v)=\widetilde{\varphi_2}(v)\widetilde{u'}(v)^{-1}.
\end{aligned}
\eeq 

For the second equation \eqref{eq: pressure-lin}, if we let
\ben
\mathfrak{R}=\mathcal{F}_1[\Pi](t, k, v)e^{-it kv}
\een
then $\mathfrak{R}$ solves 
\beno
\left\{\begin{aligned}
&\pa_v(h_1(v)\pa_v\mathfrak{R})-k^2h_2(v)\mathfrak{R}=\mathfrak{G}(v, k)\\
&\pa_v\mathfrak{R}(\al, k)=\pa_v\mathfrak{R}(\b, k)=0
\end{aligned}\right.
\eeno
where 
\beq\label{eq: coefficients-2}
\begin{aligned}
&\al=u(0),\  \b=u(1),\ \mathfrak{G}(v, k)=\widetilde{u'}(v)^{-1}\mathcal{F}_1[\Xi]e^{-iktv},\\ &h_1(v)=\widetilde{\th}\widetilde{u'}(v),\  h_2(v)=\widetilde{\th}(v)\widetilde{u'}(v)^{-1}.
\end{aligned}
\eeq 
Both equations are Sturm-Liouville Equations. 

\subsection{Dirichlet boundary conditions}\label{sec: SL-D}
In this section, we study the general Sturm-Liouville Equation
\ben\label{eq: SLequ}
\left\{\begin{aligned}
&\pa_v(h_1(v)\pa_v\mathfrak{Q})-k^2h_2(v)\mathfrak{Q}=\mathfrak{F}(v, k)\\
&\mathfrak{Q}(\al, k)=\mathfrak{Q}(\b, k)=0
\end{aligned}\right.
\een
where $h_1, h_2$ are positive functions and $\mathrm{supp}\, h_j' \subset [\al_1, \b_1]$ for $j=1,2$ with $\al<\la_1<\b_1<\b$. We also assume that 
$h_1', h_2'\in \mathcal{G}^{M, s}_{ph}([\al, \b])$. 
Let us first study the homogeneous equation. 
\begin{lemma}\label{eq: mfq}
There exist positive functions $\mfq_{\al}(v,k), \mfq_{\b}(v,k)\geq 1$ such that $\mfq_{\al}(\al,k)=\mfq_{\b}(\b, k)=1$, $\mfq_{\al}'(\al, k)=\mfq_{\b}'(\al, k)=0$, and $\mfq_{\al}(v, k)$ solves
\beno
\mfq_{\al}(v,k)=1+k^2\int_{\al}^v\f{1}{h_1(v')}\int_{\al}^{v'}h_2(v'')\mfq_{\al}(v'' ,k)d v''dv',
\eeno
and $\mfq_{\b}(v, k)$ solves
\beno
\mfq_{\b}(v,k)=1+k^2\int_{\b}^v\f{1}{h_1(v')}\int_{\b}^{v'}h_2(v'')\mfq_{\b}(v'' ,k)d v''dv'.
\eeno

Moreover, it holds for $v\in [\al, \b]$ that
\begin{align}
&\begin{aligned}
\label{eq: q'/q}
&C^{-1}\min\{|k|^2(v-\al), |k|\}\leq \f{\mfq_{\al}'(v,k)}{\mfq_{\al}(v,k)}\leq C\min\{|k|^2(v-\al), |k|\}\\
&C^{-1}\min\{|k|^2(\b-v), |k|\}\leq -\f{\mfq_{\b}'(v,k)}{\mfq_{\b}(v,k)}\leq C\min\{|k|^2(\b-v), |k|\}
\end{aligned}\\
\label{eq: q upperlower}
&e^{C^{-1}|k||v-\al|}\leq \mfq_{\al}(v,k)\leq e^{C|k||v-\al|}, \quad e^{C^{-1}|k||v-\b|}\leq \mfq_{\b}(v,k)\leq e^{C|k||v-\b|}. 
\end{align}
The Wronskian is
\beno
W[\mfq_{\al}, \mfq_{\b}](v)=\f{\mfq_{\b}'(\al)h_1(\al)}{h_1(v)}=-\f{\mfq_{\al}'(\b)h_1(\b)}{h_1(v)}<0. 
\eeno
which also implies 
\beno
-\f{\mfq_{\b}'(\al)}{\mfq_{\al}'(\b)}\approx \f{\mfq_{\b}(\al)}{\mfq_{\al}(\b)}\approx 1.
\eeno
\end{lemma}
\begin{proof}
We only show the proof for $\mfq_{\al}$. 

\no{\bf Existence. } It is easy to show that $T[\mfq_{\al}]=\int_{\al}^v\f{1}{h_1(v')}\int_{\al}^vh_2(v'')\mfq_{\al}(v'' ,k)d v''dv'$ is a contraction map in the weighted Gevrey space $\mathcal{G}_{ph, \cosh}^{M|k|,s}([\al,\b])\eqdef\{f\in C^{\infty}([\al,\b]):~\|f\|_{\mathcal{G}_{ph, \cosh}^{M|k|,s}([\al,\b])}<\infty\}$ by taking $M$ large enough, where 
\beno
\|f\|_{\mathcal{G}_{ph, \cosh}^{M|k|,s}([\al,\b])}\eqdef
\sup_{v\in [\al,\b], m\geq 0}\f{|\pa_v^mf|}{\Gamma_s(m)M^m|k|^{m}\cosh M|k|(v-\al)}
\eeno
and $\Gamma_s(m)=(m!)^{\f1s}(m+1)^{-2}$. The proof is simpler than the proof of Proposition \ref{prop: T}. We omit the details. Note that $T$ is a positive operator, then $\mfq_{\al}(v, k)\geq 1$ and $\mfq_{\al}'(v, k)\geq 0$. 

\no{\bf Estimates. } Let $\mathfrak{f}_{\al}=h_1\mfq_{\al}'/\mfq_{\al}$, then $\mathfrak{f}_{\al}'+\mathfrak{f}_{\al}^2/h_1=k^2h_2$ which gives that $0\leq \mathfrak{f}_{\al}\leq |k|\sqrt{h_2h_1}$ and thus $0\leq \mathfrak{q}_{\al}'/\mfq_{\al}\leq |k|\sqrt{h_2/h_1}\leq C|k|$. Therefore, we have for $v''\leq v$
\begin{align*}
1\geq \f{\mfq_{\al}(v'', k)}{\mfq_{\al}(v, k)}\geq e^{-C|k||v''-v|}. 
\end{align*}
We have for $|v-\al|\leq \f{1}{|k|}$, 
\begin{align*}
\f{\mathfrak{q}_{\al}'(v, k)}{\mfq_{\al}(v ,k)}=\f{k^2}{h_1(v)}\int_{\al}^{v}h_2(v'')\f{\mfq_{\al}(v'' ,k)}{\mfq_{\al}(v ,k)}d v''\approx k^2(v-\al). 
\end{align*}
As for $|v-\al|\geq \f{1}{|k|}$, we have 
\begin{align*}
\f{\mathfrak{q}_{\al}'(v, k)}{\mfq_{\al}(v ,k)}\geq \f{k^2}{h_1(v)}\int_{v-\f{1}{2|k|}}^{v}h_2(v'')\f{\mfq_{\al}(v'' ,k)}{\mfq_{\al}(v ,k)}d v'' \geq C^{-1}|k|. 
\end{align*}
Combining the above estimates, we get \eqref{eq: q'/q}, \eqref{eq: q upperlower} follows directly. 

Let us now computer the Wronskian. A direct calculation gives
\beno
\pa_v\Big(h_1W[\mfq_{\al}, \mfq_{\b}]\Big)=0
\eeno
which gives the formula of the Wronskian. 
\end{proof}

Let us now construct two linearly independent solutions $\mfq_{l}$ and $\mfq_{r}$:
\ben
\mfq_l(v,k)=\mfq_\al(v, k)-\mfq_{\b}(v, k)/\mfq_{\b}(\al, k)\\
\mfq_r(v,k)=\mfq_\b(v, k)-\mfq_{\al}(v, k)/\mfq_{\al}(\b, k)
\een

With the homogeneous solution $\mfq_r$ and $\mfq_l$, we get that 
\begin{align*}
\mathfrak{Q}(v, k)=\f{1}{h_1(v)}\f{\mfq_r(v,k)}{W[\mfq_r, \mfq_l]}\int_{0}^v\mfq_l(v',k)\mathfrak{F}(v',k)h_1(v')dv'
+\f{1}{h_1(v)}\f{\mfq_l(v,k)}{W[\mfq_r, \mfq_l]}\int_v^1\mfq_r(v',k)\mathfrak{F}(v',k)h_1(v')dv',
\end{align*}
where 
\begin{align*}
W[\mfq_r, \mfq_l]=\Big(1-\f{1}{\mfq_{\al}(\b)\mfq_{\b}(\al)}\Big)W[\mfq_{\al}, \mfq_{\b}](v)=\Big(1-\f{1}{\mfq_{\al}(\b)\mfq_{\b}(\al)}\Big)\f{\mfq_{\b}'(\al)h_1(\al)}{h_1(v)}
\end{align*}
We also express $\mathfrak{Q}(v, k)$ by using $\mfq_{\al}$ and $\mfq_{\b}$. 
\beq\label{eq: Q-represent}
\begin{aligned}
\mathfrak{Q}(v, k)&=\f{1}{h_1(v)}\f{1}{W[\mfq_r, \mfq_l]}\int_{\al}^v\Big(1+\f{1}{\mfq_{\b}(\al)\mfq_{\al}(\b)}\Big)\mfq_{\b}(v,k)\mfq_{\al}(v',k)\mathfrak{F}(v',k)h_1(v')dv'\\
&\quad+\f{1}{h_1(v)}\f{1}{W[\mfq_r, \mfq_l]}\int_v^{\b}\Big(1+\f{1}{\mfq_{\b}(\al)\mfq_{\al}(\b)}\Big)\mfq_{\al}(v,k)\mfq_{\b}(v',k)\mathfrak{F}(v',k)h_1(v')dv'\\
&\quad-\f{1}{h_1(v)}\f{1}{W[\mfq_r, \mfq_l]}\int_{\al}^{\b}\f{\mfq_{\al}(v,k)\mfq_{\al}(v',k)}{\mfq_{\al}(\b)}\mathfrak{F}(v',k)h_1(v')dv'\\
&\quad-\f{1}{h_1(v)}\f{1}{W[\mfq_r, \mfq_l]}\int_{\al}^{\b}\f{\mfq_{\b}(v,k)\mfq_{\b}(v',k)}{\mfq_{\b}(\al)}\mathfrak{F}(v',k)h_1(v')dv'
\end{aligned}
\eeq

We now introduce the Fourier kernel for the Sturm-Liouville type elliptic problem \eqref{eq: SLequ}. For any given fixed $k$, and positive Gevrey functions $h_1, h_2 \in \mathcal{G}^{M, s}_{ph}([\al, \b])$ with constant values near the boundary and any Gevrey cut-off function $\Upsilon$ with compact support $\mathrm{supp}\,\Upsilon\subset (\al, \b)$, we define a linear operator $\triangle_{D, k}^{-1}$ to be such that for a function $\mathfrak{F}$, 
\ben\label{eq:Def-inverlaplace}
\triangle_{D, k}^{-1}\mathfrak{F}=\Upsilon\mathfrak{Q}
\een 
where $\mathfrak{Q}$ solves
\beno
\left\{\begin{aligned}
&\pa_v(h_1(v)\pa_v\mathfrak{Q})-k^2h_2(v)\mathfrak{Q}=\mathfrak{F}(v, k)\Upsilon(v)\\
&\mathfrak{Q}(\al, k)=\mathfrak{Q}(\b, k)=0.
\end{aligned}\right.
\eeno
For the case $k=0$, we have
\begin{align*}
\triangle_{D, 0}^{-1}\mathfrak{F}=\int_{\al}^{\b}\mathbf{G}_D(v, v', 0)\Upsilon(v)\Upsilon(v')\mathfrak{F}(v', 0)dv'
\end{align*}
where
\begin{align*}
\mathbf{G}_D(v, v', 0)=\left\{
\begin{aligned}
\Big(\int_{\al}^{\b}\f{1}{h_1(v)}dv\Big)^{-1}\Big(\int_{\al}^v\f{1}{h_1(w)}dw\Big)\Big(\int_{\b}^{v'}\f{1}{h_1(w)}dw\Big)h_1(v')\quad \al\leq v\leq v'\leq \b\\
\Big(\int_{\al}^{\b}\f{1}{h_1(v)}dv\Big)^{-1}\Big(\int_{\al}^{v'}\f{1}{h_1(w)}dw\Big)\Big(\int_{\b}^{v}\f{1}{h_1(w)}dw\Big)h_1(v')\quad \al\leq v'\leq v\leq \b.
\end{aligned}
\right.
\end{align*}
By \eqref{eq: Q-represent}, we have for $k\neq 0$
\begin{align*}
\triangle_{D, k}^{-1}\mathfrak{F}=\int_{\al}^{\b}\mathbf{G}_D(v, v', k)\Upsilon(v)\Upsilon(v')\mathfrak{F}(v', k)dv'
\end{align*}
where
\beq\label{eq: Green-1}
\begin{aligned}
\mathbf{G}_{D}(v, v', k)
&=\mathfrak{C}^1_{\al,\b}\f{1}{|k|}\f{\mfq_{\al}(v',k)\mfq_{\b}(v,k)}{\mfq_{\b}(\al,k)}\f{h_1(v')}{h_1(v)}\chi_{\mathbb{R}^+}(v-v')\\
&\quad+\mathfrak{C}^1_{\al,\b}\f{1}{|k|}\f{\mfq_{\al}(v,k)\mfq_{\b}(v',k)}{\mfq_{\b}(\al,k)}\f{h_1(v')}{h_1(v)}\chi_{\mathbb{R}^+}(v'-v)\\
&\quad+\mathfrak{C}^2_{\al, \b}\f{1}{|k|}\f{\mfq_{\al}(v,k)\mfq_{\al}(v',k)}{\mfq_{\al}(\b, k)\mfq_{\al}(\b,k)}\f{h_1(v')}{h_1(v)}\\
&\quad+\mathfrak{C}^3_{\al, \b}\f{1}{|k|}\f{\mfq_{\b}(v,k)\mfq_{\b}(v',k)}{\mfq_{\b}(\al, k)\mfq_{\b}(\al, k)}\f{h_1(v')}{h_1(v)}
\end{aligned}
\eeq
with constants 
\beq\label{eq: constantsC}
\begin{aligned}
\mathfrak{C}^1_{\al,\b}&=\f{|k|\mfq_{\b}(\al)}{\mfq_{\b}'(\al)}\f{1}{h_1(\al)}\Big(1+\f{1}{\mfq_{\b}(\al)\mfq_{\al}(\b)}\Big)\Big(1-\f{1}{\mfq_{\al}(\b)\mfq_{\b}(\al)}\Big)^{-1}\\
\mathfrak{C}^2_{\al, \b}&=-\Big(1-\f{1}{\mfq_{\al}(\b)\mfq_{\b}(\al)}\Big)^{-1}\f{|k|\mfq_{\al}(\b)}{\mfq_{\b}'(\al)h_1(\al)}\\
\mathfrak{C}^3_{\al, \b}&=-\Big(1-\f{1}{\mfq_{\al}(\b)\mfq_{\b}(\al)}\Big)^{-1}\f{|k|\mfq_{\b}(\al)}{\mfq_{\b}'(\al)h_1(\al)},
\end{aligned}
\eeq
where $\mfq_{\al}$ and $\mfq_{\b}$ are obtained in Lemma \ref{eq: mfq}. 

A direct calculation gives that 
\begin{align}
\begin{aligned}\label{eq: pa_vK}
\pa_v\triangle_{D, k}^{-1}\mathfrak{F}
&=\int_{\al}^{\b}\mathbf{G}_D(v, v', k)\Upsilon'(v)\Upsilon(v')\mathfrak{F}(v', k)dv'\\
&\quad -\f{h_1'(v)}{h_1(v)}\int_{\al}^{\b}\mathbf{G}_D(v, v', k)\Upsilon(v)\Upsilon(v')\mathfrak{F}(v', k)dv'\\
&\quad +\int_{\al}^{\b}\mathbf{G}^1_D(v, v', k)\Upsilon(v)\Upsilon(v')\mathfrak{F}(v', k)dv',
\end{aligned}
\end{align}
and
\begin{align}
\begin{aligned}\label{eq: pa_vvK}
\pa_{vv}\triangle_{D, k}^{-1}\mathfrak{F}
&=\int_{\al}^{\b}\mathbf{G}_D(v, v', k)\Upsilon''(v)\Upsilon(v')\mathfrak{F}(v', k)dv'\\
&-2\f{h_1'(v)}{h_1(v)}\int_{\al}^{\b}\mathbf{G}_D(v, v', k)\Upsilon'(v)\Upsilon(v')\mathfrak{F}(v', k)dv'\\
&+2\int_{\al}^{\b}\mathbf{G}^1_D(v, v', k)\Upsilon'(v)\Upsilon(v')\mathfrak{F}(v', k)dv'\\
&\quad -\Big(\f{h_1'(v)}{h_1(v)}\Big)'\int_{\al}^{\b}\mathbf{G}_D(v, v', k)\Upsilon(v)\Upsilon(v')\mathfrak{F}(v', k)dv'\\
&\quad +\int_{\al}^{\b}\mathbf{G}^2_D(v, v', k)\Upsilon(v)\Upsilon(v')\mathfrak{F}(v', k)dv'\\
&\quad -2\f{h_1'(v)}{h_1(v)}\int_{\al}^{\b}\mathbf{G}_D^1(v, v', k)\Upsilon(v)\Upsilon(v')\mathfrak{F}(v', k)dv'\\
&+\mathfrak{C}^1_{\al,\b}\f{h_1(v)}{|k|}\left(\f{\mfq_{\al}(v,k)}{\mfq_{\b}(\al,k)}\Big(\f{\mfq_{\b}(v,k)}{h_1(v)}\Big)'-\f{\mfq_{\b}(v,k)}{\mfq_{\b}(\al,k)}\Big(\f{\mfq_{\al}(v,k)}{h_1(v)}\Big)'\right)\Upsilon(v)\mathfrak{F}(v, k)
\end{aligned}
\end{align}
where 
\begin{align*}
\mathbf{G}_{D}^1(v, v', k)
&=\mathfrak{C}^1_{\al,\b}\f{1}{|k|}\f{\mfq_{\b}'(v,k)}{\mfq_{\b}(v,k)}\f{\mfq_{\al}(v',k)\mfq_{\b}(v,k)}{\mfq_{\b}(\al,k)}\f{h_1(v')}{h_1(v)}\chi_{\mathbb{R}^+}(v-v')\\
&\quad+\mathfrak{C}^1_{\al,\b}\f{1}{|k|}\f{\mfq_{\al}'(v,k)}{\mfq_{\al}(v,k)}\f{\mfq_{\al}(v,k)\mfq_{\b}(v',k)}{\mfq_{\b}(\al,k)}\f{h_1(v')}{h_1(v)}\chi_{\mathbb{R}^+}(v'-v)\\
&\quad+\mathfrak{C}^2_{\al, \b}\f{1}{|k|}\f{\mfq_{\al}'(v,k)}{\mfq_{\al}(v,k)}\f{\mfq_{\al}(v,k)\mfq_{\al}(v',k)}{\mfq_{\al}(\b, k)\mfq_{\al}(\b,k)}\f{h_1(v')}{h_1(v)}\\
&\quad+\mathfrak{C}^3_{\al, \b}\f{1}{|k|}\f{\mfq_{\b}'(v,k)}{\mfq_{\b}(v,k)}\f{\mfq_{\b}(v,k)\mfq_{\b}(v',k)}{\mfq_{\b}(\al, k)\mfq_{\b}(\al, k)}\f{h_1(v')}{h_1(v)},\\
\mathbf{G}_{D}^2(v, v', k)
&=\mathfrak{C}^1_{\al,\b}\f{1}{|k|}\f{\mfq_{\b}''(v,k)}{\mfq_{\b}(v,k)}\f{\mfq_{\al}(v',k)\mfq_{\b}(v,k)}{\mfq_{\b}(\al,k)}\f{h_1(v')}{h_1(v)}\chi_{\mathbb{R}^+}(v-v')\\
&\quad+\mathfrak{C}^1_{\al,\b}\f{1}{|k|}\f{\mfq_{\al}''(v,k)}{\mfq_{\al}(v,k)}\f{\mfq_{\al}(v,k)\mfq_{\b}(v',k)}{\mfq_{\b}(\al,k)}\f{h_1(v')}{h_1(v)}\chi_{\mathbb{R}^+}(v'-v)\\
&\quad+\mathfrak{C}^2_{\al, \b}\f{1}{|k|}\f{\mfq_{\al}''(v,k)}{\mfq_{\al}(v,k)}\f{\mfq_{\al}(v,k)\mfq_{\al}(v',k)}{\mfq_{\al}(\b, k)\mfq_{\al}(\b,k)}\f{h_1(v')}{h_1(v)}\\
&\quad+\mathfrak{C}^3_{\al, \b}\f{1}{|k|}\f{\mfq_{\b}''(v,k)}{\mfq_{\b}(v,k)}\f{\mfq_{\b}(v,k)\mfq_{\b}(v',k)}{\mfq_{\b}(\al, k)\mfq_{\b}(\al, k)}\f{h_1(v')}{h_1(v)}. 
\end{align*}

\subsection{Regularity of the Green's function}
In this section, we study the Gevrey regularity of the Green's function. According to the above decompositions, we only need to estimate $\f{\mfq_{\al}(v',k)\mfq_{\b}(v,k)}{\mfq_{\b}(\al,k)}$ with $v'\leq v$, $\f{\mfq_{\al}(v,k)\mfq_{\b}(v',k)}{\mfq_{\b}(\al,k)}$ with $v'\geq v$ and $\f{\mfq_{\al}(v,k)}{\mfq_{\al}(\b,k)}$ with $v<\b-\d$, $\f{\mfq_{\b}(v,k)}{\mfq_{\b}(\al,k)}$ with $v\geq \al+\d$ for some $\d>0$. By Lemma \ref{eq: mfq}, it is easy to check all of the four terms are in Gevrey. So we only focus on the uniformity in $k$ of the Gevrey norm. Now we assume $|k|\geq k_0>0$ with $k_0$ large enough which is determined by $h_1, h_2$ and $\d$. Without loss of generality, we assume $k\geq k_0>0$. 

Now let us first treat $\f{\mfq_{\al}(v,k)}{\mfq_{\al}(\b,k)}$ with $v<\b-\d$, $\f{\mfq_{\b}(v,k)}{\mfq_{\b}(\al,k)}$ with $v\geq \al+\d$ for some $\d>0$ and give the following lemma: 
\begin{lemma}\label{lem: boundary q/q}
It holds for $\al+\d<v\leq \b-\d$
\ben\label{eq: q/q gevrey}
\left|\pa_v^{m}\Big(\f{\mfq_{\al}(v,k)}{\mfq_{\al}(\b,k)}\Big)\right|\leq C\Gamma_s(m)M^m,
\een
and for $\al+\d<v<\b-\d$
\ben\label{eq: q/q gevrey-2}
\left|\pa_v^{m}\Big(\f{\mfq_{\b}(v,k)}{\mfq_{\b}(\al,k)}\Big)\right|\leq C\Gamma_s(m)M^m.
\een
\end{lemma}
\begin{proof}
Let us prove the first estimate and the second one can be obtained by the same argument. Notice that $\mfq=\{\mfq_{\al}, \mfq_{\b}\}$ solves
\beno
\pa_v(h_1\pa_v\mfq)=k^2h_2\mfq.
\eeno
Thus for $m_1\geq 2$ or $m_1\geq 2$, a direct calculation gives 
\begin{align*}
\pa_v^{m_1}\mfq&=\pa_v^{m_1-2}\Big(\f{-h_1'}{h_1}\pa_v\mfq+k^2\f{h_2}{h_1}\mfq\Big)\\
&=\sum_{j=0}^{m_1-2}\f{(m_1-2)!}{j!(m_1-2-j)!}\pa_v^{j}\Big(\f{-h_1'}{h_1}\Big)\pa_v^{m_1-2-j+1}\mfq
+k^2\sum_{l=0}^{m_1-2}\f{(m_1-2)!}{l!(m_1-2-l)!}\pa_v^{l}\Big(\f{h_2}{h_1}\Big)\pa_v^{m_1-2-l}\mfq
\end{align*}
We now use mathematical induction to prove the following result: \\
For all $N\geq 0$, it holds that
\beq\label{eq: MI-ass-1}
\left|\pa_v^{N}\Big(\f{\mfq_{\al}(v,k)}{\mfq_{\al}(\b, k)}\Big)\right|\leq \Big(\sum_{j=0}^{N}\f{N!}{j!(N-j)!}\Gamma_s(N-j)C^{N-j}(C|k|)^{j}e^{-C^{-1}\d |k|}\Big).
\eeq
It follows directly by the monotonicity of $\mfq_{\al}(v,k)$ that
\beno
\f{\mfq_{\al}(v,k)}{\mfq_{\al}(\b, k)}\lesssim \f{\mfq_{\al}(\b-\d,k)}{\mfq_{\al}(\b, k)}\lesssim e^{-\int_{\b-\d}^{\b}\mathfrak{f}_{\al}(v)/h_1(v)dv}\lesssim e^{C^{-1}k\d},
\eeno
and by \eqref{eq: q'/q}, it holds that
\beno
\f{\mfq_{\al}'(v,k)}{\mfq_{\al}(\b, k)}\lesssim k\f{\mfq_{\al}(v,k)}{\mfq_{\al}(\b, k)}\lesssim k e^{-\int_{\b-\d}^{\b}\mathfrak{f}_{\al}(v)/h_1(v)dv}\lesssim k e^{C^{-1}k\d}.
\eeno
Let us now assume that \eqref{eq: MI-ass} holds for any $N\leq m-1$ with $m\geq 2$, then for $N=m$, and by \eqref{eq:3.9}, we get that 
\begin{align*}
&\left|\f{\pa_v^{m}\mfq_{\al}}{\mfq_{\al}(\b, k)}\right|=\left|\pa_v^{m-2}\Big(\f{-h_1'}{h_1}\pa_v\mfq+k^2\f{h_2}{h_1}\mfq\Big)/\mfq_{\al}(\b, k)\right|\\
&\leq \sum_{j=0}^{m-2}\f{(m-2)!}{j!(m-2-j)!}\left|\pa_v^{j}\Big(\f{-h_1'}{h_1}\Big)\f{\pa_v^{m-2-j+1}\mfq_{\al}(v)}{\mfq_{\al}(\b, k)}\right|
+k^2\sum_{l=0}^{m-2}\f{(m-2)!}{l!(m-2-l)!}\left|\pa_v^{l}\Big(\f{h_2}{h_1}\Big)\f{\pa_v^{m-2-l}\mfq_{\al}}{\mfq_{\al}(\b, k)}\right|\\
&\leq  \sum_{j=0}^{m-2}\f{(m-2)!}{j!(m-2-j)!}M_1^j\Gamma_s(j)\Big(\sum_{j_1=0}^{m_j}\f{m_j!}{j_1!(m_j-j_1)!}\Gamma_s(m_j-j_1)C^{m_j-j_1}(C|k|)^{j_1}e^{-C^{-1}\d |k|}\Big)\\
&\quad + k^2\sum_{l=0}^{m-2}\f{(m-2)!}{l!(m-2-l)!}M_1^l\Gamma_s(l)\Big(\sum_{l_1=0}^{m_l-1}\f{(m_l-1)!}{l_1!((m_l-1)-l_1)!}\Gamma_s(m_l-1-l_1)C^{m_l-1-l_1}(C|k|)^{l_1}e^{-C^{-1}\d |k|}\Big)\\
&\leq \sum_{j=0}^{m-2}\sum_{j_1=0}^{m_j}\f{(m-2)!}{j!(m-2-j)!}\Gamma_s(j)\f{m_j!}{j_1!(m_j-j_1)!}\Gamma_s(m-1-j-j_1)C^{m-1-j_1}(C|k|)^{j_1}e^{-C^{-1}\d |k|}\\
&\quad + \sum_{l=0}^{m-2}\sum_{l_1=0}^{m_l-1}\f{(m-2)!}{l!(m-2-l)!}\Gamma_s(l)\f{(m_l-1)!}{l_1!((m_l-1)-l_1)!}\Gamma_s(m-l-2-l_1)C^{m-2-l_1}(C|k|)^{l_1+2}e^{-C^{-1}\d |k|}\\
&\leq \sum_{j_1=0}^{m-1}\f{(m-1)!}{j_1!(m-1-j_1)!}\Gamma_s(m-1-j_1)C^{m-1-j_1}(C|k|)^{j_1}e^{-C^{-1}\d |k|}\\
&\quad +\sum_{l_1=0}^{m-2}\f{(m-l_1)!}{(l_1+2)!(m-2-l_1)!}\Gamma_s(m-2-l_1)C^{m-2-l_1}(C|k|)^{l_1+2}e^{-C^{-1}\d |k|},
\end{align*}
where $m_j=m-1-j$. Thus we proved \eqref{eq: MI-ass-1}. The lemma follows directly by applying
\beno
(C|k|)^{j}e^{-C^{-1}\d |k|}\lesssim (C\d)^{-j}(j^j/e^j)< (C\d)^{-j}\Gamma_s(j),
\eeno
and \eqref{eq:3.9}. 
\end{proof}

Let us rewrite $\mfq_{\b}$. Note that $\mfq_{\al}(v)$ and 
\beno
h_1(\al)\mfq_{\al}(v)\int_{\al}^{v}\f{1}{h_1(w)\mfq_{\al}(w)^2}dw
\eeno
are two linearly independent solutions. It is easy to get that 
\begin{align*}
\mfq_{\b}(v)
&=\mfq_{\b}(\al)\mfq_{\al}(v)+\mfq_{\b}'(\al)h_1(\al)\mfq_{\al}(v)\int_{\al}^{v}\f{1}{h_1(w)\mfq_{\al}(w)^2}dw\\
&=\f{\mfq_{\al}(v)}{\mfq_{\al}(\b)}-h_1(\b)\mfq_{\al}'(\b)\mfq_{\al}(v)\int_{\b}^{v}\f{1}{h_1(w)\mfq_{\al}(w)^2}dw.
\end{align*}

\begin{lemma}\label{lem: G-regularity-1}
It holds for $v'<v-\d$ and $v,v'\in [\al+\d,\b-\d]$, that
\beno
\left|\pa_v^{m-m_1}\pa_{v'}^{m_1}\Big(\f{\mfq_{\b}(v,k)\mfq_{\al}(v',k)}{\mfq_{\b}(\al, k)}\Big)\right|
\leq C\Gamma_s(m)M^m,
\eeno
and for $v'>v+\d$
\beno
\left|\pa_v^{m-m_1}\pa_{v'}^{m_1}\Big(\f{\mfq_{\al}(v,k)\mfq_{\b}(v',k)}{\mfq_{\b}(\al, k)}\Big)\right|
\leq C\Gamma_s(m)M^m.
\eeno
\end{lemma}
\begin{proof}
The two estimates can be proved by the same argument. Here we only prove the first inequality. 
We write
\begin{align*}
\Big(\f{\mfq_{\b}(v,k)\mfq_{\al}(v',k)}{\mfq_{\b}(\al, k)}\Big)
&=\Big(\f{\mfq_{\b}(v,k)\mfq_{\al}(v,k)}{\mfq_{\b}(\al, k)}\Big)\Big(\f{\mfq_{\al}(v',k)}{\mfq_{\al}(v,k)}\Big)\\
&=\f{\mfq_{\al}(v)\mfq_{\al}(v',k)}{\mfq_{\al}(\b)\mfq_{\b}(\al, k)}-\f{h_1(\b)\mfq_{\al}'(\b)}{\mfq_{\b}(\al, k)}\mfq_{\al}(v)\mfq_{\al}(v')\int_{\b}^{v}\f{1}{h_1(w)\mfq_{\al}(w)^2}dw
\end{align*}
For the first term, we have by Lemma \ref{lem: boundary q/q} that
\begin{align*}
\left|\pa_v^{m-m_1}\pa_{v'}^{m_1}\Big(\f{\mfq_{\al}(v)\mfq_{\al}(v',k)}{\mfq_{\al}(\b)\mfq_{\b}(\al, k)}\Big)\right|
\leq C\Gamma_s(m)M^m,
\end{align*}
Then by the monotonicity of $\mfq_{\al}$, we have $0< \f{\mfq_{\al}(v',k)}{\mfq_{\al}(v,k)}\lesssim e^{-C^{-1}\d k}$ and $\f{\mfq_{\al}(v)}{\mfq_{\al}(w)}\lesssim e^{-C^{-1}k|v-w|}$ for $v\leq w$, which gives that 
\begin{align*}
&\left|\mfq_{\al}(v)\mfq_{\al}(v')\int_{\b}^{v}\f{1}{h_1(w)\mfq_{\al}(w)^2}dw\right|
\lesssim k\int_{\b}^{v}e^{-2C^{-1}k|v-w|}dwe^{-C^{-1}k|v-v'|}\lesssim e^{-C^{-1}\d k},\\
&\left|(\pa_v,\pa_{v'})\Big(\mfq_{\al}(v)\mfq_{\al}(v')\int_{\b}^{v}\f{1}{h_1(w)\mfq_{\al}(w)^2}dw\Big)\right|
\lesssim k^2\int_{\b}^{v}e^{-2C^{-1}k|v-w|}dwe^{-C^{-1}k|v-v'|}\lesssim ke^{-C^{-1}\d k},\\
&\left|(\pa_v,\pa_{v'})^2\Big(\mfq_{\al}(v)\mfq_{\al}(v')\int_{\b}^{v}\f{1}{h_1(w)\mfq_{\al}(w)^2}dw\Big)\right|
\lesssim k^3\int_{\b}^{v}e^{-2C^{-1}k|v-w|}dwe^{-C^{-1}k|v-v'|}\lesssim k^2e^{-C^{-1}\d k}.
\end{align*}
We now use mathematical induction to prove the following result: \\
For all $N=N_1+N_2$, it holds that
\beq\label{eq: MI-ass}
\left|\pa_v^{N_1}\pa_{v}^{N_2}\Big(\f{\mfq_{\b}(v,k)\mfq_{\al}(v',k)}{\mfq_{\b}(\al, k)}\Big)\right|\leq \Big(\sum_{j=0}^{N}\f{N!}{j!(N-j)!}\Gamma_s(N-j)C^{N-j}(C|k|)^{j}e^{-C^{-1}\d |k|}\Big).
\eeq
Suppose that \eqref{eq: MI-ass} holds for $N_1+N_2=0, 1, 2, ..., m_1+m_2+1$, we get for $N_1=m_1+m_2+2$
Thus we have
\begin{align*}
&\left|\f{\pa_v^{m_1+2}\mfq_{\b}(v)\pa_{v'}^{m_2}\mfq_{\al}(v')}{\mfq_{\b}(\al, k)}\right|\\
&\leq \sum_{j=0}^{m_1}\left|\f{(m_1)!}{j!(m_1-j)!}\pa_v^{j}\Big(\f{-h_1'}{h_1}\Big)\pa_v^{m_1+1-j}\mfq_{\b}(v)\pa_{v'}^{m_2}\mfq_{\al}(v')/\mfq_{\b}(\al, k)\right|\\
&\quad+k^2\sum_{l=0}^{m_1}\left|\f{(m_1)!}{l!(m_1-l)!}\pa_v^{l}\Big(\f{h_2}{h_1}\Big)\pa_v^{m_1-l}\mfq_{\b}(v)\pa_{v'}^{m_2}\mfq_{\al}(v')/\mfq_{\b}(\al, k)\right|\\
&\leq \sum_{j=0}^{m_1}\f{(m_1)!}{j!(m_1-j)!}C^{j}\Gamma_s(j)\Big(\sum_{j_1=0}^{N_j}\f{N_j!}{j_1!(N_j-j_1)!}\Gamma_s(N_j-j_1)C^{N_j-j_1}(C|k|)^{j_1}e^{-C^{-1}\d |k|}\Big)\\
&\quad+C^{-2}\sum_{l=0}^{m_1}\f{(m_1)!}{l!(m_1-l)!}C^{l}\Gamma_s(l)\Big(\sum_{j_2=0}^{N_l}\f{N_l!}{j_2!(N_l-j_2)!}\Gamma_s(N_l-j_2)C^{N_l-j_2}(C|k|)^{j_2+2}e^{-C^{-1}\d |k|}\Big)\\
&=I_1+I_2,
\end{align*}
where $N_j=m_1+m_2+1-j$ and $N_l=m_1+m_2-l$.

For the first term, we have
\begin{align*}
I_1&=\sum_{j=0}^{m_1}\sum_{j_1=0}^{N_j}\f{(m_1)!}{j!(m_1-j)!}\Gamma_s(j)\f{N_j!}{j_1!(N_j-j_1)!}\Gamma_s(N_j-j_1)C^{m_1+m_2+1-j_1}(C|k|)^{j_1}e^{-C^{-1}\d |k|}\\
&=\sum_{j=0}^{m_1}\sum_{j_1=0}^{N_j}
\left(\f{m_1!(m_1+m_2+1-j)!}{(m_1-j)!(m_1+m_2+1-j_1)!}\right)\f{(m_1+m_2+1-j_1)!}{j_1!(m_1+m_2+1-j_1)!}\\
&\quad\quad\times\left(\f{(N_j-j_1+j)!}{j!(N_j-j_1)!}\Gamma_s(N_j-j_1)\Gamma_s(j)\right)C^{m_1+m_2+1-j_1}(C|k|)^{j_1}e^{-C^{-1}\d |k|}\\
&\leq \sum_{j_1=0}^{m_1+m_2+1}
\f{(m_1+m_2+1-j_1)!}{j_1!(m_1+m_2+1-j_1)!}\Gamma_s(m_1+m_2+1-j_1)C^{m_1+m_2+1-j_1}(C|k|)^{j_1}e^{-C^{-1}\d |k|}.
\end{align*}
Similarly, for the second term, we have
\begin{align*}
I_2&= C^{-2}\sum_{l=0}^{m_1}\sum_{j_2=0}^{N_l}\f{(m_1)!}{l!(m_1-l)!}\Gamma_s(l)\f{N_l!}{j_2!(N_l-j_2)!}\Gamma_s(N_l-j_2)C^{m_1+m_2-j_2}(C|k|)^{j_2+2}e^{-C^{-1}\d |k|}\\
&\leq C^{-2}\sum_{j_2=0}^{m_1+m_2}\f{(m_1+m_2+2-j_2)!}{(j_2+2)!(m_1+m_2-j_2)!}\Gamma_s(m_1+m_2-j_2)C^{m_1+m_2-j_2}(C|k|)^{j_2+2}e^{-C^{-1}\d |k|}.
\end{align*}
Combining the above estimate, 
\begin{align*}
\left|\f{\pa_v^{m_1+2}\mfq_{\b}(v)\pa_{v'}^{m_2}\mfq_{\al}(v')}{\mfq_{\b}(\al, k)}\right|
\leq \Big(\sum_{j=0}^{N}\f{N!}{j!(N-j)!}\Gamma_s(N-j)C^{N-j}(C|k|)^{j}e^{-C^{-1}\d |k|}\Big)
\end{align*}
with $N=m_1+m_2+2$. Using a similar argument, we can also get for any $m_1, m_2$, 
\begin{align*}
\left|\f{\pa_v^{m_1}\mfq_{\b}(v)\pa_{v'}^{m_2+2}\mfq_{\al}(v')}{\mfq_{\b}(\al, k)}\right|
\leq \Big(\sum_{j=0}^{N}\f{N!}{j!(N-j)!}\Gamma_s(N-j)C^{N-j}(C|k|)^{j}e^{-C^{-1}\d |k|}\Big)
\end{align*}
holds for $N=m_1+m_2+2$. Thus we proved \eqref{eq: MI-ass}. The lemma follows directly by applying
\beno
(C|k|)^{j}e^{-C^{-1}\d |k|}\lesssim (C\d)^{-j}(j^j/e^j)< (C\d)^{-j}\Gamma_s(j),
\eeno
and \eqref{eq:3.9}. 
\end{proof}

By using the fact that $h_1, h_2$ are constant-valued functions near the boundary, and the boundary condition of $\mfq_{\al}(v)$, we obtain that for $\al\leq v\leq \al+\d$, 
\beno
\mfq_{\al}(v)=\cosh \Big(k\int_{\al}^v\sqrt{\f{h_2}{h_1}}(w)dw\Big)
\eeno
and 
\beno
k\sqrt{h_1(\al)h_2(\al)}\mfq_{\al}(v)\int_{\al}^{v}\f{1}{h_1(w)\mfq_{\al}(w)^2}dw=\sinh  \Big(k\int_{\al}^v\sqrt{\f{h_2}{h_1}}(w)dw\Big).
\eeno
Let 
\ben\label{eq: def-e_al+}
\mfe_{\al,+}(v)=\mfq_{\al}(v)+h_1(\al)\mfq_{\al}(v)\int_{\al}^{v}\f{1}{h_1(w)\mfq_{\al}(w)^2}dw=e^{\int_{\al}^v\mfs_{\al,+}(w)/h_1(w)dw},
\een
then a direct calculation gives that $\mfe_{\al,+}(v)$ solves
\beno
\pa_v(h_1\pa_v\mfe_{\al,+}(v))=k^2h_2\mfe_{\al,+}(v)
\eeno
and that $\mfs_{\al,+}(v)=\f{h_1(v)\mfe_{\al,+}'(v)}{\mfe_{\al,+}(v)}$ solves
\beno
\mfs_{\al,+}'(v)+\f{\mfs_{\al,+}(v)^2}{h_1(v)}=k^2h_2(v),
\eeno
with boundary condition $\mfs_{\al,+}(\al)=k \sqrt{h_1(\al)h_2(\al)}$. 

Note that the above calculation for $\mfs_{\al,+}$ is similar to $\mathfrak{f}_{\al}$. 
\begin{lemma}\label{lem:mfs}
There is $k_0>0$ and there are $C$ and $M$ independent of $k$, so that  for any $k>k_0$, it holds that
\beno
\sup_{m\geq 0}\left|\pa_v^m\big(\mfs_{\al,+}(v)-k\sqrt{h_1(v)h_2(v)}\big)\right|\leq C\Gamma_s(m)M^m.
\eeno
\end{lemma}
\begin{proof}
We use the idea of WKB approximation and rewrite $\mfs_{\al,+}(v)$ as 
\begin{align*}
\mfs_{\al,+}(v)=k\sqrt{h_1(v)h_2(v)}-\f{(\sqrt{h_1(v)h_2(v)})'}{2\sqrt{h_2(v)/h_1(v)}}+r(v),
\end{align*}
then $r$ solves
\beno
r'+2k\sqrt{h_2/h_1}r-\f{(\sqrt{h_1(v)h_2(v)})'}{\sqrt{h_2(v)/h_1(v)}}r+\f{r^2}{h_1}-\Big(\f{(\sqrt{h_1(v)h_2(v)})'}{2\sqrt{h_2(v)/h_1(v)}}\Big)'+\Big(\f{(\sqrt{h_1(v)h_2(v)})'}{2\sqrt{h_2(v)}}\Big)^2=0.
\eeno
Let $R(w)=r(v)$ with $w(v)=\int_{\al}^v\sqrt{h_2/h_1}(z)dz$, then 
\beno
R'+2kR+A(w)R+B(w)R^2+C(w)=0,
\eeno
with initial data $R(0)=0$, 
where 
\begin{align*}
&A(w)=-\sqrt{h_1/h_2}\f{(\sqrt{h_1(v)h_2(v)})'}{\sqrt{h_2(v)/h_1(v)}}\\
&B(w)=\sqrt{\f{1}{h_1(v)h_2(v)}}\\
&C(w)=-\sqrt{h_1/h_2}\Big(\f{(\sqrt{h_1(v)h_2(v)})'}{2\sqrt{h_2(v)/h_1(v)}}\Big)'+\sqrt{h_1/h_2}\Big(\f{(\sqrt{h_1(v)h_2(v)})'}{2\sqrt{h_2(v)}}\Big)^2.
\end{align*}
By the Gevrey regularity of $h_1, h_2$ and composition Lemma \ref{eq: composition}, it is easy to check that there are $C$ and $M_1$ such that 
\beno
\sup_{m\geq 0}|\pa_w^m(A, B, C)(w)|\leq C\Gamma_s(m)M_1^m.
\eeno
Let $D^{m}=\f{\pa_w^m}{\Gamma_s(m)M^m}$, then we get that 
\begin{align*}
\f12\f{d}{dw}|D^mR|^2+2k|D^mR|^2
\leq C(\sup_{0\leq m_1\leq m}|D^{m_1}R|+1)^2|D^mR|.
\end{align*}
By the Gronwall's inequality, we get that for $k_0>0$ large enough and $k>k_0$, 
\beno
\sup_{m}|D^mR(w)|^2\lesssim \f{1}{k},
\eeno
which together with Lemma \ref{eq: composition} gives
\beno
\sup_{m\geq 0}|\pa_v^mr(v)|\leq C\f{\Gamma_s(m)M_1^m}{\sqrt{k}}.
\eeno
Thus we proved the lemma. 
\end{proof}

Thus we have the following remark:
\begin{remark}\label{Rmk: e_al^-1}
It holds that for $v\geq \al+\d$
\begin{align*}
\left|\pa_v^{m+1}\Big(\f{1}{\mfe_{\al,+}(v)}\Big)\right|\lesssim \Gamma_s(m)M^{m}e^{-C^{-1}k|v-\al|}
\end{align*}
\end{remark}
\begin{proof}
We rewrite $\f{1}{\mfe_{\al,+}(v)}=g_{\mfe}(v)e^{-\int_{\al}^vk\sqrt{h_1h_2(w)}dw}$. Then Lemma \ref{lem:mfs} gives that 
\beno
\left|\pa_v^{m+1}g_{\mfe}(v)\right|\lesssim \Gamma_s(m)M^{m}. 
\eeno
We also have for $v\geq \al+\d$
\begin{align*}
\left|\pa_v^{m+1}e^{-\int_{\al}^vk\sqrt{h_1h_2(w)}dw}\right|
\leq \Gamma_s(m)M^{m}e^{-C^{-1}k|v-\al|}
\end{align*}
which gives the remark. 
\end{proof}

\begin{lemma}
For any $\f{\b-\al}{100}>\d>0$, there are $M$ and $C$ independent of $k$, such that for $|v'-v|\leq \d$
\begin{align*}
\sup_{m\geq 0}\int_{\al+\d}^{\b-\d}\int_{\{0<v-v'\leq \d\}\cap [\al,\b]}\left|\f{(\pa_v+\pa_{v'})^m}{\Gamma_s(m)M^m}|k|^{-1}(1-\pa_{vv}-\pa_{v'v'})\Big(\f{\mfq_{\al}(v')}{\mfq_{\al}(v)}\Big)\right|dvdv'\leq C. 
\end{align*}
and 
\beno
\left|\f{(\pa_v+\pa_{v'})^m}{\Gamma_s(m)M^m}\Big(\f{\mfq_{\al}(v')}{\mfq_{\al}(v)}\Big)\right|\leq C
\eeno
\end{lemma}
\begin{proof}
Let us first use $\mfe_{\al,+}$ to represent $\mfq_{\al}(v)$, by matching the boundary condition, we have
\begin{align*}
\mfq_{\al}(v)=\mfe_{\al,+}(v)-k\sqrt{h_1(\al)h_2(\al)}\mfe_{\al,+}(v)\int_{\al}^v\f{1}{h_1(w)\mfe_{\al,+}(w)^2}dw. 
\end{align*}
Thus we have 
\begin{align}\label{eq: q/q by e/e}
\f{\mfq_{\al}(v)}{\mfq_{\al}(w)}=\f{\mfe_{\al,+}(v)}{\mfe_{\al,+}(w)}\left(1+\f{k\sqrt{h_1(\al)h_2(\al)}\int_{w}^v\f{1}{h_1(z)\mfe_{\al,+}(z)^2}dz}{1-k\sqrt{h_1(\al)h_2(\al)}\int_{\al}^v\f{1}{h_1(w)\mfe_{\al,+}(w)^2}dw}\right)
\end{align}
By the fact that $\mfq_{\al}(v)\geq 1$ and \eqref{eq: def-e_al+}, it is easy to 
\beno
1-k\sqrt{h_1(\al)h_2(\al)}\int_{\al}^v\f{1}{h_1(w)\mfe_{\al,+}(w)^2}dw\geq c_0>0
\eeno
for some $c_0$ independent of $k$. By Remark \ref{Rmk: e_al^-1}, we have that for $m\geq 1$ and $w\geq \al+\d$
\begin{align*}
&\left|\pa_{v}^{m}\Big(1-k\sqrt{h_1(\al)h_2(\al)}\int_{\al}^v\f{1}{h_1(w)\mfe_{\al,+}(w)^2}dw\Big)\right|\\
&\lesssim k\left|\pa_v^{m-1}\Big(\f{1}{h_1(w)\mfe_{\al,+}(w)^2}\Big)\right|\lesssim \Gamma_s(m)M^{m}
\end{align*}
and
\begin{align*}
&\left|(\pa_v+\pa_w)^m\left(k\sqrt{h_1(\al)h_2(\al)}\int_{w}^v\f{1}{h_1(z)\mfe_{\al,+}(z)^2}dz\right)\right|\\
&\lesssim k \left|\int_{w}^v\pa_z^m\left(\f{1}{h_1(z)\mfe_{\al,+}(z)^2}\right)dz\right|\lesssim \Gamma_s(m)M^{m},\\
&\left\|(\pa_v+\pa_w)^m(1-\pa_{vv}-\pa_{ww})\left(k\sqrt{h_1(\al)h_2(\al)}\int_{w}^v\f{1}{h_1(z)\mfe_{\al,+}(z)^2}dz\right)\right\|_{L_{v,w}^1(\{0<v<w-\d\}\cap [\al+\d,\b-\d]^2)}\\
&\lesssim \int_{\al+\d}^{\b-\d}k \left|\pa_v^{m+1}\left(\f{1}{h_1(v)\mfe_{\al,+}(v)^2}\right)\right|dv+\int_{\al+\d}^{\b-\d}k \left|\pa_w^{m+1}\left(\f{1}{h_1(w)\mfe_{\al,+}(w)^2}\right)\right|dw\\
&\lesssim \Gamma_s(m)M^{m}
\end{align*}
here we use the fact that
\beno
\mfe_{\al,+}(w)\geq e^{-C^{-1}k |w-\al|}.
\eeno
Note that 
\begin{align*}
\f{\mfe_{\al,+}(v)}{\mfe_{\al,+}(w)}=\f{g_{\mfe}(v)}{g_{\mfe}(w)}e^{\int_{w}^vk\sqrt{h_1h_2(z)}dz},
\end{align*}
where $g_{\mfe}(v)$ is defined in Remark \ref{Rmk: e_al^-1}. 
Thus we obtain that for $0\leq w-v\leq \d$
\begin{align*}
&(\pa_v+\pa_w)^me^{\int_{w}^vk\sqrt{h_1h_2(z)}dz}\\
&=(\pa_v+\pa_w)^{m-1}e^{\int_{w}^vk\sqrt{h_1h_2(z)}dz}k\Big(\int_{w}^v(\sqrt{h_1h_2(z)})'dz\Big)\\
&\quad +e^{\int_{w}^vk\sqrt{h_1h_2(z)}dz}k\Big(\int_{w}^v\f{d^m}{dz^m}(\sqrt{h_1h_2(z)})dz\Big),
\end{align*}
which gives that
\begin{align*}
|(\pa_v+\pa_w)^me^{\int_{w}^vk\sqrt{h_1h_2(z)}dz}|
&\leq e^{-C^{-1}k|v-w|}\sum_{j=0}^{m}(k|v-w|)^jC^{m-j}\Gamma_s(m-j)\\
&\lesssim \Gamma_s(m)M^m,
\end{align*}
and
\begin{align*}
\left|(1-\pa_{vv}-\pa_{ww})(\pa_v+\pa_w)^me^{\int_{w}^vk\sqrt{h_1h_2(z)}dz}\right|
&\lesssim k^2\Gamma_s(m)M^m e^{-C^{-1}k|w-v|}.
\end{align*}
Combining all the estimates, together with \eqref{eq: q/q by e/e}, we get the lemma. 
\end{proof}
\begin{lemma}\label{lem: G-regularity-1-v near v'}
It holds for $|v'-v|\leq \d$, $v,v'\in [\al+\d,\b-\d]$ that
\begin{align*}
\sup_{m\geq 0}\int_{\al+\d}^{\b-\d}\int_{\{0<v-v'\leq \d\}\cap [\al,\b]}\left|\f{(\pa_v+\pa_{v'})^m}{\Gamma_s(m)M^m}|k|^{-1}(1-\pa_{vv}-\pa_{v'v'})\Big(\f{\mfq_{\b}(v,k)\mfq_{\al}(v',k)}{\mfq_{\b}(\al, k)}\Big)\right|dvdv'\leq C. 
\end{align*}
and for $v-\d<v'<v$
\beno
\left|\f{(\pa_v+\pa_{v'})^m}{\Gamma_s(m)M^m}\Big(\f{\mfq_{\b}(v,k)\mfq_{\al}(v',k)}{\mfq_{\b}(\al, k)}\Big)\right|\leq C
\eeno
and
\begin{align*}
\sup_{m\geq 0}\int_{\al+\d}^{\b-\d}\int_{\{0<v'-v\leq \d\}\cap [\al,\b]}\left|\f{(\pa_v+\pa_{v'})^m}{\Gamma_s(m)M^m}|k|^{-1}(1-\pa_{vv}-\pa_{v'v'})\Big(\f{\mfq_{\al}(v,k)\mfq_{\b}(v',k)}{\mfq_{\b}(\al, k)}\Big)\right|dvdv'\leq C. 
\end{align*}
\beno
\left|\pa_v^{m-m_1}\pa_{v'}^{m_1}\Big(\f{\mfq_{\b}(v,k)\mfq_{\al}(v',k)}{\mfq_{\b}(\al, k)}\Big)\right|
\leq C\Gamma_s(m)M^m,
\eeno
and for $v'-\d<v<v'$
\beno
\left|\f{(\pa_v+\pa_{v'})^m}{\Gamma_s(m)M^m}\Big(\f{\mfq_{\al}(v,k)\mfq_{\b}(v',k)}{\mfq_{\b}(\al, k)}\Big)\right|\leq C. 
\eeno
\end{lemma}

\begin{proposition}
Let $\widehat{\mathbf{G}}_{D}$ be the Fourier transform of the kernel $\mathbf{G}_{D}$ defined in \eqref{eq: Green-1}, namely, 
\beno
\widehat{\mathbf{G}}_{D}(k, \xi_1, \xi_2)=\f{1}{2\pi}\int_{\mathbb{R}^2}\widehat{\mathbf{G}}_{D}(v, v', k)\Upsilon(v)\Upsilon(v')e^{-i v\xi_1-i v'\xi_2}dvdv'.
\eeno
Then there is $\la_{\Delta}$ such that for all $k\in\mathbb{Z}$, 
\ben\label{eq: unbound Green}
|\widehat{\mathbf{G}}_{D}(k, \xi_1, \xi_2)|\leq C\min\left\{\f{e^{-\la_{\Delta}\langle\xi_1+\xi_2\rangle^{s}}}{1+k^2+\xi_1^2}, \f{e^{-\la_{\Delta}\langle\xi_1+\xi_2\rangle^{s}}}{1+k^2+\xi_2^2}\right\}.
\een
\end{proposition}
\begin{proof}
Notice that 
\ben\label{eq: change variable green est}
\begin{aligned}
&\int_{\mathbb{R}^2}\widehat{\mathbf{G}}_{D}(v, v', k)\Upsilon(v)\Upsilon(v')e^{-i v\xi_1-i v'\xi_2}dvdv'\\
&=\int_{\mathbb{R}^2}\widehat{\mathbf{G}}_{D}(v, v+w, k)\Upsilon(v)\Upsilon(v+w)e^{-i v(\xi_1+\xi_2)-i v'\xi_2}dvdw
\end{aligned}
\een
The upper bound \eqref{eq: unbound Green} follows from \eqref{eq: Green-1}, \eqref{eq: pa_vK}, \eqref{eq: pa_vvK} Lemma \ref{lem: boundary q/q}, Lemma \ref{lem: G-regularity-1}, Lemma \ref{lem: G-regularity-1-v near v'}, Lemma \ref{Lem: A1}, Lemma \ref{lem: Fourier_type1}, Lemma \ref{Rmk: fo-g-2} and Remark \ref{Rmk: improved kernel regularity}. 
\end{proof}

By using \eqref{eq: Q-Psi} and \eqref{eq: coefficients}, it is easy for us to obtain the following proposition which is equivalent to Proposition \ref{prop: T_1}. 
\begin{proposition}\label{eq: elliptic estimate}
Let $\Psi$ solve \eqref{eq: steam function-ST} with $\Om$ compactly supported. Then there is $\mathcal{G}_{D,1}(t, k, \xi, \eta)$ such that for $k\neq 0$
\beno
\widehat{(\Psi\Upsilon)}(t, k, \xi)=\int_{\mathbb{R}}\mathcal{G}_{D,1}(t, k, \xi, \eta)\widehat{\Om}(t, k, \eta)d\eta,
\eeno
with estimate
\beno
|\mathcal{G}_{D,1}(t, k, \xi, \eta)|\leq C_{ell}\min\left\{\f{e^{-\la_{\Delta}\langle \xi-\eta\rangle^s}}{1+k^2+(\xi-kt)^2}, \f{e^{-\la_{\Delta}\langle \xi-\eta\rangle^s}}{1+k^2+(\eta-kt)^2}\right\}.
\eeno
\end{proposition}

Notice that from $\Delta \psi=\om$, we have
\beq\label{eq: steam function-2}
\left\{\begin{aligned}
&\pa_{zz}\Psi+\widetilde{u'}(\pa_v-t\pa_z)\Big(\widetilde{u'}(\pa_v-t\pa_z)\Psi\Big)=\udl{\om}\\
&\Psi(t, z, u(0))=\Psi(t, z, u(1))=0.
\end{aligned}\right.
\eeq
If we set $\mathfrak{Q}=\mathcal{F}_1(\Psi)(t,k,v)e^{-itkv}$ and $h_1(v)=\widetilde{u'}(v)$, $h_2(v)=\widetilde{u'}(v)^{-1}$ and $\mathfrak{F}=\mathcal{F}_1[\udl{\om}]e^{-itkv}$, then we also obtain the Sturm-Liouville Equation \eqref{eq: SLequ}. We then have the following corollary:
\begin{corol}\label{corol: psi-om}
Let $\Psi$ solve \eqref{eq: steam function-2} with $\udl{\om}$ compactly supported. Then there is $\mathcal{G}_2(t, k, \xi, \eta)$ such that 
\beno
\widehat{(\Psi\Upsilon)}(t, k, \xi)=\int_{\mathbb{R}}\mathcal{G}_2(t, k, \xi, \eta)\widehat{\udl{\om}}(t, k, \eta)d\eta,
\eeno
with estimate
\beno
|\mathcal{G}_2(t, k, \xi, \eta)|\leq C_{ell}\min\left\{\f{e^{-\la_{\Delta}\langle \xi-\eta\rangle^s}}{1+k^2+(\xi-kt)^2}, \f{e^{-\la_{\Delta}\langle \xi-\eta\rangle^s}}{1+k^2+(\eta-kt)^2}\right\}.
\eeno
\end{corol}

\subsection{Neumann boundary conditions}\label{sec: SL-N}
We also study the general Sturm-Liouville Equation with Neumann boundary conditions
\ben\label{eq: Sturm-Liouville Equation-N}
\left\{\begin{aligned}
&\pa_v(h_1(v)\pa_v\mathfrak{R})-k^2h_2(v)\mathfrak{R}=\mathfrak{G}(v, k)\\
&\pa_v\mathfrak{R}(\al, k)=\pa_v\mathfrak{R}(\b, k)=0
\end{aligned}\right.
\een
where $h_1, h_2$ are positive functions and $\mathrm{supp}\, h_j' \subset [\al_1, \b_1]$ for $j=1,2$. We also assume that 
$h_1', h_2'\in \mathcal{G}^{M, s}_{ph}([\al, \b])$. 

By Lemma \ref{eq: mfq}, it is easy to get that
\ben\label{eq: Rep-R-G}
\begin{aligned}
\mathfrak{R}(v, k)&=\f{1}{h_1(v)}\f{\mfq_{\b}(v,k)}{W[\mfq_{\b}, \mfq_{\al}]}\int_{0}^v\mfq_{\al}(v',k)\mathfrak{G}(v',k)h_1(v')dv'\\
&\quad+\f{1}{h_1(v)}\f{\mfq_{\al}(v,k)}{W[\mfq_{\b}, \mfq_{\al}]}\int_v^1\mfq_{\b}(v',k)\mathfrak{G}(v',k)h_1(v')dv',
\end{aligned}
\een
with 
\beno
W[\mfq_{\b}, \mfq_{\al}](v)=-W[\mfq_{\al}, \mfq_{\b}](v)=-\f{\mfq_{\b}'(\al)h_1(\al)}{h_1(v)}=\f{\mfq_{\al}'(\b)h_1(\b)}{h_1(v)}>0. 
\eeno
We now introduce the Fourier kernel for the Sturm-Liouville type elliptic problem \eqref{eq: Sturm-Liouville Equation-N}. For any given fixed $k$, and positive Gevrey functions $h_1, h_2 \in \mathcal{G}^{M, s}_{ph}([\al, \b])$ with constant values near the boundary and any Gevrey cut-off function $\Upsilon$ with compact support $\mathrm{supp}\,\Upsilon\subset (\al, \b)$, we define a linear operator $\triangle_{N, k}^{-1}$ to be such that for a function $\mathfrak{G}$, 
\ben\label{eq:Def-inverlaplace-N}
\triangle_{N, k}^{-1}\mathfrak{G}=\Upsilon\mathfrak{R}
\een 
where $\mathfrak{R}$ solves
\beno
\left\{\begin{aligned}
&\pa_v(h_1(v)\pa_v\mathfrak{R})-k^2h_2(v)\mathfrak{R}=\mathfrak{G}(v, k)\Upsilon(v)\\
&\pa_v\mathfrak{R}(\al, k)=\pa_v\mathfrak{R}(\b, k)=0.
\end{aligned}\right.
\eeno
By \eqref{eq: Rep-R-G}, we have for $k\neq 0$
\begin{align*}
\triangle_{N, k}^{-1}\mathfrak{G}=\int_{0}^1\mathbf{G}_N(v, v', k)\Upsilon(v)\Upsilon(v')\mathfrak{G}(v', k)dv'
\end{align*}
with 
\beq\label{eq: green-N}
\begin{aligned}
\mathbf{G}_N(v, v', k)&=\f{1}{h_1(v)}\f{\mfq_{\b}(v,k)}{W[\mfq_{\b}, \mfq_{\al}]}\mfq_{\al}(v',k)h_1(v')\chi_{\mathbb{R}^+}(v-v')\\
&\quad+\f{1}{h_1(v)}\f{\mfq_{\al}(v,k)}{W[\mfq_{\b}, \mfq_{\al}]}\mfq_{\b}(v',k)h_1(v')\chi_{\mathbb{R}^+}(v'-v).
\end{aligned}
\eeq
Similarly to Proposition \ref{eq: elliptic estimate}, by using the representation formula of the Green function \eqref{eq: green-N} and applying Lemma \ref{lem: G-regularity-1} and Lemma \ref{Lem: A1}, we have the following proposition. 
\begin{proposition}\label{eq: elliptic estimate-N}
Let $\Pi$ solve \eqref{eq: pressure-lin} with $\Xi$ compactly supported. Then there is $\mathcal{G}_{N}(t, k, \xi, \eta)$ such that 
\beno
\widehat{(\Pi\Upsilon)}(t, k, \xi)=\int_{\mathbb{R}}\mathcal{G}_{N}(t, k, \xi, \eta)\widehat{\Xi}(t, k, \eta)d\eta,
\eeno
with estimate
\beno
|\mathcal{G}_{N}(t, k, \xi, \eta)|\leq C\min\left\{\f{e^{-\la_{\Delta}\langle \xi-\eta\rangle^s}}{1+k^2+(\xi-kt)^2}, \f{e^{-\la_{\Delta}\langle \xi-\eta\rangle^s}}{1+k^2+(\eta-kt)^2}\right\}.
\eeno
\end{proposition}

\section{Properties of multiplier $\rmA$}\label{sec: operator A}
In this section, we list some key properties of the multiplier $\rmA$ from \cite{BM2015}. 
\begin{lemma}[Lemma 3.3 in \cite{BM2015}]\label{Lem 3.3}
For $t\in {\rm{I}}_{k,\eta}$ and $t>2\sqrt{|\eta|}$, we have the following with $\tau=t-\f{\eta}{k}$
\begin{align*}
\f{\pa_tw_{\rmN\rmR}(t,\eta)}{w_{\rmN\rmR}(t,\eta)}\approx \f{1}{1+|\tau|}\approx \f{\pa_tw_{\rmR}(t,\eta)}{w_{\rmR}(t,\eta)}.
\end{align*}
\end{lemma}
\begin{lemma}[Lemma 3.4 in \cite{BM2015}]\label{Lem 3.4}
The following holds. 
\begin{enumerate}
\item For $t\geq 1$, and $k, l, \eta, \xi$ such that $\max\{2\sqrt{|\xi|},2\sqrt{|\eta|}\}<t<2\min\{|\xi|,|\eta|\}$, 
\beno
\f{\pa_tw_{k}(t,\eta)}{w_{k}(t,\eta)}\f{w_{l}(t,\xi)}{\pa_tw_{l}(t,\xi)}\lesssim \langle\eta-\xi\rangle.
\eeno
\item For all $t\geq 1$, and $k, l, \eta, \xi$, such that for some $\al\geq 1$, $\f{1}{\al}|\xi|\leq |\eta|\leq \al|\xi|$,
\beno
\sqrt{\f{\pa_tw_{l}(t,\xi)}{w_{l}(t,\xi)}}\lesssim_{\al}\left[\sqrt{\f{\pa_tw_{k}(t,\eta)}{w_{k}(t,\eta)}}+\f{|\eta|^{\f{s}{2}}}{\langle t\rangle^s}\right]\langle \eta-\xi\rangle.
\eeno
\end{enumerate}
\end{lemma}
\begin{lemma}[Lemma 3.5 in \cite{BM2015}]\label{Lem 3.5}
For all $t, \eta, \xi$, we have
\beno
\f{w_{\rmN\rmR}(t,\xi)}{w_{\rmN\rmR}(t,\eta)}\lesssim e^{\mu|\eta-\xi|^{\f12}}.
\eeno
\end{lemma}
\begin{lemma}[Lemma 3.6 in \cite{BM2015}]\label{Lem 3.6}
We have for $t\in {\rm{I}}_{k,\eta}\cap {\rm{I}}_{k,\xi}$ with $k\neq l$ that
\beno
\f{\rmJ_{k}(\eta)}{\rmJ_{l}(\xi)}\lesssim\f{|\eta|}{k^2+(\eta-kt)^2}e^{9\mu|k-l,\eta-\xi|^{\f12}}.
\eeno
we also rewrite it as 
\beno
\f{\rmJ_{k}(\eta)}{\rmJ_{l}(\xi)}\lesssim\f{|\eta|}{k^2}\sqrt{\f{\pa_tw_{k}(t,\eta)}{w_{k}(t,\eta)}}\sqrt{\f{\pa_tw_{l}(t,\xi)}{w_{l}(t,\xi)}}e^{20\mu|k-l,\eta-\xi|^{\f12}}.
\eeno
In other cases, we have improved estimates
\beno
\f{\rmJ_{k}(\eta)}{\rmJ_{l}(\xi)}\lesssim e^{10\mu|k-l,\eta-\xi|^{\f12}}.
\eeno
Moreover, if $t\in {\rm{I}}_{l,\xi}$, $t\notin {\rm{I}}_{k,\eta}$ and $\f{1}{\al}|\xi|\leq |\eta|\leq \al|\xi|$ for some $\al>0$, then
\ben\label{EQ 3.32}
\f{\rmJ_{k}(\eta)}{\rmJ_{l}(\xi)}\lesssim\f{|l|(l+|\xi-lt|)}{|\xi|}e^{11\mu|k-l,\eta-\xi|^{\f12}}.
\een
\end{lemma}
\begin{lemma}[Lemma 3.7 in \cite{BM2015}]\label{Lem 3.7}
Let $t\leq \f{1}{2}\min\{\sqrt{|\eta|},\sqrt{|\xi|}\}$. Then
\beno
\left|\f{\rmJ_{k}(\eta)}{\rmJ_{l}(\xi)}-1\right|\lesssim \f{\langle k-l, \eta-\xi\rangle}{\sqrt{|\xi|+|\eta|+|k|+|l|}}e^{11\mu|k-l,\eta-\xi|^{\f12}}. 
\eeno
\end{lemma}
\end{appendix}

\bibliographystyle{siam.bst} 
\bibliography{references.bib}

\end{CJK*}

\end{document}